%% file: source.tex
\documentclass{memo-l}

\usepackage{}

\usepackage{amsfonts}
\usepackage{amsmath,color}
\usepackage[rgb,dvipsnames]{xcolor}
\usepackage[makeroom]{cancel}
\usepackage{amssymb,enumerate}
\usepackage{graphicx,hyperref,cite}
\usepackage{dsfont}

\usepackage[english]{babel}
\usepackage[latin1]{inputenc}
\usepackage[T1]{fontenc} 
\usepackage{mathrsfs,amssymb,amsmath,amsthm,amscd,bbm,bm}
\usepackage{dsfont,enumitem,float}
\usepackage[left = 1.15in, right = 1.15 in, top =1.2in]{geometry}
\usepackage{hyperref}
\usepackage{color}
\usepackage[rgb,dvipsnames]{xcolor}
\usepackage{cite}
\usepackage[font=small]{caption}

\usepackage[europeanresistors, cuteinductors]{circuitikz} 

\newcommand{\beq}{\begin{equation}}
\newcommand{\eeq}{\end{equation}}
\newcommand{\bea}{\begin{eqnarray}}
\newcommand{\eea}{\end{eqnarray}}
\newcommand{\beas}{\begin{eqnarray*}}
\newcommand{\eeas}{\end{eqnarray*}}

\newtheorem{theorem}{Theorem}[chapter]

\newtheorem{assumption}[theorem]{Assumption}
\newtheorem{definition}[theorem]{Definition}
\newtheorem{proposition}[theorem]{Proposition}
\newtheorem{prop}[theorem]{Proposition}
\newtheorem{corollary}[theorem]{Corollary}
\newtheorem{lemma}[theorem]{Lemma}
\newtheorem{remark}[theorem]{Remark}
\newtheorem{example}[theorem]{Example}

\newtheorem{foo}[theorem]{Remarks}

\numberwithin{section}{chapter}
\numberwithin{equation}{chapter}

\newcommand{\brak}[1]{\left(#1\right)} 
\newcommand{\norm}[1]{{\left\lVert{#1}\right\rVert}}

\newcommand{\N}[1]{||#1||} 
\newcommand{\abs}[1]{\left|#1\right|} 

\newcommand{\pat}[1]{\textcolor{magenta!50!red}{#1}}
\newcommand{\rS}[1]{{{\color{OliveGreen}{#1}}}}

\def\vint{\mathop{\mathchoice%
          {\setbox0\hbox{$\displaystyle\intop$}\kern 0.22\wd0%
           \vcenter{\hrule width 0.6\wd0}\kern -0.82\wd0}%
          {\setbox0\hbox{$\textstyle\intop$}\kern 0.2\wd0%
           \vcenter{\hrule width 0.6\wd0}\kern -0.8\wd0}%
          {\setbox0\hbox{$\scriptstyle\intop$}\kern 0.2\wd0%
           \vcenter{\hrule width 0.6\wd0}\kern -0.8\wd0}%
          {\setbox0\hbox{$\scriptscriptstyle\intop$}\kern 0.2\wd0%
           \vcenter{\hrule width 0.6\wd0}\kern -0.8\wd0}}%
          \mathopen{}\int}

\newcommand{\R}{\mathbb R}

\newcommand{\Z}{\mathbb Z}

\newcommand{\pip}{\varphi}

\newcommand{\DF}{\mathcal{E}}

\newcommand{\Ecal}{\mathcal E}

\newcommand{\B}{\mathbf B}

\newcommand{\al}{\alpha}

\newcommand{\ve}{\varepsilon}

\newcommand{\f}{e^{\lambda f}}

\newcommand{\bfB}{\mathbf B}

\newcommand{\eps}{\varepsilon}

\DeclareMathOperator{\diam}{diam}

\makeindex

\begin{document}

\frontmatter

\title{BV functions and Besov spaces associated with Dirichlet spaces}


\author{Patricia Alonso-Ruiz}
\address{P.A-R.: Department of Mathematics, University of Connecticut, Storrs, CT 06269,  U.S.A.}
\curraddr{}
\email{patricia.alonso-ruiz@uconn.edu}
\thanks{P.A-R. was partly supported by the Feodor Lynen Fellowship, Alexander von Humboldt Foundation (Germany)}

\author{Fabrice Baudoin}
\address{F.B.: Department of Mathematics, University of Connecticut, Storrs, CT 06269,  U.S.A. }
\curraddr{}
\email{fabrice.baudoin@uconn.edu}
\thanks{F.B. was partly supported by the grant DMS~\#1660031 of the NSF (U.S.A.)}

\author{Li Chen}
\address{L.C.: Department of Mathematics, University of Connecticut, Storrs, CT 06269,  U.S.A.}
\curraddr{}
\email{li.4.chen@uconn.edu}
\thanks{}


\author{Luke Rogers}
\address{L.R.: Department of Mathematics, University of Connecticut, Storrs, CT 06269,  U.S.A.}
\curraddr{}
\email{luke.rogers@uconn.edu}
\thanks{L.R. was partly supported by the grant DMS~\#1659643  of the NSF (U.S.A.)}

\author{Nageswari Shanmugalingam}
\address{N.S.: Department of Mathematical Sciences, P.O. Box 210025, University of
Cincinnati, Cincinnati, OH 45221-0025, U.S.A.}
\curraddr{}
\email{shanmun@uc.edu}
\thanks{N.S. was partly supported by the grant DMS~\#1500440 of the NSF (U.S.A.)}

\author{Alexander Teplyaev}
\address{A.T.:  Department of Mathematics, University of Connecticut, Storrs, CT 06269,  U.S.A.}
\curraddr{}
\email{alexander.teplyaev@uconn.edu}
\thanks{A.T. was partly supported by the grant DMS~\#1613025 of the NSF (U.S.A.)}

\date{}




\begin{abstract}
We develop a theory of bounded variation functions and  Besov spaces  in  abstract Dirichlet spaces which unifies several known examples and applies to new situations, including fractals.
\end{abstract}

\maketitle

\tableofcontents

\include{introduction}

\mainmatter
\part{General theory of Besov spaces related to Dirichlet spaces}
\include{chapter1}

\include{chapter2}

\include{chapter3}

\part{Specific settings}
\include{chapter4}

\include{chapter5}

\include{chapter6}

\include{chapter7}

\appendix

\backmatter
\bibliographystyle{amsalpha}
 \bibliography{BV_Refs}

\printindex

\end{document}

%% file: introduction.tex
\chapter*{Introduction}

\section*{Foreword}

 While the theory developed here is not complete, it is the hope of the authors that the readers would find this monograph
to provide a reasonable self-contained material introducing the topic, with many open questions. Most of the results here are new, but those
that are not have been marked with references to the original papers or textbooks where they appeared.

\section*{Acknowledgment}

The authors would like to thank Naotaka Kajino for several stimulating discussions.

\section*{Bird's eyeview}

Before describing our results in details, it may be useful to indicate some of the main ideas.
Let $(X,\mu,\mathcal{E},\mathcal{F}=\mathbf{dom}(\mathcal{E}))$ be a symmetric Dirichlet space, 
that is, $X$ is a topological
measure space equipped with the Radon measure $\mu$, $\mathcal{E}$ a closed Markovian bilinear form on
$L^2(X,\mu)$, and $\mathcal{F}$ the collection of all functions $u\in L^2(X)$ with $\mathcal{E}(u,u)$ finite. The book~\cite{FOT} contains
a good introduction to the theory of Dirichlet forms.
Let $\{P_{t}\}_{t\in[0,\infty)}$ denote the Markovian semigroup associated with
$(X,\mu,\mathcal{E},\mathcal{F})$. Let $p \ge 1$ and $\alpha \ge 0$. For $f \in L^p(X,\mu)$, we define the Besov type seminorm:
\[
\| f \|_{p,\alpha}= \sup_{t >0} t^{-\alpha} \left( \int_X P_t (|f-f(y)|^p)(y) d\mu(y) \right)^{1/p},
\]
and introduce the Besov space
\[
\mathbf{B}^{p,\alpha}(X)=\{ f \in L^p(X,\mu)\, :\,  \| f \|_{p,\alpha} <+\infty \}.
\]
Note that if $0<\beta<\alpha$, then $\mathbf{B}^{p,\alpha}(X)\subset\mathbf{B}^{p,\beta}(X)$.
The critical parameter
\[
\alpha_p^*(X)=\inf \{ \alpha >0\, :\,  \mathbf{B}^{p,\alpha}(X) \text{ is trivial} \}
\]
will be of special interest in our study. By trivial, we mean here that for  $\alpha >\alpha_p^*(X)$, any function $f \in \mathbf{B}^{p,\alpha}(X)$ is constant. For instance, we will see that for strongly local forms with a nice intrinsic metric in the sense of \cite{St-I}, one typically has $\alpha_p^*(X)=\frac{1}{2}$, while for   strongly local forms  with no intrinsic distance but with sub-Gaussian heat kernel estimates one typically has $\alpha_p^*(X) > \frac{1}{2}$ when $1 \le p <2$ and $\alpha_p^*(X) < \frac{1}{2}$, when $p>2$. When $p=2$, one always has $\alpha^*_2(X)=\frac{1}{2}$.

While all the spaces $\mathbf{B}^{p,\alpha}(X)$ are of interest and shall be  discussed in the monograph, we shall pay a special attention to the case $p=1$.  In particular, we will see that in many situations it may be possible and natural to define a space of bounded variation functions on $X$ as the space $\mathbf{B}^{1,\alpha_1^*(X)}(X) $. For instance, on complete Riemannian manifolds 
we have that
$\alpha_1^*(X)=\frac{1}{2}$ and  $\mathbf{B}^{1,\alpha_1^*(X)}(X)$ is the usual space of bounded variation functions. On
the Sierpinski gaskets 
$\alpha_1^*(X)=\frac{d_H}{d_W}$, where $d_H$ is the Hausdorff dimension of $X$ and $d_W$ its walk dimension. Finding the exact value of $\alpha_1^*(X)$ is  an open question for the Sierpinski carpets. 

Under a weak curvature condition on the Dirichlet space (weak Bakry-\'Emery type estimate \eqref{BE intro prelim}), a main observation of our work is that for functions in the space $\mathbf{B}^{1,\alpha_1^*(X)}(X) $, one has the following  estimate valid for all $t$'s:
\begin{equation}\label{main intro}
\| P_t f -f \|_{L^1(X,\mu)} \le C t^{\alpha_1^*(X)} \limsup_{s \to 0}  s^{-\alpha_1^*(X)}  \int_X P_s (|f-f(y)|)(y) d\mu(y) .
\end{equation}
Therefore, by using the approach of Michel Ledoux  to heat flow in the study of isoperimetry~\cite{Ledoux}, one deduces that for $f \in \mathbf{B}^{1,\alpha_1^*(X)}(X) $, the quantity $ \limsup_{s \to 0}  s^{-\alpha_1^*(X)}  \int_X P_s (|f-f(y)|)(y) d\mu(y)$ plays the role of the total variation of $f$. 
For instance, in the case  where $X$ is a complete Riemannian manifold,  one has 
\[
\limsup_{s \to 0}  s^{-1/2}  \int_X P_s (|f-f(y)|)(y) d\mu(y) \simeq \int_X | \nabla f |(y) d\mu(y),
\]
which coincides  with the usual notion of variation. However, the quantity $$\limsup_{s \to 0}  s^{-\alpha_1^*(X)}  \int_X P_s (|f-f(y)|)(y) d\mu(y)$$ is well-defined and possible to estimate in a class of Dirichlet spaces well beyond Riemannian manifolds, like for instance strongly local metric spaces satisfying a doubling condition and supporting a 2-Poincar\'e inequality,  fractal spaces, non-local Dirichlet spaces and even some infinite dimensional examples.

With \eqref{main intro} in hand, and adapting ideas from  \cite{BCLS}, we are then  able to 
deduce several Sobolev and isoperimetric inequalities with sharp exponents holding in a very general framework.

\

We summarize some of the classes of spaces, studied in this work, in the table below. The parameter $d_H$ indicates the Hausdorff dimension of the space and $d_W$ its walk dimension. The parameter $\kappa$ is the H\"older regularity exponent of the heat semigroup in the assumed weak Bakry-\'Emery estimate:
\begin{align}\label{BE intro prelim}
|P_tf (x)-P_tf(y)|\le  C \frac{ d(x,y)^\kappa}{t^{\kappa/d_W}}  \| f \|_{L^\infty(X,\mu)}.
\end{align}

\begin{table}[H]
\begin{tabular}{ |p{6cm}||p{1.2cm}|p{7.5cm}|   }
 \hline
 \multicolumn{3}{|c|}{Examples} \\
 \hline
 Dirichlet space $X$ &   $\alpha_1^*(X) $  & Characterization of $\mathbf{B}^{1,\alpha_1^*}(X)$ \\
 \hline
 1) Strongly local Dirichlet metric space with doubling and 2-Poincar\'e   & $\frac{1}{2}$ &  $BV(X)$   \\
2)  Local Dirichlet metric space with sub-Gaussian heat kernel estimates&    $1-\frac{\kappa}{d_W}$ & $ \sup\limits_{r > 0}   \iint_{ d(x,y)<r }\frac{|f(x)-f(y)|}{r^{d_H+d_W-\kappa}} \,d\mu(x)\,d\mu(y)<+\infty $  \\
3)  a) Non-local Dirichlet metric space with  heat kernel estimates and $d_W \le 1$ & 1 & $ \iint \frac{|f(x)-f(y)|}{d(x,y)^{d_H+d_W}} d\mu (x) d\mu(y) <+\infty $\\
\quad  b) Non-local Dirichlet metric space with  heat kernel estimates and $d_W >1$    &$\frac{1}{d_W}$ &  $ \sup\limits_{r > 0}   \iint_{ d(x,y)<r }\frac{|f(x)-f(y)|}{r^{d_H+1}} \,d\mu(x)\,d\mu(y)<+\infty $ \\
  \hline
\end{tabular}
\caption{Summary of examples.}
\label{Table}
\end{table}

We point out that the case~(2), in the above table, of a local Dirichlet metric space with sub-Gaussian heat kernel estimates, we actually only prove that $1-\frac{\kappa}{d_W}$ is an upper bound of the critical exponent $\alpha_1^*(X)$. We make the conjecture that $\alpha_1^*(X)=1-\frac{\kappa}{d_W}$ should hold in some generality, see the discussion after Remark \ref{crticical gasket}. The conjecture is proved in some concrete examples like the Sierpinski gasket, although the question is open for the Sierpinski carpet.

\newpage

\section*{Summary of results and structure of the monograph}

The monograph is divided into two parts. The first part develops the theory of Besov spaces and related Sobolev embeddings in the general framework of abstract Dirichlet spaces. The second part particularizes to four big classes of Dirichlet spaces: Local regular Dirichlet spaces with absolutely continuous energy measures, doubling and 2-Poincar\'e (like Riemannian manifolds with non negative Ricci curvature or more generally the $\mathbf{RCD}(0,+\infty)$ metric measure spaces), Local regular Dirichlet spaces with sub-Gaussian heat kernel estimates (like fractal spaces), Non-local regular Dirichlet spaces (like Dirichlet forms associated with jump processes), and Local quasi-regular and infinite dimensional Dirichlet spaces (like the Wiener space). We now further describe these constituent parts
of the monograph.


\subsection*{Part 1}

\subsection*{Chapter 1}
 
In Chapter 1, we lay the foundations. Let $(X,\mu,\mathcal{E},\mathcal{F}=\mathbf{dom}(\mathcal{E}))$ be a symmetric Dirichlet space. Let $p \ge 1$ and $\alpha \ge 0$. For $f \in L^p(X,\mu)$, we define the Besov type seminorm:
\[
\| f \|_{p,\alpha}= \sup_{t >0} t^{-\alpha} \left( \int_X P_t (|f-f(y)|^p)(y) d\mu(y) \right)^{1/p},
\]
where $P_t$ is the Markovian semigroup associated with $\mathcal{E}$. We then introduce the space
\[
\mathbf{B}^{p,\alpha}(X)=\{ f \in L^p(X,\mu), \| f \|_{p,\alpha} <+\infty \},
\]
and prove that it is a Banach space, which is reflexive for $p >1$. The value $\alpha=\frac{1}{2}$ plays a special role throughout this work since $\mathbf{B}^{2,1/2}(X)$ is always equal to the domain $\mathcal{F}$ of the Dirichlet form.
A first class of examples is studied: We prove that if $\mathcal{E}$ arises from a smooth H\"ormander's type diffusion operator on a manifold, then  for every $p \ge 1$, $\mathbf{B}^{p,1/2}(X)$  contains the space of smooth and compactly supported functions with the continuous embedding
\[
\left( \int_X \Gamma ( f,f) (x)^{p/2} d\mu(x)\right)^{1/p} \le 2\left( \frac{ \Gamma \left( \frac{1+p}{2} \right) }{\sqrt
{\pi}}\right)^{1/p}  \| f \|_{p,1/2},
\]
 where $\Gamma$ is the \textit{carr\'e du champ} operator of the Dirichlet form $(\mathcal{E},\mathcal{F})$ (however,
 $\Gamma((1+p)/2)$ is the standard Gamma function). 
 On the other hand,  in Corollary \ref{singular Kusuoka}, we will prove that if the Kusuoka measure is singular with respect to $\mu$, like in the Sierpinski gasket, then $\mathbf{B}^{p,1/2}(X)$ contains only constant functions when $p>2$. This dichotomy already  shows that the properties of the spaces $\mathbf{B}^{p,\alpha}(X)$ are closely related to the energy measures $\Gamma(f,g)$ being absolutely continuous or not. This dichotomy will be investigated in great details in the second part of the monograph. 
 
 An important, very general,  theorem proved in Chapter 1 is the following:
 
 \begin{theorem}\label{continuity Besov chapter 1 intro}
Let $1<p\le 2$. There exists a constant $C_p>0$ such that for every $f \in L^p(X,\mu)$ and $t \ge 0$
\[
\| P_t f \|_{p,1/2} \le \frac{C_p}{t^{1/2}} \| f \|_{L^p(X,\mu)}.
\]
In particular, when $1<p\le 2$,  $P_t: L^p(X,\mu) \to \B^{p,1/2}(X)$ is bounded for $t>0$.
\end{theorem}

The proof of the theorem will follow from some nice ideas originally due to Nick Dungey \cite{Dungey} and then developed in \cite{LiChen,LiChen2}.  As a corollary, we  obtain:

\begin{prop}
Let $1<p \le 2$. Let $L$ be the infinitesimal generator of $\mathcal{E}$ and $\mathcal{L}_p$ be the domain of $L$ in $L^p(X,\mu)$,
(that is, $\mathcal{L}_p$ is the domain of the generator of the strongly continuous semigroup $P_t:L^p(X,\mu) \to L^p(X,\mu)$) .  Then
\[
\mathcal{L}_p \subset \B^{p,1/2}(X)
\]
and for every $f \in \mathcal{L}_p$,
\begin{equation}\label{eq:multi2}
\|f\|^2_{p,1/2} \le C_p \norm{ Lf}_{L^p(X,\mu)} \| f\|_{L^p(X,\mu)}.
\end{equation}
\end{prop}

We note that due to Corollary \ref{singular Kusuoka},  one cannot hope for an extension of 
Theorem~\ref{continuity Besov chapter 1 intro} to the range $p>2$ in the general case. 

We define then the $L^p$ Besov  critical exponent of $(X,\mu,\mathcal{E},\mathcal{F})$ as
\[
\alpha^*_p(X)=\inf \{ \alpha >0 , \mathbf{B}^{p,\alpha}(X) \text{ is trivial} \}.
\]

Let us recall that the Dirichlet form $\mathcal{E}$ is said to be regular if $\mathcal{F} \cap C_0(X)$ is dense in $\mathcal{F}$ for the domain norm and uniformly dense in $C_0(X)$ (see Chapter 1 in \cite{FOT}). The form $\mathcal{E}$ is said to be irreducible if $f \in \mathcal{F}$ with $\mathcal{E}(f,f)=0$ implies that $f$ is constant.
We prove  the following result:

\begin{proposition}\label{Besov critical exponents intro}
Assume that $\mathcal{E}$ is regular and irreducible. Then
\begin{enumerate}
\item $ \alpha^*_2(X)=\frac{1}{2}$;
\item $p \to \alpha^*_p(X)$ is non increasing;
\item For $p \ge 2$, $\alpha^*_p (X) \le \frac{1}{2}$;
\item For $1 \le p \le 2$, $ \frac{1}{2} \le \alpha^*_p(X) \le \frac{1}{p}$.
\end{enumerate}
\end{proposition}

In particular (and this is of course not surprising), $\mathbf{B}^{p,\alpha}(X)$ is always trivial if $\alpha >1$. Note from 
Corollary~\ref{singular Kusuoka} that if for every non-constant $f \in \mathcal{F}$ the energy measure $\nu_f$ is singular with respect to $\mu$, then $\alpha^*_p (X) < \frac{1}{2}$ for $p>2$. We recall that for $f \in \mathcal{F}$, its energy measure $\nu_f$ is defined by 
\[
2\DF(fg,f)-\DF(f^2,g)= \int_X  2g\,d\nu_f.
 \]
 \subsection*{Chapter 2}
 
 Chapter 2 is devoted to the study of Sobolev embeddings associated to our Besov spaces, in the presence of ultracontractive estimates for the semigroup. It will follow ideas  from \cite{BCLS} (see also  \cite{saloff2002}) and the key lemma is the following simple, but powerful, observation that functions in the Besov spaces satisfy a pseudo-Poincar\'e inequality:
 
 \begin{lemma}\label{pseudo-Poincare intro }
Let $ p \ge 1$ and $\alpha >0$. For every $f \in \mathbf{B}^{p,\alpha} (X)$, and $t \ge 0$,
\[
\| P_t f -f \|_{L^p(X,\mu)} \le t^\alpha \| f \|_{p,\alpha}.
\]
\end{lemma}

The importance of such pseudo-Poincar\'e inequalities in Sobolev embeddings was explicitly noted in \cite{saloff2002} and \cite{Ledoux2003}. In particular,  the technique used by M. Ledoux in  \cite{Ledoux2003} to obtain improved Gagliardo-Nirenberg type inequalities may likely be adapted to our general framework.  However, to be concise,  we restricted ourselves to the  Sobolev embeddings in the presence of ultracontractive estimates.
 
 A general result of the work is the following weak-type Sobolev inequality:

\begin{theorem}\label{polintro}
Assume that $\{P_{t}\}_{t\in(0,\infty)}$ 
admits a heat kernel $p_{t}(x,y)$ satisfying, for some
$C>0$ and $\beta >0$,
\begin{equation*}
p_{t}(x,y)\leq C t^{-\beta}
\end{equation*}
for $\mu\times\mu$-a.e.\ $(x,y)\in X\times X$ for each $t\in\bigl(0,+\infty \bigr)$. 
Let $0<\alpha <\beta $ and $1 \le p < \frac{\beta}{\alpha} $. Then there exists a 
constant $C_{p,\alpha} >0$ such that for every $f \in \mathbf{B}^{p,\alpha}(X) $,
\[
\sup_{s \ge 0}\ s\ \mu \left( \{ x \in X, | f(x) | \ge s \} \right)^{\frac{1}{q}} \le C_{p,\alpha} \| f \|_{p,\alpha},
\]
where $q=\frac{p\beta}{ \beta -p \alpha}$.
\end{theorem}

One may consider the above to be a very weak version of the Mazya capacitary type inequality.
The case $p=1$ is particularly interesting in the previous theorem and yields an isoperimetric type inequality.

\begin{theorem}\label{iso intro}
Assume that $\{P_{t}\}_{t\in(0,\infty)}$ 
admits a heat kernel $p_{t}(x,y)$ satisfying, for some
$C>0$ and $\beta >0$,
\begin{equation*}
p_{t}(x,y)\leq C t^{-\beta}
\end{equation*}
for $\mu\times\mu$-a.e.\ $(x,y)\in X\times X$ for each $t\in\bigl(0,+\infty \bigr)$. 
Let $0 <\alpha < \beta$. There exists a constant $C_{\emph{iso}} >0$, such that for every
subset  $E\subset X$ with $1_E \in \mathbf{B}^{p,\alpha}(X) $
\[
\mu(E)^{\frac{\beta-\alpha}{\beta}} \le C_{\emph{iso}} \| 1_E \|_{1,\alpha}.
\]
\end{theorem}

We will see that these two theorems are interesting due to their level of generality, since we are able to characterize $\mathbf{B}^{p,\alpha}(X) $ in many situations and will therefore be able to prove Sobolev inequalities with sharp exponents.

%
\subsection*{Chapter 3} 
Chapter 3 studies the analogue of the Sobolev embeddings of the Besov spaces for Dirichlet forms whose semigroups are not ultracontractive but satisfy supercontractive or  hypercontractive type estimates. 

In Section 3.1, we assume that $\mathcal{E}$ satisfies a Poincar\'e inequality.  Under this assumption only,  by following ideas of M. Ledoux \cite{Le}  we are then able to prove an analogue of Buser's inequality, and thus linear isoperimetry:

\begin{theorem}
Let  $\alpha \in (0,1]$.  We define the $\alpha$-Cheeger's constant  of $X$ by
\[
h_\alpha=\inf \frac{\| \mathbf 1_E \|_{1,\alpha}}{\mu(E)},
\]
where the infimum runs over all measurable sets $E$ such that $\mu(E)\le \frac{1}{2}$ and $\mathbf 1_E \in \mathbf{B}^{1,\alpha}(X)$. Then,
\[
h_\alpha \ge (1-e^{-1}) \lambda^\alpha_1,
\]
where $\lambda_1$ is the spectral gap of $\mathcal{E}$.
\end{theorem}

In Section 3.2, we prove that if $\mathcal{E}$ satisfies a log-Sobolev inequality, then  we can  improve Buser's inequality, and prove that $X$ supports a Gaussian isoperimetric inequality, that is,
\begin{theorem}
 There exists a constant $C>0$ such that for every set $E \subset X$ with $1_E \in \B^{1,1/2}(X)$, one has:
\[
\mu(E)\sqrt{-\ln \mu(E)} \le C \| 1_E \|_{1,1/2},
\]
where the constant $C$ explicitly depends on the log-Sobolev constant of $\mathcal{E}$.
\end{theorem}
The main class of examples to which  the previous theorem apply  are infinite-dimensional spaces like the ones studied in Chapter 7.

\subsection*{Part 2}

\subsection*{Chapter 4  }

The Chapter \ref{Sec:Metric-Curvature}  will be devoted to strongly local Dirichlet spaces with absolutely continuous energy measures 
 $ \Gamma(u,u) $, also known as carr\'e du champ, as extensively studied in the literature (see for instance the foundational articles by K.T. Sturm ~\cite{St-I, St-II, St-III}). The important assumption that is essential in this chapter is that the  intrinsic distance  associated to $ \Gamma(u,u) $ defines the
 topology of $X$. Such Dirichlet spaces are called strictly local and are extensively studied in the literature (see e.g. \cite{bendikov-saloff-coste,Koskela-Zhou,Koskela-Shanmugalingam-Zhou,lenz2009allegretto,lenz2011generalized,Koskela-Shanmugalingam-Tyson04} 
 and references therein).
 
 This chapter has multiple objectives. A first objective is to develop the notion of bounded variation function in this 
 general framework.  We achieve this task in Sections~4.1 and~4.2.  Our basic definition of bounded variation 
 function in a strictly local Dirichlet space is the following:
\begin{definition}
We set $u\in L^1(X,\mu)$ to be in $BV(X)$ if there is a sequence of locally Lipschitz functions $u_k\in L^1(X,\mu)$
such that $u_k\to u$ in $L^1(X,\mu)$ and 
\[
\liminf_{k\to\infty}\int_X|\nabla u_k|\, d\mu<\infty,
\]
where $|\nabla u|$ is the square-root of the Radon-Nikodym derivative $d\Gamma(u,u)/d\mu$. For $u\in BV(X)$ and open sets $U\subset X$, we set
\[
\Vert Du\Vert(U)=\inf_{u_k\in\mathcal{C}(U), u_k\to u\text{ in }L^1(U)}\liminf_{k\to\infty} \int_U|\nabla u_k|\, d\mu,
\]
and then for sets $A\subset X$ we define 
\[
\Vert Du\Vert(A)=\inf\{\Vert Du\Vert(O)\, :\, A\subset O\text{ and }O\text{ is open in }X\}.
\]
\end{definition}

We then have the following result whose proof  is based on that of~\cite{Mr}.

\begin{theorem}\label{thm:outer-measure intro}
If $X$ is  complete and
$u \in BV(X)$, then $\Vert Du \Vert$ is a Radon outer measure on $X$.
\end{theorem}

A standard major tool in the study of bounded variation functions is the co-area formula. Such formula is established in  Lemma \ref{lem:Co-area}.

 It should be emphasized that while the notion of BV function in some strictly local metric settings has already been studied in the literature (see \cite{MMS} and the references therein), much of the current
theory on functions of bounded variation in the metric setting requires a
$1$-Poincar\'e inequality. In this chapter we will \emph{not} require the support
of $1$-Poincar\'e inequality, only the weaker $2$-Poincar\'e inequality, but in
some of the analysis we will need an additional requirement called in this monograph the weak
Bakry-\'Emery curvature condition: namely that whenever $u\in \mathcal{F}\cap L^\infty(X,\mu)$,
\begin{equation}\label{eq:weak-BE intro}
\Vert |\nabla P_tu| \Vert_{L^\infty(X)}^2\le \frac{C}{t}\Vert u\Vert_{L^\infty(X)}^2.
\end{equation}

 The weak Bakry-\'Emery curvature condition is satisfied in the following examples:
\begin{itemize}
\item Complete Riemannian manifolds with non-negative Ricci curvature and more generally, the $RCD(0,+\infty)$ spaces (\cite{CJKS, Jiang15}).
\item Carnot groups (see \cite{BB2})
\item Complete sub-Riemannian manifolds with generalized non-negative Ricci curvature (see \cite{BB,BK14})
\item Metric graphs with finite number of edges (see \cite{BK})
\end{itemize}
Several statements equivalent to the weak Bakry-\'Emery curvature condition are given in Theorem~1.2 of the recent work \cite{CJKS}.  There are some metric measure spaces equipped with a doubling measure supporting a $2$-Poincar\'e inequality but without the above Bakry-\'Emery condition, see for example~\cite{KRS}.  For instance, it should be noted, that in the setting of complete sub-Riemannian manifolds with generalized non-negative Ricci curvature in the sense of \cite{BG},  while the weak Bakry-\'Emery curvature condition is known to be satisfied (see \cite{BB,BK14}), the 1-Poincar\'e inequality has not been proven yet (though the 2-Poincar\'e is known, see \cite{BBG}).

A major result in Chapter 4, is then the following theorem that  makes the connection with Part 1 of the monograph:

\begin{theorem}\label{thm:W=BV intro}
Assume that the weak Bakry-\'Emery curvature condition \eqref{eq:weak-BE intro} is satisfied. Then, $\mathbf{B}^{1,1/2}(X)=BV(X)$ with comparable seminorms. Moreover, there exist constants $c,C>0$  such that for every $u \in BV(X)$
\[
c \limsup_{s \to 0}  s^{-1/2}  \int_X P_s (|u-u(y)|)(y) d\mu(y)  \le \Vert Du\Vert(X) \le C\liminf_{s \to 0}  s^{-1/2}  \int_X P_s (|u-u(y)|)(y) d\mu(y) .
\]
\end{theorem}

We note that by combining Theorems \ref{thm:W=BV intro} and \ref{metric intro} we deduce for instance the following corollary:

\begin{corollary}
 Assume that the weak Bakry-\'Emery curvature condition \eqref{eq:weak-BE intro} is satisfied. If the volume growth condition $\mu(B(x,r)) \ge C_1 r^Q$ is satisfied, then there exists a constant $C_2 >0$ such that for every $f \in BV(X)$,
\[
\| f \|_{L^q(X,\mu)} \le C_2 \| Df \|(X)
\]
where $q=\frac{Q}{ Q-1}$.

\end{corollary}

A second objective in the chapter is to study all of the Besov spaces $\B^{p,\alpha}(X)$, $p \ge 1$. We will first prove the following result that gives a metric characterization of our Besov spaces:

\begin{theorem}\label{metric intro}
For $1\le p<\infty$ and $0\le \alpha<\infty$ we have 
\[
\mathbf{B}^{p,\alpha/2}(X)= \left\{ u\in L^p(X,\mu)\, :\, \sup_{t>0}
\left(\int_X\int_{B(x,t)}\frac{|u(y)-u(x)|^p}{t^{\alpha p}\mu(B(x,t))}\, d\mu(y)\, d\mu(x)
\right)^{1/p} <+\infty \right\}
\]
with comparable seminorms.
\end{theorem}

We will then be interested in comparing the Besov seminorm $\| \cdot \|_{p,1/2}$ to Sobolev seminorms and our  result is the following:

\begin{theorem}\label{thm:BesovLB intro}

\

\begin{itemize}
\item Let $p >1$. There exists a constant $C>0$ such that for every $ u \in \B^{p,1/2}(X)\cap \mathcal{F}$,
\[
\| |\nabla u| \|_{L^p(X,\mu)} \le C \| u \|_{p,1/2}
\]
\item Assume furthermore that the strong Bakry-\'Emery estimate is satisfied: $| \nabla P_t u | \le C P_t | \nabla u |$. Then,  for every $p >1$, there exists a constant $C>0$ such that for every  $u \in L^p(X,\mu)\cap \mathcal{F}$ with $|\nabla u| \in L^p(X,\mu)$,
\[
 \| u \|_{p,1/2} \le C \| |\nabla u| \|_{L^p(X,\mu)} .
\]
\end{itemize}
\end{theorem}

This theorem will allow us to prove an analogue of Theorem \ref{continuity Besov chapter 1 intro} for the range $p>2$ and to study the  $L^p$ Besov  critical exponent 
\[
\alpha^*_p(X)=\inf \{ \alpha >0 \, :\, \mathbf{B}^{p,\alpha}(X) \text{ is trivial} \}.
\]
More precisely, we prove:

\begin{theorem}
Assume  that the strong Bakry-\'Emery estimate is satisfied. Let $p >1$. There exists a constant $C_p>0$ such that for every $f \in L^p(X,\mu)$ and $t \ge 0$
\[
\| P_t f \|_{p,1/2} \le \frac{C_p}{t^{1/2}} \| f \|_{L^p(X,\mu)}.
\]
As a consequence, for every $p \ge 1$, $\alpha^*_p(X)=\frac{1}{2}$ and $\B^{p,1/2}(X)$ is dense in $L^p(X,\mu)$.
\end{theorem}

\subsection*{Chapter 5}In Chapter 
\ref{LDsGHKU} we apply our analysis  to strongly local Dirichlet spaces 
which do not have absolutely continuous energy measures 
$ \Gamma(u,u) $. Thus these spaces are not strictly local. The main class of examples we are interested in are fractal spaces
where typically  
$ \Gamma(u,u) $ and $\mu$ are mutually singular. This fact was  first observed by Kusuoka in \cite{Kus89} 
(see also \cite{BenBassatStrichartzTeplyaev,Hino1,Hino2,Hino3}). 
The general framework of analysis and probability on fractals can be found in  \cite{Ba98,Gri,KigB,Kig:RFQS,Str03,StrB}.

In a first general result, we recover results from \cite{P-P10} and obtain the following characterization of the spaces $\B^{p,\alpha}(X)$.

\begin{theorem}[\cite{P-P10}]\label{sub gaussian intro}
Let $(X,\mu,\mathcal{E},\mathcal{F})$ be a 
 symmetric   Dirichlet space and let $d$ be a metric on $X$ compatible with
the topology of $X$. We assume that $(X,d)$ is Ahlfors $d_H$-regular and that  $\{P_{t}\}_{t\in(0,\infty)}$
admits a heat kernel $p_{t}(x,y)$ satisfying, for some
$c_{3},c_{4}, c_5, c_6 \in(0,\infty)$ and $d_{W}\in(1,\infty)$,
\begin{equation*}
c_{5}t^{-d_{H}/d_{W}}\exp\biggl(-c_{6}\Bigl(\frac{d(x,y)^{d_{W}}}{t}\Bigr)^{\frac{1}{d_{W}-1}}\biggr) \le p_{t}(x,y)\leq c_{3}t^{-d_{H}/d_{W}}\exp\biggl(-c_{4}\Bigl(\frac{d(x,y)^{d_{W}}}{t}\Bigr)^{\frac{1}{d_{W}-1}}\biggr)
\end{equation*}
for $\mu\times\mu$-a.e.\ $(x,y)\in X\times X$ for each $t\in\bigl(0,+\infty\bigr)$. Let $p \ge 1$ and $\alpha \ge 0$. We have 
\[
 \mathbf{B}^{p,\frac{\alpha}{d_W}}(X)=\left\{f\in L^{p}(X,\mu)\, :\,  \sup_{r>0} \frac{1}{r^{\alpha+d_{H}/p}}
 \biggl(\iint_{\Delta_r} 
 |f(x)-f(y)|^{p}\,d\mu(x)\,d\mu(y)\biggr)^{1/p}<\infty\right\}
\]
with comparable norms.
\end{theorem}

 In the above and the sequel, for $r>0$ the set $\Delta_r$ denotes the collection of all $(x,y)\in X\times X$ for which
$d(x,y)<r$.

We note that by combining Theorems \ref{iso intro} and \ref{sub gaussian intro} one immediately  obtains:

\begin{corollary}
Let $X$ be an Ahlfors $d_H$-regular space that satisfies sub-Gaussian  heat kernel estimates  as in Theorem \ref{sub gaussian intro}. Then, one has the following weak type Besov space embedding. Let $0<\delta < d_H $. Let $1 \le p < \frac{d_H}{\delta} $. There exists a constant $C_{p,\delta} >0$ such that for every $f \in \mathbf{B}^{p,\delta/d_W}(X) $,
\[
\sup_{s \ge 0} s \mu \left( \{ x \in X, | f(x) | \ge s \} \right)^{\frac{1}{q}} \le C_{p,\delta} \sup_{r>0} \frac{1}{r^{\delta+d_{H}/p}}
\biggl(\iint_{\Delta_r} 
|f(x)-f(y)|^{p}\,d\mu(x)\,d\mu(y)\biggr)^{1/p}
\]
where $q=\frac{p d_H}{ d_H -p \delta}$. Furthermore, for every $0<\delta <d_H $, there exists a constant $C_{\emph{iso},\delta}$ such that for every measurable $E \subset X$, $\mu(E) <+\infty$,
\begin{align}\label{isoperimetric intro}
\mu(E)^{\frac{d_H-\delta}{d_H}} \le C_{\emph{iso},\delta} \sup_{r>0} \frac{1}{r^{\delta+d_{H}}} (\mu \otimes \mu) \left\{ (x,y) \in E \times E^c\, :\, d(x,y)  < r\right\}.
\end{align}
\end{corollary}

 The number $\delta$ in the previous corollary plays the role of the upper codimension of the boundary of $E$. In Section \ref{Hsefcb}, we will construct several explicit examples of sets with fractal boundaries for which $\delta \neq 1$ and
\[
 \sup_{r>0} \frac{1}{r^{\delta+d_{H}}} (\mu \otimes \mu) \left\{ (x,y) \in E \times E^c\, :\,  d(x,y) < r\right\} <+\infty
\]

We then discuss in detail existence of sets for which their indicator functions are in  $\mathbf{B}^{1,\alpha}(X)$, and the related question of 
the density of $\mathbf{B}^{1,\alpha}(X)$ in $L^1(X,\mu)$.  Related co-area type formulas are studied.


We will then introduce the analogue of  the weak Bakry-\'Emery estimate \ref{eq:weak-BE intro} in this framework of Dirichlet spaces with sub-Gaussian heat kernel estimates. We will see that a convenient analogue to work with is  the following:
\begin{align}\label{BE intro prelim2}
|P_tf (x)-P_tf(y)|\le  C \frac{ d(x,y)^\kappa}{t^{\kappa/d_W}}  \| f \|_{L^\infty(X,\mu)}, \quad f \in L^\infty(X,\mu),\quad t \ge 0 .
\end{align}
The parameter $\kappa$ is the H\"older regularity exponent.
We will show that the weak Bakry-\'Emery condition is satisfied with $\kappa=1-\frac{d_H}{d_W}$ in many examples, like the infinite Sierpinski gaskets or infinite Sierpinski carpets. An important consequence of the weak Bakry-\'Emery estimate is the continuity of the heat semigroup in the Besov spaces with $p>2$.

\begin{theorem}\label{P:PtinBesovp4 intro}
Let $X$ be an Ahlfors $d_H$-regular space that satisfies sub-Gaussian  heat kernel estimates and $BE(\kappa)$ with $0<\kappa\leq \frac{d_W}{2}$. Then, for any $p\geq 2$, there exists a constant $C>0$ such that for every $t>0$ and $f\in L^p(X,\mu)$
\[
\| P_t f \|_{p,\left(1-\frac{2}{p}\right)\frac{\kappa}{d_W}+\frac{1}{p}}  \le  \frac{C}{t^{ \left(1-\frac{2}{p}\right)  \frac{\kappa }{d_W} +\frac{1}{p}}}\| f\|_{L^p(X,\mu)}.
\]
 In particular, for $t>0$, $P_t: L^p(X,\mu) \to \bm{B}^{p,\left(1-\frac{2}{p}\right)\frac{\kappa}{d_W}+\frac{1}{p}}(X)$ is bounded.
\end{theorem}

The theorem provides the analogue of Theorem \ref{continuity Besov chapter 1 intro} to the range $p>2$ and allows us to study the $L^p$ Besov critical exponents.
As before, for $p \ge 1$, if we denote

\[
\alpha^*_p(X) =\inf \left\{ \alpha >0\, :\,  \B^{p,\alpha}(X) \text{ is trivial} \right\},
\]

we obtain then:

\begin{theorem}
Let $X$ be an Ahlfors $d_H$-regular space that satisfies sub-Gaussian heat kernel estimates and $BE(\kappa)$ with $0<\kappa\leq \frac{d_W}{2}$.
The following inequalities hold:

\begin{itemize}
\item For $1 \le p \le 2$, $$ \frac{1}{2} \le  \alpha^*_p(X) \le \left(1-\frac{2}{p}\right)\frac{\kappa}{d_W}+\frac{1}{p}.$$ 
\item For $ p \ge 2$, $$ \left(1-\frac{2}{p}\right)\frac{\kappa}{d_W}+\frac{1}{p} \le  \alpha^*_p(X)  \le \frac{1}{2} .$$
\end{itemize}
\end{theorem}

In the end we discuss  generalized Riesz transforms, 
sets of finite perimeter,
and
Sobolev and isoperimetric inequalities.  Our main motivation is to connect these classical analysis notions with the heat kernel estimates, which in the context of fractals are mostly studied in probability theory.

\subsection*{Chapter 6}
Many arguments appearing in the study of local Dirichlet spaces turn out to be easily adaptable to obtain a non-local version of several results from Chapter~\ref{LDsGHKU}. In this chapter we consider non-local and regular Dirichlet spaces whose associated semigroup has a heat kernel satisfying the estimates
\begin{equation}\label{eq:HKE-non-loc_intro}
c_{5}t^{-\frac{d_{H}}{d_{W}}} \left(1+c_{6} \frac{d(x,y)}{t^{1/d_{W}}}\right)^{-d_H-d_W} 
\le p_{t}(x,y) \le
c_{3}t^{-\frac{d_{H}}{d_{W}}}\left(1+c_{4} \frac{d(x,y)}{t^{1/d_{W}}}\right)^{-d_H-d_W} 
\end{equation}
for some $c_3,\ldots ,c_6\in (0,\infty)$ and $0<d_W\leq d_H+1$.

Following the same structure as Chapter~\ref{LDsGHKU} we start by discussing in Section~\ref{S:MCB_nl} some metric properties of the space $\bm{B}^{p,\alpha}(X)$ in terms of other Besov type spaces appearing in the literature, see e.g.~\cite{Gri}. In contrast to the local case, the characterization of the Besov space given below, c.f.\ Theorem~\ref{Besov non-local}, is confined to $\alpha\in [0,1/p)$. Some situations when $\alpha\geq 1/p$ are analyzed in Proposition~\ref{P:Inclusions_Besov_nl}.
\begin{theorem}\label{Besov non-local_intro}
Let the non-local Dirichlet space $(X,d,\mu,\mathcal{E},\mathcal{F})$ have an associated heat kernel satisfying~\eqref{eq:HKE-non-loc_intro}. For any $p\in[1,\infty)$ and $ \alpha \in[0,1/p)$ we have
\[
 \| f \|_{p,\alpha} \simeq \sup_{r>0}\frac{1}{r^{\alpha d_W+d_{H}/p}}\biggl(\iint_{\Delta_r}|f(x)-f(y)|^{p}\,d\mu(x)\,d\mu(y)\biggr)^{1/p}.
\]
\end{theorem}
Sections~\ref{S:coarea_nl} and~\ref{S:FP_nl} provide coarea type estimates of the Besov norm, c.f.\ Theorem~\ref{T:non-local_coarea} as well as a characterization of sets with finite perimeter in Theorem~\ref{T:FP_char_non-local} that is analogous to their local counterparts in Chapter~\ref{LDsGHKU}. In this case the results require $0<\alpha\leq\min\{d_H/d_W,1\}$ but also include situations where $d_W<1$. The question about the critical exponent $\alpha_1^*(X)$ for the space $\bm{B}^{1,\alpha_1^*(X)}(X)$ is examined in Section~\ref{S:criticalE_nl} under an extra assumption on the underlying space $(X,d,\mu)$, c.f.\ Assumption~\ref{assumption non local}, where we prove in Theorem~\ref{T:non-local_critical_alpha} that
\begin{theorem}\label{T:non-local_critical_alpha_intro}
Under Assumption~\ref{assumption non local}, we have
\[
\alpha_1^*(X)=\begin{cases}
1&\text{if }d_W \le 1,\\
\frac{1}{d_W}&\text{if }d_W  > 1
\end{cases}
\]
and the space $\mathbf{B}^{1,\alpha_1^*(X)}(X)$ is characterized as displayed in Table~\ref{Table}.
\end{theorem}
This chapter finishes with a discussion in Section~\ref{S:Sobo_iso_nl} concerning a special type of Sobolev and isoperimetric inequalities which we can treat in the non-local setting. Assuming that the heat semigroup is transient and the Dirichlet form $(\mathcal{E},\mathcal{F})$ is regular, general arguments from Dirichlet form theory, see e.g.~\cite[Section 2.4]{FOT}\cite[Section 2.1]{ChenFukushima} provide capacitary conditions that yield a Sobolev type inequality of the form
\[
\Big(\int_X|f|^{2\kappa}d\mu\Big)^{2\kappa}\leq C\sqrt{\mathcal{E}(f,f)}
\]
for some $\kappa\geq 1$ and $C>0$, c.f.~Theorem~\ref{T:isoperim-nonlocal}. Moreover, combining the results from previous sections in this chapter and with Sections~\ref{S:Weak_Sobolev}-\ref{S:IsopIneq} in Chapter~\ref{section sobolev}, we obtain the corresponding weak Sobolev, capacitary and isoperimetric inequalities.
\subsection*{Chapter 7}

Finally, we also have some results in the setting of infinite dimensional Banach spaces. Due to the technical difficulty of the task (see \cite{FuHi}), our goal in the chapter is not to develop the general theory but rather to point out some research directions showing that our approach is also suitable to handle infinite-dimensional spaces.

We prove in particular that if $E$ is a set of finite perimeter in the Wiener space, or more generally in a path space endowed with a Gibbs measure  as in Kawabi \cite{Kawabi} or Kawabi-R\"ockner \cite{KR}, then one must have $1_E \in \mathbf{B}^{1,1/2}(X)$. More precisely, we establish the following result:

\begin{theorem}
Let $X$ be one of the Dirichlet spaces studied in Sections \ref{Wiener} or \ref{gibbs}. Let $E$ be a measurable set with finite measure. The following are equivalent:

\begin{enumerate}
\item $E$ has a finite perimeter;
\item The limit $\lim_{t \to 0^+} \int_X \| DP_t 1_E \|_{\mathcal{H}_x} d\mu(x)$ exists.
\end{enumerate}
Moreover, if $E$ is a set satisfying one of the above conditions, then $1_E \in \mathbf{B}^{1,1/2}(X)$ and 
\[
\| 1_E \|_{1,1/2}\le 2 \sqrt{2} P(E,X) =2 \sqrt{2} \lim_{t \to 0^+} \int_X \| DP_t 1_E \|_{\mathcal{H}_x} d\mu(x).
\]
\end{theorem}

In these cases, the isoperimetric inequality from Theorem \ref{iso intro} does not apply, but we prove instead a Poincar\'e and a Gaussian isoperimetric inequality. In particular we   prove:

\begin{theorem}
Let $(X,\mathcal{E},\mathcal{F},\mu)$ be the path space with Gibbs measure described in Section \ref{gibbs}. Assume that the interaction potential $U$ is strictly convex in the sense that there exists a constant $a>0$ such that $\nabla^2 U \ge a$. Then $X$ supports a Gaussian isoperimetric inequality, that is, there exists a constant $C>0$ such that for every set $E \subset X$ with finite perimeter $P(E,X)$, one has:
\[
\mu(E)\sqrt{-\ln \mu(E)} \le C P(E,X).
\]
\end{theorem}

%% file: chapter1.tex
\chapter{Besov spaces related to a Dirichlet space}\label{Sec:Dirichlet}

\section{Definitions}

Let $(X,\mu,\mathcal{E},\mathcal{F}=\mathbf{dom}(\mathcal{E}))$ be a symmetric Dirichlet space, 
that is, $X$ is a topological
measure space equipped with the Radon measure $\mu$, $\mathcal{E}$ a closed Markovian bilinear form on
$L^2(X,\mu)$, and $\mathcal{F}$ the collection of all functions $u\in L^2(X,\mu)$ with $\mathcal{E}(u,u)$ finite. The book~\cite{FOT} is a standard reference on the theory of Dirichlet forms.
Let $\{P_{t}\}_{t\in[0,\infty)}$ denote the Markovian semigroup associated with
$(X,\mu,\mathcal{E},\mathcal{F})$. From classical theory (see for instance Theorems 1.4.1 and 1.4.2 in \cite{EBD}), the semigroup $P_t$ lets $L^1(X,\mu) \cap L^\infty (X,\mu)$ invariant and may be extended to a positive, contraction semigroup on $L^p (X,\mu)$, $1 \le p \le +\infty$, that we shall still denote by $P_t$. Moreover, for $1 \le p <+\infty$, $P_t$ is strongly continuous and for $1<p<+\infty$, $P_t$ is a bounded analytic semigroup on $L^p (X,\mu)$ with angle $\theta_p=\frac{\pi}{2} \left(  1- \left| \frac{2}{p}-1 \right| \right)$.  In this monograph, we always assume that $P_t$ is conservative, i.e. $P_t 1=1$.

Our basic definition of the Besov seminorm is the following:
\begin{definition}
Let $p \ge 1$ and $\alpha \ge 0$. For $f \in L^p(X,\mu)$, we define the Besov  seminorm:
\[
\| f \|_{p,\alpha}= \sup_{t >0} t^{-\alpha} \left( \int_X P_t (|f-f(y)|^p)(y) d\mu(y) \right)^{1/p}.
\]
\end{definition}

Observe that if $P_t$ admits a heat kernel $p_t(x,y)$, then
\[
 \int_X P_t (|f-f(y)|^p)(y) d\mu(y)= \int_X \int_X |f(x)-f(y) |^p p_t (x,y) d\mu(x) d\mu(y) ,
\]
but the seminorm $\| \cdot \|_{p,\alpha}$ is well defined even if $P_t$ does not have a heat kernel.

Our goal in this work is to study the Besov  space
\begin{align}\label{eq:def:Besov}
\mathbf{B}^{p,\alpha}(X)=\{ f \in L^p(X,\mu)\, :\,  \| f \|_{p,\alpha} <+\infty \},
\end{align}
and specially $\mathbf{B}^{1,\alpha}(X)$ and its connection to the notion of bounded variation function. The norm on $\mathbf{B}^{p,\alpha}(X)$ is defined as:
\[
\| f \|_{\mathbf{B}^{p,\alpha}(X)} =\| f \|_{L^p(X,\mu)} + \| f \|_{p,\alpha}.
\]

\begin{example}
If $X=\mathbb{R}^n$ and $\mathcal{E}$ is the standard Dirichlet form on $\mathbb{R}^n$, that is, for $f,g \in W^{1,2}(\mathbb{R}^n)$
we have 
\[
\mathcal{E}(f,g) =\int_{\mathbb{R}^n} \langle \nabla f (x) ,\nabla g (x) \rangle dx, \quad f,g \in W^{1,2}(\mathbb{R}^n),
\]
then $\mathbf{B}^{p,\alpha}(X)$ coincides with the  Besov-Nikol'skii space $B^{2\alpha}_{p,\infty} (\mathbb{R}^n)$ that consists of functions $f \in L^p(\mathbb{R}^n,dx)$ such that
\[
\sup_{h \in \mathbb{R}^n, h \neq 0} \frac{ \| f( \cdot+h)-f(\cdot) \|_p}{h^{2\alpha}} <+\infty.
\]
We refer for instance to \cite{AS} and \cite{Taibleson} (Theorems 4 and 4*)  for several equivalent descriptions of those spaces. In particular, we observe therefore that $\mathbf{B}^{1,1/2}(\mathbb{R}^n)$ is the space of bounded variation functions $BV(\mathbb{R}^n)$ and that $\mathbf{B}^{2,1/2}(\mathbb{R}^n)=W^{1,2}(\mathbb{R}^n)$ is the domain of $\mathcal{E}$. The observation $\mathbf{B}^{1,1/2}(\mathbb{R}^n)=BV(\mathbb{R}^n)$ will be generalized in the framework of metric spaces with a weak Bakry-\'Emery type non-negative curvature condition and the observation $\mathbf{B}^{1,1/2}(\mathbb{R}^n)=\mathcal{F}$ is a general fact, see Proposition \ref{prop:energyasbesov}
\end{example}

\section{A preliminary class of examples}

Before discussing general properties of the spaces $\mathbf{B}^{p,\alpha}(X)$, it may be instructive to quickly look at a preliminary large class of examples in a smooth setting.

\begin{proposition}\label{Besov Rn}
Assume that $X$ is a smooth manifold. Let $L=V_0+\sum_{i=1}^d V_i^2$ be a H\"ormander's type operator on $X$, where the $V_i$'s are smooth vector fields. Let us assume that $L$ is essentially self-adjoint on $C_0^\infty(X)$ in $L^2(X,\mu)$ for some Radon measure $\mu$ on $X$. Consider the Dirichlet space $(X,\mu, \mathcal{E},\mathcal{F})$ obtained by closing the pre-Dirichlet form
\[
\mathcal{E}(f,g)=\int_X \Gamma(f,g) d\mu(x), \quad f,g \in C_0^\infty(X),
\]
where $\Gamma(f,g)$ is the carr\'e du champ operator defined by $\Gamma(f,g)=\frac{1}{2}( L(fg)-fLg -gLf)$. Assume that the associated semigroup $P_t$ is conservative.
Then, for every $p \ge 1$, $$C_0^\infty (X) \subset \mathbf{B}^{p,1/2}(X)$$ and one has for every $f \in C_0^\infty (X) $, and open set $A \subset X$,
\[
\lim_{t \to 0} t^{-1/2} \left( \int_A P_t (|f-f(x)|^p)(x) d\mu(x) \right)^{1/p} = 2\left( \frac{ \Gamma \left( \frac{1+p}{2} \right) }{\sqrt
{\pi}}\right)^{1/p} \left( \int_A \Gamma ( f,f) (x)^{p/2} d\mu(x)\right)^{1/p}.
\]
\end{proposition}
\begin{remark}
In the previous setting one has therefore for $f \in C_0^\infty(X)$,
\[
\left( \int_X \Gamma ( f,f) (x)^{p/2} d\mu(x)\right)^{1/p} \le C_p \| f \|_{p,1/2}.
\]
If $P_t$ satisfies moreover the  Bakry-\'Emery estimate $ \sqrt{\Gamma(P_t f)} \le C P_t \sqrt{\Gamma( f)} $, then we will see (Chapter \ref{Sec:Metric-Curvature}, Section 4.5) that we have a converse inequality
\[
\| f \|_{p,1/2} \le c_p \left( \int_X \Gamma ( f,f) (x)^{p/2} d\mu(x)\right)^{1/p},
\]
for $p=1$.
\end{remark}

\begin{remark}
Proposition \ref{Besov Rn} indicates that at a high level of generality, one may  expect the  Besov spaces $ \mathbf{B}^{p, 1/2}(X)$, $1 \le p <+\infty$  to be closely related to the various notions of Sobolev spaces that have been defined on metric measure spaces (see for instance \cite{Shan2000}). While in this work we shall only be concerned with the study of all the Besov  spaces $\mathbf{B}^{p,\alpha}(X)$, the comparison between Sobolev spaces and Besov spaces will be made in Chapter 4 in the framework of Dirichlet spaces with absolutely continuous energy measures. In the framework of Chapter 5,  it will be interesting to compare our results with the recent preprint \cite{HinzKochMeinert} on Sobolev spaces and calculus of variations on fractals. Such a comparison will be the subject of future study. 
\end{remark}

\begin{example}
If $X=\mathbb{R}^n$ and $\mathcal{E}$ is the standard Dirichlet form on $\mathbb{R}^n$, it is natural to expect that for every $p \ge 1$, and every $f \in \mathbf{B}^{p,1/2}(X)$
\[
 \lim_{t \to 0} t^{-1/2} \left( \int_{\mathbb{R}^n} P_t (|f-f(x)|^p)(x) dx \right)^{1/p}=2\left( \frac{ \Gamma \left( \frac{1+p}{2} \right) }{\sqrt
{\pi}}\right)^{1/p} \left( \int_{\mathbb{R}^n} | \nabla f (x)|^p dx\right)^{1/p}.
\]
The case $p=1$ is proved in \cite{MPPP}, but we did not find it in the literature for $p>1$, $p \neq 2$. 
\end{example}

\begin{proof} [Proof of Proposition \ref{Besov Rn}.]
We use here a probabilistic argument. For $x \in X$, we denote by $(B_t^x)_{t \ge 0}$ the $L$-Brownian motion on $X$ started from $x$, that is the diffusion with generator $L$. It can be constructed as the solution of a stochastic differential equation in Stratonovich form:
\[
dB_t^x=V_0(B_t^x)dt +\sqrt{2}\sum_{i=1}^d V_i(B_t^x)\circ d\beta^i_t
\]
where $\beta$ is a $d$-dimensional Brownian motion. 
Let $f \in C_0^\infty(X)$. The process
\[
M_t^f =f(B_t^x) -f(x) -\int_0^t L f(B_s^x) ds
\]
is a square integrable martingale that can be written
\[
M_t^f=\sqrt{2} \sum_{i=1}^d \int_0^t (V_i f)(B_s^x)d\beta^i_s.
\]
We have then
\begin{align*}
P_t (|f-f(x)|^p)(x) & =\mathbb{E} \left( | f(B_t^x) -f(x) |^p\right) \\
 &=\mathbb{E} \left( \left| M_t^f +\int_0^t L f(B_s^x) ds \right|^p \right).
\end{align*}
Observe now that $\frac{1}{\sqrt{t}} \int_0^t L f(B_s^x) ds$ almost surely converges to 0 when $t \to 0$. Therefore, $\frac{1}{\sqrt{t}} M_t^f$ converges in all $L^p$'s to the Gaussian random variable $\sqrt{2} \sum_{i=1}^d (V_i f) (x) \beta^i_1$. Since $f$ has a compact support, one deduces that
\[
\lim_{t \to 0} t^{-1/2} \left( \int_A P_t (|f-f(x)|^p)(x) d\mu(x) \right)^{1/p} = C_p \left( \int_A \Gamma ( f,f) (x)^{p/2} d\mu(x)\right)^{1/p},
\]
with $C_p=\sqrt{2} \mathbb{E}(| N|^p)^{1/p}=2\left( \frac{ \Gamma \left( \frac{1+p}{2} \right) }{\sqrt
{\pi}}\right)^{1/p} $, where $N$ denotes a Gaussian random variable with mean 0 and variance 1.
\end{proof}

\section{Basic  properties of the Besov spaces}

\begin{proposition}
 For $p\ge 1$ and $\alpha \ge 0$, $\mathbf{B}^{p,\alpha}(X)$ is a Banach space.
 \end{proposition}

\begin{proof}
Let $f_n$ be a Cauchy sequence in $\mathbf{B}^{p,\alpha}(X)$. Let $f$ be the $L^p$ limit of $f_n$. From Minkowski inequality and conservativeness of $P_t$, one has
\begin{align*}
 &  \left| \left(\int_X P_t (|f_n-f_n(y)|^p)(y) d\mu(y) \right)^{1/p} - \left(\int_X P_t (|f-f(y)|^p)(y) d\mu(y) \right)^{1/p} \right| \\
 \le &  \left(\int_X P_t (|(f_n-f)-(f_n(y)-f(y))|^p)(y) d\mu(y) \right)^{1/p} \\
 \le &\left(\int_X P_t (|f_n-f|^p)(y) d\mu(y) \right)^{1/p}+\left(\int_X P_t (|f_n(y)-f(y)|^p)(y) d\mu(y) \right)^{1/p} \\
 \le &2 \| f-f_n \|_{L^p(X,\mu)} .
\end{align*}
Therefore
\[
\lim_{n \to +\infty}  \left(\int_X P_t (|f_n-f_n(y)|^p)(y) d\mu(y) \right)^{1/p}=  \left(\int_X P_t (|f-f(y)|^p)(y) d\mu(y) \right)^{1/p},
\]
from which we deduce that
\[
\left(\int_X P_t (|f-f(y)|^p)(y) d\mu(y) \right)^{1/p} \le  t^\alpha \lim_{n \to +\infty} \| f_n \|_{p,\alpha}.
\]
Therefore $f \in \mathbf{B}^{p,\alpha}(X)$ and $\| f \|_{p,\alpha} \le \lim_{n \to +\infty} \| f_n \|_{p,\alpha}$. Similarly,
\[
\| f-f_m \|_{p,\alpha} \le  \lim_{n \to +\infty} \| f_n -f_m \|_{p,\alpha}
\]
and taking the limit $m \to +\infty$ finishes the proof.
\end{proof}

The following basic observation shall be useful in the sequel:

\begin{lemma}\label{Lemma limsup debut}
Let $p\ge1, \alpha \ge 0$.
\[
\mathbf{B}^{p,\alpha}(X)
=\left\{ f \in L^p(X,\mu)\, :\,  \limsup_{t \to 0} t^{-\alpha} \left( \int_X P_t (|f-f(y)|^p)(y) d\mu(y) \right)^{1/p} <+\infty \right\}.
\]
In particular if $\beta > \alpha$, $\mathbf{B}^{p,\beta}(X) \subset \mathbf{B}^{p,\alpha}(X) $.
Moreover, for $f \in \mathbf{B}^{p,\alpha}(X)$, one has for every $t >0$,
\[
\| f \|_{p,\alpha} \le \frac{2}{t^\alpha} \| f \|_{L^p(X)} +\sup_{s\in (0,t]} s^{-\alpha} \left( \int_X P_s (|f-f(y)|^p)(y) d\mu(y) \right)^{1/p}.
\]
\end{lemma}

\begin{proof}
Obviously, if $f \in \mathbf{B}^{p,\alpha}(X)$, then
\[
\limsup_{t \to 0} t^{-\alpha} \left( \int_X P_t (|f-f(y)|^p)(y) d\mu(y) \right)^{1/p} \le \| f \|_{p,\alpha}.
\]
Conversely, if $\limsup_{t \to 0} t^{-\alpha} \left( \int_X P_t (|f-f(y)|^p)(y) d\mu(y) \right)^{1/p} <+\infty$, then, for some $\ve >0$,
\[
\sup_{t \in (0,\ve]} t^{-\alpha} \left( \int_X P_t (|f-f(y)|^p)(y) d\mu(y) \right)^{1/p} <+\infty
\]
For $t >\ve$, since $|f(x)-f(y)|^p \le 2^{p-1} (|f(x)|^p+|f(y)|^p)$ and the semigroup is conservative (and hence
$P_t1(x)=1$ for all $x\in X$), one has then
\begin{align*}
 t^{-\alpha} \left( \int_X P_t (|f-f(y)|^p)(y) d\mu(y) \right)^{1/p} \le 2 \ve^{-\alpha} \| f \|_{L^p(X)}.
\end{align*}
\end{proof}

Both the following lemma and Proposition~\ref{prop:energyasbesov} are well-known, but we do not know where in the literature they appear in this form. In the standard source~\cite{FOT} the corresponding results rely on a topological assumption, see footnote on page~ of~\cite{FOT}. The argument given below is similar to that in Lemma~2.3.2.1 of~\cite{BouleauHirsch}.
\begin{lemma}\label{lem:energyasbesov}
For $t>0$, $f\in L^2(X,\mu)$ and a bounded function $g$, writing $\langle\cdot,\cdot\rangle$ for the $L^2(X,\mu)$ inner product
\begin{equation*}
 \frac1t \int_X g(y)P_t(|f-f(y)|^2)(y)\,d\mu(y) = \frac2t\langle (I-P_t)f, fg \rangle - \frac1t\langle (I-P_t) f^2, g\rangle.
\end{equation*}
\end{lemma}
\begin{proof}
Consider a simple function $f=\sum_{j=1}^n a_j \mathbf{1}_{A_j}$ where the $A_j$ are pairwise disjoint and measurable and compute (using pairwise disjointness and the symmetry of $P_t$ on $L^2$) that
\begin{align*}
 \lefteqn{\int_X g(y) P_t(|f-f(y)|^2) \, d\mu(y)}\quad&\\
 &= \sum_{j,k,l,m} \int \biggl( a_j a_l P_t \bigl(\mathbf{1}_{A_j} \mathbf{1}_{A_l}\bigr)(y) - a_j a_m P_t\bigl(\mathbf{1}_{A_j}\bigr)(y)\mathbf{1}_{A_m}(y) \\
 &\qquad- a_k a_l \mathbf{1}_{A_k}(y) P_t\bigl(\mathbf{1}_{A_l}\bigr)(y) + a_k a_m\mathbf{1}_{A_k}(y)\mathbf{1}_{A_m}(y) \biggr) g(y)\,d\mu(y)\\
 &= \sum_{j} a_j^2 \biggl( \langle P_t \mathbf{1}_{A_j},g \rangle +\langle \mathbf{1}_{A_j},g\rangle \biggr)- 2\sum_{j,l} a_j a_l \langle P_t\mathbf{1}_{A_j},g\mathbf{1}_{A_l}\rangle\\
 &= \sum_{j,l} a_j a_l \Bigl( 2\langle (I-P_t)\mathbf{1}_{A_j},g\mathbf{1}_{A_l}\rangle - \langle (I-P_t)(\mathbf{1}_{A_j}\mathbf{1}_{A_l}),g\rangle \Bigr)\\
 &=2\langle (I-P_t)f, fg \rangle - \langle (I-P_t) f^2, g\rangle.
 \end{align*}
The result then follows by density of simple functions in $L^2(X,\mu)$ and $L^p$ boundedness of $P_t$ for $t>0$ for $p=1,2$; for the former see~\cite[page 33 ]{FOT}.
\end{proof}

\begin{proposition}\label{prop:energyasbesov}
$2\DF(f)=\|f\|_{2,1/2}^2$ and hence\/ $\mathbf{B}^{2,1/2}(X)=\mathcal{F}$.
\end{proposition}
\begin{proof}
Apply Lemma~\ref{lem:energyasbesov} with $g\equiv1$. Since $(I-P_t)$ is symmetric and $(I-P_t)1=0$ this simplifies to
\begin{equation}\label{prop:energyasbesoveq}
 \frac1{2t} \int_X P_t(|f-f(x)|^2)\, d\mu(x)= \frac1t \langle (I-P_t)f,f\rangle.
 \end{equation}
However it is standard (see~\cite{BouleauHirsch} Proposition~1.2.3) that the right side of~\ref{prop:energyasbesoveq} is positive and decreasing in $t$ and has limit $\DF(f)$ as $t\downarrow0$ if $f\in\mathcal{F}$, from which this limit is the supremum and the result follows.
\end{proof}

\begin{proposition}\label{prop:BesovCEp}
Assume that $\mathcal{E}$ has the property that $\mathcal{E}(f,f)=0$ implies $f$ constant. Then, any $f \in \mathbf{B}^{p,\alpha }(X)$ with $1 \le p \le 2$ and $\alpha >1/p$ is constant.
\end{proposition}

\begin{proof}
Let $f \in \mathbf{B}^{p,\alpha }(X)$ with $1 \le p \le 2$. For $n \ge 0$, one considers $f_n(x) =f(x)$ if $|f(x)| \le n$ and $f_n(x)=0$ otherwise. Since $|f_n(x) -f_n(y)| \le |f(x) -f(y)|$ for every $x,y \in X$, it is clear that $f_n \in \mathbf{B}^{p,\alpha }(X)$. One has then
\[
P_t(|f_n-f_n(x)|^2) =P_t(|f_n-f_n(x)|^{2-p} |f_n-f_n(x)|^p)  \le 2^{2-p} \| f_n \|^{2-p}_{\infty} P_t(|f_n-f_n(x)|^p).
\]
Therefore
\[
 \frac1{2t} \int_X P_t(|f_n-f_n(x)|^2) d\mu(x) \le  2^{1-p} t^{\alpha p -1}\| f_n \|^{2-p}_{\infty} \| f_n\|_{p,\alpha}^p.
 \]
 This implies
 \[
 \lim_{t \to 0} \frac1{2t} \int_X P_t(|f_n-f_n(x)|^2) d\mu(x)  =0.
 \]
 Thus $f_n \in \mathcal{F}$ and $\mathcal{E}(f_n,f_n)=0$. This implies that $f_n$ is constant for every $n$, thus $f$ is constant.

 \end{proof}

\begin{proposition}
If $\DF$ is regular and $f\in\mathbf{B}^{p,1/2}(X)$ for $p>2$ then the energy measure $\nu_f$ (in the sense of Beurling-Deny) is absolutely continuous with respect to $\mu$.
\end{proposition}
\begin{proof}
Let $p>2$, $f \in \mathbf{B}^{p,1/2}(X)$ and $g \in L^{p/(p-2)}(X,\mu) \cap L^\infty (X,\mu)\cap \mathcal{F}$, $g \ge 0$.   From H\"older's inequality we have
\begin{align*}
 \frac1t\int_X |g(y)| P_t(|f-f(y)|^2)\,d\mu(y)
 &\leq \frac1t\int_X |g(y)| \bigl( P_t (|f-f(y)|^p)\bigr)^{2/p}\,d\mu(y) \\
 &\leq \frac1t\biggl( \int_X P_t(|f-f(y)|^p) \,d\mu(y) \biggr)^{2/p} \|g\|_{p/(p-2)}\\
 &\leq \|f\|_{p,1/2}^2\|g\|_{p/(p-2)}
 \end{align*}
Now use the result of Lemma~\ref{lem:energyasbesov} and take the limit as $t\downarrow0$ to obtain
\begin{equation*}
 \|f\|_{p,1/2}^2\|g\|_{p/(p-2)}
 \geq \lim_{t\downarrow0} \frac2t\langle (I-P_t)f, fg \rangle - \frac1t\langle (I-P_t) f^2, g\rangle
 = 2\DF(fg,f)-\DF(f^2,g)= \int_X  2g\,d\nu_f
 \end{equation*}
where, as in the previous result, the limit is by Proposition~1.2.3 of~\cite{BouleauHirsch}. The final equality is from the definition of $\nu_f$, Definition~4.1.2 of~\cite{BouleauHirsch}.

In particular, if $E_1\subset E_2$ are of finite $\mu$ measure and $\mathbf{1}_{E_1}\leq g\leq \mathbf{1}_{E_2}$ then we obtain
\begin{equation*}
\nu_{f}(E_1) \leq \int_X g\,d\nu_f \leq \frac{1}{\sqrt2} \| f\|_{p,1/2}^2 \bigl( \mu(E_2) \bigr)^{(p-2)/p}.
\end{equation*}
Regularity of $\DF$ then supplies sufficiently many functions $g\in\mathcal{F}$ to conclude that $\nu_f$ is absolutely continuous with respect to $\mu$, see  ~\cite{BouleauHirsch,FOT,ChenFukushima}.
\end{proof}

There are spaces on which the Kusuoka measure (and hence $\nu_f$ for any non-constant $f\in\mathcal{F}$) is singular to $\mu$, see~\cite{Kusuoka,BenBassatStrichartzTeplyaev}. For such spaces $\mathbf{B}^{p,1/2}(X)$ is therefore trivial when $p>2$.

\begin{corollary}\label{singular Kusuoka}
If the Kusuoka measure is singular to $\mu$ then $\mathbf{B}^{p,1/2}(X)$ contains only constant functions when $p>2$.
\end{corollary}

\section{Pseudo-Poincar\'e inequalities}

The following pseudo-Poincar\'e inequalities will play an important role in many parts of this monograph.

\begin{lemma}\label{pseudo-Poincare}
Let $ p \ge 1$ and $\alpha >0$. For every $f \in \mathbf{B}^{p,\alpha} (X)$, and $t \ge 0$,
\[
\| P_t f -f \|_{L^p(X,\mu)} \le t^\alpha \| f \|_{p,\alpha}.
\]
\end{lemma}

\begin{proof}
From conservativeness of the semigroup and H\"older's inequality, we have
\begin{align*}
\| P_t f - f \|_{L^p(X,\mu)} &=\left(\int_X | P_t f (x)-f(x)|^p d\mu(x)\right)^{1/p} \\
 & = \left(\int_X | P_t (f -f(x))(x)|^p d\mu(x)\right)^{1/p} \\
 & \le \left( \int_X P_t (|f-f(x)|^p)(x) d\mu(y) \right)^{1/p}.
\end{align*}
\end{proof}

\begin{remark}
Triebel \cite{Trie} (Section 1.13.6) introduced the interpolation spaces:
\[
(L^p(X,\mu), \mathcal{E} )_{\alpha,\infty}=\left\{ u \in L^p(X,\mu)\, :\, \sup_{t >0} t^{-\alpha} \| P_t u -u \|_{L^p(X,\mu )} <+\infty \right\}.
\]

From the previous lemma, it is therefore clear that $\mathbf{B}^{p,\alpha} (X) \subset (L^p(X,\mu), \mathcal{E})_{\alpha,\infty}.$ However, it may not be true that $\mathbf{B}^{p,\alpha} (X) = (L^p(X,\mu), \mathcal{E})_{\alpha,\infty}$, even when $X=\mathbb{R}^n$, see Remark 4.5 in \cite{MPPP} and \cite{Taibleson} (Theorems 4 and 4*).
\end{remark}

The following lemma will be useful:

\begin{lemma}
Let $L$ the generator of $\mathcal{E}$ . Let $p > 1$ and $\alpha >0$. Then, there exists a constant $C>0$ such that for every $f \in \mathbf{B}^{p,\alpha}(X)$ and $t \ge 0$,
\[
\| LP_t f \|_{L^p(X,\mu)} \le C \frac{\| f \|_{p,\alpha}}{t^{1-\alpha}}.
\]
\end{lemma}

\begin{proof}
Since $\lim_{t \to +\infty} \| LP_t f \|_{L^p(X,\mu)} =0$, we have
\begin{align*}
\| LP_t f \|_{L^p(X,\mu)} & =\left\| \sum_{k=1}^{+\infty}( LP_{2^k t} f -  LP_{2^{k-1} t} f )  \right \|_{L^p(X,\mu)} \\
 & \le  \sum_{k=1}^{+\infty} \left\| LP_{2^k t} f -  LP_{2^{k-1} t} f  \right \|_{L^p(X,\mu)} \\
 & \le \sum_{k=1}^{+\infty} \left\| LP_{2^{k-1} t} (P_{2^{k-1}t} f -  f)  \right \|_{L^p(X,\mu)} \\
 &\le \sum_{k=1}^{+\infty} \frac{1}{2^{k-1} t}  \left\| P_{2^{k-1}t} f -  f  \right \|_{L^p(X,\mu)} \\
 &\le C  \sum_{k=1}^{+\infty} \frac{(2^{k-1} t)^{\alpha}}{2^{k-1} t}  \| f \|_{p,\alpha}  \\
 &\le C \frac{\| f \|_{p,\alpha}}{t^{1-\alpha}},
\end{align*}
where we used in the proof the fact that since $P_t$ is an analytic semigroup, we have for any $g \in L^p(X,\mu) $, $\left\| LP_t g  \right \|_{L^p(X,\mu)} \le \frac{C}{t} \| g \|_{L^p(X,\mu)}$.
\end{proof}

In a very general framework, one can  resort to (Hille-Yosida) spectral theory to define the fractional powers of a closed operator $A$ on a Banach space  via the following formula
\[
(-A)^s x = \frac{\sin \pi s}{\pi} \int_0^\infty \lambda^{s-1} (\lambda I - A)^{-1} (-A)x\ d\lambda,
\]
for every $x\in D(A)$. In fact, using Bochner's subordination one can express the fractional powers of $A$ also in terms of the heat semi-group $P_t = e^{tA}$ via the following formula, see (5) on p. 260 in \cite{Yosida},
\begin{equation}\label{As}
(-A)^s x = - \frac{s}{\Gamma(1-s)} \int_0^\infty t^{-s-1} [P_t x - x]\ dt.
\end{equation}

Let us denote by $L$ the generator of $\mathcal{E}$ and, for $0< s \le 1$, by $\mathcal{L}_p^s$ the domain of the operator $(-L)^s$ in $L^p(X,\mu)$, $1 \le p <+\infty$. One has then the following proposition:

\begin{proposition}
Let $\alpha \in (0,1]$, $p \ge 1$ and $0<s<\alpha$. Then
\[
\B^{p,\alpha}(X) \subset \mathcal{L}_p^s,
\]
and there exists a constant $C=C_{p,s,\alpha}$ such that for every $\B^{p,\alpha}(X)$,
\[
\| (-L)^s f \|_{L^p(X,\mu)} \le C \|  f \|^{1-\frac{s}{\alpha}}_{L^p(X,\mu)} \| f \|_{p,\alpha}^{\frac{s}{\alpha}}.
\]
In particular, $(-L)^s: \B^{p,\alpha}(X) \to L^p(X,\mu)$ is bounded.
\end{proposition}

\begin{proof}
Let  $f \in \B^{p,\alpha}(X)$. One shall prove that the integral $\int_0^\infty t^{-s-1} [P_t f - f]\ dt$ is finite, and therefore that $f \in \mathcal{L}_p^s$. For $\delta >0$, one has
\begin{align*}
 \left\| \int_0^\infty t^{-s-1} [P_t f - f]\ dt \right\|_{L^p(X,\mu)} & \le  \int_0^\infty t^{-s-1}  \| P_t f - f \|_{L^p(X,\mu)}  dt \\
  & \le \int_0^\delta t^{-s-1}  \| P_t f - f \|_{L^p(X,\mu)}  dt +  \int_\delta^\infty t^{-s-1}  \| P_t f - f \|_{L^p(X,\mu)}  dt  \\
  & \le \| f \|_{p,\alpha} \int_0^\delta t^{-s-1 +\alpha}   dt +2 \| f \|_{L^p(X,\mu)} \int_\delta^\infty t^{-s-1}    dt \\
  & \le \| f \|_{p,\alpha} \frac{\delta^{\alpha-s}}{\alpha -s}+2 \| f \|_{L^p(X,\mu)} \frac{\delta^{-s}}{s}.
\end{align*}
The result follows then by optimizing $\delta$.
\end{proof}

\section{Reflexivity of the Besov spaces}

For any $p\ge 1$ and $\alpha>0$, recall the norm on $\mathbf{B}^{p,\alpha}(X)$:
\[
\| f \|_{\mathbf{B}^{p,\alpha}(X)} =\| f \|_{L^p(X,\mu)} + \| f \|_{p,\alpha},
\]
where
\[
\| f \|_{p,\alpha}= \sup_{t >0} t^{-\alpha} \left( \int_X P_t (|f-f(y)|^p)(y) d\mu(y) \right)^{1/p}.
\]

In the following proof, for convenience we use an equivalent norm, still denoted by $\| \cdot \|_{\mathbf{B}^{p,\alpha}(X)}$, defined as follows
\[
\| f \|_{\mathbf{B}^{p,\alpha}(X)} = \brak{\| f \|_{L^p(X,\mu)}^p +\| f \|_{p,\alpha}^p }^{\frac1p}.
\]

\begin{lemma}[Clarkson type inequalities]
Let $f,g\in \mathbf{B}^{p,\alpha}(X)$. Let $1<p<\infty$ and $q$ be its conjugate. If $2\le p<\infty$, then
\begin{equation}\label{eq:CTIge2}
\norm{\frac{f+g}{2}}_{\mathbf{B}^{p,\alpha}(X)}^p + \norm{\frac{f-g}{2}}_{\mathbf{B}^{p,\alpha}(X)}^p
\le \frac12 \|f\|_{\mathbf{B}^{p,\alpha}(X)}^p +\frac12 \|g\|_{\mathbf{B}^{p,\alpha}(X)}^p.
\end{equation}
If $1<p\le 2$, then
\begin{equation}\label{eq:CTIle2}
\norm{\frac{f+g}{2}}_{\mathbf{B}^{p,\alpha}(X)}^{q} + \norm{\frac{f-g}{2}}_{\mathbf{B}^{p,\alpha}(X)}^{q}
\le \brak{\frac12 \|f\|_{\mathbf{B}^{p,\alpha}(X)}^p +\frac12 \|g\|_{\mathbf{B}^{p,\alpha}(X)}^p}^{q-1}.
\end{equation}
\end{lemma}
The proof is essentially the same as in the $L^p$ case (see for instance \cite[Theorem 2.38]{AR}), so we omit it for concision. As a corollary, we immediately obtain

\begin{corollary}\label{lem:reflexive}
For any $p > 1$ and $\alpha>0$, $\mathbf B^{p,\alpha}(X)$ is a  reflexive Banach space.
\end{corollary}
\begin{proof}
Since $\mathbf B^{p,\alpha}(X)$ is a Banach space, from Milman-Pettis theorem it suffices to show that $\mathbf B^{p,\alpha}(X)$ is uniformly convex. That is, 
for every $0<\varepsilon \le 2$, there exists $\delta>0$ such that for any $f,g\in \mathbf B^{p,\alpha}(X)$ with $\|f\|_{\mathbf B^{p,\alpha}(X)}=\|g\|_{\mathbf B^{p,\alpha}(X)}=1$ and $\|f-g\|_{\mathbf B^{p,\alpha}(X)}\ge \varepsilon$, then 
\[
\norm{\frac{f+g}{2}}_{\mathbf B^{p,\alpha}(X)}\le 1-\delta.
\]
This can be seem as follows. 

If $2\le p<\infty$, then \eqref{eq:CTIge2} implies that
\[
\norm{\frac{f+g}{2}}_{\mathbf B^{p,\alpha}(X)}^p \le 1-\frac{\varepsilon^p}{2^p}.
\]
If $1< p\le 2$, then \eqref{eq:CTIle2} implies that
\[
\norm{\frac{f+g}{2}}_{\mathbf B^{p,\alpha}(X)}^q \le 1-\frac{\varepsilon^q}{2^q}.
\]
\end{proof}

\section{Interpolation inequalities}

We have the following basic interpolation inequalities.

\begin{proposition}\label{interpolation inequality}
Let $\theta\in [0,1]$, $1 \le p,q <+\infty$ and $\alpha,\beta >0$. Let us assume $\frac{1}{p}=\frac{\theta}{q}+\frac{1-\theta}{r}$ and $\alpha=\theta \beta+(1-\theta)\gamma$. Then, $ \B^{q,\beta}(X)\cap \B^{r,\gamma}(X) \subset \B^{p,\alpha}(X) $  and for any $f\in \B^{q,\beta}(X)\cap \B^{r,\gamma}(X)$,
\[
\|f\|_{p,\alpha} \le \|f\|_{q,\beta}^{\theta} \|f\|_{r,\gamma}^{1-\theta}.
\]
\end{proposition}

\begin{proof}
Let $f \in \B^{q,\beta}(X)\cap \B^{r,\gamma}(X) $. One has for every $t >0$
\begin{align*}
    t^{-\alpha} \left( \int_X P_t (|f-f(y)|^p)(y) d\mu(y) \right)^{1/p}  = t^{-\theta \beta-(1-\theta)\gamma} \left( \int_X P_t (|f-f(y)|^p)(y) d\mu(y) \right)^{1/p}.
\end{align*}
Then, from H\"older's inequality
\begin{align*}
\int_X P_t (|f-f(y)|^p)(y) d\mu(y) &=\int_X P_t (|f-f(y)|^{p\theta+p(1-\theta)})(y) d\mu(y)  \\
 & \le \left( \int_X P_t (|f-f(y)|^q)(y) d\mu(y) \right)^{\frac{p \theta}{q}} \left( \int_X P_t (|f-f(y)|^r)(y) d\mu(y)\right)^{\frac{p (1-\theta)}{r}}.
\end{align*}
One deduces
\begin{align*}
  & t^{-\alpha} \left( \int_X P_t (|f-f(y)|^p)(y) d\mu(y) \right)^{1/p}  \\
 \le & t^{-\theta \beta}  \left( \int_X P_t (|f-f(y)|^q)(y) d\mu(y) \right)^{\frac{ \theta}{q}}  t^{-(1-\theta)\gamma} \left( \int_X P_t (|f-f(y)|^r)(y) d\mu(y)\right)^{\frac{ 1-\theta}{r}}.
\end{align*}
Taking the supremum over $t>0$ finishes the proof.
\end{proof}

As a special case of the previous result, note that since $\B^{2,1/2}(X)=\mathcal F$ we easily deduce:

\begin{corollary}\label{duality Dirichlet}
Let $1<p \le 2$ and $q$ be its conjugate $(\frac{1}{p}+\frac{1}{q}=1)$. Let $0<\alpha<1$. Then, for any $f\in \mathcal F \cap \mathbf B^{p,\alpha}(X)$ and $g\in  \mathcal F \cap \mathbf B^{q,1-\alpha}(X)$, it holds that
\[
| \Ecal(f,g) | \le \|f\|_{p,\alpha} \|g\|_{q,1-\alpha}.
\]
\end{corollary}


\section{Continuity of the heat semigroup in the Besov spaces}

Our goal in this section is to study the continuity properties of the semigroup $P_t$ in the Besov spaces with range $1<p \le 2$ and parameter $\alpha=\frac{1}{2}$. As corollaries we will deduce several important properties of the Besov spaces themselves. The main result is the following:

\begin{theorem}\label{continuity Besov chapter 1}
Let $1<p\le 2$. There exists a constant $C_p>0$ such that for every $f \in L^p(X,\mu)$ and $t \ge 0$
\[
\| P_t f \|_{p,1/2} \le \frac{C_p}{t^{1/2}} \| f \|_{L^p(X,\mu)}.
\]
In particular $P_t: L^p(X,\mu) \to \B^{p,1/2}(X)$ is bounded for $t>0$.
\end{theorem}

A key intermediate result is the following inequality that will follow from nice ideas originally due to Nick Dungey \cite{Dungey} and then developed in \cite{LiChen,LiChen2}.

\begin{lemma}\label{Lemma interpolation}
Let $1<p \le 2$. There exists a constant $C_p>0$ such that for every non-negative $f \in L^p(X,\mu)$ and $t >0$
\[
 \left( \int_X P_t (|f-f(y)|^p)(y) d\mu(y) \right)^{1/p} \le C_p \| f \|^{1/2}_{L^p(X,\mu)} \| P_t f -f \|^{1/2}_{L^p(X,\mu)}.
\]
\end{lemma}

\begin{proof}
Let $1 <p \le 2$ and $t>0$ be fixed in the following proof. The constant $C$ in the following will denote a positive constant depending only on $p$ that may change from line to line. For simplicity of notation, we shall assume  in this proof that $P_t$ admits a measurable  heat kernel $p_t(x,y)$, however the argument can readily be extended if this is not the case, by using the heat kernel measure instead. For $\alpha,\beta \ge 0$, set 
\[
\gamma_p(\alpha,\beta):=p\alpha(\alpha-\beta)-\alpha^{2-p} (\alpha^p-\beta^p)
\]
and for a non-negative function $f \in L^p(X,\mu)$
\begin{align*}
\Gamma_p(f)(x) 
&:=p f(x) \int_X p_t(x,y)(f(x)-f(y)) d\mu(y)  - f^{2-p}(x) \int_X p_t(x,y) \brak{f^p(x)-f^p(y)}d\mu(y)
\\&= \int_X p_t(x,y)\,\gamma_p(f(x),f(y))d\mu(y).
\end{align*}
Note that from Lemma 3.5 in \cite{LiChen}, one has for any $\alpha,\beta \ge 0$ 
\[
(p-1) (\alpha -\beta)^2 \le \gamma_p(\alpha,\beta)+\gamma_p(\beta,\alpha) \le p (\alpha-\beta)^2
\]
and that, similarly to \cite{Dungey}, page 122, one has $\Gamma_p(f) \ge 0$.
Then the same argument as in \cite{LiChen}, Lemma 3.6, gives
\begin{align*}
\int_X\int_X p_t(x,y)\, |f(x)-f(y)|^p d\mu(x)d\mu(y)
&\le 
 C \int_X\int_X p_t(x,y) \brak{\gamma_p(f(x),f(y))+\gamma_p(f(y),f(x))}^{p/2} d\mu(x)d\mu(y)
\\&\le 
C \int_X\int_X p_t(x,y) \brak{\gamma_p^{p/2}(f(x),f(y))+\gamma_p^{p/2}(f(y),f(x))} d\mu(x)d\mu(y)
\\&= 
C\int_X\int_X p_t(x,y) \,\gamma_p^{p/2}(f(x),f(y))\, d\mu(x)d\mu(y)
\\&\le
C\int_X\brak{\int_X p_t(x,y)\, \gamma_p(f(x),f(y)) d\mu(y)}^{p/2} d\mu(x)
\\&= 
C\int_X \Gamma_p^{p/2}(f)(x)\, d\mu(x).
\end{align*}
We then follow the proof of Theorem 1 in \cite{LiChen} (see also Theorem 1.3 in  \cite{Dungey}). Let $u(s,x)=e^{-s\Delta_t} f (x)$ where $\Delta_t=\mathbf{Id} -P_t$. Note that
\begin{align*}
\Gamma_p(u)& =pu (  u - P_t u) -u^{2-p} (u^p -P_t (u^p)) \\
 & =pu \Delta_t u  -u^{2-p} \Delta_t (u^p) \\
 &=-pu \partial_s u  -u^{2-p} \Delta_t (u^p) \\
 & =-u^{2-p} \left( \partial_s +\Delta_t \right)u^p.
\end{align*}
Set now
\[
J(s,x)=- \left( \partial_s +\Delta_t \right)u^p(s,x),
\]
so that
\[
\Gamma_p(u)=u^{2-p} J.
\]
Note that since $u \ge 0$ and $\Gamma_p(u) \ge 0$, one has $J \ge 0$.
One has then from H\"older's inequality
\begin{align*}
\int_X \Gamma_p^{p/2}(u) d\mu & =\int_X u^{p(2-p)/2} J^{p/2} d\mu \\
 & \le \left( \int_X u^p d\mu \right)^{\frac{2-p}{2}} \left( \int_X J d\mu \right)^{p/2}.
\end{align*}
On computes then
\[
\int_X J d\mu =- \int_X \partial_s (u^p) d\mu=-p \int_X u^{p-1} \partial_s u d\mu=p \int_X u^{p-1} \Delta_t u d\mu.
\]
Thus, we have from H\"older's inequality
\[
\int_X J d\mu \le p \| u \|_{L^p(X,\mu)}^{p-1}  \| \Delta_t u \|_{L^p(X,\mu)}^{p} .
\]
From the definition of $\Delta_t$ one concludes therefore
\[
\left( \int_X\int_X p_t(x,y)\, |u(s,x)-u(s,y)|^p d\mu(x)d\mu(y)\right)^{1/p} \le C \| u(s,\cdot) \|^{1/2}_{L^p(X,\mu)} \| P_t u(s,\cdot) -u(s,\cdot) \|^{1/2}_{L^p(X,\mu)}.
\]
Letting $s \to 0^+$ yields from Fatou lemma
\[
\left( \int_X\int_X p_t(x,y)\, |f(x)-f(y)|^p d\mu(x)d\mu(y)\right)^{1/p} \le C \| f \|^{1/2}_{L^p(X,\mu)} \| P_t f -f \|^{1/2}_{L^p(X,\mu)}
\]
\end{proof}

We are now ready for the proof of Theorem \ref{continuity Besov chapter 1}.

\begin{proof}
Let $f \in L^p(X,\mu)$. We can assume $f \ge 0$. If not, it is enough to decompose $f$ as $f^+-f^-$ with $f^+=\max \{ f ,0 \}$ and $f^-=\max \{ -f ,0 \}$. Let $s,t >0$, applying Lemma \ref{Lemma interpolation} to $P_s f$, one obtains
\[
 \left( \int_X P_t (|P_sf-P_sf(y)|^p)(y) d\mu(y) \right)^{1/p} \le C_p \| P_s f \|^{1/2}_{L^p(X,\mu)} \| P_{t+s} f -P_sf \|^{1/2}_{L^p(X,\mu)}.
\]
Note that $\| P_s f \|_{L^p(X,\mu)} \le \|  f \|_{L^p(X,\mu)}$ and that
\begin{align*}
\| P_{t+s} f -P_sf \|_{L^p(X,\mu)} & = \left\| \int_0^t LP_{s+u} f  du \right\|_{L^p(X,\mu)} \\
 & =\left\| \int_0^t P_u LP_{s} f  du \right\|_{L^p(X,\mu)}  \\
 &  \le \int_0^t \left\|  P_u LP_{s} f \right\|_{L^p(X,\mu)} du \\
 & \le t  \left\|   LP_{s} f \right\|_{L^p(X,\mu)}  \\
 & \le \frac{t}{s}  \left\|    f \right\|_{L^p(X,\mu)},
\end{align*}
where in the last step we used analyticity of the semigroup. One concludes
\[
 \left( \int_X P_t (|P_sf-P_sf(y)|^p)(y) d\mu(y) \right)^{1/p} \le C_p \left(  \frac{t}{s} \right)^{1/2} \|  f \|_{L^p(X,\mu)}.
\]
Dividing both sides by $\sqrt{t}$ and taking the supremum over $t>0$ finishes the proof of Theorem \ref{continuity Besov chapter 1}.
\end{proof}

We now collect several corollaries of Theorem \ref{continuity Besov chapter 1}. The first one is the following:

\begin{corollary}
For $1 <  p \le 2$, $\B^{p,1/2}(X)$ is dense in $L^p(X,\mu)$.
\end{corollary}

\begin{proof}
Let $1<p \le 2$. Since for any $f \in L^p(X,\mu)$, $P_t f \in \B^{p,1/2}(X)$ and $P_t f \to f$ when $t \to 0$ (see for instance~\cite[Theorem 1.4.1]{EBD}), 
one concludes that $\B^{p,1/2}(X)$ is dense in $L^p(X,\mu)$.
\end{proof}

We then have the following second  corollary of Theorem \ref{continuity Besov chapter 1}. 

\begin{proposition}\label{Critical bound Chapter 1}
Assume that $\mathcal{E}$ is regular. Let $ 2 \le p <+\infty  $. For every $f \in L^p(X,\mu)$, and $t \ge 0$,
\[
\| P_t f -f \|_{L^p(X,\mu)} \le C_p   t^{1/2}  \liminf_{s \to 0}  s^{-1/2} \left( \int_X P_s (|f-f(y)|^p)(y) d\mu(y) \right)^{1/p}
\]
\end{proposition}

\begin{proof}
Let $f \in  L^p(X,\mu)$ and $g \in L^q(X,\mu) \cap \mathcal{F} $ where $q$ is the conjugate exponent of $p$. One has for $t \ge 0$,
\[
\int_X (P_t f -f ) g d\mu =  \int_0^t \mathcal{E}(P_s f ,g) ds  .
\]
The idea is now to bound $\mathcal{E}(g ,P_s f) $ by using an approximation of $\mathcal{E}$.
For  $\tau \in(0,\infty)$  we set
\begin{equation}\label{eq:energy-pt-Lp3}
\mathcal{E}_{\tau}(u,u):= \frac{1}{ \tau}  \langle (I-P_\tau)u,u\rangle.
\end{equation}
We have then $ \mathcal{E}(P_s f ,g) =\lim_{\tau \to 0} \mathcal{E}_{\tau}(P_s f ,g)$. Note now that $\mathcal{E}_{\tau}(P_s f ,g)=\mathcal{E}_{\tau}( f,P_s g)$ and that from H\"older inequality (applied as in the proof of Proposition \ref{interpolation inequality}) 
\begin{align*}
2 | \mathcal{E}_{\tau}(f,P_s g) | & \le \tau^{-1/2} \left( \int_X P_\tau  (|P_sg-P_sg(y)|^q)(y) d\mu(y) \right)^{1/q} \tau^{-1/2} \left( \int_X P_\tau (|f-f(y)|^p)(y) d\mu(y) \right)^{1/p} \\
 & \le  \tau^{-1/2} \left( \int_X P_\tau (|f-f(y)|^p)(y) d\mu(y) \right)^{1/p} \| P_s g \|_{q,1/2} \\
 &\le C_p  \tau^{-1/2} \left( \int_X P_\tau (|f-f(y)|^p)(y) d\mu(y) \right)^{1/p}  s^{-1/2} \| g \|_{L^q(X,\mu)}.
\end{align*}
One has therefore
\begin{align*}
\left| \int_X (P_t f -f ) g d\mu \right| \le  C_p t^{1/2} \| g \|_{L^q(X,\mu)} \liminf_{s \to 0}  s^{-1/2} \left( \int_X P_s (|f-f(y)|^p)(y) d\mu(y) \right)^{1/p},
\end{align*}
and we conclude by $L^p-L^q$ duality.
\end{proof}

\begin{corollary}\label{corollary 1.25}
Assume that $\mathcal{E}$ is regular. Let $ 2 \le p <+\infty  $ and $\alpha >1/2$. If $ f \in \B^{p,\alpha} (X)$ then $\mathcal{E}(f,f)=0$.
\end{corollary}

\begin{proof}
Indeed, for $ f \in \B^{p,\alpha} (X)$ with $\alpha >1/2$ one has
\[
 \liminf_{s \to 0}  s^{-1/2} \left( \int_X P_s (|f-f(y)|^p)(y) d\mu(y) \right)^{1/p}=0,
 \]
 so that for every $t \ge 0$, $P_t f=f$, and thus $\mathcal{E}(f,f)=0$.
\end{proof}

We finally point out the following other corollary of Theorem \ref{continuity Besov chapter 1}.

\begin{prop}
Let $1<p \le 2$. Let $L$ be the generator of $\mathcal{E}$ and $\mathcal{L}_p$ be the domain of $L$ in $L^p(X,\mu)$.  Then
\[
\mathcal{L}_p \subset \B^{p,1/2}(X)
\]
and for every $f \in \mathcal{L}_p$,
\begin{equation}\label{eq:multi2}
\|f\|^2_{p,1/2} \le C \norm{ Lf}_{L^p(X,\mu)} \| f\|_{L^p(X,\mu)}.
\end{equation}
\end{prop}
\begin{proof}
Write for $\lambda >0$
\[
R_{\lambda}f=(L-\lambda)^{-1} f=\int_0^{\infty} e^{-\lambda t} P_tf dt. 
\]
Consequently
\[
\|R_{\lambda}f\|_{p,1/2} \le \int_0^{\infty} e^{-\lambda t} \|P_tf\|_{p,1/2}  dt 
\le \int_0^{\infty} e^{-\lambda t}  \frac{C}{t^{1/2}} \|f\|_{L^p(X,\mu)}  dt\le C \lambda^{-1/2} \|f\|_{L^p(X,\mu)}.
\]
It follows that 
\[
\|f\|_{p,1/2} \le C  \lambda^{-1/2} \norm{(L-\lambda)f}_p  \le C ( \lambda^{-1/2} \|Lf\|_{L^p(X,\mu)}+ \lambda^{1/2} \|f\|_{L^p(X,\mu)}).
\]
Taking $\lambda=\|Lf\|_{L^p(X,\mu)} \|f\|_{L^p(X,\mu)}^{\,-1}$, we then get the result.

\end{proof}

\section{Besov critical exponents}

Let $p \ge 1$. We define the $L^p$ Besov  critical exponent of $X$ as
\[
\alpha^*_p(X)=\inf \{ \alpha >0\, :\, \mathbf{B}^{p,\alpha}(X) \text{ is trivial} \}.
\]
 $\mathbf{B}^{p,\alpha}(X)$ trivial means that any $f \in \mathbf{B}^{p,\alpha}(X)$ is constant. We have then the following result:

\begin{proposition}\label{Besov critical exponents}
Assume that $\mathcal{E}$ is regular and irreducible, i.e. $\mathcal{E}$ has the property that $\mathcal{E}(f,f)=0$ implies $f$ constant, then
\begin{enumerate}
\item $ \alpha^*_2(X)=\frac{1}{2}$;
\item $p \to \alpha^*_p(X)$ is non increasing;
\item For $p \ge 2$, $\alpha^*_p (X) \le \frac{1}{2}$;
\item For $1 \le p \le 2$, $ \frac{1}{2} \le \alpha^*_p(X) \le \frac{1}{p}$.
\end{enumerate}
\end{proposition}


\begin{proof}
\

\begin{enumerate}

\item Since $\mathbf{B}^{2,1/2}(X)=\mathcal{F}$, one has $\alpha^*_2(X) \ge \frac{1}{2}$. Then, from Corollary \ref{corollary 1.25}, one has $\alpha^*_2(X) \le \frac{1}{2}$.
\item 

For simplicity of notation, we assume in the proof that $P_t$ has a measurable heat kernel $p_t(x,y)$, but the proof easily extends to the general case. It suffices to show that $\alpha^*_p(X) \le \alpha^*_q (X)$ if $q\le p$. Indeed, first notice that  for any $a,b>0$ such that $a\neq b$, there holds
\[
|a^p-b^p||a^q-b^q|^{-1}\leq \frac{p}{q} \max\{a^q,b^q\}^{\frac{p}{q}-1}.
\]
Equivalently, 
\[
|a^p-b^p|^q \leq \brak{\frac{p}{q}}^q \max\{a^q,b^q\}^{p-q} |a^q-b^q|^{q}.
\]
Next by H\"older's inequality and Fubini's theorem, we have
 \begin{align*}
	 & \lefteqn{\int_X \int_X p_t(x,y)\bigl| |f(x)|^p-|f(y)|^p\bigr|^q \,d\mu(x)\,d\mu(y) } \quad&\\
	&\leq \brak{\frac{p}{q}}^q \int_X \int_X p_t(x,y) \bigl( |f(x)|^{p-q}+|f(y)|^{p-q}\bigr)\bigl| |f(x)|^q-|f(y)|^q\bigr|^q \,d\mu(x)\,d\mu(y) \\
	&\leq 2\brak{\frac{p}{q}}^q \int_X  |f(x)|^{p-q} \int_X p_t(x,y) \bigl| |f(x)|^q-|f(y)|^q\bigr|^q \,d\mu(y)\,d\mu(x) \\
	&\leq 2\brak{\frac{p}{q}}^q \int_X |f(x)|^{p-q} \Bigl(\int_X p_t(x,y) \bigl| |f(x)|^q-|f(y)|^q\bigr|^p \,d\mu(y)\Bigl)^{q/p} \, d\mu(x)\\
	&\leq 2\brak{\frac{p}{q}}^q \||f|^q\|_p^{p-q} \Bigl( \int_X \int_X p_t(x,y) \bigl| |f(x)|^q-|f(y)|^q\bigr|^p \,d\mu(x)\, d\mu(y) \Bigl)^{q/p}\\
	&\leq 2\brak{\frac{p}{q}}^q \||f|^q\|_p^{p-q} t^{\alpha q} \|| f|^q\|_{p,\alpha}^q,
	\end{align*}
which implies that 
\[
\| |f|^p\|_{q,\alpha}^q \leq 2\brak{\frac{p}{q}}^q  \||f|^q\|_p^{p-q} \|| f|^q\|_{p,\alpha}^q.
\]
Hence  $\alpha^*_p (X)\le \alpha^*_q(X)$.

\item This is Corollary \ref{corollary 1.25}.
\item This follows from Theorem \ref{continuity Besov chapter 1} and Proposition \ref{prop:BesovCEp}.
\end{enumerate}
\end{proof}

\begin{remark}\label{remark conjecture 1}
In view of the duality given by Corollary \ref{duality Dirichlet}, it is natural to conjecture that under suitable conditions one may have
\[
\alpha^*_p(X)+\alpha_q^*(X)=1
\]
if $p$ and $q$ are conjugate, i.e. satisfy $\frac{1}{p}+\frac{1}{q}=1$. 
\end{remark}

\begin{remark}
We will see in Chapters 4 and 5  that for local Dirichlet forms the limit $$\alpha_\infty(X)=\lim_{p \to +\infty} \alpha_p(X)$$ is closely related to a H\"older regularity property in space of the heat semigroup. If the conjecture in Remark \ref{remark conjecture 1} is true, then classical interpolation theory suggests that it is reasonable to expect that for every $p \ge 1$:
\[
\alpha^*_p (X) =\frac{1}{p} +\left( 1-\frac{2}{p} \right) \alpha_\infty(X).
\]
 \end{remark}
 
 \begin{example}
 For strongly local Dirichlet  forms with absolutely continuous energy measures, we will see in Chapter 4 that one generically has $\alpha^*_p(X)=\frac{1}{2}$ for every $p \ge 1$.
 \end{example}

%% file: chapter2.tex
\chapter{Sobolev and isoperimetric  inequalities }\label{section sobolev}

In this chapter, we are interested in Sobolev type embeddings (the case $p=1$ corresponds to isoperimetric type results) for the Besov spaces defined in the previous chapter. 
 
 Let $(X,\mu,\mathcal{E},\mathcal{F}=\mathbf{dom}(\mathcal{E}))$ be a symmetric Dirichlet space. Let $\{P_{t}\}_{t\in[0,\infty)}$ denote the Markovian semigroup associated with
$(X,\mu,\mathcal{E},\mathcal{F})$. Throughout the chapter, we shall assume that $P_t$ admits a measurable heat kernel $p_t(x,y)$.

Let $p \ge 1$ and $\alpha \ge 0$. As before, we define the Besov type seminorm:
\[
\| f \|_{p,\alpha}= \sup_{t >0} t^{-\alpha} \left( \int_X \int_X |f(x)-f(y) |^p p_t (x,y) d\mu(x) d\mu(y) \right)^{1/p}
\]
and define
\[
\mathbf{B}^{p,\alpha}(X)=\{ f \in L^p(X,\mu)\, :\,  \| f \|_{p,\alpha} <+\infty \}.
\]

Throughout the chapter, we assume that $\{P_{t}\}_{t\in(0,\infty)}$ is conservative and 
admits a heat kernel $p_{t}(x,y)$ satisfying, for some
$C>0$ and $\beta >0$,
\begin{equation}\label{eq:subGauss-upper3}
p_{t}(x,y)\leq C t^{-\beta}
\end{equation}
for $\mu\times\mu$-a.e.\ $(x,y)\in X\times X$, and for each $t\in\bigl(0,+\infty \bigr)$. 

Our goal in this chapter is to prove global Sobolev embeddings with sharp exponents for the space $\mathbf{B}^{p,\alpha}(X)$ and one of the main results will be the following weak-type Sobolev inequality and the corresponding isoperimetric inequality:

\begin{theorem}\label{polintro2}
Let $0<\alpha <\beta $. Let $1 \le p < \frac{\beta}{\alpha} $. There exists a constant $C_{p,\alpha} >0$ such that for every $f \in \mathbf{B}^{p,\alpha}(X) $,
\[
\sup_{s \ge 0}\, s\, \mu \left( \{ x \in X\, :\, | f(x) | \ge s \} \right)^{\frac{1}{q}} \le C_{p,\alpha} \| f \|_{p,\alpha},
\]
where $q=\frac{p\beta}{ \beta -p \alpha}$. Therefore, there exists a constant $C_{\emph{iso}} >0$, such that for every
subset set $E\subset X$ with $\mathbf{1}_E \in \mathbf{B}^{1,\alpha}(X) $
\[
\mu(E)^{\frac{\beta-\alpha}{\beta}} \le C_{\emph{iso}} \| \mathbf{1}_E \|_{1,\alpha}.
\]

\end{theorem}
\section{Weak type Sobolev inequality}\label{S:Weak_Sobolev}

We follow and adapt to our setting a general approach to Sobolev inequalities developed in \cite{BCLS} (see also  \cite{saloff2002}). The pseudo-Poincar\'e inequality proved in Lemma \ref{pseudo-Poincare} plays a fundamental role here.

\begin{lemma}\label{weak type}\label{gagliardo}
Let $1 \le p,q <+\infty$ and $\alpha >0$. There exists a constant $C_{p,q,\alpha}>0$ such that for every $f \in \mathbf{B}^{p,\alpha}(X) \cap L^q(X,\mu)$ and $s \ge 0$,
\[
\sup_{s \ge 0} s^{1 +q\frac{\alpha}{\beta}} \mu \left( \{ x \in X\, :\, | f(x) | > s \} \right)^{\frac{1}{p}}\le C_{p,q,\alpha} \| f \|_{p,\alpha} \| f \|_{L^q(X,\mu)}^{q\frac{\alpha}{\beta}}.
\]
\end{lemma}

\begin{proof}
We adapt an argument given in the proof of Theorem 9.1 in \cite{BCLS}.
Let $f \in \mathbf{B}^{p,\alpha}(X) $ and denote
\[
F(s)=\mu \left( \{ x \in X\, :\, | f(x) | > s \} \right)
\]
We have then
\[
F(s) \le \mu \left( \{ x \in X\, :\, | f(x) -P_t f (x) | > s/2 \} \right)+\mu \left( \{ x \in X\, :\, | P_t f (x) | > s/2 \} \right).
\]
Now, from the heat kernel upper bound
\[
p_t(x,y) \le \frac{C}{t^{\beta}}, \quad t>0,
\]
one deduces, that for $g \in L^1(X,\mu)$,
\[
| P_t g (x) | \le \frac{C}{t^{\beta}} \| g \|_{L^1(X,\mu)}.
\]
Since $P_t$ is a contraction in $L^\infty(X,\mu)$, by the Riesz-Thorin interpolation one gets that 
\[
| P_t f (x) | \le \frac{C^{1/q}}{t^{\beta /q}} \| f \|_{L^q(X,\mu)}.
\]

Therefore, for $s= 2 \frac{C^{1/q}}{t^{\frac{\beta}{ q}}} \| f \|_{L^q(X,\mu)}$, one has
\[
\mu \left( \{ x \in X\, :\, | P_t f (x) | > s/2 \} \right)=0.
\]
On the other hand, from  Theorem~\ref{pseudo-Poincare},
\[
\mu \left( \{ x \in X\, :\, | f(x) -P_t f (x) | > s/2 \} \right) \le 2^p s^{-p} t^{p\alpha} \| f \|^{p}_{p,\alpha}.
\]
We conclude that
\[
F(s)^{1/p} \le \tilde{C} s^{-1 -q\frac{\alpha}{\beta}} \| f \|_{\alpha,p} \| f \|_{L^q(X,\mu)}^{\frac{\alpha q}{\beta}}.
\]
\end{proof}

As a corollary, we are now ready to prove the weak Sobolev inequality.

\begin{theorem}\label{pol}
Let $0<\alpha <\beta $. Let $1 \le p < \frac{\beta}{\alpha} $. There exists a constant $C_{p,\alpha} >0$ such 
that for every $f \in \mathbf{B}^{p,\alpha}(X) $,
\[
\sup_{s \ge 0}\, s\, \mu \left( \{ x \in X\, :\, | f(x) | \ge s \} \right)^{\frac{1}{q}} \le C_{p,\alpha} \| f \|_{p,\alpha},
\]
where $q=\frac{p\beta}{ \beta -p \alpha}$.
\end{theorem}

\begin{proof}
Let $f \in \mathbf{B}^{p,\alpha}(X) $ be a non-negative function. For $k \in \mathbb{Z}$, we denote
\[
f_k=(f-2^k)_+ \wedge 2^k.
\]
Observe that $f_k \in L^p(X,\mu)$ and $\| f_k \|_{L^p(X,\mu)} \le \| f \|_{L^p(X,\mu)}$. Moreover, for every $x,y \in X$, 
$|f_k (x)-f_k(y)|\le |f(x)-f(y)|$ and so $\| f_k\|_{p,\alpha} \le \| f\|_{p,\alpha}$. We also note that $f_k \in L^1(X,\mu)$, with
\[
\| f_k \|_{L^1(X,\mu)} =\int_X |f_k| d\mu \le 2^k \mu (\{ x \in X\, :\, f(x) \ge 2^k \}).
\]
We now use Lemma \ref{gagliardo} to deduce:
\begin{align*}
\sup_{s \ge 0} s^{1 +\frac{\alpha}{\beta}} \mu \left( \{ x \in X\, :\, f_k(x) > s \} \right)^{\frac{1}{p}}
 & \le C_{p,\alpha} \| f_k \|_{p,\alpha} \| f_k \|_{L^1(X,\mu)}^{\frac{\alpha}{\beta}} \\
 & \le C_{p,\alpha}\| f_k \|_{p,\alpha} \left( 2^k \mu (\{ x \in X\, :\, f(x) \ge 2^k \})\right)^{\frac{\alpha}{\beta}}.
\end{align*}
In particular, by choosing $s=2^k$ we obtain
\[
2^{k \left(1 +\frac{\alpha}{\beta} \right)} \mu \left( \{ x \in X\, :\, f(x) \ge 2^{k+1} \} \right)^{\frac{1}{p}}
 \le C_{p,\alpha} \| f_k \|_{p,\alpha} \left( 2^k \mu (\{ x \in X\, :\, f(x) \ge 2^k \})\right)^{\frac{\alpha}{\beta}}.
\]
Let
\[
M(f)=\sup_{k \in \mathbb{Z}} 2^k \mu (\{ x \in X\, :\, f(x) \ge 2^k \})^{1/q}
\]
where $q=\frac{p\beta}{ \beta-p \alpha}$. Using the fact that $\frac{1}{q}=\frac{1}{p}-\frac{\alpha}{\beta}$ and the 
previous inequality we obtain:
\[
2^{k} \mu \left( \{ x \in X\, :\, f(x) \ge 2^{k+1} \} \right)^{\frac{1}{p}}
 \le 2^{ -\frac{kq\alpha}{\beta}} C_{p,\alpha}\| f \|_{p,\alpha} M(f)^{\frac{q\alpha}{\beta}}.
\]
and
\[
2^k \mu \left( \{ x \in X\, :\, f(x) \ge 2^{k+1} \} \right)^{\frac{1}{q}}
 \le C^{\frac{p}{q} }_{p,\alpha} \| f \|_{p,\alpha}^{p/q}M(f)^{\frac{p\alpha}{\beta}}.
\]
Therefore
\[
M(f)^{1-\frac{p\alpha}{\beta}} \le 2 C^{\frac{p}{q} }_{p,\alpha} \| f \|_{p,\alpha} ^{p/q}.
\]
One concludes
\[
M(f) \le 2^{q/p} C_{p,\alpha} \| f \|_{p,\alpha} .
\]
This easily yields:
\[
\sup_{s \ge 0} s \mu \left( \{ x \in X\, :\, f(x) \ge s \} \right)^{\frac{1}{q}} \le 2^{1+q/p} C_{p,\alpha} \| f \|_{p,\alpha} .
\]
Let now $f \in \mathbf{B}^{p,\alpha}(X) $, which is not necessarily non-negative. From the previous inequality applied to $|f|$, we deduce
\[ 
\sup_{s \ge 0}\, s\, \mu \left( \{ x \in X\, :\, |f(x) | \ge s \} \right)^{\frac{1}{q}} 
\le 2^{1+q/p} C_{p,\alpha} \| | f | \|_{p,\alpha} 
  \le 2^{1+q/p} C_{p,\alpha} \| f \|_{p,\alpha} .
\] 
\end{proof}

\section{Capacitary estimates}

It is well-known that Sobolev inequalities are related to capacitary estimates. Let $p \ge 1$ and $0< \alpha <\beta$. For a measurable set $A \subset X$, we define its $(\alpha,p)$ capacity:
\[
\mathbf{Cap}^\alpha_p (A)=\inf \{ \| f \|_{\alpha,p}^p\, :\, f \in \mathbf{B}^{\alpha,p}(X), \mathbf{1}_A \le f \le 1 \}.
\]

We have the following theorem:

\begin{theorem}
Let $0<\alpha <\beta $. Let $1 \le p < \frac{\beta}{\alpha} $. There exists a constant $C_{p,\alpha}>0$ such that for every measurable set $A \subset X$,
\[
\mu(A)^{1-\frac{p\alpha}{\beta}}\le C_{p,\alpha} \mathbf{Cap}^\alpha_p (A).
\]
\end{theorem}

\begin{proof}
This is an immediate corollary of Proposition \ref{pol}.
\end{proof}

\section{Isoperimetric inequality}\label{S:IsopIneq}

Let $E \subset X$ be a measurable set with finite measure. We will say that $E$ has a finite $\alpha$-perimeter if $\mathbf{1}_E \in \mathbf{B}^{1,\alpha}(X)$. In that case, we will denote
\[
P_\alpha(E)= \| \mathbf{1}_E \|_{1,\alpha}.
\]

\begin{proposition}
Let $0 <\alpha < \beta$. There exists a constant $C_{\emph{iso}} >0$, such that for every
subset  $E\subset X$ with finite $\alpha$-perimeter
\[
\mu(E)^{\frac{\beta-\alpha}{\beta}} \le C_{\emph{iso}} P_\alpha (E).
\]
\end{proposition}

\begin{proof}
Observe that we have
\[
||P_t \mathbf 1_E - \mathbf 1_E||_{L^1(X,\mu)} = 2 \left(\mu(E) -
\int \left(P_{t/2}\mathbf 1_E\right)^2 d\mu\right).
\]
Indeed, 
\begin{align*}
\Vert P_t \mathbf{1}_E -\mathbf 1_E  \Vert_{L^1(X,\mu)} =& \int_E (1-P_t \mathbf 1_E ) d\mu + \int_{E^c} P_t(\mathbf 1_E ) d\mu\\
 =& \int_E (1-P_t \mathbf 1_E ) d\mu + \int_E (P_t \mathbf1_{E^c}) d\mu\\
 =& 2 \left( \mu(E)- \int_E P_t (\mathbf 1_E ) d\mu \right)\\
 =& 2 \left( \mu(E)- \Vert P_\frac{t}{2} (\mathbf 1_E ) \Vert_{L^2(X,\mu)}^2 \right),
\end{align*}
where the last inequality is due to the fact that
\[
\int_E P_t \mathbf 1_E d\mu = \int \left(P_{t/2}\mathbf
1_E\right)^2 d\mu.
\]
We now note that 
\begin{align*}
\int (P_{t/2} \mathbf 1_E)^2 d\mu & \le \left(\int_E
\left(\int p_{t/2}(x,y)^2
d\mu(y)\right)^{\frac{1}{2}}d\mu(x)\right)^2
\\
& = \left(\int_E p_t(x,x)^{\frac{1}{2}}d\mu(x)\right)^2 \le
\frac{A}{t^{\beta}} \mu(E)^2.
\end{align*}
for some constant $A>0$.
Combining these equations yields
\[
\mu(E) \le B t^\alpha \ P_\alpha (E) +
\frac{C}{t^{\beta}} \mu(E)^2,\ \ \ \ t>0,
\]
for some positive constants $B,C$. Optimizing in $t$ concludes the proof.
\end{proof}

In the limiting case $\alpha=\beta$, the previous proof yields the following:

\begin{corollary}
There exists a constant $C_{\emph{iso}} >0$, such that for every
subset  $E\subset X$ with finite $\beta$-perimeter and $\mu(E) >0$,
\[
P_\beta(E) \ge C_{\emph{iso}} .
\]

\end{corollary}

\section{Strong Sobolev inequality}

\begin{definition}\label{chaining}
We say that the Dirichlet space satisfies the property $(P_{p,\alpha})$ if there exists a constant $C>0$ such that for every 
$f \in \mathbf{B}^{p,\alpha}(X)$,
\[
\| f \|_{p,\alpha} \le C \liminf_{t \to 0} t^{-\alpha} \left( \int_X \int_X |f(x)-f(y) |^p p_t (x,y) d\mu(x) d\mu(y) \right)^{1/p}
\]
\end{definition}

\begin{remark}
In the Chapter  4 framework,  property $(P_{p,\alpha})$  is satisfied when $p=1, \alpha=1/2$, see  Theorem \ref{thm:W=BV}.
\end{remark}

Our main theorem is the following:

\begin{theorem}\label{Sobolev}
Assume that the Dirichlet space satisfies the property $(P_{p,\alpha})$
and that $\beta$ satisfies~\eqref{eq:subGauss-upper3}.
Let $0<\alpha <\beta $. Let $1 \le p < \frac{\beta}{\alpha} $. There exists a 
constant $C_{p,\alpha,\beta}>0$ such that for every $f \in \mathbf{B}^{p,\alpha}(X)$,
\[
\| f \|_{L^q(X,\mu)} \le C_{p,\alpha,\beta} \| f \|_{p,\alpha},
\]
where $q=\frac{p\beta}{ \beta-p \alpha}$.
\end{theorem}

Note that in the standard Euclidean setting of $\mathbb{R}^n$ the Sobolev embedding theorem holds as above with
$\beta=n$.
To show that the weak type inequality implies the desired Sobolev inequality, we will need another cutoff argument and the following lemma is needed.

\begin{lemma}
For $f \in \mathbf{B}^{p,\alpha}(X)$, $f \ge 0$, denote $f_k=(f-2^k)_+ \wedge 2^k$, $k \in \mathbb{Z}$. There exists a constant $C>0$ such that for every $f \in \ \mathbf{B}^{p,\alpha}(X)$,
\[
\left( \sum_{k \in \mathbb{Z}} \| f_k\|_{p,\alpha}^p \right)^{1/p} \le C \| f \|_{p,\alpha}.
\]
\end{lemma}

\begin{proof}
By a similar type of argument, as in the the proof of Lemma 7.1 in \cite{BCLS}, one has for some constant $C_p>0$,
\[
\sum_{k \in \mathbb{Z}} \int_X \int_X |f_k (x)-f_k(y)|^p p_t(x,y)d\mu \le C_p \int_X \int_X |f (x)-f(y)|^p p_t(x,y)d\mu.
\]
As a consequence of property $(P_{p,\alpha})$,
\[
\left( \sum_{k \in \mathbb{Z}} \| f_k\|_{p,\alpha}^p \right)^{1/p} \le C'_p \| f \|_{p,\alpha}.
\]
and the proof is complete.
\end{proof}

We can now conclude the proof of Theorem \ref{Sobolev}.

\begin{proof}[\textbf{Proof of Theorem \ref{Sobolev}}]
Let $f \in \mathbf{B}^{p,\alpha}(X)$. We can assume $f \ge 0$. As before, denote $f_k=(f-2^k)_+ \wedge 2^k$, $k \in \mathbb{Z}$. From Lemma \ref{pol} applied to $f_k$, we see that

\[
\sup_{s \ge 0} s \mu \left( \{ x \in X\, :\, | f_k(x) | \ge s \} \right)^{\frac{1}{q}} \le C_{p,\alpha} \| f_k \|_{p,\alpha}
\]
In particular for $s=2^k$, we get
\[
2^k \mu \left( \{ x \in X\, :\, f (x) \ge 2^{k+1} \} \right)^{\frac{1}{q}} \le C_{p,\alpha} \| f_k \|_{p,\alpha}.
\]
Therefore,
\[
\sum_{k \in \mathbb{Z}} 2^{k q} \mu \left( \{ x \in X\, :\, f (x) \ge 2^{k+1} \} \right)
\le C_{p,\alpha}^q \sum_{k \in \mathbb{Z}} \| f_k \|_{p,\alpha}^q.
\]
Since $q \ge p$, one has 
$\sum_{k \in \mathbb{Z}} \| f_k \|_{p,\alpha}^q\le \left( \sum_{k \in \mathbb{Z}} \| f_k \|_{p,\alpha}^p \right)^{q/p}$. 
Thus, from the previous lemma
\[
\sum_{k \in \mathbb{Z}} 2^{k q} \mu \left( \{ x \in X\, :\, f (x) \ge 2^{k+1} \} \right)\le C_{p,\alpha}^q \| f_k \|_{p,\alpha}^q.
\]
Finally, we observe that
\begin{align*}
\sum_{k \in \mathbb{Z}} 2^{k q} \mu \left( \{ x \in X\, :\, f (x) \ge 2^{k+1} \} \right) 
&\ge \frac{q}{2^{q+1}-2^q} \sum_{k \in \mathbb{Z}}\int_{2^{k+1}}^{2^{k+2}} s^{q-1} \mu \left( \{ x \in X\, :\, f (x) \ge s \} \right)ds \\
 &\ge \frac{1}{2^{q+1}-2^q} \| f \|_{L^q(X,\mu)}^q.
\end{align*}
The proof is thus complete.
\end{proof}

%% file: chapter3.tex
\chapter{Cheeger constant and Gaussian isoperimetry}

While the previous chapter was devoted to Sobolev inequalities on Dirichlet spaces for which the semigroup satisfies ultracontractive estimates, the present chapter is devoted to situations where the Dirichlet form satisfies a Poincar\'e inequality or a log-Sobolev inequality.
Let $(X,\mu,\mathcal{E},\mathcal{F}=\mathbf{dom}(\mathcal{E}))$ be a symmetric Dirichlet space. Let $\{P_{t}\}_{t\in[0,\infty)}$ denote the Markovian semigroup associated with $(X,\mu,\mathcal{E},\mathcal{F})$. As usual, we assume that $P_t$ is conservative.

\section{Buser's type  inequality for the Cheeger constant of a Dirichlet space}
In the context of  a smooth compact  Riemannian manifold with Riemannian measure $\mu$, Cheeger  \cite{Ch}  introduced   the following isoperimetric constant
\[
h=\inf \frac{\mu (\partial A)}{\mu (A)},
\]
where the infimum runs over all open subsets $A$ with smooth boundary $\partial A$ such that $\mu(A)\le \frac{1}{2}$. Cheeger's constant can be used to bound from below the first non zero eigenvalue of the manifold. Indeed, it is proved in \cite{Ch} that
\[
\lambda_1 \ge \frac{h^2}{4}.
\]

Buser \cite{Bu} then proved that if the Riemannian Ricci curvature of the manifold is non-negative, then we actually have
\[
\lambda_1 \le C h^2
\]
where $C$ is a universal constant depending only on the dimension. Buser's inequality was reproved by Ledoux \cite{Le} using heat semigroup techniques. Under proper assumptions, by using the tools we introduced, Ledoux' technique can essentially reproduced in our general framework of Dirichlet spaces.

We assume in this section that $\mathcal{E}$ satisfies a spectral gap inequality and normalize the measure $\mu$ in such a way that $\mu(X)=1$.  For $\alpha \in (0,1]$, we define the $\alpha$-Cheeger's constant  of $X$ by
\[
h_\alpha=\inf \frac{\| \mathbf 1_E \|_{1,\alpha}}{\mu(E)}
\]
where the infimum runs over all measurable sets $E$ such that $\mu(E)\le \frac{1}{2}$ and $\mathbf 1_E \in \mathbf{B}^{1,\alpha}(X)$. We denote by $\lambda_1$ the spectral gap of $\mathcal{E}$.

\begin{theorem}
We have $h_\alpha \ge (1-e^{-1}) \lambda^\alpha_1$.
\end{theorem}

\begin{proof}
Let $A$ be a set with $P_\alpha (A) :=\| \mathbf 1_A \|_{1,\alpha}< +\infty$. As before, by symmetry and stochastic completeness of the semigroup, we have
\begin{align*}
\Vert \mathbf 1_A - P_t \mathbf 1_A \Vert_{L^1(X,\mu)} =& \int_A (1-P_t \mathbf 1_A) d\mu + \int_{A^c} P_t(\mathbf 1_A) d\mu\\
 =& \int_A (1-P_t \mathbf 1_A) d\mu + \int_A (P_t \mathbf 1_{A^c}) d\mu\\
 =& 2 \left( \mu(A)- \int_A P_t (\mathbf 1_A) d\mu \right)\\
 =& 2 \left( \mu(A)- \Vert P_\frac{t}{2} (\mathbf 1_A) \Vert_{L^2(X,\mu)}^2 \right).
\end{align*}
By definition, we have
\[
\| P_t \mathbf 1_A -\mathbf 1_A \|_{L^1(X,\mu)} \le t^\alpha P_\alpha(A).
\]
We deduce that
\[
\mu(A) \le \frac{1}{2} t^\alpha P_\alpha(A)+ \Vert P_\frac{t}{2} (\mathbf 1_A) \Vert_{L^2(X,\mu)}^2.
\]
Now, by the spectral theorem,
\[
\Vert P_\frac{t}{2} (\mathbf 1_A) \Vert_{L^2(X,\mu)}^2=\mu(A)^2 + \Vert P_\frac{t}{2} (\mathbf 1_A-\mu(A)) \Vert_{L^2(X,\mu)}^2 \le \mu(A)^2 +e^{-\lambda_1 t} \| \mathbf 1_A-\mu(A) \|_{L^2(X,\mu)}^2
\]
This yields
\[
\mu(A) \le \frac{1}{2} t^\alpha P_\alpha(A)+ \mu(A)^2 +e^{-\lambda_1 t} \| \mathbf 1_A-\mu(A) \|_{L^2(X,\mu)}^2.
\]
Equivalently, one obtains
\[
\frac{1}{2} t^\alpha P_\alpha(A) \ge \mu(A) (1 -\mu(A))(1-e^{-\lambda_1 t}).
\]
Therefore,
\[
h_\alpha \ge \sup_{t >0} \left( \frac{1-e^{-\lambda_1 t}}{t^\alpha} \right).
\]
\end{proof}

As already noted in \cite{BK}, let us observe that it is known that the Cheeger lower bound on $\lambda_1$ may be obtained under further assumptions on the Dirichlet space $(X,d,\mathcal{E})$. Indeed, assume  that $\mathcal{E}$ is strictly local with a carr\'e du champ $\Gamma$,  that Lipschitz functions are in the domain of $\mathcal{E}$ and that $\sqrt{\Gamma(f)}$ is an upper gradient in the sense that for any Lipchitz function $f$,
\[
\sqrt{\Gamma(f)}(x) =\lim \sup_{d(x,y)\to 0} \frac{ |f(x)-f(y)|}{d(x,y)}.
\]
In that case, if $A$ is a closed set of $X$, one defines its Minkowski exterior boundary measure by
\[
\mu_+(A)=\lim \inf_{\varepsilon \to 0} \frac{1}{\varepsilon} \left( \mu(A_\varepsilon) -\mu(A) \right),
\]
where $A_\varepsilon=\{ x \in X, d(x,A) <\varepsilon \}$. We can then define a Cheeger's constant of $X$ by
\[
h_+=\inf \frac{\mu_+(E)}{\mu(E)}
\]
where the infimum runs over all closed sets $E$ such that $\mu(E)\le \frac{1}{2}$. Then, according to Theorem 8.5.2 in \cite{BGL}, one has
\[
\lambda_1 \ge \frac{h^2_+}{4}.
\]

\section{Log-Sobolev and Gaussian isoperimetric inequalities}

We assume in this section that $\mathcal{E}$ satisfies a log-Sobolev inequality inequality and normalize the measure $\mu$ in such a way that $\mu(X)=1$. 

We define the Gaussian isoperimetric constant of $X$ by
\[
k=\inf \frac{\| \mathbf{1}_E \|_{1,1/2}}{\mu(E)\sqrt{-\ln \mu(E)}}
\]
where the infimum runs over all sets $E$ such that $\mu(E)\le \frac{1}{2}$ and $\mathbf{1}_E \in \mathbf{B}^{1,1/2}(X)$. We denote by $\rho_0$ the log-Sobolev constant of $X$, that is the best constant in the inequality 
\begin{equation}\label{log Sob}
\int f^2 \ln f^2 d\mu -\int f^2 d\mu \ln \int f^2 d\mu \le \frac{1}{\rho_0} \mathcal{E}(f,f).
\end{equation}

Following an argument of M. Ledoux \cite{Le}, one obtains:

\begin{theorem} 

\[
\rho_0 \le C_{\text{ledoux}} k^2
\]
where $C_{\text{ledoux}}$ is a numerical constant.
\end{theorem}

\begin{proof}
Let $A$ be a measurable set such that $P(A):=\| \mathbf 1_A \|_{1,1/2} <+\infty$. By the same computations as before we have
\[
\mu(A) \le \frac{1}{2} \sqrt{t} P(A)+ \Vert P_\frac{t}{2} (\mathbf 1_A) \Vert_{L^2(X,\mu)}^2.
\]
Now we can use the hypercontractivity constant to bound $\Vert P_\frac{t}{2} (\mathbf 1_A) \Vert_2^2$. Indeed, from Gross' theorem it is well known that the logarithmic Sobolev inequality 
\[
\int f^2 \ln f^2 d\mu -\int f^2 d\mu \ln \int f^2 d\mu \le \frac{1}{\rho_0} \mathcal{E}(f,f),
\]
is equivalent to hypercontractivity property
$$
\Vert P_t f \Vert_{L^q(X,\mu)} \leq \Vert f \Vert_{L^p(X,\mu)}
$$
 for all $f$ in $L^p(X,\mu)$ whenever $1<p<q<\infty$ and $e^{\rho_0 t}\geq \sqrt \frac{q-1}{p-1}$.

Therefore, with $p(t)= 1+e^{-2\rho_0 t}<2$, we get,
\begin{align*}
 \sqrt{t} P(A) \geq & 2 \left( \mu(A)- \mu(A)^\frac{2}{p(t)} \right)\\
 &\geq 2 \mu(A) \left(1- \mu(A)^\frac{1-e^{2-\rho_0 t}}{1+e^{-2\rho_0 t}} \right).
\end{align*}
By using then the computation page 956 in \cite{Le}, one deduces that if $A$ is a set which has a finite $P(A)$ and such that $0\leq \mu(A)\leq \frac{1}{2}$, then 
$$
P( A)\geq \tilde{C} \sqrt{ \rho_0} \mu(A)\left(\ln \frac{1}{\mu(A)} \right)^\frac{1}{2},
$$
where $\tilde{C}$ is a numerical constant.
\end{proof}

%% file: chapter4.tex
\chapter{Strictly local Dirichlet spaces with doubling measure and 2-Poincar\'e inequality}\label{Sec:Metric-Curvature}

In this chapter we define the space  $BV(X)$ of functions of 
bounded variation when $\mathcal{E}$ is strcitly local (see the exact assumptions in Section 4.1). 
We shall then prove that under a weak Bakry-\'Emery curvature dimension condition, 
one actually has $BV(X)=\mathbf{B}^{1,1/2}(X)$ with comparable seminorms. One of the key tools used there is the co-area
formula for BV functions; we establish this formula in Lemma~\ref{lem:Co-area} of the present chapter. We will also study  in that framework  the  spaces $\B^{p,\alpha}(X)$ introduced in Chapter 1 and relate them to Besov spaces previously considered in the literature. Finally, applications to Sobolev and isoperimetric inequalities are given.

\section{Local Dirichlet spaces and standing assumptions}\label{SubSec:StrongLocal}

We will assume in this section that $(X,\mu)$ is a topological space equipped with a Radon measure $\mu$,
and that there is a strongly local Markovian Dirichlet form $\mathcal{E}$ on $X$.
As $\mathcal{E}$ is strongly local, it is represented by a signed Radon measure 
$\Gamma(u,v)$ for each $u,v\in \mathcal{F}$ such that whenever $\varphi$ is a bounded function in $\mathcal{F}$ and
$u,v\in\mathcal{F}$, we have
\[
\mathcal{E}(\varphi u,v)+\mathcal{E}(\varphi v,u)-\mathcal{E}(\varphi, uv)=2\int_X\varphi\, d\Gamma(u,v).
\]
We then also have $\mathcal{E}(u,v)=\int_X d\Gamma(u,v)$, see for example~\cite{St-I}.
With respect to $\mathcal{E}$ we can define the following \emph{intrinsic metric} $d_{\mathcal{E}}$:
\[
d_{\mathcal{E}}(x,y)=\sup\{u(x)-u(y)\, :\, u\in\mathcal{F}\cap C(X)\text{ and } \Gamma(u,u)\le \mu\}.
\]
Here by $\Gamma(u,u)\le \mu$ we mean that for each Borel set $A\subset X$ we have 
$\Gamma(u,u)(A)=\int_Ad\Gamma(u,u)\le \mu(A)$.
In general there is no reason why $d_{\mathcal{E}}$ is a metric on $X$ (for it could be infinite for a given pair of points $x,y$
or zero for some distinct pair of points), but \emph{suppose that} $d_{\mathcal{E}}$ is a metric on $X$ such that the topology
generated by $d_{\mathcal{E}}$ coincides with the topology on $X$ and that balls with respect to $d_{\mathcal{E}}$ have
compact closures in this topology. Then by~\cite[Lemma~1, Lemma~$1^\prime$]{St-I}, functions $\varphi$ of the form 
$x\mapsto d_{\mathcal{E}}(x,y)$ for each $y\in X$, and maps $\varphi$ of the form $x\mapsto (r-d_{\mathcal{E}}(x,y))_+$ for each
$r>0$ and $y\in X$ are in $\mathcal{F}_{loc}(X)$ with $d\Gamma(\varphi,\varphi)/d\mu\le 1$.
So we will also assume in this section that $d_\mathcal{E}$ gives a metric on $X$ that is commensurate with the topology
on $X$, and that with respect to this metric the measure $\mu$ is \emph{doubling}, that is, there is a constant $C>0$ such that
whenever $x\in X$ and $r>0$, we have $\mu(B(x,2r))\le C\, \mu(B(x,r))$.

\begin{lemma}\label{lem:diff-of-Lip}
Let $f:X\to\R$ be locally Lipschitz continuous with respect to the metric $d_{\mathcal{E}}$. Then $f\in\mathcal{F}_{loc}(X)$
with $\Gamma(f,f)\ll\mu$. If $f$ is locally $L$-Lipschitz, then $\Gamma(f,f)\le L\, \mu$.
\end{lemma} 

\begin{proof}
Let $Q$ be a countable dense subset of $X$.
Let $U\subset X$ be a bounded open set and let $\{q_i\}_{i\in I\subset\mathbb{N}}$ be an enumeration of $Q\cap U$.
Note that $Q\cap U$ is dense in $U$. For each $i\in I$ let $\psi_i(x)=d_{\mathcal{E}}(x,q_i)$. Then as explained above,
$\psi_i\in\mathcal{F}(U)$ with $\Gamma(\psi_i,\psi_i)\le \mu$. For $j\in I$ set
\[
f_j(x):=\inf\{f(q_i)+L\, \psi_i(x)\, :\, i\in I\text{ with }i\le j\},
\]
where $L\ge 0$ is the Lipschitz constant of $f$ in $U$. The above functions are inspired by the proof of the McShane extension
theorem (see for example~\cite{Hei}). By the lattice properties of Dirichlet form, it is seen that each 
$f_j\in\mathcal{F}(U)$ with $d\Gamma(f_j,f_j)\le L$. Furthermore, $f_j$ are $L$-Lipschitz in $U$ with $f_j(q_i)=f(q_i)$ for
$i\in I$ with $i\le j$. 

We can see that $f_j\to f$ monotonically and hence (as $f$ and $f_j$ are bounded in $U$ because $U$ is bounded)
$f_j\to f$ in $L^2(U)$, with $d\Gamma(f,f)/d\mu\le L$ on $U$. 
\end{proof}

Thus locally Lipschitz functions $f$ with respect to $d_{\mathcal{E}}$ satisfy 
\begin{enumerate}
\item $f\in \mathcal{F}_{loc}(X)$,
\item $\Gamma(f,f)$ is absolutely continuous with respect to the underlying measure $\mu$; we denote its Radon-Nikodym 
derivative by $|\nabla f|^2$, that is, $|\nabla f|$ is the square-root of its Radon-Nikodym derivative.
\end{enumerate}
We say that $(X,\mu,\mathcal{E},\mathcal{F})$ supports a $1$-Poincar\'e inequality if there are constants
$C,\lambda>0$ such that whenever $B$ is a ball in $X$ (with respect to the metric $d_{\mathcal{E}}$)
and $u\in\mathcal{F}$, we have
\[
\frac{1}{\mu(B)}\int_B|u-u_B|\, d\mu\le C\, \text{rad}(B)\, \frac{1}{\mu(\lambda B)}\, \int_{\lambda B}|\nabla u|\, d\mu.
\]
Of course, the $1$-Poincar\'e inequality does not make sense if $\mathcal{E}$
does not satisfy the condition of strong locality. 
We say that $(X,\mu,\mathcal{E},\mathcal{F})$ supports  $2$-Poincar\'e
inequality if there are constants $C,\lambda>0$ such that 
whenever $B$ is a ball in $X$ (with respect to the metric $d_{\mathcal{E}}$)
and $u\in\mathcal{F}$, we have
\[
\frac{1}{\mu(B)}\int_B|u-u_B|\, d\mu\le C\, \text{rad}(B)\, \left(\frac{1}{\mu(\lambda B)}\, \int_{\lambda B}d\Gamma(u,u)\right)^{1/2}.
\]
The requirement that $\mathcal{E}$ supports a $1$-Poincar\'e inequality is significantly
a stronger requirement than that of $2$-Poincar\'e inequality. Much of the current
theory on functions of bounded variation in the metric setting require a
$1$-Poincar\'e inequality. In this chapter we will \emph{not} require the support
of $1$-Poincar\'e inequality, only the weaker $2$-Poincar\'e inequality, but in
some of the analysis we would need an additional requirement called the weak
Bakry-\'Emery curvature condition, see the discussion below.

Should the 2-Poincar\'e inequality be satisfied, a standard argument, due to Semmes, tells us that locally Lipschitz continuous functions
forms a dense subclass of $\mathcal{F}$, where $\mathcal{F}$ is equipped with the norm
\[
\Vert u\Vert_{\mathcal{F}}:=\Vert u\Vert_{L^2(X)}+\sqrt{\mathcal{E}(u,u)},
\]
see for example~\cite[Chapter~8]{HKST}.

\

\noindent {\bf Standing assumptions for this chapter:} For this chapter we will assume that $\mathcal{E}$ is a strongly local
Markovian Dirichlet form that induces a metric $d=d_\mathcal{E}$ with respect to which $\mu$ is doubling and supports a 
$2$-Poincar\'e inequality. We will also assume that the class of locally Lipschitz continuous functions on $X$ forms a dense
subclass of $L^1(X)$ and that $X$ is complete with respect to $d_\mathcal{E}$. As a consequence of
the doubling property of $\mu$ it follows that closed and bounded subsets of $X$ are compact.

The work of Saloff-Coste~\cite{Sal-Cos} tells us that the above standing assumptions are equivalent to the property that there is an associated
heat kernel function $p_t(x,y)$ on $[0,\infty)\times X\times X$ for which there are constants $c_1, c_2, C >0$ such that
whenever $t>0$ and $x,y\in X$,
\begin{equation}\label{eq:heat-Gauss}
 \frac{1}{C}\, \frac{e^{-c_1 d(x,y)^2/t}}{\sqrt{\mu(B(x,\sqrt{t}))\mu(B(y,\sqrt{t}))}}\le p_t(x,y)
   \le C\, \frac{e^{-c_2 d(x,y)^2/t}}{\sqrt{\mu(B(x,\sqrt{t}))\mu(B(y,\sqrt{t}))}}.
\end{equation}
The above inequality is called the \emph{Gaussian bounds} for the heat kernel. Examples of spaces that support such a Dirichlet
form include complete Riemannian manifolds with non-negative Ricci curvature, Carnot groups and other complete sub-Riemannian
manifolds, and doubling metric measure spaces that support a $2$-Poincar\'e inequality with respect to the upper gradient
structure of Heinonen and Koskela (see~\cite{HKST}).

\

In some parts of this chapter we need a condition that guarantees that the space is not negatively curved with respect to the
Dirichlet form. To do so, we will assume that the metric space $(X,d_\mathcal{E},\mu)$
satisfies a weak Bakry-\'Emery curvature condition,
namely that whenever $u\in \mathcal{F}\cap L^\infty(X)$,
\begin{equation}\label{eq:weak-BE}
\Vert |\nabla P_tu| \Vert_{L^\infty(X)}^2\le \frac{C}{t}\Vert u\Vert_{L^\infty(X)}^2.
\end{equation}
Here $P_tu$ denotes the heat extension of $u$ via the heat semigroup associated with the Dirichlet form $\mathcal{E}$, see~\cite{FOT}.

\

\begin{example}
\em The weak Bakry-\'Emery curvature condition is satisfied in the following examples:
\begin{itemize}
\item Complete Riemannian manifolds with non-negative Ricci curvature and more generally, the $RCD(0,+\infty)$ spaces (\cite{Jiang15}).
\item Carnot groups (see \cite{BB2})
\item Complete sub-Riemannian manifolds with generalized non-negative Ricci curvature (see \cite{BB,BK14})
\item Metric graphs with finite number of edges (see \cite{BK})
\end{itemize}
Several statements equivalent to the weak Bakry-\'Emery curvature condition are given in Theorem~1.2 in \cite{CJKS}.  There are some metric measure spaces equipped with a doubling measure supporting a $2$-Poincar\'e inequality but without the above Bakry-\'Emery condition, see for example~\cite{KRS}.  For instance, it should be noted, that in the setting of complete sub-Riemannian manifolds with generalized non-negative Ricci curvature in the sense of \cite{BG},  while the weak Bakry-\'Emery curvature condition is known to be satisfied (see \cite{BB,BK14}), the 1-Poincar\'e inequality has not been proven yet (though the 2-Poincar\'e is known, see \cite{BBG}).
\end{example}

\section{Construction of BV class using the Dirichlet form}\label{subsec:BV}

In this section we use the Dirichlet form and the associated family $\Gamma(\cdot,\cdot)$ of measures to construct a BV class
of functions on $X$. To do so, we only need $\mu$ to be a doubling measure on $X$ for $d_\mathcal{E}$ and that the class of 
locally Lipschitz functions to be dense in $L^1(X)$. So in this section we do \emph{not} need the Poincar\'e inequality
nor do we need the weak Bakry-\'Emery curvature condition.

We now set the \emph{core} of the Dirichlet form,
$\mathcal{C}(X)$, to be the class of all $f\in\mathcal{F}_{loc}(X)$ such that $\Gamma(f,f)\ll \mu$.
In this case, we can also set $u\in L^1(X)$ to be in $BV(X)$ if there is a sequence of local Lipschitz functions $u_k\in L^1(X)$
such that $u_k\to u$ in $L^1(X)$ and 
\[
\liminf_{k\to\infty}\int_X|\nabla u_k|\, d\mu<\infty.
\]
For $u\in BV(X)$ and open sets $U\subset X$, we set
\[
\Vert Du\Vert(U)=\inf_{u_k\in\mathcal{C}(U), u_k\to u\text{ in }L^1(U)}\liminf_{k\to\infty} \int_U|\nabla u_k|\, d\mu,
\]
and then for sets $A\subset X$ we define
\[
\Vert Du\Vert(A)=\inf\{\Vert Du\Vert(O)\, :\, A\subset O\text{ and }O\text{ is open in }X\}.
\]

\begin{lemma}\label{lem:BV-Leibniz}
If $u,v\in BV(X)$ and $\eta$ is a Lipschitz continuous function on $X$ with $0\le \eta\le 1$ on $X$, then
$\eta u+(1-\eta)v\in BV(X)$ with
\[
\Vert D(\eta u+(1-\eta)v)\Vert(X)\le \Vert Du\Vert(X)+\Vert Dv\Vert(X) +\int_X|u-v|\, |\nabla \eta|\, d\mu.
\]
\end{lemma}

\begin{proof}
From Lemma~\ref{lem:diff-of-Lip} we already know that given such $\eta$,
$\eta\in\mathcal{F}_{loc}$ with $|\nabla \eta|\in L^\infty(X)$.

We can choose sequences $u_k,v_k\in \mathcal{F}(X)$ such that $u_k\to u$ and $v_k\to v$ in $L^1(X)$ and
$\int_X |\nabla u_k|\, d\mu\to \Vert Du\Vert(X)$ and $\int_X |\nabla v_k|\, d\mu\to \Vert Dv\Vert(X)$ as $k\to\infty$.
Now an application of the properties of Dirichlet form (see~\cite[page~190]{St-I} as well
as the references therein) to the functions
$\eta u_k+(1-\eta)v_k$ tells us that
\begin{align*}
\Vert D(\eta u+(1-\eta)v)\Vert(X)&\le \liminf_{k\to\infty}\int_X |\nabla[\eta u_k+(1-\eta)v_k]|\, d\mu\\
 & \le \liminf_{k\to \infty}\left(\int_X\eta |\nabla u_k|\, d\mu+\int_X(1-\eta) |\nabla v_k|\, d\mu
 +\int_X |u_k-v_k|\, |\nabla \eta|\, d\mu\right).
\end{align*}
Now using the facts that $0\le \eta\le 1$ and that $u_k-v_k\to u-v$ in $L^1(X)$ we obtain the required inequality.
\end{proof}

We first list some elementary properties of $\Vert Du\Vert$.

\begin{lemma}\label{lem:elementary}
Let $U$ and $V$ be two open subsets of $X$. Then
\begin{enumerate}
\item $\Vert Du\Vert(\emptyset)=0$;
\item $\Vert Du\Vert(U)\le \Vert Du\Vert(V)$ if $U\subset V$,
\item $\Vert Du\Vert(\bigcup_i U_i)=\sum_i\Vert Du\Vert(U_i)$ if $\{U_i\}_i$ is a pairwise disjoint subfamily of open
subsets of $X$.
\end{enumerate}
\end{lemma}

\begin{proof}
We will only prove the third property here, as the other two are quite direct consequences of the definition of $\Vert Du\Vert$.
Since any function $f\in \mathcal{F}(\bigcup_i U_i)$ has restrictions $u_i=f\vert_{U_i}\in \mathcal{F}(U_i)$ with
$\int_{\bigcup_i U_i}|\nabla f|\, d\mu=\sum_i\int_{U_i}|\nabla u_i|\, d\mu$, it follows that
\[
\Vert Du\Vert(\bigcup_i U_i)\ge \sum_i\Vert Df\Vert(U_i).
\]
In the above we also used the fact that as $f$ gets closer to $u$ in the $L^1(\bigcup_i U_i)$ sense,
$u_i$ gets closer to $u$ in the $L^1(U_i)$ sense.

To prove the reverse inequality, for $\eps>0$ we can choose locally Lipschitz continuous
$u_i\in \mathcal{F}(U_i)$ for each $i$ such that
\[
\int_{U_i}|u-u_i|\, d\mu<2^{-i-2}\eps
\]
and
\[
\int_{U_i}|\nabla u_i|\, d\mu<\Vert Du\Vert(U_i)+2^{-i-2}\eps.
\]
Now the function $f_\eps=\sum_i u_i\mathbf{1}_{U_i}$ is in $\mathcal{F}(\bigcup_i U_i)$ because 
the $U_i$ are pairwise disjoint open sets, and $\mathcal{E}$ is local.
Therefore
\[
\int_{\bigcup_i U_i}|u-f_\eps|\, d\mu\le \sum_i\int_{U_i}|u-u_i|\, d\mu\le \frac{\eps}{2}
\]
and
\[
\int_{\bigcup_i U_i}|\nabla f_\eps|\, d\mu=\sum_i\int_{U_i}|\nabla u_i|\, d\mu\le \frac{\eps}{2}
+\sum_i\Vert Df\Vert(U_i).
\]
From the first of the above two inequalities it follows that $\lim_{\eps\to 0^+}f_\eps=u$ in 
$L^1(\bigcup_i U_i)$,
and therefore
\[
\Vert Du\Vert(\bigcup_i U_i)\le \liminf_{\eps\to 0^+}\left(\frac{\eps}{2}+\sum_i\Vert Du\Vert(U_i)\right)
=\sum_i\Vert Du\Vert(U_i).
\]
\end{proof}

We use the above definition of $\Vert Du\Vert$ on open sets to consider the following Caratheodory construction.

\begin{definition}
For $A\subset X$, we set
\[
\Vert Du\Vert^*(A):=\inf\{\Vert Du\Vert(O)\, :\, O\text{ is open subset of }X, A\subset O\}.
\]
\end{definition}

By the second property listed in the above lemma, we note that if $A$ is an open subset of $X$, then
$\Vert Du\Vert^*(A)=\Vert Du\Vert(A)$. With this observation, we re-name $\Vert Du\Vert^*(A)$ as
$\Vert Du\Vert(A)$ even when $A$ is not open.

\subsection{Radon measure property of BV energy}

The following lemma is due to De Giorgi and Letta~\cite[Theorem~5.1]{DeLe}, see 
also~\cite[Theorem~1.53]{AFP}.

\begin{lemma}\label{lem:DeGLett}
If $\nu$ is a non-negative function on the class of all open subsets of $X$ such that 
\begin{enumerate}
\item $\nu(\emptyset)=0$, \label{DeLett1}
\item if $U_1\subset U_2$ and they are both open sets, then $\nu(U_1)\le \nu(U_2)$, \label{DeLett2}
\item $\nu(U_1\cup U_2)\le \nu(U_1)+\nu(U_2)$ for open sets $U_1,U_2$ in $X$, \label{DeLett3}
\item if $U_1\cap U_2$ is empty and $U_1,U_2$ are open, then $\nu(U_1\cup U_2)=\nu(U_1)+\nu(U_2)$, \label{DeLett4}
\item for open sets $U$ \label{DeLett5}
\[
\nu(U)=\sup\{\nu(V)\, :\, V\text{ is bounded and open in }X\text{ with }\overline{V}\subset U\}.
\]
\end{enumerate}
Then the Caratheodory extension of $\nu$ to all subsets of $X$ gives a Borel regular outer measure on $X$.
\end{lemma}

The proof of the following theorem is based on that of~\cite{Mr}.

\begin{theorem}\label{thm:outer-measure}
If $X$ is  complete and
$f\in BV(X)$, then $\Vert Df\Vert$ is a Radon outer measure on $X$.
\end{theorem}

\begin{proof}
For simplicity of notation we will assume that $X$ is itself bounded.
Thanks to the lemma of De Giorgi and Letta (Lemma~\ref{lem:DeGLett}), 
it suffices to verify that $\Vert Du\Vert$ satisfies the five conditions 
set forth in Lemma~\ref{lem:DeGLett}. By Lemma~\ref{lem:elementary}, we know that $\Vert Du\Vert$ satisfies
Conditions~\ref{DeLett1},~\ref{DeLett2}, and Condition~\ref{DeLett4}. Thus it suffices for us to verify 
Condition~\ref{DeLett3} and Condition~\ref{DeLett5}. We will first show the validity of Condition~\ref{DeLett5},
and use it (or rather, it's proof) to show that Condition~\ref{DeLett3} holds. We will do so for bounded open subsets of $X$;
a simple modification (by truncating $U_\delta$ by balls) would complete the proof for unbounded
sets; we leave this part of the extension as
an exercise.

\noindent {\bf Proof of Condition~\ref{DeLett5}:} From the monotonicity condition~\ref{DeLett2}, it suffices to prove 
that
\[
\Vert Df\Vert(U)\le  \sup\{\Vert Df\Vert(V)\, :\, V\text{ is open in }X, \overline{V}\text{ is a compact subset of }U\}.
\]
For $\delta>0$ we set 
\[
U_\delta=\{x\in U\, :\, \text{dist}(x,X\setminus U)>\delta\}.
\]
For $0<\delta_1<\delta_2<\text{diam}(U)/2$, let $V=U_{\delta_1}$ and $W=U\setminus\overline{U_{\delta_2}}$.
Then $V$ and $W$ are open subsets of $U$, and the closure of $V$ is a compact
subset of $U$. Note also that
$U=V\cup W$ and that 
$\partial V\cap \partial W$ is empty. Thus we can find a Lipschitz function $\eta$ on $U$ that can be 
used as a ``needle+thread" to stitch Sobolev functions on $V$ to Sobolev functions on $W$ as follows to obtain a Sobolev
function on $U$:  $0\le \eta\le 1$ on $U$, $\eta=1$ on $V\setminus W=\overline{U_{\delta_2}}$,
$\eta=0$ on $W\setminus V=U\setminus U_{\delta_1}$, and 
\[
\text{Lip}\, \eta\le \frac{2}{\delta_2-\delta_1}\, \mathbf{1}_{V\cap W}.
\]
Now, for $v\in \mathcal{F}(V)$ and $w\in \mathcal{F}(W)$ we set $u=\eta v+(1-\eta)w$. As
we have the Leibniz rule (see~\cite{St-I}), we can see that $u\in \mathcal{F}(U)$ and 
\begin{equation}\label{eq:B}
\int_U |\nabla u|\, d\mu\le \int_V |\nabla v|\, d\mu+\int_W |\nabla w|\, d\mu+ \frac{2}{\delta_2-\delta_1}\int_{V\cap W}|v-w|\, d\mu.
\end{equation}
Furthermore, whenever $h\in L^1(U)$, we can write $h=\eta h+(1-\eta)h$ to see that
\begin{equation}\label{eq:A}
\int_U|u-h|\, d\mu\le \int_V|v-h|\, d\mu+\int_W|w-h|\, d\mu.
\end{equation}
Now, we take $v_k$ from $\mathcal{F}(V)$ such that $v_k\to f$ in $L^1(V)$ and $\lim_{k\to\infty}\int_V|\nabla {v_k}|\, d\mu=\Vert Df\Vert(V)$,
and take $w_k\in \mathcal{F}(W)$ analogously. We then follow through by stitching together $v_k$ and $w_k$ into
the function $u_k$ as prescribed above. By~\eqref{eq:A} with $h=f$, we have that 
\[
\int_U|f-u_k|\, d\mu\le \int_V|v_k-f|\, d\mu+\int_W|w_k-f|\, d\mu\to 0\text{ as }k\to\infty.
\]
It follows from~\eqref{eq:B} and the fact $\int_{V\cap W}|v_k-w_k|\, d\mu\to 0$ as $k\to\infty$  that
\[
\Vert Df\Vert(U)\le \liminf_{k\to\infty}\int_U|\nabla {u_k}|\, d\mu\le \Vert Df\Vert(V)+\Vert Df\Vert(W).
\]
Remembering again that the closure of $V$ is a compact subset of $U$, we see that
\[
\Vert Df\Vert(U)\le \sup\{\Vert Df\Vert(V)\, :\, V\text{ is open in }X, \overline{V}\text{ is compact subset of }U\}+
  \Vert Df\Vert(U\setminus \overline{U_{\delta_2}}).
\]
So now it suffices to prove that
\begin{equation}\label{eq:est-annuli}
\lim_{\delta\to 0^+}\Vert Df\Vert(U\setminus \overline{U_{\delta}})=0.
\end{equation}
To prove this, we note first that the above limit exists as 
$\Vert Df\Vert(U\setminus \overline{U_{\delta}})$ decreases as
$\delta$ decreases. We fix a strictly monotone decreasing sequence of real numbers $\delta_k$ with
$\lim_{k\to\infty}\delta_k=0$, and for $k\ge 2$ we set 
$V_k:=U_{\delta_{2k+3}}\setminus\overline{U_{\delta_{2k}}}$. 
Observe that the family $\{V_{2k}\}_k$ is a pairwise disjoint family of open subsets of $X$ and that
the family $\{V_{2k+1}\}_k$ is also a pairwise disjoint family of open subsets of $X$. 

By Lemma~\ref{lem:elementary}, we know that
\[
\infty>\Vert Df\Vert(U)\ge \Vert Df\Vert(\bigcup_{k\ge 1} V_{2k})=\sum_{k=1}^\infty \Vert Df\Vert(V_{2k}),
\]
and
\[
\infty>\Vert Df\Vert(U)\ge \Vert Df\Vert(\bigcup_{k\ge 1} V_{2k+1})=\sum_{k=1}^\infty \Vert Df\Vert(V_{2k+1}).
\]
It follows that for $\eps>0$ there is some positive integer $k_\eps\ge 2$ such that
\[
\sum_{k=k_\eps}^\infty \Vert Df\Vert(V_{2k})+\sum_{k=k_\eps}^\infty \Vert Df\Vert(V_{2k+1})<\eps.
\]
Now we stitch together approximations on $V_{2k}$ to approximations on $V_{2k+1}$, and from there to $V_{2k+2}$ and so
on. For each $k$ we choose a ``stitching function" $\eta_k$ as a Lipschitz function on $\bigcup_{j=k_\eps}^{k+1}V_j$
such that $0\le \eta_k\le 1$, with $\eta_k=1$ on $V_{k}\setminus V_{k-1}$, $\eta_k=0$ on 
$\bigcup_{j=k_\eps}^{k-1} V_j\setminus V_{k}$, and $|\nabla {\eta_k}|\le C_k\mathbf{1}_{V_{k}\cap V_{k-1}}$.

Next, for each $k$ we can find $v_{k,j}\in \mathcal{F}(V_k)$ such that 
\[
 \int_{V_k}|v_{k,j}-f|\, d\mu\le \frac{2^{-k-j}}{3(1+C_k)}
\]
and 
\[
\int_{V_k}|\nabla {v_{k,j}}|\, d\mu\le \Vert Df\Vert(V_k)+2^{-j-k}.
\]
We now inductively stitch the functions together. To do so, we first fix $i\in\mathbb N$.

Starting with $k=k_\eps$, we stitch $u_{k,i}$ to $u_{k+1,i}$ using $\eta_{k+1}=\eta_{k_\eps+1}$ to obtain
$w_{i,k}\in \mathcal{F}(V_{k_\eps}\cup V_{k_\eps+1})$ so that we have
\[
\int_{V_{k_\eps}\cup V_{k_\eps+1}}|w_{i,k}-f|\, d\mu\le \frac{2^{-i-k_\eps}}{1+C_{k_\eps+1}}
\]
and 
\[
\int_{V_{k_\eps}\cup V_{k_\eps+1}}|\nabla {w_{i,k}}|\, d\mu\le \sum_{j=k_\eps}^{k_\eps+1}\Vert Df\Vert(V_j)+2^{1-i-k_\eps}.
\]
Suppose now that for some $k\in\mathbb N$ with $k\ge k_\eps+1$
we have constructed $w_{i,k}\in \mathcal{F}(\bigcup_{j=k_\eps}^k V_j)$ such that
\[
\int_{\bigcup_{j=k_\eps}^k V_j}|w_{i,k}-f|\, d\mu\le \sum_{k=k_\eps}^k\frac{2^{-i-j}}{1+C_j}
\]
and 
\[
\int_{\bigcup_{j=k_\eps}^k V_j}|\nabla {w_{i,k}}|\, d\mu\le \sum_{j=k_\eps}^{k}[\Vert Df\Vert(V_j)+2^{1-i-j}].
\]
Then we stitch $u_{k+1,i}$ to $w_{i,k}$ using $\eta_{k+1}$ to obtain $w_{i,k+1}$ satisfying inequalities analogous to the
above two. Note that $w_{i,k+1}=w_{i,k-1}$ on $V_{k-1}$ for $k\ge k_\eps+2$.
Thus, in the limit, we obtain a function $w_i=\lim_k w_{i,k}\in \mathcal{F}(\bigcup_{k=k_\eps}^\infty V_k)$ satisfying
\[
\int_{\bigcup_{j=k_\eps}^\infty V_j}|w_{i}-f|\, d\mu\le \sum_{k=k_\eps}^k\frac{2^{-i-j}}{1+C_j}<2^{1-i},
\]
\[
\int_{\bigcup_{j=k_\eps}^\infty V_j}|\nabla {w_{i}}|\, d\mu\le \sum_{j=k_\eps}^\infty\Vert Df\Vert(V_j)+2^{2-i}
   <\eps+2^{2-i}.
\]
From the first of the above two inequalities, we see that $w_i\to f$ in $L^1(\bigcup_{j=k_\eps}^\infty V_j)$
as $i\to\infty$,
and so from the second of the above two inequalities we obtain
\[
\Vert Df\Vert(\bigcup_{j=k_\eps}^\infty V_j)=\Vert Df\Vert(U\setminus \overline{U_{\delta_{k_\eps}}})
\le \liminf_{i\to\infty}\int_{\bigcup_{j=k_\eps}^\infty V_j}|\nabla {w_{i}}|\, d\mu\le \eps.
\]
This last inequality above tells us that the claim we set out to prove, namely
\[
\lim_{\delta\to 0^+}\Vert Df\Vert(U\setminus \overline{U_{\delta}})=0.
\]
this completes the proof of Condition~\ref{DeLett5}.

\noindent {\bf Proof of Condition~\ref{DeLett3}:}
By Condition~\ref{DeLett5} (which we have now proved above, so no circular argument
here!), for each $\eps>0$ we can find relatively compact open subsets $U_1^\prime\Subset U_1$ and
$U_2^\prime\Subset U_2$ such that
$\Vert Df\Vert(U_1\cup U_2)\le \Vert Df\Vert(U_1^\prime\cup U_2^\prime)+\eps$.
We then choose a Lipschitz ``stitching function" $\eta$ on $X$ such that $0\le \eta\le 1$ on $X$,
$\eta=1$ on $U_1^\prime$, $\eta=0$ on $X\setminus U_1$,  and 
\[
|\nabla \eta|\le \frac{1}{C_{U_1,U_1^\prime}} \, \mathbf{1}_{U_1\setminus U_1^\prime}.
\]
For $u_1\in \mathcal{F}(U_1)$ and $u_2\in \mathcal{F}(U_2)$, we obtain the stitched function 
$w=\eta u_1+(1-\eta)u_2$ and note that $w\in \mathcal{F}(U_1^\prime\cup U_2^\prime)$. 
Observe that we cannot in general have $w\in \mathcal{F}(U_1\cup U_2)$ as $w$ is not defined in 
$U_1\setminus (U_1^\prime\cup U_2)$ because $1-\eta$ is non-vanishing there and $u_2$ is not defined there, for example.
Then we have
\[
\int_{U_1^\prime\cup U_2^\prime}|\nabla w|\, d\mu\le \int_{U_1}|\nabla {u_1}|\, d\mu+\int_{U_2}|\nabla u_2\, d\mu
 +\frac{1}{C_{U_1,U_1^\prime}}\, \int_{U_1\cap U_2}|u_1-u_2|\, d\mu
\]
and 
\[
\int_{U_1^\prime\cup U_2^\prime}|w-f|\, d\mu\le \int_{U_1}|u_1-f|\, d\mu+\int_{U_2}|u_2-f|\, d\mu.
\]
As before, choosing $u_{1k}$ to be the optimal approximating sequence for  $f$ on $U_1$ and
$u_{2,k}$ correspondingly for $f$ on $U_2$, we see from the first of the above two inequalities 
that the stitched sequence $w_k$ approximates
$f$ on $U_1^\prime\cup U_2^\prime$. Therefore we obtain
\[
\Vert Df\Vert(U_1\cup U_2)\le \eps+\Vert Df\Vert(U_1^\prime\cup U_2^\prime)
  \le \eps+\liminf_{k\to\infty}\int_{U_1\cup U_2}|\nabla {w_k}|\, d\mu
  \le \Vert Df\Vert(U_1)+\Vert Df\Vert(U_2)+\eps.
\]
Letting $\eps\to 0$ now gives the desired Condition~\ref{DeLett3}.
\end{proof}


\subsection{Co-area formula}

\begin{definition}
We say that a measurable set $E\subset X$ is of \emph{finite perimeter} if $\mathbf{1}_E\in BV_{loc}(X)$ with
$\Vert D\mathbf{1}_E\Vert(X)<\infty$. We denote by $P(E,A):=\Vert D\mathbf{1}_E\Vert(A)$ the perimeter measure of $E$.
\end{definition}

\begin{lemma}\label{lem:Co-area}
The co-area formula holds true, that is, for Borel sets $A\subset X$ and $u\in L^1_{loc}(X)$,
\[
\Vert Du\Vert(A)=\int_{\R}P(\{u>s\}, A)\, ds.
\]
\end{lemma}

\begin{proof}
We first prove the formula for open sets $A$.

Suppose first that $u\in BV_{loc}(X)$ with $\Vert Du\Vert(A)<\infty$.
For $s\in\R$ we set $E_s:=\{x\in X\, :\, u(x)>s\}$. Consider the function $m:\R\to\R$ given by 
\[
m(t)=\Vert Du\Vert(A\cap E_t).
\]
Then $m$ is a monotone decreasing function, and hence is differentiable almost everywhere. Let $t\in\R$ such that
$m^\prime(t)$ exists. Then
\[
m^\prime(t)=\lim_{h\to 0^+} \frac{\Vert Du\Vert(A\cap E_t\setminus E_{t+h})}{h}.
\]
Note that the functions
\[
u_{t,h}:=\frac{\max\{t,\min\{t+h,u\}\}-t}{h}
\]
converge in $L^1(X)$ to $\mathbf{1}_{E_t}$ as $h\to 0^+$. It follows that as $A$ is open,
\[
P(E_t, A)\le \liminf_{h\to 0^+}\Vert Du_{t,h}\Vert(A)=\liminf_{h\to0^+}\frac{\Vert Du\Vert(A\cap E_t\setminus E_{t+h})}{h}=m^\prime(t).
\]
Note also that by this lower semicontinuity of BV energy, $t\mapsto P(E_t, A)$ is a lower semicontinuous function, and hence
is measurable; and as it is non-negative, we can talk about its integral, whether that integral is finite or not.
Therefore, by the fundamental theorem of calculus for monotone functions, 
\[
\int_{\R}P(E_t, A)\, dt\le \int_{\R} m^\prime(t)\, dt\le \lim_{s,\tau\to\infty}m(s)-m(-\tau)=\Vert Du\Vert(A).
\]
The above in particular tells us that if $u\in BV_{loc}(X)$ then almost all of its superlevel sets $E_t$ have finite perimeter.
If $u$ is not a BV function on $A$, then $\Vert Du\Vert(A)=\infty$, and hence 
we also have
\begin{equation}\label{eq:1}
\int_{\R}P(E_t, A)\, dt\le \Vert Du\Vert(A).
\end{equation}
In particular, it also follows that $\int_{\R} P(E_t,A)\, dt<\infty$ if $u\in BV(X)$.

We still continue to assume that $A$ is open, and prove the reverse of the above inequality. If
$\int_{\R}P(E_t,A)\, dt=\infty$, then trivially
\[
\Vert Du\Vert(A)\le \int_{\R}P(E_t, A)\, dt.
\] 
So we may assume without loss of generality that $\int_{\R}P(E_t,A)\, dt$ is finite. Note also by the Markovian property of
Dirichlet forms, filtered down to the level of the measure $|\nabla f|$, we have that
\[
\Vert Du\Vert(A)=\lim_{s,\tau\to\infty} \Vert Du_{s,\tau}\Vert(A),
\]
where $u_{s,\tau}=\max\{-\tau, \min\{u,s\}\}$. So without loss of generality we may assume that $0\le u\le 1$ for some
finite $a,b\in\R$. For positive integers $X$ we can divide $[0,1]$ into $X$ equal sub-intervals $[t_i,t_{i+1}]$, $i=0,\cdots, k$
with $t_{i+1}-t_i=1/k$. Then we can find $\rho_{k,i}\in (t_i, t_{i+1})$ such that 
\[
\frac{1}{k}\, P(E_{\rho_{k,i}},A)\le \int_{t_i}^{t_{i+1}}\, P(E_s,A)\, ds.
\]
We set
\[
u_X=\sum_{j=1}^k\frac{1}{k} \mathbf{1}_{E_{\rho_{k,i}}}.
\]
Then as $|u_k-u|\le 1/k$ on $X$, we have that $u_k\to u$ in $L^1(A)$ as $k\to\infty$, and so
\begin{equation}\label{eq:2}
\Vert Du\Vert(A)\le \liminf_{k\to\infty}\Vert Du_k\Vert(A)=\liminf_{k\to\infty}\sum_{j=1}^k\frac{1}{k}\, P(E_{\rho_{k,i}},A)
 \le \int_0^1 P(E_s,A)\, ds.
\end{equation}

Note now that by the proofs of inequalities~\eqref{eq:1} and~\eqref{eq:2}, if $A$ is an open set then
$u\in BV(A)$ if and only if $\int_{\R}P(E_t,A)\, dt$ is finite.

Finally, we remove the requirement that $A$ be open. By the above comment, it suffices to prove this for the case that 
$u\in BV(X)$. In this case, the maps $A\mapsto\Vert Du\Vert(A)$ and $A\mapsto\int_{\R}P(E_t,A)\, dt$ are
both Radon measures on $X$ that agree on open subsets of $X$ (that is, they are equal for open $A$). Hence it follows
that they agree on Borel subsets of $X$.
%
This completes the proof of the coarea formula.
\end{proof}

\section{$\mathbf{B}^{1,1/2}(X)$ is equal to $BV(X)$}\label{weak BE metric space}

%
%


In this section we will assume that $(X,d_\mathcal{E},\mu)$ is doubling, supports a $2$-Poincar\'e inequality, 
and satisfies a weak Bakry-\'Emery curvature
condition~\eqref{eq:weak-BE}. Under this condition, we will compare the BV class, constructed in Section~\ref{subsec:BV},
with the Besov class $\mathbf{B}^{1,1/2}(X)$ from~\eqref{eq:def:Besov}.
We will also assume that
for each $u\in\mathcal{F}$ the measure $\Gamma(u,u)$ is absolutely continuous with respect to the underlying measure
$\mu$ on $X$; as in the previous section we will denote by $|\nabla u|^2$ the Radon-Nikodym derivative of
$\Gamma(u,u)$ with respect to $\mu$.

In what follows, $\mathcal{F}_1(X)$ denotes the collection of all $u\in L^1(X)\cap\mathcal{F}_{loc}(X)$ for which 
$|\nabla u|\in L^1(X)$.
This means that for each ball $B$ in $X$ there is a compactly supported Lipschitz function $\varphi$ with $\varphi=1$ on $B$
such that $u\varphi\in\mathcal{F}$; in this case we can set $|\nabla u|=|\nabla (u\varphi)|$ in $B$, thanks to the strong locality 
property of $\mathcal{E}$.

\begin{lemma}\label{lem:L1-norm-control}
Suppose that~\eqref{eq:weak-BE} holds.
For $u\in \mathcal{F}\cap \mathcal{F}_1(X)$, we have that
\[
\Vert P_t u-u\Vert_{L^1(X)}\le C\, \sqrt{t}\, \int_X|\nabla u|\, d\mu.
\]
Hence, if $u\in BV(X)$, then
\[
\Vert P_t u-u\Vert_{L^1(X)}\le C\, \sqrt{t}\, \Vert Du\Vert(X).
\]
\end{lemma}

\begin{proof}
To see the first part of the claim, we note that for each $x\in X$ and $s>0$, $\tfrac{\partial}{\partial s}P_su(x)$ exists,
and so by the fundamental theorem of calculus, for $0<\tau<t$ and $x\in X$,
\[
P_tu(x)-P_\tau u(x)=\int_\tau^t \frac{\partial}{\partial s}P_su(x)\, ds.
\]
Thus for each compactly supported function $\varphi\in \mathcal{F} \cap L^\infty(X)$, 
by the facts that $P_t u$ satisfies the heat equation and that $P_s$ commutes with the infinitesimal generator (Laplacian),
\begin{align*}
\bigg\lvert\int_X\varphi(x)[P_tu(x)-P_\tau u(x)]\, d\mu(x)\bigg\rvert
 &=\bigg\lvert-\int_X\int_\tau^t\varphi(x)\frac{\partial}{\partial s}P_su(x)\, ds\, d\mu(x)\bigg\rvert\\
 &=\bigg\lvert\int_\tau^t \int_X d\Gamma(\varphi,P_su)(x)\, ds\bigg\rvert\\
 &=\bigg\lvert\int_\tau^t \int_X d\Gamma(P_s \varphi,u)(x)\, ds\bigg\rvert\\
 &\le \int_\tau^t\int_X|\nabla P_s\varphi|\, |\nabla u|\, d\mu\, ds\\
 &\le \Vert|\nabla P_s\varphi|\Vert_{L^\infty(X)}\, \int_\tau^t\int_X|\nabla u|\, ds\, d\mu.
\end{align*}
An application of~\eqref{eq:weak-BE} gives
\begin{align*}
\bigg\lvert\int_X\varphi(x)[P_tu(x)-P_\tau u(x)]\, d\mu(x)\bigg\rvert
&\le \frac{C}{\sqrt{t}}\Vert\varphi\Vert_{L^\infty(X)}\int_X\int_\tau^t |\nabla u|\, ds\, d\mu\\
&=C\, \frac{t-\tau}{\sqrt{t}}\, \Vert\varphi\Vert_{L^\infty(X)}\, \int_X|\nabla u|\, d\mu.
\end{align*}
As the above holds for all $\varphi\in \mathcal{F}\cap L^\infty(X)$, and as
compactly supported Lipschitz functions are dense in $L^\infty(X)$ (note that $X$ is separable as closed and bounded 
subsets of $X$ are compact because $X$ is complete and supports a doubling measure),
we obtain
\[
\Vert P_t u-P_\tau u\Vert_{L^1(X)}\le C\, \frac{t-\tau}{\sqrt{t}}\, \int_X|\nabla u|\, d\mu.
\]
Now by the fact that $P_\tau u\to u$ as $\tau\to 0^+$ in $L^2(X)$, and by the fact that $\{P_t\}_{t>0}$ has an
extension as a contraction semigroup to $L^1(X)$ (see~\cite[Section~1.5 of page~37, Section~4.7 of page~201]{FOT}), we have that
\[
\Vert P_tu-u\Vert_{L^1(X)}\le C\, \sqrt{t}\, \int_X|\nabla u|\, d\mu
\]
as desired. 

Finally, if $u\in BV(X)$, then we can find a sequence $u_k\in \mathcal{F}\cap W^{1,1}(X)$ 
such that
$u_k\to u$ in $L^1(X)$ and $\lim_k\int_X|\nabla u_k|\, d\mu=\Vert Du\Vert(X)$. 
As $u\in L^1(X)$ and for $t>0$ we have that the heat kernel $p_t(x,y)$ is bounded in $X\times X$,
it follows that
$P_tu\in L^1(X)$ is well-defined with $\Vert P_tu\Vert_{L^1(X)}\le C_t\Vert u\Vert_{L^1(X)}$, see~\cite[Section~1.5]{FOT}. Now,
\begin{align*}
\Vert P_tu-u\Vert_{L^1(X)}
 &\le \Vert P_t(u-u_k)\Vert_{L^1(X)}+\Vert P_t u_k-u_k\Vert_{L^1(X)}+\Vert u_k-u\Vert_{L^1(X)}\\
 &\le C_t\Vert u-u_k\Vert_{L^1(X)}+C\, \sqrt{t}\, \int_X|\nabla u_k|\, d\mu+\Vert u-u_k\Vert_{L^1(X)}\\
 &\to C\, \sqrt{t}\, \lim_{k\to\infty} \int_X|\nabla u_k|\, d\mu=C\, \sqrt{t}\, \Vert Du\Vert(X)\text{ as }k\to\infty.
\end{align*}
\end{proof}
 
Recall the definition of $\mathbf{B}^{1,1/2}(X)$ from~\eqref{eq:def:Besov}: $f$ is in this class if 
\[
\Vert f\Vert_{B^{1,1/2}(X)}:=\Vert f\Vert_{L^1(X)}+\sup_{t>0}\frac{1}{\sqrt{t}}\, \int_X P_t(|f-f(x)|)(x)\, d\mu(x)<\infty.
\]
Equivalently,
\[
\Vert f\Vert_{B^{1,1/2}(X)}=\Vert f\Vert_{L^1(X)}+\sup_{t>0}\frac{1}{\sqrt{t}}\, \int_X\int_Xp_t(x,y)|f(y)-f(x)|\, d\mu(y)\, d\mu(x)<\infty.
\]

We now wish to prove the following theorem.

\begin{theorem}\label{thm:W=BV}
Under the hypotheses of this section, $\mathbf{B}^{1,1/2}(X)=BV(X)$ with comparable seminorms.  Moreover, there exist constants $c,C>0$ such that for every $u \in BV(X)$
\[
c \limsup_{s  \to 0 }  s^{-1/2}  \int_X P_s (|u-u(y)|)(y) d\mu(y)  \le \Vert Du\Vert(X) \le C\liminf_{s \to 0}  s^{-1/2}  \int_X P_s (|u-u(y)|)(y) d\mu(y) .
\]

\end{theorem}

Note from the results of~\cite[Theorem~4.1]{MMS} that if the measure $\mu$ is doubling and supports a $1$-Poincar\'e 
inequality, then a measurable set $E\subset X$ is in the BV class if 
\[
\liminf_{t\to 0^+} \frac{1}{\sqrt{t}} \int_{E^{\sqrt{t}}\setminus E} P_t\mathbf{1}_E\, d\mu<\infty.
\]
Here $E^\varepsilon=\bigcup_{x\in E}B(x,\varepsilon)$. Note that 
\[
\int_{E^{\sqrt{t}}\setminus E} P_t\mathbf{1}_E\, d\mu\le \int_{X\setminus E}P_t\mathbf{1}_E\, d\mu
=\frac{1}{2}\Vert P_t\mathbf{1}_E-\mathbf{1}_E\Vert_{L^1(X)}.
\]
Thus if $\mu$ is doubling and supports a $1$-Poincar\'e inequality, and in addition
\[
\sup_{t>0}\frac{1}{\sqrt{t}}\, \Vert P_t\mathbf{1}_E-\mathbf{1}_E\Vert_{L^1(X)}<\infty,
\]
then $E$ is of finite perimeter. Note that we do not assume the validity of 
$1$-Poincar\'e inequality, but the weaker
version of~\eqref{eq:weak-BE}.

\begin{proof}
First we assume that $u\in BV(X)$. Then we know that for almost every $t\in\R$ the set $E_t$ is of finite perimeter,
where
\[
E_t=\{x\in X\, :\, u(x)>t\},
\]
and by the co-area formula for BV functions (see for example Lemma~\ref{lem:Co-area} or~\cite{Mr}),
\[
\Vert Du\Vert(X)=\int_{\mathbb{R}}\Vert D\mathbf{1}_{E_t}\Vert(X)\, dt.
\]
For such $t$, by Lemma~\ref{lem:L1-norm-control} we know that
\[
\sup_{s>0}\frac{1}{\sqrt{s}}\int_X|P_s\mathbf{1}_{E_t}(x)-\mathbf{1}_{E_t}(x)|\, d\mu (x) \le C\, \Vert D\mathbf{1}_{E_t}\Vert(X).
\]
Now, setting $A=\{(x,y)\in X\times X\, :\, u(x)<u(y)\}$, we have for $s>0$,
\begin{align*}
  & \int_{X}\int_X p_s(x,y)  |u(x)-u(y)|\, d\mu(x)d\mu(y) \\
 &=2\iint_A p_s(x,y)|u(x)-u(y)|\, d\mu(x)d\mu(y) \\
 &=2\iint_A \int_{u(x)}^{u(y)}\, p_s(x,y)\, dt\, d\mu(x)d\mu(y)\\
 &=2\int_X\int_X \int_{\mathbb{R}}\mathbf{1}_{[u(x),u(y))}(t)\, \mathbf{1}_A(x,y)\, p_s(x,y)\, dt\, d\mu(x)d\mu(y) \\
 &=2\int_{\mathbb{R}} \int_X\int_X \mathbf{1}_{E_t}(y)[1-\mathbf{1}_{E_t}(x)]\, p_s(x,y)\, d\mu(x)d\mu(y)\, dt\\
 &=2\int_{\mathbb{R}}\int_X P_s\mathbf{1}_{E_t}(x)[1-\mathbf{1}_{E_t}(x)]\, d\mu(x)\, dt\\
 &=2\int_{\mathbb{R}}\int_{X\setminus E_t}P_s\mathbf{1}_{E_t}(x)\, d\mu(x)\, dt.
\end{align*}
Observe that
\[
\int_X|P_s\mathbf{1}_{E_t}(x)-\mathbf{1}_{E_t}(x)|\, d\mu(x)
 \ge \int_{X\setminus E_t}|P_s\mathbf{1}_{E_t}(x)-\mathbf{1}_{E_t}(x)|\, d\mu(x)=\int_{X\setminus E_t}P_s\mathbf{1}_{E_t}(x)\, d\mu(x).
\]
Therefore we obtain
\[
\int_X\int_X p_s(x,y) \ |u(x)-u(y)|\, d\mu(x) d\mu(y)\le 2\int_{\mathbb{R}}\Vert P_s\mathbf{1}_{E_t}-\mathbf{1}_{E_t}\Vert_{L^1(X)}\, dt.
\]
An application of Lemma~\ref{lem:L1-norm-control} now gives
\[
\int_X\int_X p_s(x,y) \ |u(x)-u(y)|\, d\mu(x)d\mu(y)\le C\sqrt{s}\, \int_{\mathbb{R}}\Vert D\mathbf{1}_{E_t}\Vert(X)\, dt,
\]
whence with the help of the co-area formula we obtain
\[
\Vert u\Vert_{1,1/2}\le C\, \Vert Du\Vert(X),
\]
that is, $u\in \mathbf{B}^{1,1/2}(X)$. 
Thus $BV(X)\subset \mathbf{B}^{1,1/2}(X)$ boundedly.

\

Now we show that $\mathbf{B}^{1,1/2}(X)\subset BV(X)$; this inclusion holds
even when $\mathcal{E}$ does not support a Bakry-\'Emery curvature
condition; only a $2$-Poincar\'e inequality and the doubling condition on
$\mu$ are needed.
Now suppose that $u\in \mathbf{B}^{1,1/2}(X)$. 
Then there is some $C\ge 0$ such that for each $t>0$,
\[
\int_X\int_Xp_t(x,y)|u(y)-u(x)|\, d\mu(y)\, d\mu(x)\le C\sqrt{t}.
\]
By~\eqref{eq:heat-Gauss}, we have a Gaussian lower bound for the heat kernel:
\[
 p_t(x,y)\ge \frac{e^{-cd(x,y)^2/t}}{C\, \sqrt{\mu(B(x,\sqrt{t}))\mu(B(y,\sqrt{t}))}}.
\]
Let $C_0=\Vert u\Vert_{\B^{1,1/2}(X)}-\Vert f\Vert_{L^1(X)}$.
Therefore, setting $\Delta_\varepsilon=\{(x,y)\in X\, :\, d(x,y)<\varepsilon\}$ for some $\varepsilon>0$, we get
\begin{align*}
C_0\, \sqrt{t}&\ge \int_X\int_X\frac{e^{-cd(x,y)^2/t}}{C\, \sqrt{\mu(B(x,\sqrt{t}))\mu(B(y,\sqrt{t}))}}|u(y)-u(x)|\, d\mu(y)\, d\mu(x)\\
&\ge C^{-1}\iint_{\Delta_\varepsilon} \frac{e^{-cd(x,y)^2/t}}{\sqrt{\mu(B(x,\sqrt{t}))\mu(B(y,\sqrt{t}))}}|u(y)-u(x)|\, d\mu(x)d\mu(y) \\
&\ge \frac{e^{-c\varepsilon^2/t}}{C}\, 
 \iint_{\Delta_\varepsilon}\frac{|u(y)-u(x)|}{\sqrt{\mu(B(x,\sqrt{t}))\mu(B(y,\sqrt{t}))}}\, d\mu(x)d\mu(y).
\end{align*}
With the choice of $\varepsilon=\sqrt{t}$, we now get
\[
C_0\, \varepsilon
\ge \frac{1}{C}\iint_{\Delta_\varepsilon}\frac{|u(y)-u(x)|}{\sqrt{\mu(B(x,\varepsilon))\mu(B(y,\varepsilon))}}\, d\mu(x)d\mu(y).
\]
It follows that
\begin{equation}\label{eq:Ledoux1}
\liminf_{\varepsilon\to 0^+}\frac{1}{\varepsilon}
 \iint_{\Delta_\varepsilon}\frac{|u(y)-u(x)|}{\sqrt{\mu(B(x,\varepsilon))\mu(B(y,\varepsilon))}}\, d\mu(x)d\mu(y) \le C\, C_0<\infty.
\end{equation}
Now an argument as in the second half of the proof of~\cite[Theorem~3.1]{MMS} tells us that 
$u\in BV(X)$. We point out here that although~\cite[Theorem~3.1]{MMS} assumes that $X$ supports a $1$-Poincar\'e 
inequality, this part of the proof of~\cite[Theorem~3.1]{MMS} does not need this assumption; the argument using
discrete convolution there is valid also in our setting.
We also then obtain that
\[
\Vert Du\Vert(X)\le \liminf_{\varepsilon\to 0^+}\frac{1}{\varepsilon}
 \iint_{\Delta_\varepsilon}\frac{|u(y)-u(x)|}{\sqrt{\mu(B(x,\varepsilon))\mu(B(y,\varepsilon))}}\, d\mu(x)d\mu(y) 
 \le 
 \Vert u\Vert_{\B^{1,1/2}(X)}-\Vert f\Vert_{L^1(X)}.
\]
\end{proof}

\begin{remark}\label{chaining metric}
As a byproduct of this proof,  we also obtain that there exists a constant $C>0$ such that for every $u \in BV(X)$,
\begin{align*}
\sup_{\varepsilon >0} \frac{1}{\varepsilon}
 \iint_{\Delta_\varepsilon}\frac{|u(y)-u(x)|}{\sqrt{\mu(B(x,\varepsilon))\mu(B(y,\varepsilon))}}\, d\mu(x)d\mu(y) 
 \le C \liminf_{\varepsilon\to 0^+}\frac{1}{\varepsilon}
 \iint_{\Delta_\varepsilon}\frac{|u(y)-u(x)|}{\sqrt{\mu(B(x,\varepsilon))\mu(B(y,\varepsilon))}}\, d\mu(x)d\mu(y)
\end{align*}
because both sides are comparable to $\| Du(X) \|$. This property of the metric space $(X,d)$ can be viewed as 
an interesting consequence of the weak Bakry-\'Emery estimate.
\end{remark}

We say that $X$ supports a $1$-Poincar\'e inequality if there are constants $C,\lambda>0$ such that 
whenever $B$ is a ball in $X$ and $u\in W^{1,1}(X)$, we have
\[
\int_B|u-u_B|\, d\mu\le C\, \text{rad}(B)\, \int_{\lambda B}|\nabla u|\, d\mu.
\]

Given $A\subset X$ we set
\[
\mathcal{H}(A):=\lim_{\varepsilon\to 0^+} \inf\bigg\lbrace\sum_i\frac{\mu(B_i)}{\text{rad}(B_i)}\, 
:\, A\subset\bigcup_iB_i, \text{ and }\forall i, \ \text{rad}(B_i)<\varepsilon\bigg\rbrace.
\]
In the event that $X$ supports a $1$-Poincar\'e inequality, from the results of
Ambrosio, Miranda and Pallara, we know that $P(E,X)\simeq \mathcal{H}(\partial_mE)$
whenever $E$ is of finite perimeter. Here $\partial_mE$ consists of points $x\in\partial E$
for which both
\[
\limsup_{r\to 0^+}\frac{\mu(B(x,r)\cap E)}{\mu(B(x,r))}>0
\]
and
\[
\limsup_{r\to 0^+}\frac{\mu(B(x,r)\setminus E)}{\mu(B(x,r))}>0.
\]
In this current work we do not assume the strong condition that $X$ supports a $1$-Poincar\'e inequality,
but we assume that $X$ supports a $2$-Poincar\'e inequality; in this case, we have
the double Gaussian bounds~\eqref{eq:heat-Gauss} for the heat kernel $p_t(x,y)$. 
Note that even without the assumption of $2$-Poincar\'e inequality, we know that if
$\mathcal{H}(\partial E)<\infty$, then $E$ is of finite perimeter.

For $r_0>0$ and $0<\alpha<1$ we set $\partial_\alpha^{r_0}E$ to be the collection of all
points $x\in\partial E$ such that for all $0<r\le r_0$,
\[
\frac{\mu(B(x,r)\cap E)}{\mu(B(x,r))}>\alpha, \qquad 
\frac{\mu(B(x,r)\setminus E)}{\mu(B(x,r))}>\alpha.
\]
Observe that $\partial_mE=\bigcup_{0<\alpha<1}\bigcup_{0<r_0<1}\partial_\alpha^{r_0}(E)$.
This union can be made a countable union by taking $\alpha$ and $r_0$ to be 
rational numbers.

\begin{lemma}
Suppose that $\mu$ is doubling and supports a $2$-Poincar\'e inequality, and that
$E\subset X$ with $\Vert \mathbf{1}_E\Vert_{\mathbf{B}^{1,1/2}(X)}$ finite. Then for all
$r_0>0$ and $0<\alpha<1$,
\[
\mathcal{H}(\partial_\alpha^{r_0}E)\le \frac{C}{\alpha}\, P(E,X).
\]
Consequently, $\mathcal{H}\vert_{\partial_mE}$ is a $\sigma$-finite measure.
\end{lemma} 

\begin{proof}
Note that the proof of Theorem~\ref{thm:W=BV} also tells us that even without the
Bakry-\'Emery condition~\eqref{eq:weak-BE}, we know that if $\mathbf{1}_E\in\mathbf{B}^{1,1/2}(X)$ then
$\mathbf{1}_E\in BV(X)$. By the definition of $\B^{1,1/2}(X)$, we know that
\[
\sup_{t>0}\frac{1}{\sqrt{t}}\int_X\int_Xp_t(x,y)|\mathbf{1}_E(x)-\mathbf{1}_E(y)|\, d\mu(x)\, d\mu(y)
\le C\, P(E,X).
\]
Fix $t<(r_0/3)^2$.
We cover $\partial_\alpha^{r_0}E$ by countably many balls $B_i$ of radius $\sqrt{t}$
such that the balls $5B_i$ have a bounded overlap (the bound depending solely on the
doubling constant of $\mu$, see~\cite{Hei} for example).
Then by the doubling property of $\mu$ and by the Gaussian lower bound for $p_t(x,y)$
(which follows from the $2$-Poincar\'e inequality together with the 
doubling property of $\mu$),
\begin{align*}
C\, \sqrt{t}\, P(E,X)&\ge \sum_i\int_{B_i\cap E}\int_{B_i\setminus E}
p_t(x,y)\, d\mu(x)\, d\mu(y)\\
&\ge C^{-1}\sum_i\int_{B_i\cap E}\int_{B_i\setminus E} 
 \frac{e^{-C}}{\mu(B_i)}\, d\mu(x)\, d\mu(y)\\
&\ge C^{-1}\sum_i\frac{\mu(B_i\cap E)\,\mu(B_i\setminus E)\, \mu(B_i)}{\mu(B_i)^2}.
\end{align*}
In the above computations, $C$ stands for various generic constants that depend only
on the doubling and Poincar\'e constants of the space, and the value of $C$ could change
at each occurrence. 
Note that at least one of $\mu(B_i\cap E)/\mu(B_i)$ and $\mu(B_i\setminus E)/\mu(B_i)$
is larger than $1/2$. 
Now by the definition of $\partial _\alpha^{r_0}E$
we obtain
\[
C\, P(E,X)\ge\alpha \sum_i\frac{\mu(B_i)}{\sqrt{t}}.
\]
Since $\sqrt{t}$ is the radius of each $B_i$, we get
\[
C\, P(E,X)\ge \alpha\limsup_{t\to 0^+}\sum_i\frac{\mu(B_i)}{\sqrt{t}}
\ge \alpha\, \mathcal{H}(\partial _\alpha^{r_0}E).
\]

If $0<r_1<r_0$, then 
\[
\partial_\alpha^{r_0}E\subset\partial_\alpha^{r_1}E\subset\partial_mE.
\]
Setting 
\[
\partial_\alpha E:=\bigcup_{0<r_0<1}\partial_\alpha^{r_0}E=\bigcup_{(0,1)\cap\mathbb{Q}}\partial_\alpha^{r_0}E,
\]
we now see by the continuity of measure that if  the sets $\partial_\alpha^{r_0}E$
are Borel sets, then
\[
\mathcal{H}(\partial_\alpha E)\le \frac{C}{\alpha}\, P(E,X).
\]
To see that $\partial_\alpha^{r_0}E$ is a Borel set we argue as follows. Recall that
we assume $\mu$ to be Borel regular. Therefore, given a $\mu$-measurable set $E$ and $r>0$, the function
\[
x\mapsto \mu(B(x,r)\cap E)
\]
is lower semicontinuous, the map
\[
\varphi_{E,r}(x):= \frac{\mu(B(x,r)\cap E)}{\mu(B(x,r))} 
\]
is a Borel function. Hence the function $\Phi_{E,r_0}:=\inf_{r\in\mathbb{Q}\cap(0,r_0]}\varphi_{E,r}$ is also a Borel function,
and hence 
\[
\partial_\alpha^{r_0}E=\{x\in X\, :\, \Phi_{E,r_0}(x)>\alpha\}\cap\{x\in X\, :\, \Phi_{X\setminus E,r_0}(x)>\alpha\}
\]
is a Borel set.
\end{proof}

Should $X$ support a $1$-Poincar\'e inequality, then by the results of~\cite{AMP},
there is a number $\gamma\in (0,1/2]$ such that
$\mathcal{H}(\partial_m E\setminus \partial_\gamma E)=0$; this number $\gamma$ depends
solely on the doubling and the $1$-Poincar\'e constants. Observe that the results of~\cite{AMP}
also tells us that if $E$ is of finite perimeter, then
$\mathcal{H}(\partial_m E)\simeq P(E,X)$.

\section{Metric characterization of $\mathbf{B}^{p,\alpha}(X)$}

Now we turn our attention to the study of Besov spaces in the metric setting. In this section,
we follow the notions given in~\cite{GKS}. For $0\le\alpha<\infty$, $1\le p<\infty$
and $p<q\le \infty$, we set $B^\alpha_{p,q}(X)$ to be the collection of all functions
$u\in L^p(X,\mu)$ for which 
\begin{equation}\label{eq:BesovMetric-q<infty}
\Vert u\Vert_{B^\alpha_{p,q}(X)}:=\left(\int_0^\infty
\left(\int_X\int_{B(x,t)}\frac{|u(y)-u(x)|^p}{t^{\alpha p}\mu(B(x,t))}\, d\mu(y)\, d\mu(x)
\right)^{q/p}\, \frac{dt}{t}\right)^{1/q}
\end{equation}
is finite when $q<\infty$, and
\begin{equation}\label{eq:BesovMetric-q=infty}
\Vert u\Vert_{B^\alpha_{p,\infty}(X)}:=\sup_{t>0}
\left(\int_X\int_{B(x,t)}\frac{|u(y)-u(x)|^p}{t^{\alpha p}\mu(B(x,t))}\, d\mu(y)\, d\mu(x)
\right)^{1/p}
\end{equation}
is finite when $q=\infty$. In the setting considered in the next chapter,
with $\mu$ an Ahlfors $d_H$-regular measure (that is, $\mu(B(x,r))\simeq r^{d_H}$
for all $x\in X$ and $0<r<2\text{diam}(X)$)
$\mathfrak{B}^{\alpha}_{p}(X)=B^\alpha_{p,\infty}(X)$.

\begin{proposition}\label{prop:B=B}
Suppose that $\mu$ is doubling and supports a $2$-Poincar\'e inequality. Then for
$1\le p<\infty$ and $0\le \alpha<\infty$ we have
\[
\mathbf{B}^{p,\alpha/2}(X)=B^{\alpha}_{p,\infty}(X),
\]
with equivalent seminorms.
\end{proposition}

\begin{proof}
Since $\mu$ is doubling and supports a $2$-Poincar\'e inequality, we have the
Gaussian double bound for $p_t(x,y)$. Hence if $u\in \mathbf{B}^{p,\alpha}(X)$, we then
must have
\begin{align*}
\Vert u\Vert_{\mathbf{B}^{p,\alpha/2}(X)}^p
&\ge C^{-1}\sup_{t>0}\int_X\int_X\frac{|u(y)-u(x)|^p}{t^{\alpha p/2}}\,
 \frac{e^{-C\, d(x,y)^2/t}}{\mu(B(x,\sqrt{t}))}\, d\mu(y)\, d\mu(x)\\
&\ge C^{-1}\sup_{\sqrt{t}>0}\int_X \int_{B(x,\sqrt{t})}
 \frac{|u(y)-u(x)|^p}{t^{\alpha p/2}}\,
 \frac{e^{-C\, d(x,y)^2/t}}{\mu(B(x,\sqrt{t}))}\, d\mu(y)\, d\mu(x)\\
&\ge C^{-1} \sup_{\sqrt{t}>0}\int_X \int_{B(x,\sqrt{t})}
 \frac{|u(y)-u(x)|^p}{\sqrt{t}^{\alpha p}\, \mu(B(x,\sqrt{t}))}
 \, d\mu(y)\, d\mu(x)\\
&=C^{-1}\Vert u\Vert_{B^\alpha_{p,\infty}(X)}^p,
\end{align*}
and from this it follows that $\mathbf{B}^{p,\alpha/2}(X)$ embeds boundedly into
$B^\alpha_{p,\infty}(X)$.

Now we focus on proving the converse embedding. 
Since $X$ supports $2$-Poincar\'e inequality, it is connected. Therefore by the doubling 
property of $\mu$ (see~\cite{Hei}), there are constants $0<Q\le s<\infty$
and $C\ge 1$
such that whenever $0<r\le R$, $x\in X$, and $y\in B(x,R)$,
\begin{equation}\label{eq:mass-bounds}
\frac{1}{C}\left(\frac{r}{R}\right)^s\le \frac{\mu(B(y,r))}{\mu(B(x,R))}
 \le C\left(\frac{r}{R}\right)^Q.
\end{equation}
Therefore,
\begin{align*}
&\frac{1}{t^{\alpha p/2}}\int_X\int_X|u(y)-u(x)|^p\, p_t(x,y)\, d\mu(y)\, d\mu(x)\\
&\le \frac{C}{t^{\alpha p/2}}\int_X\sum_{i=-\infty}^\infty\int_{B(x,2^i\sqrt{t})\setminus B(x,2^{i-1}\sqrt{t})}
 \frac{|u(y)-u(x)|^p\, e^{-C4^i}}{\sqrt{\mu(B(\sqrt{t}))\, \mu(B(y,\sqrt{t}))}}\, d\mu(y)\, d\mu(x)\\
 &\le \frac{C}{t^{\alpha p/2}}\int_X\sum_{i=-\infty}^\infty\int_{B(x,2^i\sqrt{t})}
 \frac{|u(y)-u(x)|^p\, e^{-C4^i}}{\mu(B(x,2^i\sqrt{t}))}\, \sqrt{\frac{\mu(B(x,2^i\sqrt{t}))}{\mu(B(x,\sqrt{t}))}}
 \, \sqrt{\frac{\mu(B(y,2^i\sqrt{t}))}{\mu(B(y,\sqrt{t}))}}\, d\mu(y)\, d\mu(x)\\
 &\le \frac{C}{t^{\alpha p/2}}\sum_{i=-\infty}^\infty e^{-C4^i}\,\max\{2^{is},2^{iQ}\}\, (2^i\sqrt{t})^{\alpha p}\,
 \int_X\int_{B(x,2^i\sqrt{t})}\frac{|u(y)-u(x)|^p}{(2^i\sqrt{t})^{\alpha p}\, \mu(B(x,2^i\sqrt{t}))}\, d\mu(y)\, d\mu(x)\\
 &\le C\, \Vert u\Vert_{B^\alpha_{p,\infty}(X)}^p\, \sum_{i=-\infty}^\infty e^{-C4^i}\, 2^{i\alpha p}\, \max\{2^{is},2^{iQ}\}.
\end{align*}
Since 
\[
\sum_{i=-\infty}^\infty e^{-C4^i}\, 2^{i\alpha p}\, \max\{2^{is},2^{iQ}\}\le 
\sum_{i\in\mathbb{N}}e^{-C4^i}\,  2^{i[\alpha p+s]}\, +\, \sum_{i=0}^\infty 2^{-i[\alpha p+Q]}<\infty,
\]
the desired bound follows.
\end{proof}

\begin{remark}
While the previous theorem gave us a way to control, from above, the $\mathcal{H}$-measure of 
$\partial_mE$ for a set $E$ of finite perimeter, Proposition~\ref{prop:B=B} give us a way
to control, from below, the co-dimension $1$ Minkowski measure of $\partial E$. Given set $A\subset X$,
the co-dimension $1$-Minkowski measure of $A$ is given by
\[
\mathcal{M}_{-1}(A):=\liminf_{\varepsilon\to 0^+} \frac{\mu(A_\varepsilon)}{\varepsilon},
\]
where $A_\varepsilon=\bigcup_{x\in A}B(x,\varepsilon)$. As $\mu$ is doubling, we see 
with the aid of Proposition~\ref{prop:B=B} that if
$E\subset X$ is a measurable set such that $\mathbf{1}_E\in \mathbf{B}^{1,1/2}(X)$, then
\[
\liminf_{t\to 0^+}\int_X\int_{B(x,t)}\frac{|\mathbf{1}_E(x)-\mathbf{1}_E(y)|}{t^{1/2}\mu(B(x,t))}\, d\mu(y)\, d\mu(x)
\le \mathcal{M}_{-1}(\partial E),
\]
with the limit infimum on the left-hand side a finite number.
\end{remark}

\begin{remark}
Another application of Proposition~\ref{prop:B=B} is the following. It is in general not true that
if $\Vert Du\Vert(X)=0$ then $u$ is constant almost everywhere in $X$, even if $X$ is connected. 
Should $X$ support a $1$-Poincar\'e inequality, the fact that $\Vert Du\Vert(X)=0$ implies $u$ is constant
follows immediately. We can use the above proposition to show that even if we do not have $1$-Poincar\'e 
inequality, if $X$ supports the Bakry-\'Emery curvature condition~\eqref{eq:weak-BE}, then
we still have the constancy of $u$ from $\Vert Du\Vert(X)=0$, for then
\[
\sup_{t>0} \int_X\int_{B(x,t)}\frac{|\mathbf{1}_E(x)-\mathbf{1}_E(y)|}{t^{1/2}\mu(B(x,t))}\, d\mu(y)\, d\mu(x)
 \simeq \Vert Du\Vert(X).
\]
\end{remark}

To conclude this section, we point out the following result that shows that $1/2$ is a  critical index for the spaces $\mathbf{B}^{1,\alpha}(X)$. 

\begin{proposition}
Let $f \in \mathbf{B}^{1,\alpha}(X)$ with $\alpha >1/2$. Then, $f$ is constant.
\end{proposition}

\begin{proof}
Let $f \in \mathbf{B}^{1,\alpha}(X)$ with $\alpha >1/2$. Since $\mathbf{B}^{1,\alpha}(X) \subset \mathbf{B}^{1,1/2}(X)=BV(X)$, we deduce that $f$ is a BV function. Now since $f \in \mathbf{B}^{1,\alpha}(X)$, one has for every $ t >0$,
\[
\int_X \int_X p_t(x,y) |f(x)-f(y)| d\mu(x) d\mu(y) \le t^\alpha \| f \|_{1,\alpha}.
\]
By using the heat kernel lower bound as before, it implies
\[
\liminf_{\varepsilon\to 0^+}\frac{1}{\varepsilon}
 \iint_{\Delta_\varepsilon}\frac{|f(y)-f(x)|}{\sqrt{\mu(B(x,\varepsilon))\mu(B(y,\varepsilon))}} d\mu(x) d\mu(y) =0.
 \]
 Therefore $\| Df(X) \|=0$, so from the previous remark $f$ is constant.
\end{proof}

\section{Comparison between Sobolev and Besov seminorms}

In this section, we compare the Besov and Sobolev seminorms for $p>1$. The case $p=1$ was studied in detail in Section 4.3.

Throughout the section we will assume that $(X,d_\mathcal{E},\mu)$ is doubling, supports a $2$-Poincar\'e inequality.
As before, we will also assume that
for each $u\in\mathcal{F}$ the measure $\Gamma(u,u)$ is absolutely continuous with respect to the underlying measure
$\mu$ on $X$; as in the previous section we will denote by $|\nabla u|^2$ the Radon-Nikodym derivative of
$\Gamma(u,u)$ with respect to $\mu$.  We also assume that $\mathcal{E}$ is regular and strictly local.

\subsection{Lower bound for the Besov seminorm}

\begin{theorem}\label{thm:BesovLB}
Let $p >1$. There exists a constant $C>0$ such that for every $ u \in \B^{p,1/2}(X)\cap \mathcal{F}$,
\[
\| |\nabla u| \|_{L^p(X,\mu)} \le C \| u \|_{p,1/2}
\]
\end{theorem}

\begin{proof}
Let $u\in\B^{p,1/2}(X)\cap\mathcal{F}$. Then as in the derivation of~\eqref{eq:Ledoux1} we obtain that
for each $\eps>0$, 
\[
\frac{1}{\eps^p}\iint_{\Delta_\eps}\frac{|u(x)-u(y)|^p}{\mu(B(x,\eps))}\, d\mu(y)\, d\mu(x)\le \| u \|_{p,1/2}<\infty.
\]
Fix $\eps>0$ and cover $X$ by family of balls $B_i^\eps=B(x_i^\eps,\eps)$ such that $\frac12 B_i^\eps$ are pairwise
disjoint, and let $\pip_i^\eps$ be a $C/\eps$-Lipschitz partition of unity subordinate to this cover: that is, 
$0\le \pip_i^\eps\le 1$ on $X$, $\sum_i\pip_i^\eps=1$ on $X$, and $\pip_i^\eps=0$ in $X\setminus B_i^\eps$. We then
set 
\[
u_\eps:=\sum_i u_{B_i^\eps}\, \pip_i^\eps.
\]
Then as each $\pip_i^\eps$ is Lipschitz, we know that $u_\eps$ is locally Lipschitz and hence is in $\mathcal{F}_{loc}(X)$.
For $x,y\in B_j^\eps$ we see that
\begin{align*}
|u_\eps(x)-u_\eps(y)|&\le \sum_{i:2B_i^\eps\cap 2B_j^\eps\ne\emptyset}|u_{B_i^\eps}-u_{B_j^\eps}||\pip_i^\eps(x)-\pip_i^\eps(y)|\\
 &\le \frac{C\, d(x,y)}{\eps} \sum_{i:2B_i^\eps\cap 2B_j^\eps\ne\emptyset}
    \left(\vint_{B_i^\eps}\vint_{B(x,2\eps)}|u(y)-u(x)|^p\, d\mu(y)\, d\mu(x)\right)^{1/p}.
\end{align*}
Therefore, by Lemma~4.1, we see that
\begin{align*}
|\nabla u_\eps|&\le \frac{C}{\eps}\sum_{i:2B_i^\eps\cap 2B_j^\eps\ne\emptyset}
    \left(\vint_{B_i^\eps}\vint_{B(x,2\eps)}|u(y)-u(x)|^p\, d\mu(y)\, d\mu(x)\right)^{1/p}\\
    &\le C\left(\vint_{2B_j^\eps} \vint_{B(x,2\eps)}\frac{|u(y)-u(x)|^p}{\eps^p}\, d\mu(y)\, d\mu(x)\right)^{1/p},
\end{align*}
and so by the bounded overlap property of the collection $2B_j^\eps$ (which is a consequence of the doubling 
property of $\mu$),
\begin{align*}
\int_X|\nabla u_\eps|^p\, d\mu &\le \sum_j \int_{B_j^\eps}|\nabla u_\eps|^p\, d\mu\\
  &\le C\, \sum_j \int_{2B_j^\eps} \vint_{B(x,2\eps)}\frac{|u(y)-u(x)|^p}{\eps^p}\, d\mu(y)\, d\mu(x)\\
  &\le C\, \int_X \vint_{B(x,2\eps)}\frac{|u(y)-u(x)|^p}{\eps^p}\, d\mu(y)\, d\mu(x)\\
  &\le C\, \frac{1}{\eps^p}\int_{\Delta_{2\eps}}\frac{|u(x)-u(y)|^p}{\mu(B(x,\eps))}\, d\mu(y)\, d\mu(x)\le 2M.
\end{align*}
In a similar manner, we can also show that 
\[
\int_X|u_\eps(x)-u(x)|^p\, d\mu(x)\le C \eps^p \int_{\Delta_{2\eps}}\frac{|u(x)-u(y)|^p}{\eps^p\, \mu(B(x,\eps))}\, d\mu(y)\, d\mu(x)
  \le CM\, \eps^p,
\]
that is, $u_\eps\to u$ in $L^p(X)$ as $\eps\to 0^+$. So we get that if $u\in\mathcal{F}$, then
\[
C\, \| u\|_{p,1/2}\ge \| |\nabla u| \|_{L^p(X)}.
\]
\end{proof}

\subsection{Upper bound for the Besov seminorm}

We now turn to the proof of the upper bound for the Besov seminorm in terms of the Sobolev seminorm. We will use an additional assumption: The strong Bakry-\'Emery curvature condition. As before, $P_tu$ will denote the heat extension of $u$ via the heat semigroup associated with the Dirichlet form $\mathcal{E}$. We recall that from classical theory (see for instance Theorems 1.4.1 and 1.4.2 in \cite{EBD}), the semigroup $P_t$ lets $L^1(X,\mu) \cap L^\infty (X,\mu)$ invariant and may be extended to a positive, contraction semigroup on $L^p (X,\mu)$, $1 \le p \le +\infty$, that we still denote by $P_t$. Moreover, from the same reference, for $1 \le p <+\infty$, $P_t$ is strongly continuous  (i.e. $P_t u \to u$ in $L^p(X,\mu)$ when $t \to 0$). We will denote by $L$ the generator of $P_t$.

\

\textbf{Assumption:}
\textit{In this subsection, in addition to the assumptions stated at the beginning of Section 4.5, we will furthermore assume the strong Bakry-\'Emery curvature condition: There exists a constant $C>0$ such that for every $u \in \mathcal{F}$ and $t \ge 0$ we have $\mu$ a.e.
\begin{equation}\label{eq:strong-BE}
 |\nabla P_t u|  \le C P_t |\nabla u|  .
\end{equation}
}

\

We note that the strong Bakry-\'Emery curvature condition  implies the weak one, see the proof of Theorem 3.3 in  \cite{BK}. Examples where the strong Bakry-\'Emery estimate is satisfied include: Riemannian manifolds with non negative Ricci curvature and more generally $\mathbf{RCD}(0,+\infty)$ spaces (see \cite{EKS}), some metric graphs (see \cite{BK}), the Heisenberg group and more generally H-type groups (see \cite{BBBC, Eldredge})

A first important corollary of the strong Bakry-\'Emery estimate is the following Hamilton's type gradient estimate for the heat kernel. Such type of estimate is well-known on Riemannian manifolds with non-negative Ricci curvature (see for instance  \cite{Kotschwar}), but is new in our general framework. 

\begin{theorem}\label{Hamilton estimate}
There exists a constant $C>0$ such that for every $t>0$, $x,y \in X$,
\[
| \nabla_x \ln p_t (x,y) |^2 \le \frac{C}{t} \left( 1 +\frac{d(x,y)^2}{t} \right)
\]
\end{theorem}

\begin{proof}
The proof proceeds in two steps. 

\textbf{Step 1:}  We first collect a gradient bound for the heat kernel. We observe that \eqref{eq:strong-BE} implies a weaker $L^2$ version as follows
\[
 |\nabla P_t u|^2  \le C P_t (|\nabla u|^2), 
\]
and hence the pointwise  heat kernel gradient bound (see \cite[Lemma 3.3]{ACDH})
\[
|\nabla_x p_t(x,y)| \le \frac{C}{\sqrt t} \frac{e^{-cd(x,y)^2/t}}{\sqrt{\mu(B(x,\sqrt t)) \mu(B(y,\sqrt t))}}. 
\]
In particular, we note that $|\nabla_x p_t(x,\cdot)| \in L^p(X,\mu)$ for every $p \ge 1$.

\

\textbf{Step 2:} In a second step, we prove a reverse log-Sobolev inequality for the heat kernel. Let $\tau,\ve>0$ and $x\in X$ be fixed in this Step 2.

We denote $u=p_\tau (x,\cdot)+\ve$. One has, from the chain rule for stricly local forms \cite[Lemma 3.2.5]{FOT},

\begin{align*}
P_t (u \ln u)-P_t u \ln P_t u &=\int_0^t \partial_s \left( P_s (P_{t-s} u  \ln P_{t-s} u) \right) ds \\
&=\int_0^t LP_s (P_{t-s} u \ln P_{t-s} u) -P_s (LP_{t-s}u \ln P_{t-s} u) -P_s(LP_{t-s}u)  ds \\
&=\int_0^tP_s (L(P_{t-s} u \ln P_{t-s} u)) -P_s (LP_{t-s}u \ln P_{t-s} u) -P_s(LP_{t-s}u)  ds\\
&=\int_0^tP_s \left[ L(P_{t-s} u \ln P_{t-s} u)) -LP_{t-s}u \ln P_{t-s} u -LP_{t-s}u \right]  ds \\
 &=\int_0^t   P_s ((P_{t-s} u) | \nabla \ln P_{t-s} u|^2)  ds, \\
 \end{align*}
 where the above computations may be justified by using the Gaussian heat kernel estimates for the heat kernel and the Gaussian upper bound for the gradient of the heat kernel obtained in Step 1. In particular, we point out that the commutation $ LP_s (P_{t-s} u \ln P_{t-s} u) =P_s (L(P_{t-s} u \ln P_{t-s} u)$ is justified by noting that $P_{t-s} u \ln P_{t-s} u -\ve \ln \ve$ is in the domain of $L$ in $L^2(X,\mu)$.
Here, $L$ is the infinitesimal generator (the Laplacian operator) associated with $\mathcal{E}$. 
 
Thus,  from the strong Bakry-\'Emery estimate and the Cauchy-Schwarz inequality in the form of $P_s\left( \frac{f^2}{g}\right) \ge \frac{(P_sf)^2}{P_s g}$, we obtain
\begin{align*}
P_t (u \ln u)-P_t u \ln P_t u  &=\int_0^t   P_s\left(\frac{| \nabla P_{t-s}u |^2 }{P_{t-s} u } \right)   ds \\
 & \ge \int_0^t  \frac{(P_s | \nabla P_{t-s} u |)^2}{  P_s(P_{t-s} u) }    ds \\
 &\ge \frac{1}{C} \frac{1}{P_t u} \int_0^t | \nabla P_t u|^2 ds \\
  &\ge \frac{t}{C}   \frac{1}{P_t u}  | \nabla P_t u|^2 
\end{align*}


Coming back to the definition of $u$, noting that $P_t p_\tau(x,\cdot)=P_{t +\tau}(x,\cdot)$ and applying the previous inequality with $t=\tau$ one deduces
\[
| \nabla_{y} \ln (p_{2t} (x,y)+\ve)|^2\le \frac{C}{t} P_t \left[  \ln \left( \frac{M_t(x)+\ve }{p_{2t} (x,\cdot) +\ve} \right) \right] (y)
\]
where $M_t(x)=\sup_{y \in X} p_t (x,y) $. Therefore, by letting $\ve \to 0$ and using the Gaussian heat kernels, one concludes
\[
| \nabla_y \ln p_{2t} (x,y) |^2 \le \frac{C}{t} \left( 1 +\frac{d(x,y)^2}{t} \right)
\]
and thus
\[
| \nabla_y \ln p_{t} (x,y) |^2 \le \frac{C}{t} \left( 1 +\frac{d(x,y)^2}{t} \right).
\]
Finally, we note that the inequality
\[
| \nabla_x \ln p_{t} (x,y) |^2 \le \frac{C}{t} \left( 1 +\frac{d(x,y)^2}{t} \right).
\]
is obtained similarly be exchanging the roles of $x$ and $y$ in our proof.
\end{proof}

\begin{corollary}
Let $p>1$. There exists a constant $C>0$ such that for every $u \in L^p(X,\mu)$,
\[
| \nabla P_t u| \le \frac{C}{\sqrt{t}} P_t( |u|^p)^{1/p}.
\]
\end{corollary}

\begin{proof}
Let $p>1$, $q$ the conjugate exponent and $u \in L^p(X,\mu)$. One has from H\"older's inequality
\begin{align*}
| \nabla P_t u| (x) & \le \int_X |\nabla _x p_t (x,y)| u(y) d\mu(y) \\
& \le \left( \int_X \frac{|\nabla _x p_t (x,y)|^q}{p_t(x,y)^{q/p}}  d\mu(y)\right)^{1/q} P_t( |u|^p)^{1/p} \\
&\le  \left( \int_X | \nabla_x \ln p_t(x,y)|^q p_t(x,y)  d\mu(y)\right)^{1/q} P_t( |u|^p)^{1/p}
\end{align*}
The proof follows then from Theorem \ref{Hamilton estimate} and the Gaussian upper bound for the heat kernel.
\end{proof}

Note that, by integrating over $X$ the previous proposition immediately yields:
\begin{lemma}\label{lem:LpGradient}
Let $p>1$. There exists a constant $C>0$ such that for every $u \in L^p(X,\mu)$
\[
\Vert |\nabla P_tu| \Vert_{L^p(X,\mu)}^2\le \frac{C}{t}\Vert u\Vert_{L^p(X,\mu)}^2.
\]
\end{lemma}

We deduce then:

\begin{lemma}\label{lem:HSminusF1}
Let  $p > 1$. There exists a constant $C>0$ such that for every $u \in L^p(X,\mu)\cap\mathcal{F}$ with
$|\nabla u|\in L^p(X,\mu)$
\[
\| P_t u -u \|_{L^p(X,\mu)} \le C \sqrt{t} \| |  \nabla u | \|_{L^p(X,\mu)}
\]
\end{lemma}

\begin{proof}
With the previous lemma in hand, the proof is similar to the one in Lemma~4.10, with $\varphi$ in 
$\mathcal{F}\cap L^q(X,\mu)$ where $p^{-1}+q^{-1}=1$ As compactly supported Lipschitz functions form a 
dense subclass of $L^p(X,\mu)$ and also belong to $\mathcal{F}$ from the discussion in Chapter~4, we recover
the $L^p$-norm of $P_tu-u$ by taking the supremum over all such $\varphi$ with $\int_X|\varphi|^q\, d\mu<\infty$.
\end{proof}

\begin{lemma}\label{lem:HSminusF2}
Let $p >1$, then
\[
\left( \int_X \int_X | P_tu(x)-u (y) |^p p_t (x,y) d\mu(x) d\mu(y) \right)^{1/p} \le C \sqrt{t}  \| |  \nabla u | \|_{L^p(X)}.
\]
\end{lemma}

\begin{proof}
 Let $u\in L^p(X)$ and $t >0 $ be fixed in the above argument. By an application of Fubini's theorem we have
\[
\left( \int_X \int_X | P_tu(x)-u (y) |^p p_t (x,y) d\mu(x) d\mu(y) \right)^{1/p}=\left( \int_X P_t ( | P_t u(x) -u|^p)(x) d\mu(x)\right)^{1/p}.
\]

The main idea now is to  adapt the proof of~\cite[Theorem~6.2]{BBBC}.

As above, let $q$ be the conjugate of $p$. Let $x \in X$ be fixed. Let $g$ be a function in $L^\infty(X,\mu)$  such that $P_t (|g|^q)(x) \le 1$. 

We first note that from the chain rule:

\begin{align*}
& \partial_s \left[ P_s  ((P_{t-s} u) (P_{t-s} g))(x) \right]  \\
=& LP_s  ((P_{t-s} u) (P_{t-s} g))(x) -P_s  ((LP_{t-s} u) (P_{t-s} g))(x) -P_s  ((P_{t-s} u) (LP_{t-s} g))(x) \\
=&P_s  (L(P_{t-s} u) (P_{t-s} g))(x) -P_s  ((LP_{t-s} u) (P_{t-s} g))(x) -P_s  ((P_{t-s} u) (LP_{t-s} g))(x)  \\
=&2P_s (\Gamma(P_{t-s}u,P_{t-s}g))
\end{align*}

Therefore, we have

\begin{align*}
P_t ( ( u-P_t u(x) ) g)(x) & =P_t (ug)(x)-P_t u(x) P_t g (x) \\
 &=\int_0^t \partial_s \left[ P_s  ((P_{t-s} u) (P_{t-s} g))(x) \right] ds \\
 &=2 \int_0^t P_s \left( \Gamma(P_{t-s} u,P_{t-s} g)  \right) (x) ds \\
 &\le 2 \int_0^t P_s \left( | \nabla P_{t-s} u | | \nabla P_{t-s} g| )  \right) (x) ds \\
 &\le 2 \int_0^t P_s \left( | \nabla P_{t-s} u |^p \right)^{1/p}(x) P_s\left(  | \nabla P_{t-s} g|^q  \right)^{1/q} (x) ds
\end{align*}
Now from the strong BE estimate and H\"older's inequality we have
\[
P_s \left( | \nabla P_{t-s} u |^p \right)^{1/p}(x) \le  CP_s  \left( P_{t-s}( | \nabla  u |^p) \right)^{1/p}(x) =C P_t ( | \nabla  u |^p)^{1/p}(x) 
\]
On the other hand
\[
 | \nabla P_{t-s} g|^q\le \frac{C}{(t-s)^{q/2}} P_{t-s} (|g|^q)
 \]
 Thus,
 \[
 P_s\left(  | \nabla P_{t-s} g|^q  \right)^{1/q} (x) \le \frac{C}{(t-s)^{1/2}} P_{t} (|g|^q)^{1/q}(x) \le \frac{C}{(t-s)^{1/2}} .
 \]
 One concludes
 \[
 P_t ( ( u-P_t u(x) ) g)(x) \le C \sqrt{t} P_t ( | \nabla  u |^p)^{1/p}(x) 
 \]
 Thus by $L^p-L^q$ duality, one concludes
 \[
  P_t ( | u-P_t u(x)| ^p)(x)^{1/p} \le C \sqrt{t} P_t ( | \nabla  u |^p)^{1/p}(x) 
 \]
 and finishes the proof by integration over $X$.
\end{proof}

We are now finally in position to prove the main result of the section.

\begin{theorem}\label{thm:BesovUB}
Let $p>1$. There exists a constant $C>0$ such that for every  $u \in L^p(X,\mu)\cap \mathcal{F}$ with $|\nabla u| \in L^p(X,\mu)$,
\[
\| u \|_{p,1/2} \le C \| |  \nabla u | \|_{L^p(X)}
\]
\end{theorem}

\begin{proof}
One has
\begin{align*}
 & \left( \int_X \int_X | u(x)-u(y)|^p p_t (x,y) d\mu(x) d\mu(y) \right)^{1/p}  \\
 \le &  \left( \int_X \int_X | u(x)-P_tu (x) |^p p_t (x,y) d\mu(x) d\mu(y) \right)^{1/p} + \left( \int_X \int_X | P_tu(x)-u(y)|^p p_t (x,y) d\mu(x) d\mu(y) \right)^{1/p}   \\
 \le &  \| P_t u -u \|_{L^p(X)} +  \left( \int_X \int_X | P_tu(x)-u (y) |^p p_t (x,y) d\mu(x) d\mu(y) \right)^{1/p} 
\end{align*}
One has from the previous lemma
\[
 \| P_t u -u \|_{L^p(X)} \le C \sqrt{t} \| |  \nabla u | \|_{L^p(X)}
\]
and the term $\left( \int_X \int_X | P_tu(x)-u (y) |^p p_t (x,y) d\mu(x) d\mu(y) \right)^{1/p}$ can be controlled by the above lemma.
\end{proof}

In view of  Theorems \ref{thm:BesovLB} and \ref{thm:BesovUB}, as well as  \cite[Theorem 1.4]{ACDH},   we conclude this section by the following result on the Riesz transform.
\begin{corollary}\label{thm:RT}
Let $p>1$. Then for any $f\in L^p(X) \cap \mathcal F$,
\[
 \Vert  f\Vert_{p,1/2} \simeq \Vert \sqrt{-L} f \Vert_{L^p(X)}.
\]
Consequently,
\[
\mathbf B^{p,1/2}(X) =\mathcal L_p^{1/2},
\]
where $\mathcal{L}_p^{1/2}$ defined in Section 1.4 is the domain of the operator $\sqrt{-L}$ in $L^p(X,\mu)$.
\end{corollary}

\section{Critical exponents}

In this section, our assumptions are the same as Subsection 4.5.2., that is, the standing assumptions of this chapter as well as the strong Bakry-\'Emery curvature condition \eqref{eq:strong-BE}. We first notice that as a consequence of Lemma \ref{lem:LpGradient} and Theorem \ref{thm:BesovUB},
\begin{theorem} 
Let $p>1$. There exists a constant $C>0$ such that for every $f\in L^p(X, \mu)$ and $t>0$
\[
\Vert P_t f \Vert_{p,1/2} \le \frac{C}{t^{1/2}} \Vert f \Vert_{L^p(X)}.
\]
\end{theorem}


\begin{remark}
The bound 
\[
\Vert P_t f \Vert_{p,1/2} \le \frac{C}{t^{1/2}} \Vert f \Vert_{L^p(X)}, \quad p>1,
\]
likely holds if one only assumes the weak Bakry-\'Emery estimate. This may be seen  by adapting to this setting the argument yielding Theorem \ref{P:PtinBesovp4} in Chapter 5. For concision, we will not expand this remark.
\end{remark}

Similarly as in Section 1.7, we have several corollaries of the above result.

\begin{corollary}
Let $p > 1$. Let $L$ be the infinitesimal generator of $\mathcal{E}$ and $\mathcal{L}_p$ be the generator of $L$  
in $L^p(X,\mu)$.  Then
\[
\mathcal{L}_p \subset \B^{p,1/2}(X)
\]
and for every $f \in \mathcal{L}_p$,
\begin{equation}\label{eq:multi2}
\|f\|^2_{p,1/2} \le C_p \norm{ Lf}_{L^p(X,\mu)} \| f\|_{L^p(X,\mu)}.
\end{equation}
\end{corollary}

Recall that $L$ plays the role of Laplacian in the theory of Dirichlet forms, see~\cite{FOT}.

\begin{corollary}
For $p>1$, $\mathbf B^{p,1/2}(X)$ is dense in $L^p(X,\mu)$.
\end{corollary}

\begin{corollary}
Let $p>1$. For every $f\in L^p(X,\mu)$, and $t>0$,
\[
\| P_t f -f \|_{L^p(X,\mu)} \le C_p   t^{1/2}  \liminf_{s \to 0}  s^{-1/2} \left( \int_X P_s (|f-f(y)|^p)(y) d\mu(y) \right)^{1/p}
\]
\end{corollary}

\begin{corollary}
Let $p>1$. If $f\in \mathbf B^{p,\alpha}(X)$ with $\alpha>1/2$ then $\mathcal E(f,f)=0$.
\end{corollary}

In particular, with the notation of  Section 1.8, one concludes:

\begin{proposition}
For every $p \ge 1$, $\alpha^*_p(X)=1/2$.
\end{proposition}

\section{Sobolev and isoperimetric inequalities}\label{section Sobolev local}

Combining the conclusions of this chapter with the results in Chapter \ref{section sobolev}, we immediately obtain the 
following  results that encompass many of the known results found in the literature.

\begin{corollary}
Suppose that $\mu$ is doubling and 
supports a $2$-Poincar\'e inequality. Assume moreover that  the volume growth condition 
$\mu(B(x,r)) \ge C_1 r^Q$,  $r \ge 0$,  is satisfied for some $Q>0$,
 Then, one has the following weak type Besov space embedding. Let $0<\delta < Q $. Let $1 \le p < \frac{Q}{\delta} $.   There exists a constant $C_{p,\delta} >0$ such that for every $f \in \mathbf{B}^{p,\delta/2}(X) $,
\[
\sup_{s \ge 0} s \mu \left( \{ x \in X, | f(x) | \ge s \} \right)^{\frac{1}{q}} \le C_{p,\delta} \sup_{r>0} \frac{1}{r^{\delta+Q/p}}\biggl(\iint_{\{(x,y)\in X\times X\mid d(x,y)<r\}}|f(x)-f(y)|^{p}\,d\mu(x)\,d\mu(y)\biggr)^{1/p}
\]
where $q=\frac{p Q}{ Q -p \delta}$. Furthermore, for every $0<\delta <Q $, there exists a constant $C_{\emph{iso},\delta}$ such that for every measurable $E \subset X$, $\mu(E) <+\infty$,
\begin{align*}
\mu(E)^{\frac{d_H-\delta}{d_H}} 
\le C_{\emph{iso},\delta} \sup_{r>0} \frac{1}{r^{\delta+Q}} (\mu \otimes \mu) \left\{ (x,y) \in E \times E^c\, :\, d(x,y) \le r\right\} 
\end{align*}
\end{corollary}

\begin{proof}
From  the heat kernel upper bound, the  volume growth condition $\mu(B(x,r)) \ge C_1 r^Q$,  $r \ge 0$, implies the ultracontractive estimate
\[
p_t(x,y) \le \frac{C}{t^{Q/2}}.
\]
We are therefore in the framework of Chapter 2. From Theorem \ref{pol}, one deduces therefore the following. Let $0<\delta < Q $. Let $1 \le p < \frac{Q}{\delta} $.   There exists a constant $C_{p,\delta} >0$ such that for every $f \in \mathbf{B}^{p,\delta/2}(X) $,
\[
\sup_{s \ge 0}\, s\, \mu \left( \{ x \in X\, :\, | f(x) | \ge s \} \right)^{\frac{1}{q}} \le C_{p,\delta}  \| f \|_{p,\delta/2}
\]
where $q=\frac{p Q}{ Q -p \delta}$. We conclude then with Theorem \ref{prop:B=B}.

\end{proof}

In the case where $\delta=1/2$ and the weak Bakry-\'Emery estimate is satisfied, then one gets a strong Sobolev inequality.

\begin{theorem}
Suppose that $\mu$ is doubling and 
supports a $2$-Poincar\'e inequality and that the weak Bakry-\'Emery estimate is satisfied. If the volume growth condition 
$\mu(B(x,r)) \ge C_1 r^Q$,  $r \ge 0$,  is satisfied for some $Q>0$, then there exists a constant $C_2 >0$ such 
that for every $f \in BV(X)$,
\[
\| f \|_{L^q(X,\mu)} \le C_2 \| Df \|(X)
\]
where $q=\frac{Q}{ Q-1}$. In particular, if $E$ is a set with finite perimeter in $X$, then
\[
\mu(E)^{\frac{Q-1}{Q}} \le C_2 P(E,X).
\]
\end{theorem}

\begin{proof}
We apply Theorem \ref{Sobolev}. Since $\| f \|_{1,1/2} \le C \| Df \|(X)$, it is enough to prove that $(P_{1,1/2})$ is satisfied. This follows from
\[
\sup_{s >0 }  s^{-1/2}  \int_X P_s (|f-f(y)|)(y) d\mu(y)  \le C\liminf_{s \to 0}  s^{-1/2}  \int_X P_s (|f-f(y)|)(y) d\mu(y) ,
\]
which is a consequence from Theorem \ref{thm:W=BV}.
\end{proof}

%% file: chapter5.tex
\chapter{Strongly local Dirichlet metric spaces with sub-Gaussian heat kernel estimates}
\label{LDsGHKU}

We now turn to the study  strongly local Dirichlet spaces 
which \emph{do not have Gaussian heat kernel bounds}~\eqref{eq:heat-Gauss}.  
 The main class of examples we are interested in are fractal spaces. 
We refer to \cite{Ba98,Gri,KigB,Kig:RFQS}  for further details on the following framework/assumptions. We note that some of the most basic facts of 
Chapter~\ref{Sec:Metric-Curvature}, such as Lemma~\ref{lem:diff-of-Lip}, do not hold true in this chapter because energy measures $\Gamma(\cdot,\cdot)$ are not  
$\mu$-absolutely continuous. Therefore we can not use locally Lipschitz functions for our analysis, and need to develop a different set of tools. 

Let $(X,d,\mu)$ be a metric measure space. We assume that $B(x,r):=\{y\in X\mid d(x,y)<r\}$ has compact closure for any $x\in X$ and any $r\in(0,\infty)$, and that $\mu$ is Ahlfors $H$-regular, i.e.\ there exist $c_{1},c_{2},d_{H}\in(0,\infty)$ such that $c_{1}r^{d_{H}}\leq\mu\bigl(B(x,r)\bigr)\leq c_{2}r^{d_{H}}$ for any $r\in\bigl(0,+\infty\bigr)$. Furthermore, we assume that on $(X,\mu)$ there is a measurable  heat kernel $p_t(x,y)$ satisfying, for some $c_{3},c_{4}, c_5, c_6 \in(0,\infty)$ and $d_{W}\in [2,+\infty)$,
\begin{equation}\label{eq:subGauss-upper}
c_{5}t^{-d_{H}/d_{W}}\exp\biggl(-c_{6}\Bigl(\frac{d(x,y)^{d_{W}}}{t}\Bigr)^{\frac{1}{d_{W}-1}}\biggr) 
\le p_{t}(x,y)\leq c_{3}t^{-d_{H}/d_{W}}\exp\biggl(-c_{4}\Bigl(\frac{d(x,y)^{d_{W}}}{t}\Bigr)^{\frac{1}{d_{W}-1}}\biggr)
\end{equation}
for $\mu\!\times\!\mu$-a.e.\ $(x,y)\in X\times X$ and each $t\in\bigl(0,+\infty\bigr)$. We refer to \cite{P-P10}, assumptions (A1) to (A5), page 201, for the definition of  heat kernels on metric measure spaces.
The corresponding Dirichlet space will be denoted by $(X,\mu,\mathcal{E},\mathcal{F})$ and the corresponding semigroup by $\{P_t, t\ge 0 \}$.  We assume that $(X,\mu,\mathcal{E},\mathcal{F})$ is a strongly local regular symmetric Dirichlet space. The parameter $d_H$ is the Hausdorff dimension and the parameter $d_W$ the so-called walk dimension. It is possible to prove that if the metric space $(X,d)$ satisfies a chain condition, then $ 2 \le d_W \le d_H+1$. When $d_W=2$, one speaks of Gaussian estimates and when $d_W > 2$, one speaks then of sub-Gaussian estimates. 

 In this framework, it is known that $(X,\mu,\mathcal{E},\mathcal{F})$ is conservative. Let $p \ge 1$ and $\alpha \ge 0$. As before, we define the Besov seminorm
\[
\| f \|_{p,\alpha}= \sup_{t >0} t^{-\alpha} \left( \int_X \int_X |f(x)-f(y) |^p p_t (x,y) d\mu(x) d\mu(y) \right)^{1/p}
\]
and define
\[
\mathbf{B}^{p,\alpha}(X)=\{ f \in L^p(X,\mu)\, :\,  \| f \|_{p,\alpha} <+\infty \}.
\]

\section{Metric characterization of Besov spaces}

Our goal in this section is to compare the space $\mathbf{B}^{p,\alpha}(X)$ to Besov type spaces previously 
considered in a similar framework (see \cite{Gri}). In the following, for $r>0$ we set
\[
\Delta_r:=\{(x,y)\in X\times X\, :\, d(x,y)<r\}.
\]

For $\alpha\in[0,\infty)$ and $p\in[1,\infty)$, we introduce the following
Besov seminorm: for $f\in L^{p}(X,\mu)$,
\begin{equation}\label{eq:Besov-seminorm-r}
N^{\alpha}_{p}(f,r):=\frac{1}{r^{\alpha+d_{H}/p}}\biggl(\iint_{\Delta_r}|f(x)-f(y)|^{p}\,d\mu(x)\,d\mu(y)\biggr)^{1/p}
\end{equation}
for $r\in(0,\infty)$, and
\begin{equation}\label{eq:Besov-seminorm}
N^{\alpha}_{p}(f):=\sup_{r\in(0,1]}N^{\alpha}_{p}(f,r).
\end{equation}
We then define the Besov space $\mathfrak{B}^{\alpha}_{p}(X)$ by
\begin{equation}\label{eq:Besov}
\mathfrak{B}^{\alpha}_{p}(X):=\bigl\{f\in L^{p}(X,\mu)\, :\, N^{\alpha}_{p}(f)<\infty\bigr\}.
\end{equation}
With the notation of the previous section, note that $\Vert f\Vert_{B^\alpha_{p,\infty}(X)}\simeq \sup_{r>0} N_p^\alpha(f,r)$.
It is clear that $\mathfrak{B}^{\alpha_{2}}_{p}(X)\subset\mathfrak{B}^{\alpha_{1}}_{p}(X)$
for $\alpha_{1},\alpha_{2}\in[0,\infty)$ with $\alpha_{1}\leq\alpha_{2}$. Note also that
\begin{equation}\label{eq:Besov-limsup}
\mathfrak{B}^{\alpha}_{p}(X)=\biggl\{f\in L^{p}(X,\mu)\, :\, \limsup_{r\downarrow 0}N^{\alpha}_{p}(f,r)<\infty\biggr\}.
\end{equation}
Indeed, if $f\in\mathfrak{B}^{\alpha}_{p}(X)$ then $f\in L^{p}(X,\mu)$ and
$\limsup_{r\downarrow 0}N^{\alpha}_{p}(f,r)\leq N^{\alpha}_{p}(f)<\infty$.
Conversely, for any $f\in L^{p}(X)$ with
$\limsup_{r\downarrow 0}N^{\alpha}_{p}(f,r)<\infty$,
we have $\sup_{r\in(0,\varepsilon]}N^{\alpha}_{p}(f,r)<\infty$ for some
$\varepsilon\in(0,\infty)$, and for any $r\in(\varepsilon,\infty)$
we see from $|f(x)-f(y)|^{p}\leq 2^{p}(|f(x)|^{p}+|f(y)|^{p})$ and
$\mu\bigl(B(x,r)\bigr)\leq c_{2}r^{d_{H}}$ that
\begin{align*}
N^{\alpha}_{p}(f,r)^{p}&=\frac{1}{r^{p\alpha+d_{H}}}\iint_{\{(x,y)\in X\times X\, :\, d(x,y)<r\}}|f(x)-f(y)|^{p}\,d\mu(x)\,d\mu(y)\\
&\leq\frac{1}{r^{p\alpha+d_{H}}}\iint_{\{(x,y)\in X\times X\, :\,  d(x,y)<r\}}2^{p}(|f(x)|^{p}+|f(y)|^{p})\,d\mu(x)\,d\mu(y)\\
&=\frac{2^{p+1}}{r^{p\alpha+d_{H}}}\int_X|f(y)|^{p}\mu\bigl(B(y,r)\bigr)\,d\mu(y)\\
&\leq\frac{2^{p+1}}{r^{p\alpha+d_{H}}}\int_X|f(y)|^{p}\cdot c_{2}r^{d_{H}}\,d\mu(y)
 =\frac{2^{p+1}c_{2}}{r^{p\alpha}}\|f\|_{L^{p}(X,\mu)}^{p}\leq\frac{2^{p+1}c_{2}}{\varepsilon^{p\alpha}}\|f\|_{L^{p}(X,\mu)}^{p},
\end{align*}
so that
\[
\sup_{r\in(0,\infty)}N^{\alpha}_{p}(f,r)
 \leq\max\bigl\{2(2c_{2})^{1/p}\varepsilon^{-\alpha}\|f\|_{L^{p}(\mu)},\sup_{r\in(0,\varepsilon]}N^{\alpha}_{p}(f,r)\bigr\}<\infty.
 \]
Hence $f\in\mathfrak{B}^{\alpha}_{p}(X)$, proving \eqref{eq:Besov-limsup}. 
The above argument also shows that with the notation of the previous section $B^\alpha_{p,\infty}(X)=\mathfrak{B}^\alpha_p(X)$.
The purpose of this section is to prove the following theorem. 

\begin{theorem}\cite[Theorem~3.2]{P-P10}\label{Besov characterization}
Let $p \ge 1$ and $\alpha \ge 0$. We have $\mathfrak{B}^{\alpha}_{p}(X) = \mathbf{B}^{p,\frac{\alpha}{d_W}}(X)$ and 
there exist constants $c_{p,\alpha},C_{p,\alpha}>0$ such that for every $f \in \mathfrak{B}^{\alpha}_{p}(X)$ and $r >0$,
\[
c_{p,\alpha} \sup_{s\in(0,r]}N^{\alpha}_{p}(f,s) \le \| f \|_{p,\alpha/d_W} 
\le C_{p,\alpha} \left( \sup_{s\in(0,r]}N^{\alpha}_{p}(f,s)+\frac{1}{r^{\alpha}} \|f\|_{L^{p}(X,\mu)} \right).
\]
In particular, $ \| f \|_{p,\alpha/d_W} \simeq \sup_{s\in(0,+\infty) }N^{\alpha}_{p}(f,s)$.
\end{theorem}

\begin{remark}
The above theorem is essentially a rephrasing of~\cite[Theorem~3.2]{P-P10}. However, 
the notion of Besov spaces given in~\cite{P-P10} considers dyadic 
jumps in the parameter $t$; hence the proof given there is slightly more complicated than ours. We also include the relatively short proof because it shall  repeatedly be used in the next sections.
\end{remark}
\begin{proof}
We first prove the lower bound. For $s,t>0$ and $\alpha>0$,
\begin{align*}
   \int_X \int_X & |f(x)-f(y) |^p p_t (x,y) d\mu(x) d\mu(y)  \\
   \ge & \int_X \int_{B(y,s)} |f(x)-f(y) |^p p_t (x,y) d\mu(x) d\mu(y) \\
  \ge & c_{5}t^{-d_{H}/d_{W}} \int_X \int_{B(y,s)} |f(x)-f(y) |^p \exp\biggl(-c_{6}\Bigl(\frac{d(x,y)^{d_{W}}}{t}\Bigr)^{\frac{1}{d_{W}-1}}\biggr) d\mu(x) d\mu(y) \\
 \ge &  c_{5}t^{-d_{H}/d_{W}} \exp\biggl(-c_{6}\Bigl(\frac{s^{d_{W}}}{t}\Bigr)^{\frac{1}{d_{W}-1}}\biggr)  s^{\alpha p+d_{H}} N^{\alpha}_{p}(f,s)^p.
\end{align*}
Therefore we have
\[
t^{-\frac{\alpha p}{d_W}} \int_X \int_X |f(x)-f(y) |^p p_t (x,y) d\mu(x) d\mu(y)
 \ge c_5 t^{-\frac{\alpha p}{d_W}-\tfrac{d_H}{d_W}} \exp\biggl(-c_{6}\Bigl(\frac{s^{d_{W}}}{t}\Bigr)^{\frac{1}{d_{W}-1}}\biggr)  s^{\alpha p+d_{H}} N^{\alpha}_{p}(f,s)^p.
\]


We now choose $s=t^{1/d_W}$. This yields
\[
t^{-\frac{\alpha p}{d_W}}  \int_X \int_X |f(x)-f(y) |^p p_t (x,y) d\mu(x) d\mu(y)
 \ge c_5  \exp{(-c_6)}   N^{\alpha}_{p}(f,t^{1/d_W})^p.
\]
From the definition of the $\| \cdot \|_{p,\alpha/d_W} $ seminorm, one concludes
\[
c_5  \exp{(-c_6)}   N^{\alpha}_{p}(f,t^{1/d_W})^p \le \| f \|^p_{p,\alpha/d_W} .
\]
Since, it is true for every $t>0$, the conclusion follows.
\

We now turn to the upper bound. Fixing $r>0$, we set
\begin{align}
A(t)&:=\int_X\int_{X\setminus B(y,r)}p_{t}(x,y)|f(x)-f(y)|^{p}\,d\mu(x)\,d\mu(y),\label{E:A(t)}\\
B(t)&:=\int_X\int_{B(y,r)}p_{t}(x,y)|f(x)-f(y)|^{p}\,d\mu(x)\,d\mu(y),\label{E:B(t)}
\end{align}
so that $ \int_X \int_X |f(x)-f(y) |^p p_t (x,y) d\mu(x) d\mu(y) =A(t)+B(t)$.
By \eqref{eq:subGauss-upper} and the inequality
$|f(x)-f(y)|^{p}\leq 2^{p-1}(|f(x)|^{p}+|f(y)|^{p})$,

\begin{align}
A(t)&\leq\frac{c_{3}}{t^{d_{H}/d_{W}}}\int_X\int_{X\setminus B(y,r)}\exp\biggl(-c_{4}\Bigl(\frac{d(x,y)^{d_{W}}}{t}\Bigr)^{\frac{1}{d_{W}-1}}\biggr)\cdot 2^{p}|f(y)|^{p}\,d\mu(x)\,d\mu(y)
\notag\\
&=
\frac{2^{p}c_{3}}{t^{d_{H}/d_{W}}}\sum_{k=1}^{\infty}\int_X\int_{B(y,2^{k}r)\setminus B(y,2^{k-1}r)}\exp\biggl(-c_{4}\Bigl(\frac{d(x,y)^{d_{W}}}{t}\Bigr)^{\frac{1}{d_{W}-1}}\biggr)|f(y)|^{p}\,d\mu(x)\,d\mu(y)
\notag\\
&\leq
\frac{2^{p}c_{3}}{t^{d_{H}/d_{W}}}\sum_{k=1}^{\infty}\int_X\mu\bigl(B(y,2^{k}r)\bigr)\exp\biggl(-c_{4}\Bigl(\frac{2^{(k-1)d_{W}}r^{d_{W}}}{t}\Bigr)^{\frac{1}{d_{W}-1}}\biggr)|f(y)|^{p}\,d\mu(y)
\notag\\
&\leq
\frac{2^{p}c_{3}}{t^{d_{H}/d_{W}}}\sum_{k=1}^{\infty}c_{2}r^{d_{H}}2^{kd_{H}}\|f\|_{L^{p}}^{p}\exp\biggl(-c_{4}\Bigl(\frac{r^{d_{W}}}{t}\Bigr)^{\frac{1}{d_{W}-1}}\Bigl(2^{\frac{d_{W}}{d_{W}-1}}\Bigr)^{k-1}\biggr)
\notag\\
&=
\|f\|_{L^{p}}^{p}2^{p}c_{2}c_{3}\sum_{k=1}^{\infty}2^{d_{H}}\Bigl(\frac{r^{d_{W}}}{t}2^{d_{W}(k-1)}\Bigr)^{d_{H}/d_{W}}\exp\biggl(-2^{-\frac{d_{W}}{d_{W}-1}}c_{4}\Bigl(\frac{r^{d_{W}}}{t}2^{d_{W}k}\Bigr)^{\frac{1}{d_{W}-1}}\biggr)
\notag\\
&\leq
\|f\|_{L^{p}}^{p}2^{p+d_{H}}c_{2}c_{3}\sum_{k=1}^{\infty}\int_{(r^{d_{W}}/t)(2^{d_{W}})^{k-1}}^{(r^{d_{W}}/t)(2^{d_{W}})^{k}}s^{d_{H}/d_{W}}\exp\Bigl(-2^{-\frac{d_{W}}{d_{W}-1}}c_{4}s^{\frac{1}{d_{W}-1}}\Bigr)\frac{1}{(d_{W}\log 2)s}\,ds
\notag\\
&=
\frac{2^{p+d_{H}}c_{2}c_{3}}{d_{W}\log 2}\|f\|_{L^{p}}^{p}\int_{r^{d_{W}}/t}^{\infty}s^{d_{H}/d_{W}-1}\exp\Bigl(-2^{-\frac{d_{W}}{d_{W}-1}}c_{4}s^{\frac{1}{d_{W}-1}}\Bigr)\,ds\notag\\
&\leq c_{8}\exp\biggl(-c_{9}\Bigl(\frac{r^{d_{W}}}{t}\Bigr)^{\frac{1}{d_{W}-1}}\biggr)\|f\|_{L^{p}(X,\mu)}^{p},
\label{eq:HKBesov-norms-upper-proof1}
\end{align}
where $c_{9}:=2^{-1-\frac{d_{W}}{d_{W}-1}}c_{4}$ and
$c_{8}:=2^{p+d_{H}}c_{2}c_{3}(d_{W}\log 2)^{-1}\int_{0}^{\infty}s^{d_{H}/d_{W}-1}\exp\bigl(-c_{9}s^{\frac{1}{d_{W}-1}}\bigr)\,ds$.

On the other hand, for $B(t)$, by \eqref{eq:subGauss-upper} we have
\begin{align}
B(t)&\leq\frac{c_{3}}{t^{d_{H}/d_{W}}}\int_X\int_{B(y,r)}\exp\biggl(-c_{4}\Bigl(\frac{d(x,y)^{d_{W}}}{t}\Bigr)^{\frac{1}{d_{W}-1}}\biggr)|f(x)-f(y)|^{p}\,d\mu(x)\,d\mu(y)\notag\\
&\leq\frac{c_{3}}{t^{d_{H}/d_{W}}}\sum_{k=1}^{\infty}\int_X\int_{B(y,2^{1-k}r)\setminus B(y,2^{-k}r)}\exp\biggl(-c_{4}\Bigl(\frac{d(x,y)^{d_{W}}}{t}\Bigr)^{\frac{1}{d_{W}-1}}\biggr)|f(x)-f(y)|^{p}\,d\mu(x)\,d\mu(y)\notag\\
&\leq\frac{c_{3}}{t^{d_{H}/d_{W}}}\sum_{k=1}^{\infty}\int_X\int_{B(y,2^{1-k}r)}\exp\biggl(-c_{4}\Bigl(\frac{2^{-kd_{W}}r^{d_{W}}}{t}\Bigr)^{\frac{1}{d_{W}-1}}\biggr)|f(x)-f(y)|^{p}\,d\mu(x)\,d\mu(y)\notag\\
&\leq c_{3} \sum_{k=1}^{\infty}\frac{(2^{1-k}r)^{p\alpha+d_{H}}}{t^{d_{H}/d_{W}}}\exp\biggl(-2c_{9}\Bigl(\frac{r^{d_{W}}}{t}2^{d_{W}(1-k)}\Bigr)^{\frac{1}{d_{W}-1}}\biggr)\frac{1}{(2^{1-k}r)^{p\alpha+d_{H}}}\int_X\int_{B(y,2^{1-k}r)}|f(x)-f(y)|^{p}\,d\mu(x)\,d\mu(y)\notag\\
&\leq c_{3}2^{p\alpha+d_{H}} t^{\frac{p\alpha}{d_W} }\sup_{s\in(0,r]}N^{\alpha}_{p}(f,s)^{p}\sum_{k=1}^{\infty}\Bigl(\frac{r^{d_{W}}}{t}2^{-d_{W}k}\Bigr)^{\frac{d_H+p\alpha}{d_W}}\exp\biggl(-2c_{9}\Bigl(\frac{r^{d_{W}}}{t}2^{-d_{W}(k-1)}\Bigr)^{\frac{1}{d_{W}-1}}\biggr)\notag\\
&\leq c_{7}t^{\frac{p\alpha}{d_W} }\sup_{s\in(0,r]}N^{\alpha}_{p}(f,s)^{p}.
\label{eq:HKBesov-norms-upper-proof2}
\end{align}

As a conclusion, one has 

\[
 \int_X \int_X |f(x)-f(y) |^p p_t (x,y) d\mu(x) d\mu(y) 
 \le c_{8}\exp\biggl(-c_{9}\Bigl(\frac{r^{d_{W}}}{t}\Bigr)^{\frac{1}{d_{W}-1}}\biggr)\|f\|_{L^{p}}^{p}
    +c_{7}t^{\frac{p\alpha}{d_W} }\sup_{s\in(0,r]}N^{\alpha}_{p}(f,s)^{p}.
\]
This yields
\[
\sup_{t>0}t^{-\frac{p\alpha}{d_W}}\int_X \int_X |f(x)-f(y) |^p p_t (x,y) d\mu(x) d\mu(y) 
\le c_{7} \sup_{s\in(0,r]}N^{\alpha}_{p}(f,s)^{p}+\frac{c_{10}}{r^{p\alpha}} \|f\|_{L^{p}(X,\mu)}^{p}.
\]
The proof is thus complete.
\end{proof}


\section{Co-area type estimates}
In the sequel, for any $E \subset X$ we will denote by $E_r$ the $r$-neighborhood of $E$ and define the distance from $x\in X$ to $E$ by $d(x,E):=\inf_{y\in E}d(x,y)$. We will also refer to the standing assumptions for this section, outlined at the beginning of the chapter.

\begin{theorem}\label{ahlfors-coarea}
Let $X$ be Ahlfors $d_H$-regular with a heat kernel that satisfies the upper sub-Gaussian estimate~\eqref{eq:subGauss-upper}. For $u\in L^1(X,\mu)$ and $s\in\mathbb{R}$, let $E_s(u):=\{ x \in X, u(x)>s\}$. Assume that there is $0<\alpha\leq \frac{d_H}{d_W}$ and $R>0$ such that 
\begin{equation}\label{extended-inner-boundary-estimate}
h(u,s) = \sup_{r \in(0,R]} \frac1{r^{\alpha d_W}} \mu \bigl\{ x\in E_s(u): d(x, X\setminus E_s(u)) < r  \bigr\}
\end{equation}
is in $L^1 (\mathbb{R},ds)$. Then, $u \in \mathbf{B}^{1,\alpha}(X)$ and there exist constants $C_1,C_2>0$ independent of $u$ such that
\begin{equation}\label{E:JKLM}
\| u \|_{1, \alpha} \le C_1 R^{-\alpha}\|u\|_{L^1(X,\mu)}+  C_2 \int_{\mathbb{R}} h(u,s) ds.
\end{equation}
\end{theorem}

\begin{proof}
As in the proof of Theorem \ref{thm:W=BV} we have
\begin{align*}
\lefteqn{\int_{X\times X} p_t(x,y) \ |u(x)-u(y)|\, d\mu(x) d\mu(y)}\quad&\\
&=2\int_{\mathbb{R}}\int_X \int_X \mathbf{1}_{E_s(u)^c} (x)  \mathbf{1}_{E_s(u)}(y)p_t(x,y)\, d\mu(x)d\mu(y)\, ds\\
&=2\int_\mathbb{R} \int_{E_s(u)} \int_{X\setminus E_s(u)} p_t(x,y)\, d\mu(y)\, d\mu(x)\, ds
\end{align*}
Fix $s$ and $t$ and decompose the integral over $E_s(u)$ as follows: For $j\geq 1$, let
\begin{equation*}
	F_j := \bigl\{ x\in E_s(u):  2^{j-1} t^{1/d_W} \leq d(x, X\setminus E_s(u)) < 2^j t^{1/d_W}  \bigr\}
	\end{equation*}
and $F_0:= \bigl\{ x\in E_s(u):  d(x, X\setminus E_s(u)) < t^{1/d_W}  \bigr\}$. Evidently $E_s(u)=\cup_0^\infty F_j$. In the following computation we use the bound $\int_{X\setminus E_s(u)} p_t(x,y)\, d\mu(y)\leq1$ for $x\in F_0$. For the remaining range of $j$ and $x\in F_j$ we instead note that $d(x,y)\geq 2^{j-1} t^{1/d_W}$ for any $y\in X\setminus E_s(u)$ and integrate the sub-Gaussian estimate~\ref{eq:subGauss-upper} over this range of radii.
\begin{align}
	\lefteqn{\int_{X\times X} p_t(x,y) \ |u(x)-u(y)|\, d\mu(x) d\mu(y)}\quad&\notag\\
	&\leq 2\int_{\mathbb{R}} \sum_{j=0}^\infty \int_{F_j} \int_{X\setminus E_s(u)} p_t(x,y)\, d\mu(y)\, d\mu(x)\, ds\notag\\
	&\leq 2 \int_{\mathbb{R}} \biggl(  \mu(F_0) + c_3t^{-d_H/d_W}  \sum_{j=1}^\infty \int_{F_j} \int_{X\setminus E_s(u)} \exp \biggl( -c_4\Bigl(\frac{d(x,y)^{d_W}}{t}\Bigr)^{1/(d_W-1)}\biggr)\, d\mu(y)\, d\mu(x)\biggr)\, ds\notag\\
	&\leq 2 \int_{\mathbb{R}} \biggl( \mu(F_0) + c_3t^{-d_H/d_W}  \sum_{j=1}^\infty \mu(F_j) \int_{ 2^{j-1} t^{1/d_W}}^\infty \exp\biggl( -c_4\Bigl(\frac{r^{d_W}}t \Bigr)^{1/(d_W-1)}\biggr)\, \mu(B(x,r))\,\frac{dr}r\biggr)\, ds\notag\\
	&\leq 2 \int_{\mathbb{R}} \biggl( \mu(F_0) + c_3  \sum_{j=1}^\infty \mu(F_j) \int_{2^{(j-1)d_W/(d_W-1)}}^\infty e^{-c_4u} u^{(d_W-1)d_H/d_W} \frac{du}u \biggr) \,ds \label{eqn:suffconditforuinBesov}
	\end{align}

Our hypotheses ensure that $\mu(F_j)\leq (2^{j} t^{1/d_W})^{\alpha d_W}h(u,s)$ provided $2^j t^{1/d_w}<R$; for notational convenience we let the largest $j$ satisfy this be $j=J$ and thus obtain
\begin{align*}
	\lefteqn{\sum_{j=1}^J \mu(F_j) \int_{2^{(j-1)d_W/(d_W-1)}}^\infty e^{-c_4u} u^{(d_W-1)d_H/d_W} \frac{du}u }\quad&\\
	&\leq t^{\alpha}  h(u,s) \sum_{j=1}^J 2^{j\alpha d_W} \int_{ 2^{(j-1)d_W/(d_W-1)}}^\infty e^{-c_4u} u^{(d_W-1)d_H/d_W} \frac{du}u \\
	&\leq  t^{\alpha} h(u,s)  \int_1^\infty  \Bigl( \sum_{j=1}^\infty 2^{j\alpha d_W}\mathds{1}_{\{u\geq 2^{(j-1)d_W/(d_W-1)}\}}\Bigr) e^{-c_4u} u^{(d_W-1)d_H/d_W} \frac{du}u \\
	&\leq \tilde{C}_1 t^{\alpha} h(u,s) \int_1^\infty e^{-c_4u} u^{(d_W-1)(d_H+\alpha d_W)/d_W} \frac{du}u \\
	&\leq \tilde{C}_2 t^{\alpha} h(u,s),
	\end{align*}
where we note that $\tilde{C}_2=\tilde{C}_2(c_4,\delta,d_W,d_H)$ is independent of $R$. For $j\geq J+1$ we instead use that $2^{(J-1)d_W/(d_W-1)}\geq \tilde{C}_3 R^{d_W/(d_W-1)}t^{-1/(d_W-1)}$. Moreover, there is $\tilde{C}_4=\tilde{C}_4(c_4,d_W,d_H)$ such that $R^{d_W/(d_W-1)}t^{-1/(d_W-1)}\geq \tilde{C}_4$ implies the bound
\begin{equation*}
	\int_{ 2^{(j-1)d_W/(d_W-1)}}^\infty e^{-c_4u} u^{(d_W-1)d_H/d_W} \frac{du}u
	\leq \tilde{C}_5  R^{d_H} t^{-d_H/d_W} \exp \Bigl( -c_4 \tilde{C}_4 R^{d_W/(d_W-1)}t^{-1/(d_W-1)} \Bigr).
	\end{equation*}
Using $\sum_{j\geq J} \mu(F_j)\leq \mu(E_s(u))$ we conclude
\begin{align*}
	\lefteqn{\sum_{j=J}^\infty \mu(F_j) \int_{2^{(j-1)d_W/(d_W-1)}}^\infty e^{-c_4u} u^{(d_W-1)d_H/d_W} \frac{du}u} \quad&\\
	&\leq  \tilde{C}_5 R^{d_H} t^{-d_H/d_W} \exp \Bigl( -c_4 \tilde{C}_4 R^{d_W/(d_W-1)}t^{-1/(d_W-1)} \Bigr)\mu(E_s(u)).
	\end{align*}
Combining these with the estimate $\mu(F_0)\leq t^{\alpha} h(u,s)$ we obtain from~\eqref{eqn:suffconditforuinBesov} that if $R^{\frac{d_W}{d_W-1}}t^{-\frac{1}{d_W-1}}\geq \tilde{C}_4$,
\begin{align*}
\lefteqn{\int_{X\times X} p_t(x,y) \ |u(x)-u(y)|\, d\mu(x) d\mu(y)}\quad&\\
	&\leq \int_\mathbb{R} \tilde{C}_2 t^{\alpha}  h(u,s) +  \tilde{C}_5 R^{d_H} t^{-d_H/d_W} \exp \Bigl( -c_4 C R^{d_W/(d_W-1)}t^{-1/(d_W-1)} \Bigr) \mu(E_s(u)) \, ds\\
	&\leq t^{\alpha}\int_\mathbb{R}  \tilde{C}_2 h(u,s)\, ds +  \tilde{C}_5 R^{d_H} t^{-d_H/d_W} \exp \Bigl( -c_4 C R^{d_W/(d_W-1)}t^{-1/(d_W-1)} \Bigr) \|u\|_{L^1(X,\mu)}
	\end{align*}
Finally, recall from the proof of Theorem~\ref{Besov characterization} that
\begin{equation*}
\|u\|_{1,\alpha}	\leq 2 T^{-\alpha} \|u\|_{L^1(X,\mu)} + 
\sup_{t\in(0,T]}  t^{-\alpha}\int_{X\times X} p_t(x,y) \ |u(x)-u(y)|\, d\mu(x) d\mu(y).
\end{equation*}
Our estimate for the integral is valid for $t \leq T= \tilde{C}_4^{-(d_W-1)} R^{d_W}$, hence there is $\tilde{C}_6$ so that
\begin{align*}
\|u\|_{1,\alpha} &\leq \tilde{C}_6 R^{-\alpha d_W} \|u\|_1+ \tilde{C}_2\int_\mathbb{R} h(u,s)\, ds \\
&\quad + \sup_{t\leq T}\tilde{C}_5 R^{d_H} t^{-(\alpha d_W+d_H)/d_W} \exp \Bigl( -c_4 \tilde{C}_4 R^{d_W/(d_W-1)}t^{-1/(d_W-1)} \Bigr) \|u\|_{L^1(X,\mu)}\\
&\leq C_1 R^{-\alpha d_W} \|u\|_{L^1(X,\mu)} +  C_2\int_\mathbb{R} h(u,s)\, ds \qedhere
\end{align*}
\end{proof}
The latter result has several useful applications. Recall that a space is uniformly locally connected if there is $C_{LC}>0$ such that any $x$ and $y$ cannot be disconnected in $B(x,C_{LC}d(x,y))$. Note, for example, that a geodesic space is uniformly locally connected.

\begin{corollary}
Let $X$ be a uniformly locally connected space with a heat kernel that satisfies the upper sub-Gaussian estimate~\eqref{eq:subGauss-upper} and let $u\in L^1(X,\mu)$. Let $E_s(u)$ be as in Theorem~\ref{ahlfors-coarea} and $\partial E_s(u)$ be its topological boundary. Suppose there is $0<\alpha\leq \frac{d_H}{d_W}$ and $R>0$ such that
\begin{equation*}
H(u,s) = \sup_{r\in(0,R]} \frac1{r^{\alpha d_W}} \mu\bigl\{ x\in E_s(u): d(x,\partial E_s(u)) <r\bigr\}
\end{equation*}
is in $L^1(\mathbb{R},ds)$. Then, $u \in \mathbf{B}^{1,\alpha}(X)$ and there exist constants $C_1,C_2>0$ independent of $u$ such that
\begin{equation}
 \| u \|_{1, \alpha} \le C_1 R^{-\alpha d_W}\|u\|_{L^1(X,\mu)}+  C_2 \int_{\mathbb{R}} H(u,s) ds.
\end{equation}
\end{corollary}
\begin{proof}
It suffices to show that $H(u,s)$ and $h(u,s)$ are comparable. This follows from the inclusions 
\begin{align*}
\bigl\{ x\in E_s(u): d(x,\partial E_s(u)) <r\bigr\} 
&\subset \bigl\{ x\in E_s(u): d(x, X\setminus E_s(u)) <2r\bigr\}\\
&\subset \bigl\{ x\in E_s(u): d(x,\partial E_s(u)) < 2C_{LC}r\bigr\}, 
\end{align*}
where $C_{LC}$ is the local connectivity constant. The first inclusion is trivial. Let us now assume that the second inclusion  fails, so that there are $x\in E_s(u)$ and $z\in X\setminus E_s(u)$ with $d(x,z)<2r$ and yet $d(x,y)>2C_{LC}r$ for all $y\in\partial E_s(u)$. Clearly neither $x$ nor $z$ is in $\partial E_s(u)$, so they are in the interior of $E_s(u)$ and the interior of $X\setminus E_s(u)$ respectively. These sets are disjoint and open, and they cover $B(x,2C_{LC}d(x,z))$ because $d(x,y)>2C_{LC}r$ for all $y\in\partial E_s(u)$ implies $B(x,2C_{LC}d(x,z))\cap\partial E_s(u)=\emptyset$. Thus, $x$ and $z$ are disconnected in $B(x,2C_{LC}d(x,y))$ in contradiction to uniform local connectedness.
\end{proof}

From the observation that $\partial E_s(\mathbf{1}_E)=\partial E$ if $s\in[0,1]$ and is empty otherwise we obtain a bound that related to the inner Minkowski content of $E$, see Section~\ref{Sec:Examples_1E}.
\begin{corollary}\label{C:ahlfors-coarea-1E-ulc}
Let $X$ be a uniformly locally connected space with a heat kernel that satisfies the upper sub-Gaussian estimate~\eqref{eq:subGauss-upper}. If $E \subset X$ has finite measure and for some $0<\alpha\leq \frac{d_H}{d_W}$
\begin{equation*}
\sup_{r\in(0,R]} \frac1{r^{\alpha d_W}} \mu\bigl\{ x\in E: d(x,\partial E) <r\bigr\}<\infty,
\end{equation*}
then $\mathbf{1}_E \in \mathbf{B}^{1,\alpha}(X) $ and
\begin{equation*}
\| \mathbf{1}_E \|_{1, \alpha} \le C_1 R^{-\alpha d_W} \mu(E) + C_2  \sup_{r\in(0,R]} \frac1{r^{\alpha d_W}} \mu\bigl\{ x\in E: d(x,\partial E) <r\bigr\}.
\end{equation*}
\end{corollary}
 
\begin{remark} 
One can drop the assumption on $X$ of being uniformly locally connected by considering \emph{an extended metric notion of the $r$-boundary of a set}:
\begin{equation}\label{e-b-ext}
\tilde\partial_r  E=(E\cap E^c_r)\cup (E^c\cap E_r) 
\end{equation} 
This is done in Theorems~\ref{ahlfors-coarea} and~\ref{T:non-local_coarea} when $h_R(u,s)$ is defined in~\eqref{extended-inner-boundary-estimate} and~\eqref{extended-inner-boundary-estimate-non-loc}. This is especially natural in the context of nonlocal Dirichlet forms in Chapter~\ref{Ch-non-loc}, where spaces considered can be totally disconnected.  
\end{remark}
 
The next corollary considers a situation that often applies to fractals.

\begin{corollary}\label{JKLM}
Let $X$ be a uniformly locally connected space with a heat kernel that satisfies the upper sub-Gaussian estimate~\eqref{eq:subGauss-upper} and let $u\in L^1(X,\mu)$. Let $E_s(u)$ be as in Theorem~\ref{ahlfors-coarea} and $\partial E_s(u)$ be its topological boundary. Assume that $\partial E_s(u)$ is finite and $s\mapsto  | \partial E_s(u) |$ is in $L^1(\mathbb{R})$, where $| \partial E_s(u) |$ denotes the cardinality of $\partial E_s(u)$. Then, $u \in \mathbf{B}^{1,\frac{d_H}{d_W}}(X)$ and there exists a constant $C>0$ independent of $u$ such that
\begin{equation}\label{E:JKLM}
 \| u \|_{1, d_H/d_W} \le C \int_{\mathbb{R}} | \partial E_s(u) | ds.
\end{equation}
In particular, if $E \subset X$ is a set of finite measure whose boundary is finite, then $\mathbf{1}_E \in \mathbf{B}^{1,\frac{d_H}{d_W}}(X)$ and 
\[
\| \mathbf{1}_E \|_{1, d_H/d_W} \le C  | \partial E |.
\]
\end{corollary}
\begin{proof}
Writing $\partial E_s(u)=\{z_1,\ldots ,z_{|\partial E_s(u)|}\}$, the Ahlfors regularity condition implies that, for any $r>0$,
\begin{equation*}
\mu((\partial E_s(u))_r)\leq \sum_{i=1}^{|\partial E_s(u)|} \mu(B(z_i,r))\leq |\partial E_s(u)|c_2 r^{d_H}.
\end{equation*}
\end{proof}

\begin{remark}
For the lower bound in this corollary one needs extra assumptions. One nice assumption is to assume that $u$ is a finite linear combination of indicator functions of open sets with finite boundary. This is natural for the Sierpinski gasket and similar spaces. For the Sierpinski gasket type fractals,  the geometry of the fractal will force $u$ to be totally discontinuous piecewise-constant function, and so the lower bound follows from Theorem~\ref{thm-HKE-E-Ec2}, Corollary~\ref{C:0P_characterization} or Proposition~\ref{prop-reg-b} under the mild extra assumption. This will be the subject of further work. 
\end{remark}

\begin{remark}
If Corollary~\ref{JKLM} is applied to the Sierpinski gasket or a similar fractal, then the function $u$ has to be discontinuous. This is because only countably many triangles have finite boundary, and any set which is not a finite union of triangles will have infinite boundary.  
However, on the Vicsek   set or a similar fractal, there are infinitely continuous functions for which level sets are finite. This is because the Vicsek   set is a topological tree. For instance, it is easy to see that the level sets of $n$-harmonic on the Vicsek  functions are finite. 
For some related general theory see \cite{KigamiDendrites}. 
\end{remark}

\begin{proposition}\label{P:BesovLB}
Let $X$ be Ahlfors $d_H$-regular with a heat kernel that satisfies the sub-Gaussian lower estimate~\eqref{eq:subGauss-upper}. For $u\in L^1(X,\mu)$ let us assume that the level sets $E_s(u)$ are uniformly porous at uniformly small scale, i.e.\ there exists $c_u>0$ and $R>0$ small such that for any $0<r<R$, $s\in\mathbb{R}$ and $y\in E_s(u)$ with $d(y, X\setminus E_s(u))<r$, there exists $z_y\in\partial E_s(u)$ such that
\begin{equation}\label{E:unif_porosity_levelset}
B(z_y,c_u r)\subset B(y,r)\cap (X\setminus E_s(u)).
\end{equation}
Furthermore, define
\begin{equation}
\tilde{h}(u,s):=\liminf_{r\to 0^+}\frac{1}{r^{\alpha d_W}}\mu(\{y\in E_s(u)\colon d(y,X\setminus E_s(u))<r\}).
\end{equation}
If $h\in L^1(\mathbb{R})$, then there exists $C_u>0$ depending on the porosity constant $c_u$ such that
\begin{equation}\label{E:BesovLB}
\|\mathbf{1}_E\|_{1,\alpha}\geq C_u\int_{\mathbb{R}}\tilde{h}(u,s)\,ds.
\end{equation}
\end{proposition}
\begin{proof}
Applying in the proof of Theorem~\ref{Besov characterization} the porosity condition~\eqref{E:unif_porosity_levelset}, for any $t>0$ and $0<r<R$ we have 
\begin{align*}
&t^{-\alpha}\int_{\mathbb{R}}\int_X\int_Xp_t(x,y)|\mathbf{1}_{E_s(u)}(x)-\mathbf{1}_{E_s(u)}(y)|\,\mu(dx)\,\mu(dy)\,ds\\
&=2t^{-\alpha}\int_{\mathbb{R}}\int_X\int_X p_t(x,y)\mathbf{1}_{E_s(u)^c}(x)\mathbf{1}_{E_s(u)}(y)\,\mu(dx)\,\mu(dy)\,ds\\
& \ge 2c_5t^{-\alpha-\frac{d_H}{d_W}}\exp\Big(\!\!-c_6\Big(\frac{r^{d_W}}{t}\Big)^{\frac{1}{d_W-1}}\Big)\int_{\mathbb{R}}\int_{\{y\in E_s(u)\colon d(y,X\setminus E_s(u))<r\} }\mu(B(y,r)\cap (X\setminus E_s(u)))\,\mu(dy)\,ds\\
&\geq 2c_5t^{-\alpha-\frac{d_H}{d_W}}\exp\Big(\!\!-c_6\Big(\frac{r^{d_W}}{t}\Big)^{\frac{1}{d_W-1}}\Big)\int_{\mathbb{R}}\int_{\{y\in E_s(u)\colon d(y,X\setminus E_s(u))<r\}}\mu(B(z_y,c_ur))\,\mu(dy)\,ds\\
&\geq 2c_5t^{-\alpha-\frac{d_H}{d_W}}\exp\Big(\!\!-c_6\Big(\frac{r^{d_W}}{t}\Big)^{\frac{1}{d_W-1}}\Big)c_2c_u^{d_H}r^{d_H+\alpha d_W}\int_{\mathbb{R}}\frac{1}{r^{\alpha d_W}}\mu(\{y\in E_s(u)\colon d(y,X\setminus E_s(u))<r\})\,ds.
\end{align*}
Thus, (with $t=r^{d_W}$) for any $0<r<\diam_d E$ it holds that
\begin{align*}
\|\mathbf{1}_E\|_{1,\alpha}&=\sup_{t>0}t^{-\alpha}\int_X\int_Xp_t(x,y)|\mathbf{1}_E(x)-\mathbf{1}_E(y)|\,\mu(dx)\,\mu(dy)\\
&\geq 2c_5 e^{-c_6}c_2 c_u^{d_H}\int_{\mathbb{R}}\frac{1}{r^{\alpha d_W}}\mu(\{y\in E_s(u)\colon d(y,X\setminus E_s(u))<r\})\,ds\\
&\geq C_u\liminf_{r\to 0^+}\int_{\mathbb{R}}\frac{1}{r^{\alpha d_W}}\mu(\{y\in E_s(u)\colon d(y,X\setminus E_s(u))<r\})\,ds\\
&\geq C_u\int_{\mathbb{R}}\liminf_{r\to 0^+}\frac{1}{r^{\alpha d_W}}\mu(\{y\in E_s(u)\colon d(y,X\setminus E_s(u))<r\})\,ds,
\end{align*}
where last inequality is from Fatou's lemma.
\end{proof}
The same argument as Corollary~\ref{C:ahlfors-coarea-1E-ulc} gives Proposition~\ref{P:BesovLB} when $X$ is uniformly locally constant with 
\begin{equation}
\tilde{h}(u,s)=\liminf_{r\to 0^+}\frac{1}{r^{\alpha d_W}}\mu(\{y\in E_s(u)~\colon~d(y,\partial E_s(u))<r\})
\end{equation}
that we can call the $(d_H-\alpha d_W)$-dimensional inner lower Minkowski content of $\partial E_s(u)$.
\

\section{Sets $E$ with $\mathbf{1}_E \in \mathbf{B}^{1,\alpha}(X)$ and fractional content of boundaries}\label{Hsefcb}
\renewcommand{\pat}[1]{\textcolor{magenta!50!red}{#1}}
\renewcommand{\rS}[1]{{{\color{OliveGreen}{#1}}}}
We continue working under the standing assumptions of this section which are outlined at the beginning of the chapter. As before and in the sequel, for $E \subset X$, we will denote its $r$-neighborhood by $E_r$. In addition, for any $r>0$ and $E\subset X$ we will denote $(\partial E)^-_r:=(\partial E)_r\cap E$ and $(\partial E)^+_r:=(\partial E)_r\cap E^c$, where $\partial E$ is the topological boundary of $E$.

\subsection{Several characterizations of the sets $E$ with $\mathbf{1}_E \in \mathbf{B}^{1,\alpha}(X)$}
With the arguments used in the proofs of the previous section in hand we can now provide different characterizations of sets of finite $\alpha$-perimeter.

\begin{theorem}\label{T:FP_char}
Under the assumptions of Theorem~\ref{Besov characterization}, for a bounded measurable set $E \subset X$  and $\alpha >0$ we consider the following properties:
\begin{enumerate}
\item\label{FP_char1} $\mathbf{1}_E \in \mathbf{B}^{1,\alpha}(X)$;
\item\label{FP_char2} $\displaystyle \sup_{r>0} \frac{1}{r^{\alpha d_W+d_{H}}} (\mu \otimes \mu) \big(\{ (x,y) \in E \times E^c\, :\, d(x,y) < r\}\big) <+\infty$;
\item\label{FP_char3} $\displaystyle \limsup_{r \to 0^+} \frac{1}{r^{\alpha d_W+d_{H}}} (\mu \otimes \mu) \big(\{ (x,y) \in E \times E^c\, :\, d(x,y) < r\}\big) <+\infty$;
\item\label{FP_char4} $\displaystyle \limsup_{r\to 0^+}\frac{1}{r^{\alpha d_W}}\mu\big(\{x\in E~\colon~d(x,E^c)<r\}\big)<+\infty$;
\item\label{FP_char5} $\displaystyle \limsup_{r\to 0^+}\int_{\{x\in E\colon d(x,E^c)<r\}}\int_{B(x,r)\cap E^c}\frac{1}{d(x,y)^{d_H+\alpha d_W}}d\mu(y)d\mu(x)<+\infty$;
\item\label{FP_char6} $\displaystyle \int_E \int_{E^c} \frac{1}{d(x,y)^{d_H+\alpha d_W}}d\mu(y)d\mu(x) <+\infty$.
\end{enumerate}
Then, the following relations hold:
\begin{enumerate}[label=(\roman*)]
\item \eqref{FP_char1}~$\Leftrightarrow$~\eqref{FP_char2}~$\Leftrightarrow$~\eqref{FP_char3}; 
moreover there exist constants $c,C >0$ independent of $E$ such that
\begin{multline*}
c\, \sup_{r>0} \frac{1}{r^{\alpha d_W+d_{H}}} (\mu \otimes \mu) \big(\{ (x,y) \in E \times E^c\, :\, d(x,y) < r\} \big) \\
 \le  \| \mathbf{1}_E \|_{1,\alpha} \le C \sup_{r>0} \frac{1}{r^{\alpha d_W+d_{H}}} (\mu \otimes \mu) \big(\{ (x,y) \in E \times E^c\, :\, d(x,y) < r\}\big);
\end{multline*}
\item \eqref{FP_char4}~$\Rightarrow$~\eqref{FP_char5}~$\Rightarrow$~\eqref{FP_char1};
\item \eqref{FP_char6}~$\Rightarrow$~\eqref{FP_char1};
\item If in addition $E$ is porous, i.e.\ there exists $c_E>0$ such that for any $0<r<\diam_d E$ and  $y\in E$ with $d(y, X\setminus E)<r$, there exists $z_y\in\partial E$ such that
\begin{equation}\label{D:porous}
B(z_y,c_Er)\subset B(y,r)\cap (X\setminus E),
\end{equation}
then \eqref{FP_char1}~$\Rightarrow$~\eqref{FP_char4}; 
\end{enumerate}
\end{theorem}
The relations stated in this theorem follow from the proof of Theorem~\ref{Besov characterization}, Theorem~\ref{ahlfors-coarea} and Proposition~\ref{P:BesovLB} applied to $\mathbf{1}_E$ along with some modifications of those arguments which are presented in the next paragraph. Assuming uniformly locally connectedness on the space, which includes the case of $X$ being geodesic, we obtain as in Corollary~\ref{C:ahlfors-coarea-1E-ulc} a characterization of sets with finite $\alpha$-perimeter that involves their inner Minkowski content.
\begin{corollary}
Let $X$ be uniformly locally connected and satisfy the assumptions of Theorem~\ref{Besov characterization}. Then, Theorem~\ref{T:FP_char} holds with the inner neighborhood $(\partial E)_r^-$ in place of $\{x\in E~\colon~d(x,E^c)<r\}$. In particular, if $E$ is closed, the quantity
\begin{equation}
\limsup_{r\to 0^+}\frac{1}{r^{\alpha d_W}}\mu\big(\{x\in E~\colon~d(x,E^c)<r\}\big)
\end{equation}
becomes the inner Minkowski content of $E$.
\end{corollary}

In the case when a set $E\subset X$ is porous and has a finite boundary, as for instance triangular cells in the infinite Sierpinski gasket, see Figure~\ref{F:porositySG}, we obtain a simpler characterization.
\begin{corollary}\label{C:0P_characterization}
If $(X,d)$ is uniformly locally connected, for any bounded measurable porous set $E\subset X$ whose topological boundary $\partial E$ is finite there exist $c,c_E>0$ such that
\begin{equation}
 c_E|\partial E|\leq \|\mathbf{1}_E\|_{1,d_H/d_W}\leq c |\partial E|,
\end{equation} 
where $c_E$ is the porosity constant of $E$.
\end{corollary}
\begin{proof}
The upper bound is Corollary~\ref{JKLM}. Assuming $\partial E=\{z_1,\ldots,z_{|\partial E|}\}$ and using the fact that there exists $0<r_E<\diam_d E$ such that for any $0<r<r_E$, $B(z_i,r)$ are pairwise disjoint with
\begin{equation*}\label{E:finite_bdry}
(\partial E)_r\supseteq \bigcup_{i=1}^{|\partial E|}B(z_i,r),
\end{equation*}
the lower bound follows analogously to Proposition~\ref{P:1E_BesovLB}.
\end{proof}
\begin{figure}[H]
\begin{tikzpicture}[scale=1/9]
\foreach \c in {0,1,2} {
\foreach \b in {0,1,2} {
\foreach \a in {0,1,2} {
\foreach \e in {0,1,2} {
\foreach \f in {0,1,2} {
\draw[] ($(90+120*\f:16)+(90+120*\e:8) + (90+120*\c:4)+(90+120*\b:2)+(90+120*\a:1)+(90:1)$) -- ($(90+120*\f:16)+(90+120*\e:8) + (90+120*\c:4)+(90+120*\b:2)+(90+120*\a:1)+(210:1)$) -- ($(90+120*\f:16)+(90+120*\e:8) + (90+120*\c:4)+(90+120*\b:2)+(90+120*\a:1)+(330: 1)$)--cycle; 
}
}
}
}
}
\foreach \c in {0,1,2} {
\foreach \b in {0,1,2} {
\foreach \a in {0,1,2} {
\foreach \e in {0,1,2} {
\foreach \f in {0,1,2} {
\draw[] ($(180:55)+(90+120*\f:16)+(90+120*\e:8) + (90+120*\c:4)+(90+120*\b:2)+(90+120*\a:1)+(90:1)$) -- ($(180:55)+(90+120*\f:16)+(90+120*\e:8) + (90+120*\c:4)+(90+120*\b:2)+(90+120*\a:1)+(210:1)$) -- ($(180:55)+(90+120*\f:16)+(90+120*\e:8) + (90+120*\c:4)+(90+120*\b:2)+(90+120*\a:1)+(330: 1)$)--cycle; 
}
}
}
}
}
\filldraw[fill=green!75!black, fill opacity=0.2, text opacity =1] ($(90+120*1:16)+(90+120*0:8) + (90+120*0:4)+(90+120*0:2)+(90+120*0:1)+(90:1)$) -- ($(90+120*1:16)+(90+120*0:8) + (90+120*0:4)+(90+120*1:2)+(90+120*1:1)+(210:1)$) -- ($(90+120*1:16)+(90+120*0:8) + (90+120*0:4)+(90+120*2:2)+(90+120*2:1)+(330: 1)$) -- ($(90+120*1:16)+(90+120*0:8) + (90+120*0:4)+(90+120*0:2)+(90+120*0:1)+(90:1)$)  node [midway, right] {\color{green!75!black}\small{$\; r/2$}};

\filldraw[fill=green!75!black, fill opacity=0.2, text opacity =1] ($(90+120*1:16)+(90+120*1:8) + (90+120*1:4)+(90+120*0:2)+(90+120*0:1)+(90:1)$) -- ($(90+120*1:16)+(90+120*1:8) + (90+120*1:4)+(90+120*1:2)+(90+120*1:1)+(210:1)$) -- ($(90+120*1:16)+(90+120*1:8) + (90+120*1:4)+(90+120*2:2)+(90+120*2:1)+(330: 1)$)  node [midway, below] {\color{green!75!black}\small{$\; r/2$}} -- ($(90+120*1:16)+(90+120*1:8) + (90+120*1:4)+(90+120*0:2)+(90+120*0:1)+(90:1)$);

\filldraw[fill=green!75!black, fill opacity=0.2, text opacity =1] ($(90+120*1:16)+(90+120*2:8) + (90+120*2:4)+(90+120*0:2)+(90+120*0:1)+(90:1)$) -- ($(90+120*1:16)+(90+120*2:8) + (90+120*2:4)+(90+120*1:2)+(90+120*1:1)+(210:1)$) -- ($(90+120*1:16)+(90+120*2:8) + (90+120*2:4)+(90+120*2:2)+(90+120*2:1)+(330: 1)$)  node [midway, below] {\color{green!75!black}\small{$\; r/2$}} -- ($(90+120*1:16)+(90+120*2:8) + (90+120*2:4)+(90+120*0:2)+(90+120*0:1)+(90:1)$);


\filldraw[fill=blue, fill opacity=0.2, text opacity =1] ($(90+120*0:16)+(90+120*1:8) + (90+120*1:4)+(90+120*1:2)+(90+120*0:1)+(90:1)$) -- ($(90+120*0:16)+(90+120*1:8) + (90+120*1:4)+(90+120*1:2)+(90+120*1:1)+(210:1)$) -- ($(90+120*0:16)+(90+120*1:8) + (90+120*1:4)+(90+120*1:2)+(90+120*2:1)+(330:1)$) -- ($(90+120*0:16)+(90+120*1:8) + (90+120*1:4)+(90+120*1:2)+(90+120*0:1)+(90:1)$);

\filldraw[fill=blue, fill opacity=0.2, text opacity =1] ($(180:3)+(90+120*1:16)+(90+120*1:8) + (90+120*1:4)+(90+120*1:2)+(90+120*0:1)+(90:1)$) -- ($(180:3)+(90+120*1:16)+(90+120*1:8) + (90+120*1:4)+(90+120*1:2)+(90+120*1:1)+(210:1)$) -- ($(180:3)+(90+120*1:16)+(90+120*1:8) + (90+120*1:4)+(90+120*1:2)+(90+120*2:1)+(330: 1)$) -- ($(180:3)+(90+120*1:16)+(90+120*1:8) + (90+120*1:4)+(90+120*1:2)+(90+120*0:1)+(90:1)$);

\filldraw[fill=blue, fill opacity=0.2, text opacity =1] ($(90+120*2:16)+(90+120*1:8) + (90+120*1:4)+(90+120*1:2)+(90+120*0:1)+(90:1)$) -- ($(90+120*2:16)+(90+120*1:8) + (90+120*1:4)+(90+120*1:2)+(90+120*1:1)+(210:1)$) -- ($(90+120*2:16)+(90+120*1:8) + (90+120*1:4)+(90+120*1:2)+(90+120*2:1)+(330: 1)$) -- ($(90+120*2:16)+(90+120*1:8) + (90+120*1:4)+(90+120*1:2)+(90+120*0:1)+(90:1)$);

\draw[draw=red, thick] ($(90+120*1:16)+(90+120*0:8) + (90+120*0:4)+(90+120*0:2)+(90+120*0:1)+(90:1)$) circle (10pt) -- ($(90+120*1:16)+(90+120*0:8) + (90+120*1:4)+(90+120*1:2)+(90+120*1:1)+(210:1)$)  node [midway, left] {\color{red}{$ r$}} -- ($(90+120*1:16)+(90+120*0:8) + (90+120*2:4)+(90+120*2:2)+(90+120*2:1)+(330: 1)$) -- ($(90+120*1:16)+(90+120*0:8) + (90+120*0:4)+(90+120*0:2)+(90+120*0:1)+(90:1)$)  circle (10pt);
\draw[draw=red, thick] ($(90+120*1:16)+(90+120*1:8) + (90+120*0:4)+(90+120*0:2)+(90+120*0:1)+(90:1)$) -- ($(90+120*1:16)+(90+120*1:8) + (90+120*1:4)+(90+120*1:2)+(90+120*1:1)+(210:1)$) circle (10pt) -- ($(90+120*1:16)+(90+120*1:8) + (90+120*2:4)+(90+120*2:2)+(90+120*2:1)+(330: 1)$) -- ($(90+120*1:16)+(90+120*1:8) + (90+120*0:4)+(90+120*0:2)+(90+120*0:1)+(90:1)$);
\draw[draw=red, thick] ($(90+120*1:16)+(90+120*2:8) + (90+120*0:4)+(90+120*0:2)+(90+120*0:1)+(90:1)$) -- ($(90+120*1:16)+(90+120*2:8) + (90+120*1:4)+(90+120*1:2)+(90+120*1:1)+(210:1)$) -- ($(90+120*1:16)+(90+120*2:8) + (90+120*2:4)+(90+120*2:2)+(90+120*2:1)+(330: 1)$) circle (10pt) -- ($(90+120*1:16)+(90+120*2:8) + (90+120*0:4)+(90+120*0:2)+(90+120*0:1)+(90:1)$);
\end{tikzpicture}
\caption{Marked red cells correspond to $(\partial E)_{r}^-$, the green colored ones to $(\partial E)_{r/2}^-$. 
In this case, any ball $B(y,r/2)$ with $y\in (\partial E)_{r/2}^-$ contains at least a triangular cell (colored blue) of side length $r/4$.}
\label{F:porositySG}
\end{figure}
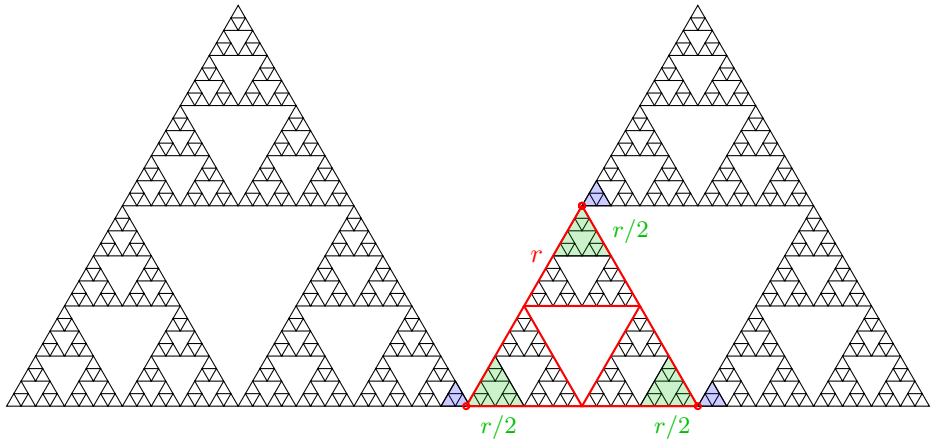
\subsection*{Detailed statements and proofs for Theorem~\ref{T:FP_char}}
For the sake of clarity, in the statements in this paragraph the assumptions on the underlying space $X$ are explicitly written.
\begin{proposition}\label{P:1E_BesovLB}
Let $X$ be Ahlfors $d_H$-regular with a heat kernel that satisfies the sub-Gaussian lower estimate~\eqref{eq:subGauss-upper} and let $0<\alpha\leq \frac{d_H}{d_W}$. For any porous set $E\subset X$ there exists $C_E>0$ such that
\begin{equation}\label{E:1E_BesovLB}
\|\mathbf{1}_E\|_{1,\alpha}\geq C_E\, \limsup_{r\to 0^+}\frac{1}{r^{\alpha d_W}}\mu(\{x\in E\colon d(x,E^c)<r\}).
\end{equation}
\end{proposition}
\begin{proof}
Applying Proposition~\ref{P:BesovLB} with $u=\mathbf{1}_E$, for any $t>0$ and $0<r<\diam_d E$ we have 
\begin{multline*}
\|\mathbf{1}_E\|_{1,\alpha}=\sup_{t>0}t^{-\alpha}\int_X\int_Xp_t(x,y)|\mathbf{1}_E(x)-\mathbf{1}_E(y)|\,\mu(dx)\,\mu(dy)\\
\geq 2c_5 e^{-c_6}c_2 c_E^{d_H}\frac{1}{r^{\alpha d_W}}\mu(\{y\in E\colon d(y,E^c)<r\}).
\end{multline*}
Taking $\limsup_{r\to 0^+}$ yields~\eqref{E:1E_BesovLB}.
\end{proof}

\begin{proposition}\label{P:1Ein_Besov}
Let $X$ be Ahlfors $d_H$-regular with a heat kernel that satisfies the sub-Gaussian upper estimate~\eqref{eq:subGauss-upper} and let $E\subset X$ be a bounded measurable for which
\begin{equation}\label{E:local_geom_cond}
\limsup_{r\to 0^+}\int_{\{y\in E\colon d(y,E^c)<r\}}\int_{B(y,r)\cap E^c}\frac{1}{d(x,y)^{d_H+\alpha d_W}}\mu(dx)\mu(dy)<\infty
\end{equation}
for some $\alpha >0$. Then, $\mathbf{1}_E\in \bm{B}^{1,\alpha}(X)$.
\end{proposition}
\begin{proof}
In order to bound $\|\mathbf{1}_E\|_{1,\alpha}$ we follow the proof of Theorem~\ref{Besov characterization} with some small modifications. 
Using the same notation, for any fixed $t,r>0$ we have, on the one hand
\begin{align*}
A(t,r)\leq c_8\exp\Big(\!\!-c_9\Big(\frac{r^{d_W}}{t}\Big)^{\frac{1}{d_W-1}}\Big)\|\mathbf{1}_E\|_{L^1(X,\mu)}=c_8\exp\Big(\!\!-c_9\Big(\frac{r^{d_W}}{t}\Big)^{\frac{1}{d_W-1}}\Big)\mu(E),
\end{align*}
with the corresponding constants in the mentioned proof. 
On the other hand, following the estimate for $B(t)$ we get
\begin{align*}
&B(t,r)=2\int_X\int_{B(y,r)}\mathbf{1}_{E^c}(x)\mathbf{1}_{E}(y)\,p_t(x,y)\,\mu(dx)\,\mu(dy)\nonumber\\
&\leq \frac{2c_3}{t^{d_H/d_W}}\sum_{k=1}^\infty\int_X\int_{B(y,2^{1-k}r)}\!\!\!\!\exp\bigg(\!\!\!-c_4\Big(\frac{(2^{-k}r)^{d_W}}{t}\Big)^{\frac{1}{d_W-1}}\bigg)\mathbf{1}_{E^c}(x)\mathbf{1}_{E}(y)\,\mu(dx)\,\mu(dy)\nonumber\\
&\leq 2c_3\sum_{k=1}^\infty \frac{(2^{1-k}r)^{\alpha d_W+d_H}}{t^{d_H/d_W}}\exp\bigg(\!\!-c_42^{-\frac{d_W}{d_W-1}}\Big(\frac{(2^{1-k}r)^{d_W}}{t}\Big)^{\frac{1}{d_W-1}}\bigg)\nonumber\\
&\quad\times \frac{1}{(2^{1-k}r)^{\alpha d_W+d_H}}\int_X\int_{B(y,2^{1-k}r)}\mathbf{1}_{E^c}(x)\mathbf{1}_{E}(y)\,\mu(dx)\,\mu(dy)\nonumber\\
&\leq 2c_3t^{\alpha}\int_E\int_{B(y,r)\cap E^c}\frac{1}{d(x,y)^{\alpha d_W+d_H}}\mu(dx)\mu(dy) \sum_{k=1}^\infty\Big(\frac{2^{1-k}r}{t^{1/d_W}}\Big)^{\alpha d_W+d_H}\!\!\!\exp\bigg(\!\!-2c_9\Big(\frac{(2^{1-k}r)^{d_W}}{t}\Big)^{\frac{1}{d_W-1}}\bigg)\nonumber\\
&\leq c_{10} t^{\alpha}\int_E\int_{B(y,r)\cap E^c}\frac{1}{d(x,y)^{\alpha d_W+d_H}}\mu(dx)
\mu(dy)\\
&= c_{10} t^{\alpha}\int_{\{y\in E\colon d(y,E^c)<r\}}\int_{B(y,r)\cap E^c}\frac{1}{d(x,y)^{\alpha d_W+d_H}}\mu(dx)
\mu(dy),
\end{align*}
where $c_{10}:=c_32^{\alpha d_W+d_H+1}(2d_W\log 2)^{-1}\int_0^\infty s^{\alpha+\frac{d_H}{d_W}}\exp\big(\!\!-2c_9 s^{\frac{1}{d_W-1}}\big)\,ds$. For any $t>0$ it thus holds that
\begin{multline*}\label{E:1Ein_Besov_help03}
t^{-\alpha}\int_X\int_X p_t(x,y)|\mathbf{1}_E(x)-\mathbf{1}_E(y)|\,\mu(dx)\,\mu(dy)\\
\leq t^{-\alpha}c_8\exp\Big(\!\!-c_9\Big(\frac{r^{d_W}}{t}\Big)^{\frac{1}{d_W-1}}\Big)\mu(E)+ c_{10} \int_{\{y\in E\colon d(y,E^c)<r\}}\int_{B(y,r)\cap E^c}\frac{1}{d(x,y)^{\alpha d_W+d_H}}\mu(dx)\mu(dy).
\end{multline*}
By assumption, we can now choose $r_\varepsilon>0$ small enough so that 
\begin{equation*}
\sup_{t>0}t^{-\alpha}\int_X\int_X p_t(x,y)|\mathbf{1}_E(x)-\mathbf{1}_E(y)|\,\mu(dx)\,\mu(dy)\leq \frac{c_{11}}{r_\varepsilon^{\alpha d_W}}\mu(E)+c_{10}c_\varepsilon<\infty
\end{equation*}
for some $c_\varepsilon>0$.
\end{proof}
\begin{proposition}\label{P:local_geom_cond}
Let $X$ be Ahlfors $d_H$-regular with a heat kernel that satisfies the sub-Gaussian upper estimate~\eqref{eq:subGauss-upper} and let $E\subset X$ be a bounded measurable for which
\begin{equation}\label{E:inner_coarea}
\limsup_{r\to 0^+}\frac{1}{r^{\alpha d_W}}\mu(\{y\in E\colon d(y,E^c)<r\})<\infty
\end{equation}
for some $0\leq\alpha\leq \frac{d_H}{d_W}$. Then,~\eqref{E:local_geom_cond} holds for that set $E$.
\end{proposition}

\begin{proof}
Let $E\subset X$ be a bounded measurable set. For any $r>0$, write
\begin{multline*}
\int_{\{y\in E\colon d(y,E^c)<r\}}\int_{B(y,r)\cap E^c}\frac{1}{d(x,y)^{d_H +\alpha d_W}}\mu(dx)\mu(dy)\\
=\int_0^\infty(\mu\otimes\mu)\big(\big\{(x,y)\in E\times E^c~\colon~d(x,y)<r,d(x,y)<s^{-(d_H+\alpha d_W)^{-1}}\big\}\big)\,ds.
\end{multline*}
Since $r<s^{-(d_H+\alpha d_W)^{-1}}$ if and only if $s<r^{-(d_H+\alpha d_W)}$, the latter integral equals
\begin{align*}
&\int_0^{r^{-(d_H+\alpha d_W)}}\!\!(\mu\otimes\mu)\big(\big\{(x,y)\in E\times E^c~\colon~d(x,y)<r\big\}\big)\,ds\\
&+\int_{r^{-(d_H+\alpha d_W)}}^\infty(\mu\otimes\mu)\big(\big\{(x,y)\in E\times E^c~\colon~d(x,y)<s^{-(d_H+\alpha d_W)^{-1}}\big\}\big)\,ds\\
&=\int_0^{r^{-(d_H+\alpha d_W)}}\int_{\{y\in E\colon d(y,E^c)<r\}}\mu(B(x,r)\cap E^c)\,\mu(dx)\,ds\\
&+\int_{r^{-(d_H+\alpha d_W)}}^\infty\int_{\{y\in E\colon d(y,E^c)<s^{-(d_H+\alpha d_W)^{-1}}\}}\mu(B(y,s^{-(d_H+\alpha d_W)^{-1}})\cap E^c)\,\mu(dy)\,ds.
\end{align*}
The Ahlfors regularity now yields
\begin{align*}
\int_{\{y\in E\colon d(y,E^c)<r\}}&\int_{B(y,r)\cap E^c}\frac{1}{d(x,y)^{d_H+\alpha d_W}}\mu(dx)\mu(dy) \\
&\leq \frac{c_2}{r^{\alpha d_W}}\mu(\{y\in E\colon d(y,E^c)<r\})\\
&+c_2\int_{r^{-(d_H+\alpha d_W)}}^{\infty} \mu\big(\{y\in E\colon d(y,E^c)<s^{-(d_H+\alpha d_W)^{-1}}\}\big)s^{-\frac{d_H}{d_H+\alpha d_W}}\, ds.
\end{align*}
Letting $r\to 0^+$, $r^{-(d_H+\alpha d_W)}\to\infty$ 
and~\eqref{E:local_geom_cond} follows from~\eqref{E:inner_coarea}.
\end{proof}
\subsection{Further results}

For the following result the examples include the finite and infinite Sierpinski gasket (Figure~\ref{fig-sig} and 
\cite{BP}) and, more generally, finite and infinite nested fractals, 
\cite{Lind,Ba98,Ba03,BaASC}. Other examples can be constructed in the class of p.c.f. self-similar sets, 
\cite{KigB,Kig:RFQS}, fractafolds 
\cite{Str,ST} and, more generally, finitely ramified cell structures 
\cite{T08}. To obtain similar results for fractals that are not finitely ramified, 
one can use the methods of 
\cite{BB89,BB99,BBKT,KZ,Kajino}.

\begin{theorem}\label{thm-HKE-E-Ec1}
Let $E\subseteq X$ be a measurable set and let $p_t(x,y)$ satisfy the upper heat kernel estimate in~\eqref{eq:subGauss-upper}. If there exists $0<\delta\leqslant d_H$ such that for all $\varepsilon>0$
\begin{equation}\label{E:nbd_estimateU}
(\mu\times\mu)\big(\{(x,y)\in E^c\times E\, :\, d(x,y)<\varepsilon\}\big)
 \leqslant c_7\left(\varepsilon^{d_H+\delta}+\varepsilon^{2d_H }\right),
\end{equation}
then for all $t>0$ it holds that 
\begin{equation}\label{eq:PtE-upper}
\left(t^{-\delta/d_{W}}+t^{-{{d_H}}/d_{W}}\right)\| P_t \mathbf{1}_E -\mathbf{1}_E\|_{L^1(X,\mu)} \leqslant c_{11}<\infty,
\end{equation}
where $c_{11}$ is a constant given by~\eqref{eq:PtE-upperc11} that depends on $\delta,d_H$ and $d_W$.
\end{theorem}

\begin{theorem}\label{thm-HKE-E-Ec2}
Let $E\subseteq X$ be a measurable set and let $p_t(x,y)$ satisfy the lower heat kernel estimate in~\eqref{eq:subGauss-upper}. If there exist $\delta>0$ and $\varepsilon>0$ such that
\begin{equation}\label{E:nbd_estimateL}
(\mu\times\mu)\big(\{(x,y)\in E^c\times E\, :\, d(x,y)<\varepsilon\}\big)\geq c_8\varepsilon^{d_H+\delta},
\end{equation}
then, for $t=\varepsilon^{d_W}$, it holds that 
\begin{equation}\label{eq:PtE-lower}
t^{-\delta/d_{W}}\| P_t \mathbf{1}_E -\mathbf{1}_E \|_{L^1(X,\mu)} \geqslant 2 c_{5}c_8e^{-c_{6}} >0.
\end{equation}
\end{theorem}
\begin{corollary}
If there exists $\delta>0$ and $\varepsilon_0>0$ such that~\eqref{E:nbd_estimateL} holds for any $0<\varepsilon<\varepsilon_0$, then~\eqref{eq:PtE-lower} holds for any $0<t<\varepsilon^{d_W}$.
\end{corollary}
Localized version of Theorem~\ref{thm-HKE-E-Ec1}.
\begin{theorem}\label{thm-HKE-E-Ec1-ball}
Let $p_t(x,y)$ satisfy the upper heat kernel estimate~\eqref{eq:HKE-non-loc} and let $E\subset B\subset X$ be bounded measurable set. If there exists $0<\delta\leqslant d_H$ such that for all $\varepsilon>0$\begin{equation}\label{E:nbd_estimateU-ball}
(\mu\times\mu)\big(\{(x,y)\in (B\cap E^c)\times E\, :\, d(x,y)<\varepsilon\}\big)\leqslant c_7\varepsilon^{d_H+\delta},
\end{equation}
then for all $t>0$ it holds that 
\begin{equation}\label{eq:PtE-upper-ball}
t^{-\delta/d_{W}}\| P_t \mathbf{1}_E -\mathbf{1}_E\|_{L^1(B,\mu)} \leqslant c_{11}<\infty,
\end{equation}
where $c_{11}$ is a constant given by~\eqref{eq:PtE-upperc11-ball} that depends on $\delta,d_H$ and $d_W$.
\end{theorem}
\begin{remark}\mbox{ }
\begin{enumerate}
\item In Theorem~\ref{thm-HKE-E-Ec1}, ${{d_H}} $ controls the behavior of the left hand side in \eqref{E:nbd_estimateU} when $\varepsilon$ is large. For small $t$ the leading term is $t^{-\delta/d_{W}}$. 
\item The constant $\delta$ plays the role of the (upper) co-dimension of the boundary and the constant $c_7$ plays the role of the (possibly fractal) upper Minkowski content corresponding to that dimension (see Proposition~\ref{prop-reg-bu}, \cite{Fal03,Mattila}). 	There is an extensive literature on the subject for fractal subsets $E$ of a Euclidean space, such as \cite[and references therein]{ambrosio2017perimeter,WZ13,W15,PW14,RZ12,FK12,LPW11}. 	We note that, although~\cite{BobkovHoudre} does not explicitly refer to the Minkowski content, it appears in the proof of the main theorem as $\mu^+(A)$. 
\item In Theorem~\ref{thm-HKE-E-Ec2}	the constant $\delta$ also plays the role of a (lower) co-dimension of the boundary and the constant $c_8$ plays the role of an weaker version of the (possibly fractal) 
lower Minkowski content corresponding to that dimension. A useful sufficient condition for \eqref{E:nbd_estimateL} is given in Proposition~\ref{prop-reg-b}.
\item The values of the constant $\delta$ can be different in Theorem~\ref{thm-HKE-E-Ec1} and in Theorem~\ref{thm-HKE-E-Ec2}, so one can speak about an upper-co-dimension $\overline{\delta}$ and about a lower-co-dimension $\underline{\delta}$. In general we have $\overline{\delta}\geqslant\underline{\delta}$, but in most examples we have $\overline{\delta}=\underline{\delta}$.
\end{enumerate}
\end{remark}

\begin{proof}[Proof of Theorem~\ref{thm-HKE-E-Ec1}] 
Using the fact thate the semigroup $P_t$ is conservative, we can rewrite
\begin{equation*}\label{eq:PtE-i}
\| P_t \mathbf{1}_E -\mathbf{1}_E \|_{L^1(X,\mu)}= \| \mathbf{1}_{E^c}P_t \mathbf{1}_E \|_{L^1(X,\mu)}+\| \mathbf{1}_E P_t\mathbf{1}_{E^c} \|_{L^1(X,\mu)} =2\| \mathbf{1}_{E^c}P_t \mathbf{1}_E \|_{L^1(X,\mu)}.
\end{equation*}
For any fixed $t>0$, since $p_t(x,y)\geq 0$, we have that
\begin{equation*}\label{E:moti}
\|\mathbf{1}_{E^c} P_t\mathbf{1}_{E}\|_{L^1(X,\mu)}=\int_{E^c}\int_{E}p_t(x,y)\,\mu(dy)\,\mu(dx)=\int_0^\infty\!\!\mu\otimes\mu(\{(x,y)\in E^c\times E\, :\, p_t(x,y)>s\})\,ds.
\end{equation*}
Moreover, in view of~\eqref{eq:subGauss-upper}, $p_t(x,y)>s$ implies 
\begin{equation*}
d(x,y)\leq t^{\frac{1}{d_W}}\Big(-c_4^{-1}\log(c_3^{-1}st^{\frac{d_H}{d_W}})\Big)^{\frac{d_W-1}{d_W}}=:F(t,s)
\end{equation*}
and $F(t,s)>0$ if and only if $s<c_5t^{-\frac{d_H}{d_W}}$. Therefore,	we obtain
\begin{equation}\label{e-u-dd'}
\|\mathbf{1}_{E^c} P_t\mathbf{1}_{E}\|_{L^1(X,\mu)} \leq c_7\int_0^{c_3t^{-d_H/d_W}} (F(t,s))^{d_H+\delta}ds+c_7\int_0^{c_3t^{-d_H/d_W}} (F(t,s))^{2d_H }ds. 
	\end{equation}
The first integral can be estimated by 
\begin{multline*}
c_7\int_0^{c_3t^{-d_H/d_W}} t^{\frac{d_H+\delta}{d_W}}\Big(-c_4^{-1}\log(c_5^{-1}st^{\frac{d_H}{d_W}})\Big)^{\frac{(d_W-1)(d_H+\delta)}{d_W}}ds\\
\leq c_7c_3t^{\frac{\delta}{d_W}}\int_0^1\big(-c_4^{-1}\log u\big)^{\frac{(d_W-1)(d_H+\delta)}{d_W}}du\\
=\ c_7c_3 c_4^{-(d_W-1)(d_H+\delta)/d_W}\Gamma\Big(\frac{(d_W-1)(d_H+\delta)}{d_W}\Big)t^{\frac{\delta}{d_W}},
\end{multline*}
where we have used the substitution $u=c_3^{-1}st^{\frac{d_H}{d_W}}$. A similar estimate is obtained for the second integral in~\eqref{e-u-dd'} by substituting $\delta$ by ${{d_H}}$.	Thus,~\eqref{eq:PtE-upper} holds with 
\begin{equation}\label{eq:PtE-upperc11}
c_{11}=2c_3c_7\left\{ \Gamma\left(\tfrac{(d_W-1)(d_H+\delta)}{d_W}\right) c_4^{-\tfrac{(d_W-1)(d_H+\delta)}{d_W}}+
\Gamma\left(\tfrac{(d_W-1)(d_H+{{d_H}})}{d_W}\right) c_4^{-\frac{2(d_W-1)d_H }{d_W}} \right\}.
\end{equation}
\end{proof}

\begin{proof}[Proof of Theorem~\ref{thm-HKE-E-Ec2}] 
For each fixed $t>0$, define 
\begin{equation}
A_{E,t^{1/d_W}}:=\{(x,y)\in E^c\times E~|~d(x,y)\leq t^{1/d_W}\}.
\end{equation}
Due to the lower estimate~\eqref{eq:subGauss-upper}, $p_t(x,y)\geq c_5 \exp\left(-c_{6}\right)t^{-d_H/d_W}$	for any $(x,y)\in A_{E,t^{1/d_W}}$. Assumption~\eqref{E:nbd_estimateL} now yields
\begin{multline*}
\|\mathbf{1}_{E^c} P_t\mathbf{1}_{E}\|_{L^1(X,\mu)} =\int_{E}\int_{E^c}p_t(x,y)\,\mu(dx)\,\mu(dy)\ge \int_{A_{E,t^{1/d_W}}}p_t(x,y)\,\mu(dx)\,\mu(dy)\\
\geq c_5 e^{-c_{6}} t^{-\frac{d_H}{d_W}}\mu\times\mu (A_{E,t^{1/d_W}})\geq c_5e^{-c_{6}} t^{-\frac{d_H}{d_W}}c_8t^{\frac{(d_H+\delta)}{d_W}}=c_5c_8e^{-c_{6}}t^{\frac{\delta}{d_W}}
\end{multline*}
as we wanted to prove.
\end{proof}

\begin{proof}[Proof of Theorem \ref{thm-HKE-E-Ec1-ball}]
As in Theorem~\ref{thm-HKE-E-Ec1}, it suffices to show $t^{-\delta/d_{W}}\|  \mathbf{1}_{E^c} P_t \mathbf{1}_E \|_{L^1(B,\mu)} \leqslant c_{11}$. Moreover, we have that $p_t(x,y)>s$ implies $d(x,y)<F(t,s)$ as in~\eqref{E:def_F(t,s)} and $F(s,t)>0$ if and only if $s<c_3t^{-d_H/d_W}$. Thus,
\begin{align*}
\|  \mathbf{1}_{E^c} P_t \mathbf{1}_E \|_{L^1(B,\mu)}&=\int_{B\cap E^c}\int_Ep_t(x,y)\,\mu(dy)\,\mu(dx)\\
&=\int_0^\infty\mu\times\mu\big(\{(x,y)\in(B\cap E^c)\times E~|~p_t(x,y)>s\}\big)\,ds\\
&\leq \int_0^\infty\mu\times\mu\big(\{(x,y)\in(B\cap E^c)\times E~|~d(x,y)<F(t,s)\}\big)\,ds\\
&= \int_0^\infty\mu\times\mu\big(\{(x,y)\in E^c\times E~|~d(x,y)\leq\int_0^{c_3t^{-d_H/d_W}}\!\!\min\{F(t,s),d(E,B)\}\}\big)\,ds,
\end{align*}
where $d(E,B):=\sup_{x\in E}\inf_{y\in B}d(x,y)$. Applying~\eqref{E:nbd_estimateU-ball} and estimating the integral as in the proof of Theorem~\ref{thm-HKE-E-Ec1} yields
\begin{align*}
\| \mathbf{1}_{E^c} P_t \mathbf{1}_E \|_{L^1(B,\mu)}
& \leq c_7c_3 t^{\frac{\delta}{d_W}}c_4^{-\frac{(d_W-1)(d_H+\delta)}{d_W}}\Gamma\Big(\frac{(d_W-1)(d_H+\delta)}{d_W}\Big)t^{\frac{\delta}{d_W}}
\end{align*}
and~\eqref{eq:PtE-upper-ball} holds with 
\begin{equation}\label{eq:PtE-upperc11-ball}
c_{11}=c_7c_3 c_4^{-\frac{(d_W-1)(d_H+\delta)}{d_W}}\Gamma\Big(\frac{(d_W-1)(d_H+\delta)}{d_W}\Big)t^{\frac{\delta}{d_W}}.
\end{equation}
\end{proof}

We give next a sufficient condition for~\eqref{E:nbd_estimateU} when $X$ is an Ahlfors $d_H$-regular geodesic space.
\begin{proposition}\label{prop-reg-bu}
Let $X$ be Ahlfors $d_H$-regular and geodesic. If $E\subseteq X$ has a compact boundary $\partial E$ and	for any $\varepsilon>0$ the measure of the $\varepsilon$-neighborhood $(\partial E)_\varepsilon$ of $\partial E$ satisfies
\begin{equation}\label{E:nbhd_estimateU-}
\mu((\partial E)_\varepsilon)=O\left(\varepsilon^\delta\right)_{\varepsilon\to0^+}
\end{equation} 
then~\eqref{E:nbd_estimateU} holds.
\end{proposition} 

\begin{proof}
Since the space $X$ is geodesic, for any $x\in E$ and $y\in E^c$ there exists $z=z(x,y)\in\partial E$ such that $d(x,y)=d(x,z)+d(z,y)$. Therefore $x,y\in\left(\partial E\right)_\varepsilon$ if	$x\in E$, $y\in E^c$ and 	$d(x,y)<\varepsilon$, so that
\begin{align*}
(\mu\otimes\mu)\big(\{(x,y)\in E^c\!\times\! E\, :\, d(x,y)<\varepsilon\}\big)	
&= \int_{E^c}\int_{E} \mathbf{1}_{B(x,\varepsilon)}(y) \,\mu(dx)\mu(dy)\\
&\leq\int_{\left(\partial E\right)_\varepsilon}\int_{E} \mathbf{1}_{B(x,\varepsilon)}(y) \, \,\mu(dx)\mu(dy) 
\leq \mu\left({\left(\partial E\right)_\varepsilon}\right)c_{2}\varepsilon^{d_H}.
\end{align*}
Due to~\eqref{E:nbhd_estimateU-} and the fact that 
\begin{equation*}
\mu((\partial E)_\varepsilon)=O\left(\varepsilon^{d_H}\right)_{\varepsilon\to\infty}
\end{equation*} 
because $\partial E$ is compact, the proof is complete.
\end{proof}

Without assuming the underlying space $X$ to be geodesic we can give a sufficient condition ensuring~\eqref{E:nbd_estimateL} which by Theorem~\ref{T:FP_char} provides a lower bound of the $\alpha$-perimeter of the set $E\subseteq X$.
\begin{proposition}\label{prop-reg-b}
Let $X$ be an Ahlfors $d_H$-regular space. Suppose that $E\subseteq X$ has a compact boundary $\partial E$ and 
for all $n\geq 1$ there exist $r_n>0$ and $x_1,\ldots, x_n\in\partial E$ such that $B_{n,k}:=B\big(x_k,\frac{r_n}{2}\big)\subseteq X$, $k=1,\ldots, n$, are disjoint and satisfy
\begin{equation}\label{E:Cond_ball}
\mu(B_{n,k}\cap E)\geq c_9r_n^{d_H}\qquad\text{and}\qquad\mu(B_{n,k}\cap E^c)\geq c_{10}r_n^{d_H}
\end{equation}
for all $k=1,\ldots, n$. Furthermore, assume that for some $c_{12}>0$ and $d_H>\dim (\partial E)\geqslant0$ we have 
\begin{equation}	\label{E:Cond_n-r_n}
n\geqslant c_{12}\cdot r_n^{-\dim (\partial E)}.
\end{equation} 
Then,~\eqref{E:nbd_estimateL} holds with $c_8=c_9c_{10}c_{12}$, $\varepsilon=r_n$ and $\delta=d_H-\dim (\partial E)$.
\end{proposition}

\begin{proof}
Notice that
\begin{equation*}
\bigcup_{X=1}^n \big(B_{n,k}\cap E^c\big)\times \big(B_{n,k}\cap E\big)\subseteq\{(x,y)\in E^c\times E\, :\, d(x,y)<\varepsilon\}.
\end{equation*}
In view of~\eqref{E:Cond_ball} we thus have that
\begin{equation*}
\mu\otimes\mu (\{(x,y)\in E^c\!\times\! E\, :\, d(x,y)<r_n\})
\geq \sum_{k=1}^n \mu\big(B_{n,k}\cap E\big)\mu\big(B_{k,n}\cap E^c\big)\geq n\,c_9 c_{10}\, r_n^{2d_H}.
\end{equation*}
\end{proof} 

\begin{remark}
Note that \eqref{E:Cond_ball} means that both $E$ and $E\times E$ are Ahlfors $d_H$-regular.
We do not specifically use what notion of dimension $\dim_H(\partial E)$ of $ \partial E $ is the most appropriate in this proposition, but usually it is the Hausdorff dimension, which in most of examples coincides with the box counting, Minkowski, and fractal self-similarity dimensions (see e.g.\cite[Section 3.1]{Fal03}).
	
If $\{B_{n,k}\}_{X=1}^n$ is a disjoint cover of $\partial E$ by sets of diameter $r_n$, it follows that $n\sim r_n^{-\dim_B\partial E}$~see e.g.\cite[Section 3.1]{Fal03}, where $\dim_B$ stands for box dimension. Due to the fact that $\dim_H\partial E\leq \dim_B\partial E$, see~\cite[p.46]{Fal03} and $r_n$ is small, we can assume $r_n<1$, hence $n\geq c_{12}r_n^{-\dim_H\partial E}$ and
\begin{equation*}
(\mu\otimes\mu)\big(\{(x,y)\in E^c\!\times\! E\, :\, d(x,y)<r_n\}\big)\geq c_{11} r_n^{2d_H-\dim_H\partial E}=c_{12}r_n^{d_H+\delta}.
\end{equation*}
\end{remark}

\section{Density of $\mathbf{B}^{1,\alpha}(X)$ in $L^1(X,\mu)$ }

A noteworthy application of the results yielding Theorem~\ref{T:FP_char} is that it actually provides sufficient conditions to ensure that $\mathbf{B}^{1,\alpha}(X)$ is dense in $L^1(X,\mu)$.

\begin{theorem}\label{T:L1embedding}
Let $X$ be Ahlfors $d_H$-regular geodesic space with sub-Gaussian upper heat kernel estimate and $\alpha>0$. Assume that there is a family of bounded open sets $E\subset X$ which generates the topology of $X$ and such that
\begin{equation}
\limsup_{r\to 0^+}\frac{1}{r^{\alpha d_W}}\mu\big(\{x\in E~\colon~d(x,E^c)<r\}\big)<\infty
\end{equation}
for some $0<\alpha\leq\frac{d_H}{d_W}$. Then, $\bm{B}^{1,\alpha}(X)$ is dense in $L^1(X,\mu)$.
\end{theorem}
\begin{proof}
By Theorem~\ref{T:FP_char} and Corollary~\ref{C:0P_characterization} respectively, $\mathbf{1}_E\in\bm{B}^{1,\alpha}(X)$ holds for bounded open sets and thus for simple functions. Since these are dense in $L^1(X,\mu)$ the assertion follows.
\end{proof}
\begin{example}\label{Example density gasket}
\begin{enumerate}[itemsep=.5em,label=(\roman*),wide=0em]
\item If $X$ is the infinite Sierpinski gasket as in \cite{BP} (see also Section \ref{examples BE}), a suitable family of open sets consists of (open) triangular cells of any side length. For each such set $E\subset X$, the boundary $\partial E$ consists of three points and $(\partial E)_r^-$ is the union of three triangular cells of side length $r$. Hence, $\mu(\partial E)_r^-=c r^{d_H}$ and Theorem~\ref{T:L1embedding} holds with $\alpha=\frac{d_H}{d_W}$. 
\item A similar result as for the Sierpinski gasket can be expected for any nested fractal~\cite{Lind,Ba98}.
\end{enumerate}
\end{example}

\begin{remark}\label{R:density_SC}
When $X$ is the Sierpinski carpet or a similar fractal the situation is more delicate. We mention some points here; a complete answer remains open.
\begin{enumerate}[itemsep=.5em,label=(\roman*),wide=0em]
\item We believe that in this case Theorem~\ref{T:L1embedding} will not be applicable because there may not be such family of open sets. Intuitively and in view of Example~\ref{Ex:1E_SC}, one expects open sets $E\subset X$ to asymptotically fulfill $\mu(\partial E)_r^-\sim r^{d_H-\dim_H\partial E}$. Therefore, only sets of finite (or countable) boundary would produce the critical exponent but these may not generate the topology of $X$. In fact, in the Sierpinski carpet or the Menger sponge, there are no non-empty open sets with non-empty finite boundary \cite{BB99}. 
This is true for any generalized Sierpinski carpet in the sense of Barlow and Bass. 
\item It may still be true that $\bm{B}^{1,\frac{d_H}{d_W}}(X)$ is dense in $L^1(X,\mu)$, only dense functions will not be generated by indicator functions of a base of open sets.
\end{enumerate}
\end{remark}

\subsection{Examples}\label{Sec:Examples_1E}

Let us now discuss several examples of different nature where the previous results apply. In particular, we can identify sets of finite $\alpha$-perimeter for the corresponding ranges of the parameter $\alpha$.
\begin{example}
[Smooth domains in the $d$-dimensional Euclidean space]
Consider $X=\mathbb{R}^d$, $d_H=d$, $d_W=2$, $\mu=\lambda^d$ the $d$-dimensional Lebesgue measure. A smooth domain $E$ with $\partial E$ closed and $(d-1)$-rectifiable satisfies 
\begin{equation}
\limsup_{r\to 0+}\frac{1}{r} \lambda^d((\partial E)_r)= \frac{2\lambda^{d-1}(\partial E)\Gamma\big(\frac{1}{2}\big)^{d-1}}{(d-1)\Gamma\big(\frac{d-1}{2}\big)}
\end{equation}
because the $(d-1)$Minkowski content of $\partial E$ coincides with $\lambda^{d-1}(\partial E)$, see e.g.~\cite[Theorem 3.2.39]{Fed69}. Since $\mathbb{R}^d$ is geodesic, Theorem~\ref{T:FP_char}~\eqref{FP_char4} is satisfied with $\alpha d_W=1$ and hence $\mathbf{1}_E \in \mathbf{B}^{1,1/2}(\mathbb{R}^d)$.
\end{example}

\begin{example}[Rectifiable domains in the $d$-dimensional Euclidean space]
More generally, if $\partial E\subset \mathbb{R}^d$ is closed and $m$-rectifiable, again by~\cite[Theorem 3.2.39]{Fed69} we have
\begin{equation}
\limsup_{r\to 0+}\frac{1}{r^{d-m}} \lambda^d((\partial E)_r)= \frac{2\lambda^{m}(\partial E)\Gamma\big(\frac{1}{2}\big)^m}{m\Gamma\big(\frac{m}{2}\big)}.
\end{equation}
Thus, Theorem~\ref{T:FP_char}~\eqref{FP_char4} is satisfied with $\alpha d_W=d-m$, which is the codimension of $\partial E$ and $\mathbf{1}_E \in \mathbf{B}^{1,\frac{d-m}{2}}(\mathbb{R}^d)$. 

For instance, if $E$ is the so-called Koch snowflake domain in $\mathbb R^2$ whose boundary is displayed in Figure~\ref{F:snow}, then $d_H=2$, $d_W=2$ and $m=\frac{\log4}{\log3}$. More precisely, from~\cite[Theorem 1]{LP06} we have the tube formula
\begin{equation}
\lambda^2((\partial E)_r^-)=G_1(r)r^{2-\frac{\log4}{\log3}}+G_2(r)r^2,
\end{equation}
where $G_1,G_2$ are periodic functions. 
Thus, in this case
Theorem~\ref{T:FP_char} yields $\mathbf{1}_E\in\bm{B}^{1,\alpha}(X)$ with $\alpha=1-\frac{\log 4}{2\log 3}$.

\begin{figure}[htb]\centering
\includegraphics 
{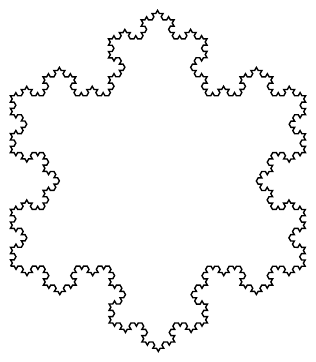}
\caption{von Koch snowflake domain.}
\label{F:snow}
\end{figure}
\end{example}

\begin{example}[Subsets of the Sierpinski carpet]\label{Ex:1E_SC}
Consider $X$ to be the Sierpinski carpet built in the standard way on the unit square $[0,1]^2$, so that $d_H=\frac{\log8}{\log3}$ and $d_W=\text{unknown but unique}>2$.
\begin{enumerate}[leftmargin=*]
\item Let $E=X\cap[0,\frac13]^2$ be a sub-square cell in the standard self-similar construction of $X$. Setting $r=3^{-k}$, an inner neighborhood of the boundary consists of $2\cdot 3^k$ copies of a carpet $X_r$ with side length $r$. Thus,
\[
\mu((\partial E)_r^-)=2\cdot 3^{k}\mu(X_r)=2\cdot 3^{k} r^{\frac{\log 8}{\log 3}}=2r^{\frac{\log 8}{\log 3}-1}.
\]
Theorem~\ref{T:FP_char}~\eqref{FP_char4} is satisfied with $\alpha d_W=\frac{\log 8}{\log 3}-1$ and hence $\mathbf{1}_E \in \mathbf{B}^{1,\frac{\delta}{d_W}}(X) $ with $\delta=\frac{\log8}{\log3}-1$. 
\item If $E=X\cap\big([0,\frac12]\times[0,1]\big)$ is the left half of the carpet $X$, a similar computation as the previous case yields $\mathbf{1}_E \in \mathbf{B}^{1,\frac{\delta}{d_W}}(X) $ with $\delta=\frac{\log8}{\log3}-\frac{\log2}{\log3}=\frac{2\log2}{\log3} $.
\end{enumerate}
Notice that in each case, the dimension of the boundary of $E$ was $1$ or  $\frac{\log 2}{\log 3}$ 
respectively, supporting the conjecture from Remark~\ref{R:density_SC}.
\end{example}

\begin{figure}[htb]\centering
\includegraphics[width=.30\textwidth] {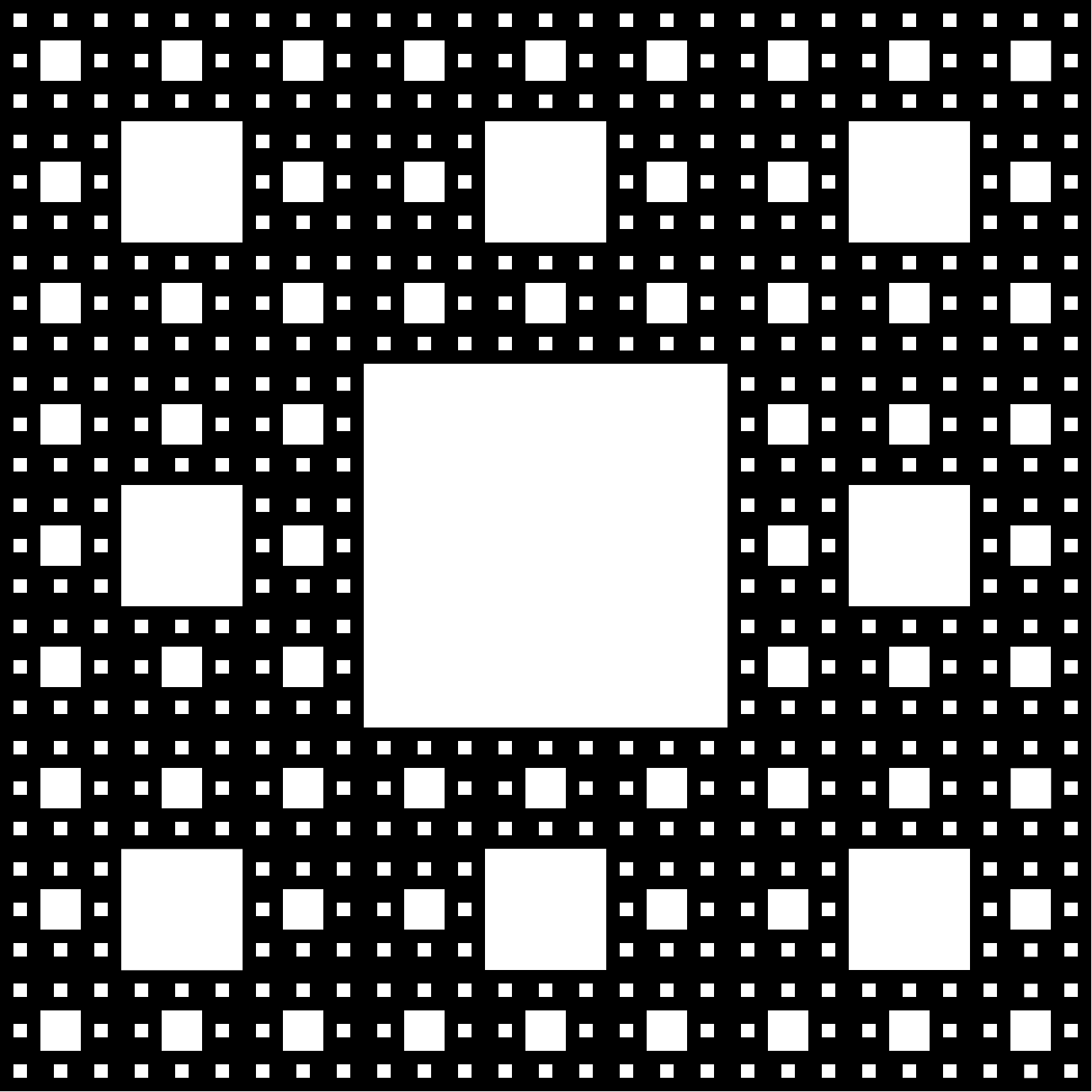}
\caption{Sierpinski carpet}
\label{fig-SC}
\end{figure}

\begin{example}[Subsets of the Menger sponge]
Let $X$ be the standard Menger sponge on the unit cube $[0,1]^3$. Here, $d_H=\frac{\log20}{\log3}$ and $d_W=\text{unknown but unique}>2$.
\begin{enumerate}[leftmargin=*]
\item For $E=X\cap[0,\frac13]^3$ is a sub-cube cell in the standard self-similar construction of $X$, then $\mathbf{1}_E \in \mathbf{B}^{1,\frac{\delta}{d_W}}(X) $ with $\delta=\frac{\log20}{\log3}-2 $.
\item For the left half of the cubical sponge, $E=X\cap\big([0,\frac12]\times[0,1]^2\big)$, we have $\mathbf{1}_E \in \mathbf{B}^{1,\frac{\delta}{d_W}}(X) $ with $\delta=\frac{\log20}{\log3}-\frac{\log4}{\log3}=\frac{ \log5}{\log3} $.
\end{enumerate}
\end{example}

\section{Weak Bakry-\'Emery curvature condition $BE(\kappa)$}

\begin{definition}
We will say that $(X,\mu,\mathcal{E},\mathcal{F})$ satisfies the weak Bakry-\'Emery non-negative curvature condition $BE(\kappa)$ if there exist a constant $C>0$ and a parameter $0 < \kappa < d_W$ such that for every $t >0$, $g \in L^\infty(X,\mu)$ and $x,y \in X$,
\begin{align}\label{WBECD}
| P_t g (x)-P_tg(y)| \le C \frac{d(x,y)^\kappa}{t^{\kappa/d_W}} \| g \|_{ L^\infty(X,\mu)}.
\end{align}
\end{definition}

Several examples of spaces satisfying the condition $BE(\kappa)$ will be given later. We first observe that the weak Bakry-\'Emery is related to the H\"older regularity of the heat kernel.

\begin{lemma}
Assume that $(X,\mu,\mathcal{E},\mathcal{F})$ satisfies the weak Bakry-\'Emery  condition $BE(\kappa)$. Then, there exist a constant $C>0$ such that for every $t>0$, $x,y,z \in X$,
\[
| p_t(x,z)-p_t(y,z)| \le C \frac{d(x,y)^\kappa}{t^{(\kappa+d_H)/d_W}}
\]
\end{lemma}

\begin{proof}
Indeed, it is easily seen that the weak Bakry-\'Emery estimate is equivalent to:
\[
\int_X | p_t(x,z)-p_t(y,z)| d\mu(z)\le C \frac{d(x,y)^\kappa}{t^{\kappa/d_W}}.
\]
One has then
\begin{align*}
| p_t(x,z)-p_t(y,z)| &=\left| \int_X (p_{t/2}(x,u)-p_{t/2}(y,u)) p_{t/2}(u,z) d\mu(u)\right| \\
 &\le \frac{C}{t^{d_H/d_W}}  \int_X | p_t(x,z)-p_t(y,z)| d\mu(z)
\end{align*}
\end{proof}

Under the weak Bakry-\'Emery curvature condition, the main result is the following. It can be thought as an analogue of Lemma \ref{lem:L1-norm-control} in a situation where we do not necessarily have a carr\'e du champ operator. 

\begin{theorem}\label{JKLN}
Assume that the weak Bakry-\'Emery curvature condition \eqref{WBECD} is satisfied. Then, for every $f \in \mathbf{B}^{1,1-\frac{\kappa}{d_W}} (X)$, and $t \ge 0$,
\[
\| P_t f -f \|_{L^1(X,\mu)} \le  C t^{1-\frac{\kappa}{d_W}}  \limsup_{r \to 0} \frac{1}{r^{d_H+d_W-\kappa}}  
 \iint_{\Delta_r}|f(x)-f(y)|\,d\mu(x)\,d\mu(y)
\]
\end{theorem}

\begin{proof}
In the proof $C$ will denote a constant (depending only on $\kappa,d_W,d_H$) that may change from line to line.
Let $g \in \mathcal{F}$, $\| g \|_{L^\infty} \le 1$ and $f \in \mathbf{B}^{1,1-\frac{\kappa}{d_W}} (X)$.  We have
\begin{align*}
\left| \int_X  g(f-P_tf) d\mu \right| & =  \left| \int_0^t \int_X  g \frac{\partial P_sf}{\partial s}
d\mu ds \right| = \left| \int_0^t \mathcal{E}(g ,P_s f)  ds \right|.
\end{align*}
The idea is now to bound $\mathcal{E}(g ,P_s f) $ by using an approximation of $\mathcal{E}$.
For  $\tau \in(0,\infty)$ and $u\in L^{2}(X,\mu)$, we set
\begin{equation}\label{eq:energy-pt-Lp}
\mathcal{E}_{\tau}(u.u):=\frac{1}{2\tau}\int_X\int_Xp_{\tau}(x,y)|u(x)-u(y)|^{2}\,d\mu(x)\,d\mu(y).
\end{equation}
Note that  by the symmetry and the conservativeness of $P_{t}$,
\begin{align*}
\mathcal{E}_{\tau}(u,u)&=\frac{1}{2\tau}\int_X\int_Xp_{\tau}(x,y)\bigl(u(x)^{2}-2u(x)u(y)+u(y)^{2}\bigr)\,d\mu(x)\,d\mu(y)\\
&=\frac{1}{2\tau}\Bigl(\langle u,u\rangle-2\langle u,P_{\tau}u\rangle+\langle u,u\rangle\Bigr)\\
&=\frac{1}{\tau}\langle u-P_{\tau}u,u\rangle\xrightarrow{t\downarrow 0}
  \begin{cases}\mathcal{E}(u,u) &\textrm{if $u\in\mathcal{F}$,}\\
    \infty &\textrm{if $u\in L^{2}(X,\mu)\setminus\mathcal{F}$,}\end{cases}
\end{align*}
and in particular that $(0,\infty)\ni \tau \mapsto\mathcal{E}_{\tau}(u,u)$ is non-increasing. We now have
\begin{align*}
|\mathcal{E}_\tau (g ,P_s f) | &=| \mathcal{E}_\tau (P_s g , f)| \\
 &\le \frac{1}{2\tau}\int_X\int_Xp_{\tau}(x,y)|P_s g(x)-P_s g(y)| |f(x)-f(y)| \,d\mu(x)\,d\mu(y) \\
 &\le \frac{C}{2s^{\kappa/d_W} \tau}\int_X\int_Xp_{\tau}(x,y) d(x,y)^\kappa |f(x)-f(y)| \,d\mu(x)\,d\mu(y).
\end{align*}
We can now prove, using the same arguments as in the proof of theorem \ref{Besov characterization} that since $f \in \mathbf{B}^{1,1-\frac{\kappa}{d_W}} (X)$, one has
\begin{align}\label{estimate BE rs}
 & \limsup_{\tau \to 0} \frac{1}{ \tau}\int_X\int_Xp_{\tau}(x,y) d(x,y)^\kappa |f(x)-f(y)| \,d\mu(x)\,d\mu(y)  \\
 \le & C  \limsup_{r \to 0} \frac{1}{r^{d_H+d_W-\kappa}}  \iint_{\Delta_r}|f(x)-f(y)|\,d\mu(x)\,d\mu(y). \notag
\end{align}
As a conclusion we deduce
\begin{align*}
|\mathcal{E}(g ,P_s f)| =\lim_{\tau \to 0 }  |\mathcal{E}_\tau (g ,P_s f)| \le \frac{C}{s^{\kappa/d_W} }  \limsup_{r \to 0} \frac{1}{r^{d_H+d_W-\kappa}}  \iint_{\Delta_r}|f(x)-f(y)|\,d\mu(x)\,d\mu(y).
\end{align*}
Therefore, one has
\begin{align*}
\left| \int_X  g(f-P_tf) d\mu \right| \le C t^{ 1-\frac{\kappa}{d_W} } \limsup_{r \to 0} \frac{1}{r^{d_H+d_W-\kappa}}  
\iint_{\Delta_r}|f(x)-f(y)|\,d\mu(x)\,d\mu(y),
\end{align*}
and we conclude by $L^1-L^\infty$ duality.
\end{proof}

We now briefly comment on the significance of the parameter $\kappa$  with two corollaries.  In particular, the following lemma shows that $\kappa$ is a critical parameter in the $L^1$ theory of $X$.

\begin{corollary}
Assume that $BE(\kappa)$ is satisfied. If $f \in \mathbf{B}^{1,\alpha}(X)$ with $\alpha >1-\frac{\kappa}{d_W}$, then $f=0$. 
\end{corollary}

\begin{proof}
Let $f \in \mathbf{B}^{1,\alpha}(X)$ with $\alpha >1-\frac{\kappa}{d_W}$. One has then
\[
 \limsup_{r \to 0} \frac{1}{r^{d_H+d_W-\kappa}}  \iint_{\Delta_r}|f(x)-f(y)|\,d\mu(x)\,d\mu(y)=0.
 \]
 Therefore, from the previous theorem, one has for every $t>0$, $\| P_t f -f \|_{L^1(X,\mu)}=0$, which immediately implies $f \in \mathcal{F}$ and $P_t f=f$, $t \ge 0$. The sub-Gaussian heat kernel upper bound implies then $f=0$. 
 
 \end{proof}

\begin{corollary}\label{estimate 45}
One has $\kappa \le \frac{d_W}{2}$ and thus $1-\frac{\kappa}{d_W}\ge \frac{1}{2}$.
\end{corollary}

\begin{proof}
For $f \ge 0$, we have by Cauchy-Schwartz inequality
\begin{align*}
	&\lefteqn{\iint p_t(x,y)\bigl| |f(x)|^2-|f(y)|^2\bigr|\,d\mu(x)\,d\mu(y) } \quad&\\
\le	& \iint p_t(x,y) \bigl( |f(x)|+|f(y)| \bigr)\bigl| |f(x)|-|f(y)|\bigr|\,d\mu(x)\,d\mu(y) \\
\le	& 2 \int |f(x)| P_t(|f-f(x)|)(x)\, d\mu(x)\\
\le	& 2 \|f\|_2 \Bigl( \int  P_t(|f-f(x)|^2) (x) d\mu(x)\Bigr)^{1/2}\\
\le	&\ 2 \|f\|_2 t^\alpha \| f\|_{2,\alpha}
	\end{align*}
so that $\| |f|^2\|_{1,\alpha} \leq 2\|f\|_2 \|f\|_{2,\alpha}$. Hence, if $\alpha=\frac{1}{2}$, one deduces that if $f \in \mathbf{B}^{2,1/2}(X)=\mathcal{F}$, then $f^2 \in \mathbf{B}^{1,1/2}(X)$.  One concludes that $\mathbf{B}^{1,1/2}(X)$ is always non-trivial, so that one must have $1-\frac{\kappa}{d_W}\ge \frac{1}{2}$.
\end{proof}



We will see  that the weak Bakry-\'Emery estimate can be used to improve Sobolev and isoperimetric inequalities.  We show here that it also yields an improvement in Theorem \ref{ahlfors-coarea}.

\begin{theorem}\label{T:BE-ahlfors-coarea}
Let $X$ be Ahlfors $d_H$-regular with sub-Gaussian heat kernel estimates and the $BE(\kappa)$ estimate. Let $u\in L^1(X,\mu)$. For $s \in \mathbb{R}$, denote $E_s(u)=\{ x \in X, u(x)>s\}$. Let $ \alpha =1-\frac{\kappa}{d_W}$ and assume that $s\mapsto \limsup_{r  \to 0^+} \frac{1}{r^{\alpha d_W}} \mu (\{x\in E_s(u)~\colon~d(x,X\setminus E_s(u))<r\})$ is in $L^1 (\R)$. Then, $u \in \mathbf{B}^{1,\alpha}(X) $ and there exists a constant $C>0$ independent from $u$ such that:
\[
\| u \|_{1, \alpha} \le C \int_\R \limsup_{r \to  0^+} \frac{1}{r^{\alpha d_W}} \mu (\{x\in E_s(u)~\colon~d(x,X\setminus E_s(u))<r\}) ds.
\]
In particular, if  for a set $E \subset X$ with finite measure, $\limsup_{r \to  0^+} \frac{1}{r^{\alpha d_W}} \mu (\{x\in E~\colon~d(x,E^c)<r\})<+\infty$, then $\mathbf{1}_E \in \mathbf{B}^{1,\alpha}(X) $ and:
\[
\| \mathbf{1}_E \|_{1, \alpha} \le C \limsup_{r \to  0^+} \frac{1}{r^{\alpha d_W}} \mu (\{x\in E~\colon~d(x,E^c)<r\}) .
\]
\end{theorem}
\begin{proof}
This follows from the proof of Theorem~\ref{ahlfors-coarea} and the statement of Theorem \ref{JKLN}.
\end{proof}

\begin{corollary}
If in addition $X$ is uniformly locally connected, then Theorem~\ref{T:BE-ahlfors-coarea} holds with $\mu((\partial E_s(u))_r^-)$ instead of $\mu(\{x\in E_s(u)~\colon~d(x,X\setminus E_s(u))<r\})$. 
\end{corollary}

\section{Examples of spaces satisfying the weak Bakry-\'Emery condition}

\subsection{Unbounded Sierpinski gasket and unbounded Sierpinski carpet}\label{examples BE}

We show here first  that the unbounded Sierpinski gasket $(X, \mathcal E, \mathcal F, \mu)$ considered in \cite{BP} satisfies the weak Bakry-\'Emery estimate.

\begin{theorem}
The unbounded Sierpinski gasket satisfies $BE(\kappa)$ with $\kappa=d_W-d_H$.
\end{theorem}

\begin{proof}
Let $f \in L^\infty(X,\mu)$. From Theorem 5.22 in \cite{BP}, one has for $t >0$ and $\lambda >0$,
\[
| P_t f(x) -P_t f(y) | \le \frac{C}{\lambda^{d_H/d_W}} d(x,y)^{d_W-d_H} \| (L-\lambda)P_t f \|_{L^\infty(X,\mu)},
\]
where $L$ is the generator of $\mathcal{E}$. We now observe that
\[
 | (L-\lambda)P_t f (x)| \le \int_X |\partial_t p(t,x,y)-\lambda p(t,x,y) | d\mu(y) \|f \|_{L^\infty(X,\mu)}.
\]
From the sub-Gaussian upper bound for the heat kernel and its time derivative (see \cite{BP}), one deduces
\[
 \int_X |\partial_t p(t,x,y)-\lambda p(t,x,y) | d\mu(y) \le C \left(\frac{1}{t}+\lambda\right).
\]
Therefore, by choosing $\lambda=1/t$ one concludes:
\[
| P_t f(x) -P_t f(y) | \le C t^{d_H/d_W-1} d(x,y)^{d_W-d_H} \|  f \|_{L^\infty(X,\mu)}.
\]
\end{proof}

In that case, one has therefore $1-\frac{\kappa}{d_W}=\frac{d_H}{d_W}$, so from Theorem \ref{JKLN} the space of interest is $\mathbf{B}^{1,\frac{d_H}{d_W}}(X)$. We note that from Theorem \ref{T:L1embedding}, $\mathbf{B}^{1,\frac{d_H}{d_W}}(X)$ is dense in $L^1(X,\mu)$.

\begin{figure}[htb]\centering
	\includegraphics{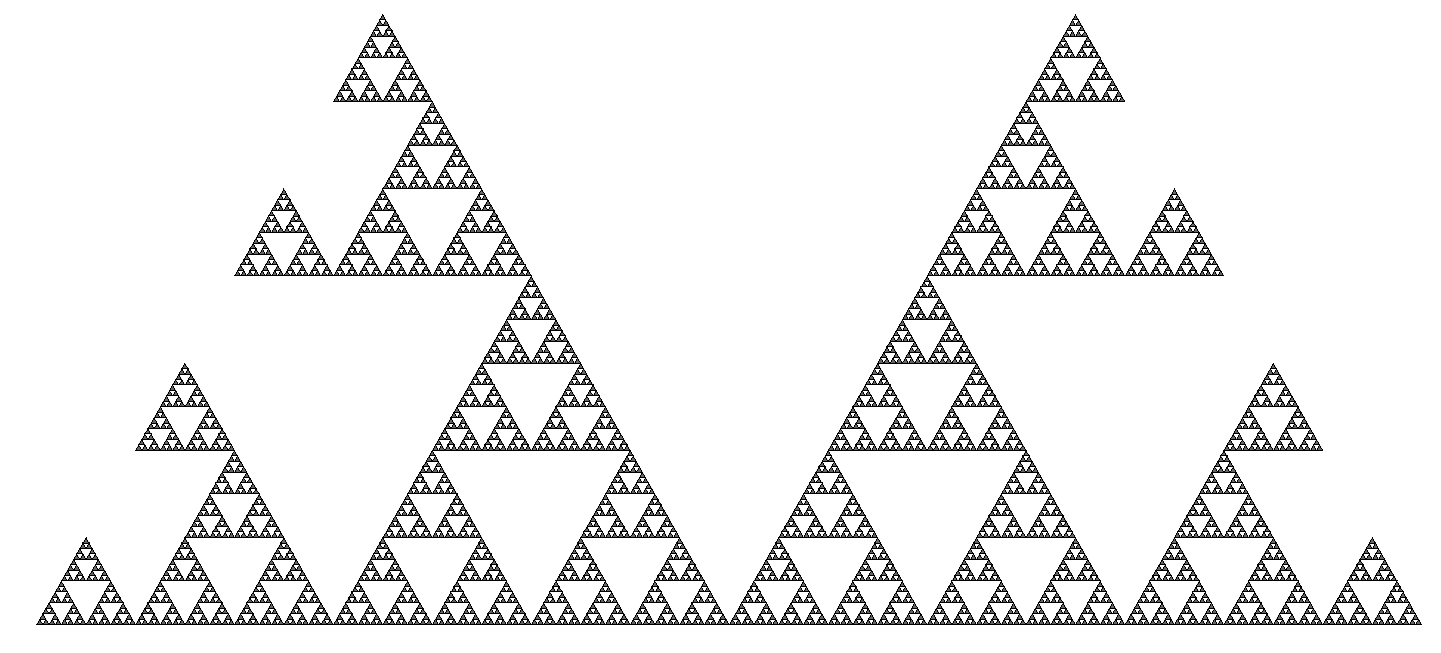}
	\caption{A part of an infinite Sierpinski gasket.}\label{fig-sig}
\end{figure}

We now  show   that the unbounded Sierpinski carpet $(X, \mathcal E, \mathcal F, \mu)$ considered in \cite{BB92} satisfies the weak Bakry-\'Emery estimate.

\begin{theorem}
The unbounded Sierpinski carpet satisfies $BE(\kappa)$ with $\kappa=d_W-d_H$.
\end{theorem}

\begin{proof}
The proof is the same as the case of Sierpinski gasket, thanks to \cite[Theorem 4.9 (b)]{BB92} and \cite[Theorem 4 or Corollary 5]{Da97}.
\end{proof}

\subsection{Fractional metric spaces with fractional diffusion}
In~\cite[Section 3]{Ba98} a class of processes called \textit{fractional diffusions} were introduced in the context of fractional metric spaces, which are complete Ahlfors regular metric spaces with the midpoint property, c.f.~\cite[Definition 3.2, Definition 3.5]{Ba98}. In this setting, applying~\cite[Theorem 3.40]{Ba98} and~\cite[Theorem 4 or Corollary 5]{Da97} as in the previous example yields the corresponding weak Bakry-\'Emery estimate.

\begin{theorem}
A fractional metric space of dimension $d_H\geq 1$ and associated fractional diffusion with parameters $2\leq d_W\leq d_H+1$ satisfies $BE(\kappa)$ with $\kappa=d_W-d_H$.
\end{theorem}

\begin{remark}
Notice that nested fractals (\cite{Lind,Ba98}), and in particular the Vicsek set, fall into this category. In the view of the discussion of Sections~\ref{S:critical_exp_local}, we can view the case of the Vicsek set as critical since $\kappa=1$. 
Note that the Vicsek set is a dendrite, see \cite{KigamiDendrites}. 
\end{remark}

\subsection{Stability of $BE(\kappa)$ by tensorization}

The condition $BE(\kappa)$ is stable by tensorization. This yields a lot of examples.

Consider $(X^n, d_{X^n},\mu^{\otimes n})$, where the metric $ d_{X^n}$ is defined as follows: for any $\mathbf x=(x_1, x_2,\cdots, x_n) \in X^n$ and $\mathbf x'=(x_1', x_2',\cdots, x_n')\in X^n$,
\[
d_{X^n}(\mathbf x, \mathbf x')^{\frac{d_W}{d_W-1}}=\sum_{i=1}^n d_X(x_i,x_i')^{\frac{d_W}{d_W-1}}.
\]
Let $p_t^X(x, y)$ be the heat kernel on $X$. Then the heat kernel on $X^n$ is given by
\[
p_t^{X^n}(\mathbf x, \mathbf y)=p_t^X(x_1, y_1)\cdots p_t^X(x_n, y_n).
\]

\begin{proposition}
 if $X$ satisfies $BE(\kappa)$, then $X^n$ also satisfies $BE(\kappa)$.
\end{proposition}

\begin{proof}
Let $f\in L^{\infty}(X^n, \mu^{\otimes n})$. Then for any $\mathbf x, \mathbf x'\in X^n$, we have
\begin{align*}
\abs{P_t^{X^n} f(\mathbf x)-P_t^{X^n} f(\mathbf x')}
\le &
\abs{P_t^{X^n} f(x_1, x_2,\cdots, x_n)-P_t^{X^n} f(x_1', x_2,\cdots, x_n)}
\\ &+\abs{P_t^{X^n} f(x_1', x_2,\cdots, x_n)-P_t^{X^n} f(x_1', x_2',x_3,\cdots, x_n)}
\\ &+\cdots + \abs{P_t^{X^n} f(x_1', \cdots,x_{n-1}', x_n)-P_t^{X^n} f(x_1', x_2',\cdots, x_n')}.
\end{align*}
Notice that for any $1\le i\le n$, since $X$ satisfies $BE(\kappa)$, then
\begin{align*}
&\abs{P_t^{X^n} f(x_1', \cdots,x_{i-1}',x_{i},\cdots, x_n)-P_t^{X^n} f(x_1', \cdots,x_i',x_{i+1},\cdots, x_n)}
\\ &\le 
\abs{P_t^{X}(P_t^{X^{n-1}} f(x_1', \cdots,x_{i-1}',\cdot,x_{i+1},\cdots, x_n))(x_i)-P_t^X(P_t^{X^{n-1}} f(x_1', \cdots,x_{i-1}',\cdot,x_{i+1},\cdots, x_n))(x_i')}
\\ &\le
C \frac{d_X(x_i,x_i')^{\kappa}}{t^{\kappa/d_W}} \left\|P_t^{X^{n-1}} f(x_1', \cdots,x_{i-1}',\cdot,x_{i+1},\cdots, x_n)\right\|_{L^{\infty}(X)}
\\ &\le
C \frac{d_X(x_i,x_i')^{\kappa}}{t^{\kappa/d_W}} \|f\|_{L^{\infty}(X^n)}.
\end{align*}
Hence
\[
\abs{P_t^{X^n} f(\mathbf x)-P_t^{X^n} f(\mathbf x')} 
\le C\frac{1}{t^{\kappa/d_W}}  \brak{\sum_{i=1}^n d_X(x_i,x_i')^{\kappa}} \|f\|_{L^{\infty}(X^n)}
\le
C \frac{d_{X^n}(\mathbf x,\mathbf x')^{\kappa}}{t^{\kappa/d_W}} \|f\|_{L^{\infty}(X^n)}.
\]
(The constant $C$ may or may not depend on $n$.)
\end{proof}

\section{Continuity of the heat semigroup in the Besov spaces: the case $p \ge 2$}

In this section, we study the continuity of the heat semigroup in the Besov spaces in the range $p \ge 2$. This complements the results of Chapter 1, under the additional weak Bakry-\'Emery estimate assumption. For the sake of clarity, in the statements in this paragraph the assumptions on the underlying space $X$ are explicitly written. Throughout this section, let $X$ be an Ahlfors $d_H$-regular space that satisfies sub-Gaussian heat kernel estimates and $BE(\kappa)$ with $0<\kappa\leq \frac{d_W}{2}$.

\begin{theorem}\label{P:PtinBesovp}
 For any $p\geq 2$, there exists a constant $C>0$ such that for every $t>0$ and $f\in L^2(X,\mu)\cap L^\infty(X,\mu)$
\begin{equation*}
\|P_tf\|^p_{p,\kappa/d_W}\leq \frac{C}{t^{\frac{\kappa p}{d_W}}}\| f\|_{L^2(X,\mu)}^2  \| f \|^{p-2}_{L^\infty(X,\mu)}.
\end{equation*}
 In particular, for $t>0$, $P_tf\in\bm{B}^{p,\frac{\kappa}{d_W}}(X)$.
\end{theorem}

\begin{proof}
First, note that $P_tf\in L^p(X,\mu)\cap \mathcal{F}$ for any $p\geq 2$. In order to prove that $\|P_t f\|_{p,\kappa/d_W}$ is finite, we proceed as in the proof of Theorem~\ref{Besov characterization} and consider, for any $s,r>0$,
\begin{align*}
&A(s,r):=\int_X\int_{X\setminus B(y,r)}p_s(x,y)|P_tf(x)-P_tf(y)|^p\mu(dx)\mu(dy),\\
&B(s,r):=\int_X\int_{X(y,r)}p_s(x,y)|P_tf(x)-P_tf(y)|^p\mu(dx)\mu(dy).
\end{align*}
In what follows, $C$ will denote a positive constant that may change from line to line.
On the one hand, from the proof of Theorem \ref{Besov characterization} we know that the upper bound on the heat kernel implies
\begin{equation*}
A(s,r)\leq c_8\exp\Big(\!\!-c_9\Big(\frac{r^{d_W}}{s}\Big)^{\frac{1}{d_W-1}}\Big)\|P_tf\|_{L^p(X,\mu)}^p.
\end{equation*}
On the other hand, the aforementioned proof also yields
\begin{multline*}
B(s,r)\leq c_3\sum_{n=1}^\infty\frac{(2^{1-n}r)^{p\kappa+d_H}}{s^{d_H/d_W}} \exp\Big(\!\!-c_9\Big(\frac{(2^{1-n}r)^{d_W}}{s}\Big)^{\frac{1}{d_W-1}}\Big)\\
\times\frac{1}{(2^{1-n}r)^{p\kappa+d_H}}\int_X\int_{B(y,2^{1-n}r)}|P_tf(x)-P_tf(y)|^p\mu(dy)\mu(dx).
\end{multline*}
Writing $|P_tf(x)-P_tf(y)|^p=|P_tf(x)-P_tf(y)|^{p-2}|P_tf(x)-P_tf(y)|^2$ and applying $BE(\kappa)$ to the first factor we obtain
\begin{align*}
B(s,r)\leq &\frac{C\| f \|_\infty^{p-2}}{t^{\kappa(p-2)/d_W}}\sum_{n=1}^\infty\frac{(2^{1-n}r)^{p\kappa+d_H}}{s^{d_H/d_W}} \exp\Big(\!\!-c_9\Big(\frac{(2^{1-n}r)^{d_W}}{s}\Big)^{\frac{1}{d_W-1}}\Big)\\
&\times\frac{1}{(2^{1-n}r)^{2\kappa+d_H}}\int_{\{(x,y)\in X\times X\,\colon\,d(x,y)<2^{1-n}r\}}|P_tf(x)-P_tf(y)|^2\mu(dy)\mu(dx)\\
\leq &\frac{C\| f \|_\infty^{p-2} s^{p\kappa/d_W}}{t^{\kappa(p-2)/d_W}}\sup_{\tilde{s}\in (0,r]}\!\!\big(N_2^\kappa(P_tf,\tilde{s})\big)^2\sum_{n=1}^\infty\Big(\frac{(2^{1-n}r)^{d_W}}{s}\Big)^{\frac{p\kappa+d_H}{d_W}} \!\!\exp\Big(\!\!-c_9\Big(\frac{(2^{1-n}r)^{d_W}}{s}\Big)^{\frac{1}{d_W-1}}\Big)\\
\leq &\frac{C\| f \|_\infty^{p-2}}{t^{\kappa(p-2)/d_W}}s^{p\kappa/d_W}\|P_t f\|_{2,\kappa/d_W}^2,
\end{align*}
where the last inequality follows from the proof of the lower bound in Theorem \ref{Besov characterization}. We have thus proved that for any $r>0$ 
\begin{multline*}
s^{-p\kappa/d_W}\int_X\int_X p_s(x,y)|P_tf(x)-P_tf(y)|^p\mu(dx)\mu(dy)\\
\leq c_8s^{-p\kappa/d_W}\exp\Big(\!\!-c_9\Big(\frac{r^{d_W}}{s}\Big)^{\frac{1}{d_W-1}}\Big)\|P_tf\|_{L^p(X,\mu)}^p+\frac{C\| f \|_\infty^{p-2}}{t^{\kappa(p-2)/d_W}}\|P_t f\|_{2,\kappa/d_W}^2.
\end{multline*}
This yields
\begin{align*}
 & \sup_{s>0}s^{-p\kappa/d_W}\int_X\int_X p_s(x,y)|P_tf(x)-P_tf(y)|^p\mu(dx)\mu(dy)\\
&\qquad\qquad\qquad\qquad \le \frac{C}{r^{p\kappa}}\|P_tf\|_{L^p(X,\mu)}^p+\frac{C\| f \|_\infty^{p-2}}{t^{\kappa(p-2)/d_W}}\|P_t f\|_{2,\kappa/d_W}^2
\end{align*}
and letting $r\to\infty$ we obtain
\begin{equation*}
\|P_tf\|^p_{p,\kappa/d_W}\leq\frac{C\| f \|_\infty^{p-2}}{t^{\kappa(p-2)/d_W}}\|P_t f\|_{2,\kappa/d_W}^2.
\end{equation*}
We now find an upper bound for $\|P_t f\|_{2,\kappa/d_W}$. From Lemma \ref{Lemma limsup debut}, one has for $t>0$
\begin{align*}
\| P_t f \|_{2,\kappa/d_W} &\le \frac{2}{t^{\kappa/d_W}} \| P_t f \|_{L^2(X,\mu)} +\sup_{s\in (0,t]} s^{-\kappa/d_W} \left( \int_X P_s (|P_tf-P_tf(y)|^2)(y) d\mu(y) \right)^{1/2} \\
 &\le \frac{2}{t^{\kappa/d_W}} \| f \|_{L^2(X,\mu)} +t^{\frac{1}{2} -\frac{\kappa}{d_W}} \sup_{s\in (0,t]} s^{-1/2} \left( \int_X P_s (|P_tf-P_tf(y)|^2)(y) d\mu(y) \right)^{1/2}.
\end{align*}
We now note that from spectral theory,
\[
\sup_{s\in (0,t]} s^{-1/2} \left( \int_X P_s (|P_tf-P_tf(y)|^2)(y) d\mu(y) \right)^{1/2}\le C \sqrt{ \mathcal{E}(P_t f,P_t f)}\le \frac{C}{t^{1/2}} \| f \|_{L^2(X,\mu)}
\]
and the proof is complete.
\end{proof}

We first point out two corollaries.

\begin{corollary}\label{P:PtinBesovp2}
For any $p\geq 2$, there exists a constant $C>0$ such that for every $t>0$ and $f\in L^2(X,\mu)$
\begin{equation*}
\|P_tf\|^p_{p,\kappa/d_W}\leq \frac{C}{t^{\frac{\kappa p}{d_W} + \frac{(p-2)d_H}{2d_W}}}\| f\|_{L^2(X,\mu)}^p .
\end{equation*}
 In particular, for $t>0$, $P_t: L^2(X,\mu) \to \bm{B}^{p,\frac{\kappa}{d_W}}(X)$ is bounded for every $p \ge 2$.
\end{corollary}

\begin{proof}
From the heat kernel upper bound, we have that $P_t :L^2(X,\mu) \to L^\infty(X,\mu) $ is continuous with
\[
\|P_t f \|_{L^\infty(X,\mu) } \le \frac{C}{t^ \frac{d_H}{2d_W}} \|f \|_{L^2(X,\mu) }.
\]
The result then easily follows from the theorem \ref{P:PtinBesovp} applied to $P_t f$ instead of $f$.
\end{proof}

Since $P_t$ is also bounded from $ L^2(X,\mu) \to \bm{B}^{1,\frac{1}{2}}(X)=\mathcal{F}$,   by using the interpolation inequalities in proposition \ref{interpolation inequality}, one obtains:

\begin{corollary}\label{P:PtinBesovp3}
For any $p\geq 2, t>0$,  $P_t: L^2(X,\mu) \to \bm{B}^{p,\left(1-\frac{2}{p}\right)\frac{\kappa}{d_W}+\frac{1}{p}}(X)$ is bounded. More precisely,
\[
\| P_t f \|_{p,\left(1-\frac{2}{p}\right)\frac{\kappa}{d_W}+\frac{1}{p}}  \le  \frac{C}{t^{ \left(1-\frac{2}{p}\right) \left(\frac{\kappa }{d_W} +\frac{d_H}{2d_W} \right)+\frac{1}{p}}}\| f\|_{L^2(X,\mu)}.
\]
\end{corollary}

\begin{proof}
Let $p \ge 2$ and $q \ge p$. By applying  proposition \ref{interpolation inequality} with 
\[
\theta=\frac{\frac{1}{2} -\frac{1}{p}}{\frac{1}{2} -\frac{1}{q}}
\]
and 
\[
\alpha=\theta \frac{\kappa}{d_W} +\frac{1}{2} \left( 1- \theta \right),
\]
we get
\[
\| P_t f \|_{p,\alpha} \le \| P_t f \|^\theta_{q, \kappa /d_W}  \|P_t f \|^{1-\theta}_{2,1/2}.
\]
Now, from corollary \ref{P:PtinBesovp2},
\[
\| P_t f \|_{q, \kappa /d_W} \le  \frac{C}{t^{\frac{\kappa }{d_W} + \frac{(q-2)d_H}{2qd_W}}}\| f\|_{L^2(X,\mu)},
\]
and from spectral theory
\[
\| P_t f \|_{2,1/2} \le \frac{C}{\sqrt{t}} \| f\|_{L^2(X,\mu)}.
\]
One deduces
\[
\| P_t f \|_{p,\alpha}  \le  \frac{C}{t^{\theta \left(\frac{\kappa }{d_W} + \frac{(q-2)d_H}{2qd_W}\right)+\frac{1-\theta}{2}}}\| f\|_{L^2(X,\mu)}.
\]
By letting $q \to +\infty$, one has , $\theta \to 1-\frac{2}{p}$ and $\alpha \to \left(1-\frac{2}{p}\right)\frac{\kappa}{d_W}+\frac{1}{p}$. It follows therefore
\[
\| P_t f \|_{p,\left(1-\frac{2}{p}\right)\frac{\kappa}{d_W}+\frac{1}{p}}  \le  \frac{C}{t^{ \left(1-\frac{2}{p}\right) \left(\frac{\kappa }{d_W} + \frac{d_H}{2d_W}\right)+\frac{1}{p}}}\| f\|_{L^2(X,\mu)}.
\]
\end{proof}

Finally, for $p \ge 2$ we now prove the boundedness of $P_t:
L^p(X,\mu)\to \bm{B}^{p,\left(1-\frac{2}{p}\right)\frac{\kappa}{d_W}+\frac{1}{p}}(X)$.

\begin{theorem}\label{P:PtinBesovp4}
For any $p\geq 2$, there exists a constant $C>0$ such that for every $t>0$ and $f\in L^p(X,\mu)$
\[
\| P_t f \|_{p,\left(1-\frac{2}{p}\right)\frac{\kappa}{d_W}+\frac{1}{p}}  \le  \frac{C}{t^{ \left(1-\frac{2}{p}\right)  \frac{\kappa }{d_W} +\frac{1}{p}}}\| f\|_{L^p(X,\mu)}.
\]
 In particular, for $t>0$, $P_t: L^p(X,\mu) \to \bm{B}^{p,\left(1-\frac{2}{p}\right)\frac{\kappa}{d_W}+\frac{1}{p}}(X)$ is bounded.
\end{theorem}

\begin{proof}
As usual, in the sequel, $C$ is a constant that may change from line to line.
 From theorem \ref{P:PtinBesovp}, for any $q\geq 2$, there exists a constant $C>0$ such that for every $t>0$ and $f\in L^2(X,\mu)\cap L^\infty(X,\mu)$
\begin{equation*}
\|P_tf\|^q_{q,\kappa/d_W}\leq \frac{C}{t^{\frac{\kappa q}{d_W}}}\| f\|_{L^2(X,\mu)}^2  \| f \|^{q-2}_{L^\infty(X,\mu)}.
\end{equation*}
From the heat kernel upper bound, we have that $P_t :L^q(X,\mu) \to L^\infty(X,\mu) $ is continuous with
\[
\|P_t f \|_{L^\infty(X,\mu) } \le \frac{C}{t^ \frac{d_H}{qd_W}} \|f \|_{L^q(X,\mu) }.
\]
Therefore, one deduces that exists a constant $C>0$ such that for every $t>0$ and $f\in L^2(X,\mu)\cap L^q(X,\mu)$
\begin{equation*}
\|P_tf\|^q_{q,\kappa/d_W}\leq \frac{C}{t^{\frac{\kappa q}{d_W}+\frac{(q-2)d_H}{qd_W}}}  \| f\|_{L^2(X,\mu)}^2 \| f \|^{q-2}_{L^q(X,\mu)}.
\end{equation*}
If $A \subset X$ is a set of finite measure, one deduces that
\begin{equation*}
\|P_t(f\mathbf{1}_A) \|^q_{q,\kappa/d_W}\leq \frac{C\mu(A)^{ 1-\frac{1}{2q}}}{t^{\frac{\kappa q}{d_W}+\frac{(q-2)d_H}{qd_W}}}   \| f \|^{q}_{L^q(A,\mu)}.
\end{equation*}
Thus, for every $f \in L^q(A,\mu)$,
\[
\|P_t(f\mathbf{1}_A) \|_{q,\kappa/d_W}\leq \frac{C \mu(A)^{ \frac{1}{q}-\frac{1}{2q^2}}}{t^{\frac{\kappa }{d_W}+\frac{(q-2)d_H}{q^2d_W}}}   \| f \|_{L^q(A,\mu)}.
\]
On the other hand,  for every $f \in L^2(A,\mu)$
\[
\| P_t (f\mathbf{1}_A) \|_{2,1/2} \le \frac{C}{\sqrt{t}} \| f\|_{L^2(A,\mu)}.
\]
Let now $2 \le p \le q$ and
\[
\theta=\frac{\frac{1}{2} -\frac{1}{p}}{\frac{1}{2} -\frac{1}{q}}
\]
and 
\[
\alpha=\theta \frac{\kappa}{d_W} +\frac{1}{2} \left( 1- \theta \right).
\]

From the Riesz-Thorin interpolation theorem one deduces that  for every $f \in L^p(A,\mu)$
\[
\| P_t (f\mathbf{1}_A) \|_{p,\alpha} \le \frac{C \mu(A)^{ \frac{\theta}{q}-\frac{\theta}{2q^2}}}{t^{\theta \left(\frac{\kappa }{d_W} + \frac{(q-2)d_H}{q^2d_W}\right)+\frac{1-\theta}{2}}}\| f\|_{L^p(A,\mu)}.
\]
As before, by letting $q \to +\infty$, one has , $\theta \to 1-\frac{2}{p}$ and $\alpha \to \left(1-\frac{2}{p}\right)\frac{\kappa}{d_W}+\frac{1}{p}$. One deduces that there exists a constant $C>0$ (independent from $A$) such that for every $t>0$ and $f\in L^p(X,\mu)$
\[
\| P_t (f\mathbf{1}_A) \|_{p,\left(1-\frac{2}{p}\right)\frac{\kappa}{d_W}+\frac{1}{p}}  \le  \frac{C}{t^{ \left(1-\frac{2}{p}\right)  \frac{\kappa }{d_W} +\frac{1}{p}}}\| f\|_{L^p(A,\mu)}\le  \frac{C}{t^{ \left(1-\frac{2}{p}\right)  \frac{\kappa }{d_W} +\frac{1}{p}}}\| f\|_{L^p(X,\mu)}.
\]
The conclusion easily follows since $A$ is arbitrary.
\end{proof}

We finish the section with two interesting consequence of the Besov continuity of the semigroup.

\begin{proposition}
Let $L$ be the generator of $\mathcal{E}$, $p \ge 2$. Denote
\[
\alpha^*_p =\left(1-\frac{2}{p}\right)\frac{\kappa}{d_W}+\frac{1}{p}.
\]
and $\mathcal{L}_p$ be the domain of $L$ in $L^p(X,\mu)$. Then
\[
\mathcal{L}_p \subset \B^{p,\alpha_p^*}(X)
\]
and for every $f \in \mathcal{L}_p$,
\begin{equation}\label{eq:multi2}
\|f\|_{p,\alpha_p^*} \le C \norm{ Lf}^{\alpha_p^*}_{L^p(X,\mu)} \| f\|^{1-\alpha_p^*}_{L^p(X,\mu)}.
\end{equation}
\end{proposition}

\begin{proof}
For $\lambda >0$, write 
\[
R_{\lambda}f=(L-\lambda)^{-1} f=\int_0^{\infty} e^{-\lambda t} P_tf dt. 
\]
We have then
\[
\|R_{\lambda}f\|_{p,\alpha_p^*} \le \int_0^{\infty} e^{-\lambda t} \|P_tf\|_{p,\alpha_p^*}  dt 
\le \int_0^{\infty} e^{-\lambda t}  \frac{C}{t^{\,\alpha_p^*}} \|f\|_{p}  dt\le C \lambda^{\alpha^*-1} \|f\|_p.
\]
It follows that 
\[
\|f\|_{p,\alpha_p^*} \le C  \lambda^{\,\alpha^*-1} \norm{(L-\lambda)f}_p  \le C ( \lambda^{\,\alpha^*-1} \|Lf\|_p+ \lambda^{\,\alpha^*} \|f\|_p).
\]
Taking $\lambda=\|Lf\|_p \|f\|_p^{\,-1}$, we then get the conclusion.

\end{proof}

For $p \ge 1$, we denote:

\[
\alpha^*_p =\left(1-\frac{2}{p}\right)\frac{\kappa}{d_W}+\frac{1}{p}.
\]
We have then the following theorem that complements Theorem \ref{Critical bound Chapter 1} of Chapter 1.
\begin{theorem}\label{JKLNM}
Let $1 \le p \le 2$. For every $f \in \mathbf{B}^{p,\alpha_p^*} (X)$, and $t \ge 0$,
\[
\| P_t f -f \|_{L^p(X,\mu)} \le  C t^{\alpha^*_p}  \limsup_{r \to 0} \frac{1}{r^{d_H+d_W\alpha_p^* }} 
\left( \iint_{\Delta_r}|f(x)-f(y)|^p\,d\mu(x)\,d\mu(y)\right)^{1/p}
\]
\end{theorem}

\begin{proof}
With Theorem \ref{P:PtinBesovp4}  in hands, the proof is similar to that of Theorem \ref{Critical bound Chapter 1} .
\end{proof}

\section{Critical exponents and generalized Riesz transforms}\label{S:critical_exp_local}

Let $X$ be an Ahlfors $d_H$-regular space that satisfies sub-Gaussian  heat kernel estimates and $BE(\kappa)$ with $0<\kappa\leq \frac{d_W}{2}$.
For $p \ge 1$, we denote in this section:

\[
\alpha^*_p(X) =\inf \left\{ \alpha >0\, :\, \B^{p,\alpha}(X) \text{ is trivial} \right\}.
\]

\begin{theorem}
The following inequalities hold:
\begin{itemize}
\item For $1 \le p \le 2$, $$ \frac{1}{2} \le  \alpha^*_p(X) \le \left(1-\frac{2}{p}\right)\frac{\kappa}{d_W}+\frac{1}{p}.$$ 
\item For $ p \ge 2$, $$ \left(1-\frac{2}{p}\right)\frac{\kappa}{d_W}+\frac{1}{p} \le  \alpha^*_p(X)  \le \frac{1}{2} .$$
\end{itemize}
\end{theorem}

\begin{proof}
Let $p \ge 2$. Let $\beta_p=\left(1-\frac{2}{p}\right)\frac{\kappa}{d_W}+\frac{1}{p}$. Thanks, to Theorem \ref{P:PtinBesovp4}, $ \B^{p,\beta_p}(X)$ is dense in $L^p(X,\mu)$, because for $f \in L^p(X,\mu)$, $P_tf \in  \B^{p,\beta_p}(X)$ and $P_t f \to f$ when $t \to 0$. Next, we prove that for $\alpha > \beta_p$, and $1 \le p \le 2 $,   $\B^{p,\alpha}(X)$  is trivial, which is a consequence of Theorem \ref{JKLNM} since for $f \in \B^{p,\alpha}(X)$  one has then
\[
\limsup_{r \to 0} \frac{1}{r^{d_H+d_W\beta_p }} 
\left( \iint_{\Delta_r}|f(x)-f(y)|^p\,d\mu(x)\,d\mu(y)\right)^{1/p}=0
\]
and thus $P_t f =f$, $t \ge 0$.
\end{proof}

\begin{remark}\label{crticical gasket}
Note that for the unbounded Sierpinski gasket, from previous results (see Example  \ref{Example density gasket}), one has
\[
\alpha_1^*(X)=\frac{d_H}{d_W}.
\]
\end{remark}

It is natural to conjecture that if $X$ is an Ahlfors $d_H$-regular space that satisfies sub-Gaussian  heat kernel estimates and $BE(\kappa)$ with $0<\kappa\leq \frac{d_W}{2}$, under possible additional conditions, we may actually always have 
\[
\alpha^*_p(X) =\left(1-\frac{2}{p}\right)\frac{\kappa}{d_W}+\frac{1}{p}, \quad p \ge 1.
\]
This is will be the object of a further study. We however briefly comment on this problem and its connection to a natural notion of Riesz transform. 

For $p>1$, $\alpha \in (0,1]$ let us say that $X$ satisfies $(R_{p,\alpha})$ if there exists a constant $C=C_{p,\alpha}$ such that 
\[
\| f \|_{p,\alpha} \le C \| (-L)^{\alpha} f \|_{L^p(X,\mu)},
\]
where $f$ is in the proper domain. For instance, in the strictly local framework of Chapter 4, under the strong Bakry-\'Emery curvature condition, $\| f \|^p_{p,1/2} \sim \int_X \Gamma(f)^{p/2} d\mu$ and $(R_{p,1/2})$ is therefore equivalent to boundedness of the Riesz transform in $L^p(X,\mu)$.

One has then the following:

\begin{lemma}
 Assume that for every $1<p \le 2$, $(R_{p,\beta_b})$ is satisfied with $\beta_p=\left(1-\frac{2}{p}\right)\frac{\kappa}{d_W}+\frac{1}{p}$.  Then for every $p \ge 1$,
\[
\alpha^*_p(X) =\left(1-\frac{2}{p}\right)\frac{\kappa}{d_W}+\frac{1}{p}.
\]
\end{lemma}

\begin{proof}
By analyticity of the semigroup, one has then for every $f \in L^p(X,\mu)$, $1 < p\le 2$, 
\[
\| P_t f \|_{p,\beta_p} \le C \| (-L)^{\beta_p} P_tf \|_{L^p(X,\mu)} \le \frac{C}{t^{\beta_p}} \| f \|_{L^p(X,\mu)}.
\]
Therefore,  for $t>0$, $P_t: L^p(X,\mu) \to \bm{B}^{p,\beta_p}(X)$ is bounded. In particular, $\bm{B}^{p,\beta_p}(X)$ is dense in $L^p(X,\mu)$ and we can argue as before. Details are omitted.
\end{proof}

\section{Sets of finite perimeter}

\begin{definition}
Let $X$ be an Ahlfors $d_H$-regular space that satisfies sub-Gaussian  heat kernel estimates and $BE(\kappa)$ with $0<\kappa\leq \frac{d_W}{2}$. A set $E \subset X$ is said to be of finite perimeter if $\mathbf{1}_E \in \mathbf{B}^{1-\frac{\kappa}{d_W}}(X)$.
\end{definition}

An important property of sets of finite perimeter is the following:
\begin{theorem}
There exist constants $c,C>0$ such that for every set $E \subset X$ of finite perimeter,
\begin{align*}
 & c \sup_{r>0} \frac{1}{r^{d_W-\kappa+d_{H}}} (\mu \otimes \mu) \left\{ (x,y) \in E \times E^c\, :\, d(x,y) < r\right\}  \\
 \le &  \| \mathbf{1}_E \|_{1,\alpha} \le C \limsup_{r \to 0^+} \frac{1}{r^{ d_W-\kappa+d_{H}}} (\mu \otimes \mu) \left\{ (x,y) \in E \times E^c\, :\, 
 d(x,y) < r\right\};
 \end{align*}
\end{theorem}

\begin{proof}
This is a consequence of Theorem \ref{Besov characterization} and Theorem \ref{JKLN} applied to $\mathbf{1}_E$.
\end{proof}
\begin{remark}
In the previous theorem, note that
\[
\limsup_{r \to 0} \frac{1}{r^{ d_W-\kappa+d_{H}}} (\mu \otimes \mu) \left\{ (x,y) \in E \times E^c\, :\, d(x,y) \le r\right\} \le C \limsup_{r \to 0} \frac{1}{r^{d_W-\kappa}} \mu ((\partial E)_r),
\]
where, as everywhere in the section, $(\partial E)_r$ denotes the $r$-neighborhood of $\partial E$.
\end{remark}

\section{Sobolev and isoperimetric inequalities}

Combining the results of Part I, Chapter 2,  with the results of this chapter, one immediately obtains the following corollaries. The proofs are similar to the ones in Section  \ref{section Sobolev local}, so will be omitted for concision.

\begin{corollary}
Let $(X,\mu,\mathcal{E},\mathcal{F})$ be a symmetric Dirichlet space with sub-Gaussian heat kernel estimates. Then, one has the following weak type Besov space embedding. Let $0<\delta < d_H $. Let $1 \le p < \frac{d_H}{\delta} $.   There exists a constant $C_{p,\delta} >0$ such that for every $f \in \mathbf{B}^{p,\delta/d_W}(X) $,
\[
\sup_{s \ge 0} s \mu \left( \{ x \in X\, :\, | f(x) | \ge s \} \right)^{\frac{1}{q}} 
\le C_{p,\delta} \sup_{r>0} \frac{1}{r^{\delta+d_{H}/p}}
\biggl(\iint_{\Delta_r}|f(x)-f(y)|^{p}\,d\mu(x)\,d\mu(y)\biggr)^{1/p}
\]
where $q=\frac{p d_H}{ d_H -p \delta}$. Furthermore, for every $0<\delta <d_H $, there exists a constant $C_{\emph{iso},\delta}$ such that for every measurable $E \subset X$, $\mu(E) <+\infty$,
\begin{align}\label{isoperimetric intro2}
\mu(E)^{\frac{d_H-\delta}{d_H}} 
 \le C_{\emph{iso},\delta} \sup_{r>0} \frac{1}{r^{\delta+d_{H}}} (\mu \otimes \mu) \left\{ (x,y) \in E \times E^c\, :\, d(x,y) \le r\right\} 
\end{align}
\end{corollary}

In the previous corollary, when $p=1$ (the most interesting for isoperimetry), one may replace the $\sup$ by the $\limsup$ provided that the weak Bakry-\'Emery curvature condition is satisfied.

\begin{corollary}
Let $(X,\mu,\mathcal{E},\mathcal{F})$ be a symmetric Dirichlet space with sub-Gaussian heat kernel estimates that satisfies the weak Bakry-\'Emery curvature condition  \eqref{WBECD} with a parameter $\kappa$ such that $d_W-d_H<\kappa <d_W $ . Let $\delta=d_W-\kappa $. There exists a constant $C_{1,\delta} >0$ such that for every $f \in \mathbf{B}^{1,\delta/d_W}(X) $,
\[
\sup_{s \ge 0} s \mu \left( \{ x \in X, | f(x) | \ge s \} \right)^{\frac{1}{q}} \le C_{p,\delta} \limsup_{r>0} \frac{1}{r^{\delta+d_{H}}} 
\iint_{\Delta_r}|f(x)-f(y)|\,d\mu(x)\,d\mu(y)
\]
where $q=\frac{p d_H}{ d_H - \delta}$.
Furthermore, there exists a constant $C_{\emph{iso},\delta}$ such that for every measurable $E \subset X$, $\mu(E) <+\infty$,
\begin{align}\label{isoperimetric intro2}
\mu(E)^{\frac{d_H-\delta}{d_H}} \le C_{\emph{iso},\delta} \limsup_{r \to 0} \frac{1}{r^{\delta+d_{H}}} (\mu \otimes \mu) \left\{ (x,y) \in E \times E^c\, :\, d(x,y) \le r\right\}.
\end{align}
\end{corollary}

\begin{remark}
In the limiting case $\kappa=d_W-d_H$ (like the unbounded Sierpinski gasket), one has the following result. There exists a constant $C_{\emph{iso}}>0$ such that for every measurable $E \subset X$, $0<\mu(E) <+\infty$,
\begin{align}\label{isoperimetric intro3}
 \limsup_{r \to 0} \frac{1}{r^{2d_{H}}} (\mu \otimes \mu) \left\{ (x,y) \in E \times E^c\, :\, d(x,y) \le r\right\} \ge C_{\emph{iso}} ,
\end{align}

\end{remark}

\begin{remark}
In the previous theorem, if $E\subseteq X$ has a compact topological boundary $\partial E$
\[
\limsup_{r \to 0} \frac{1}{r^{\delta+d_{H}}} (\mu \otimes \mu) \left\{ (x,y) \in E \times E^c\, :\, 
d(x,y) \le r\right\} \le C \limsup_{r \to 0} \frac{1}{r^{\delta}} \mu ((\partial E)_r),
\]
where $(\partial E)_r$ denotes the $r$-neighborhood of $\partial E$.
\end{remark}

%% file: chapter6.tex
\chapter{Non-local Dirichlet spaces with heat kernel estimates}\label{Ch-non-loc}
Let $(X,d,\mu)$ be a metric measure space. We assume that $B(x,r):=\{y\in X\mid d(x,y)<r\}$ has compact closure for any $x\in X$ and any $r\in(0,\infty)$, and that $\mu$ is Ahlfors $d_H$-regular, i.e.\ there exist $c_{1},c_{2},d_{H}\in(0,\infty)$ such that $c_{1}r^{d_{H}}\leq\mu\bigl(B(x,r)\bigr)\leq c_{2}r^{d_{H}}$ for any $r\in\bigl(0,+\infty\bigr)$. Furthermore, we assume that on $(X,\mu)$ there is a heat kernel $p_t(x,y)$.
If the metric measure space $(X,d,\mu)$ satisfies a certain chain condition and all its metric balls are relatively compact, it is known from~\cite[Theorem 4.1]{GK08} that the sub-Gaussian heat kernel estimates~\eqref{eq:subGauss-upper} will force the associated Dirichlet form to be local. If the Dirichlet form is non-local, the mentioned result tells us that the only possible heat kernel estimates are
\def\Phihke#1{\left(1+#1\right)^{-d_H-d_W}}
\begin{equation}\label{eq:HKE-non-loc}
c_{5}t^{-\frac{d_{H}}{d_{W}}}
\Phihke{ c_{6} \frac{d(x,y)}{t^{1/d_{W}}}}
\le 
p_{t}(x,y)
\le
c_{3}t^{-\frac{d_{H}}{d_{W}}}
\Phihke{ c_{4} \frac{d(x,y)}{t^{1/d_{W}}}},
\end{equation}
with $0<d_W\leq d_H+1$. We refer to~\cite{GK08,GHL:TAMS2003,GHL14,GrigHuHu17,GrigKajino17,GrigLiu15,BeGriPittVess14} for further details and results in this direction. The present chapter is devoted to the study of the spaces $\mathbf{B}^{p,\alpha}(X)$ in this framework of non-local Dirichlet spaces with heat kernel estimates~\eqref{eq:HKE-non-loc}. To compare both situations more easily, we will follow the same structure as the previous chapter about the local case.

\section{Metric characterization of  Besov spaces}\label{S:MCB_nl}
We start by relating our Besov space $\mathbf{B}^{p,\alpha}(X)$ with other Besov type spaces that have been considered in the literature in the non-local context; see e.g.~\cite{Gri,Str03}.

For $\alpha\in [0,\infty)$ and $p\in [1,\infty)$, recall the Besov  seminorm
\[
\| f \|_{p,\alpha}= \sup_{t >0} t^{-\alpha} \left( \int_X \int_X |f(x)-f(y) |^p p_t (x,y) d\mu(x) d\mu(y) \right)^{1/p}
\]
and the Besov space
\[
\mathbf{B}^{p,\alpha}(X)=\{ f \in L^p(X,\mu)\, :\,
 \| f \|_{p,\alpha} <+\infty \}.
\]
 
Recall also the metric Besov type seminorm defined in \eqref{eq:Besov-seminorm-r}. That is, for $f\in L^p(X,\mu)$, 
\[
N^{\alpha}_{p}(f,r):=\frac{1}{r^{\alpha+d_{H}/p}}\biggl(\iint_{\Delta_r}|f(x)-f(y)|^{p}\,d\mu(x)\,d\mu(y)\biggr)^{1/p}.
\]
Here, as in the previous chapters, for $r>0$ the set $\Delta_r$ denotes the collection of all 
$(x,y)\in X\times X$ with $d(x,y)<r$.
For any $q\in [1,\infty)$ such that $p\le q$, define
\[
N_{p,q}^{\alpha}(f):=\brak{\int_0^\infty \brak{N_p^{\alpha}(f,r)}^{q}\, \frac{dr}{r}}^{1/q},
\]
and for $q=\infty$, define
\[
N_{p,\infty}^{\alpha}(f):=\sup_{r>0} N_p^{\alpha}(f,r).
\]
The Besov space is then defined by 
\[
\mathfrak{B}^{\alpha}_{p,q}(X)=\{u\in L^p(X)\, :\, N_{p,q}^{\alpha}(f)<\infty\}.
\]
We determine first which of the these spaces coincides with $\bm{B}^{p,\alpha}(X)$ when $p\in[1,\infty)$ and $\alpha\in[0,1/p)$.

\begin{theorem}\label{Besov non-local}
Let $p \ge 1$ and $ 0\le \alpha < 1/p$. We have $\mathfrak{B}^{\alpha d_W}_{p,\infty}(X)=\mathbf{B}^{p,\alpha}(X)$ and there exist two constants $c,C>0$ (depending on $p,\alpha, d_H, d_W$) such that 
\[
c\,N_{p,\infty}^{\alpha d_W}(f) \le \| f \|_{p,\alpha} \le C\,N_{p,\infty}^{\alpha d_W}(f).\footnote{Here and after, the letters $c,C$ denote constants independent of important parameters which may change at each circumstance.}
\]
\end{theorem}

\begin{proof}
We follow the proof for Theorem \ref{Besov characterization} by  first proving the lower bound. By the upper bound in \eqref{eq:HKE-non-loc}, one has for $r,t>0$
\begin{align*}
& \int_X \int_X |f(x)-f(y) |^p p_t (x,y) d\mu(x) d\mu(y)  
\\ \ge & 
\frac{c}{t^{d_H/d_W}}\int_X \int_{B(y,r)}\frac{|f(x)-f(y) |^p}{\brak{1+\frac{d(x,y)}{t^{1/d_W}}}^{d_H+d_W}} d\mu(x) d\mu(y) 
\\ \ge & 
\frac{c}{t^{d_H/d_W}}\frac{1}{\brak{1+\frac{r}{t^{1/d_W}}}^{d_H+d_W}}  \int_X \int_{B(y,r)} |f(x)-f(y) |^p d\mu(x) d\mu(y) 
\\ \ge & 
\frac{c}{t^{d_H/d_W}} \frac{ r^{\alpha d_W p+d_{H}} }{\brak{1+\frac{r}{t^{1/d_W}}}^{d_H+d_W}} N^{\alpha d_W}_{p}(f,r)^p 
\end{align*}
Choosing $r=t^{1/d_W}$, it easily follows that 
\[
\| f \|_{p,\alpha} \ge c \sup_{t>0} N^{\alpha d_W}_{p}(f,t^{1/d_W})=c\,N_{p,\infty}^{\alpha d_W}(f).
\]
We now turn to the upper bound. We set $A(t)$ and $B(t)$ as in~\eqref{E:A(t)}, respectively~\eqref{E:B(t)}, 
so that $ \int_X \int_X |f(x)-f(y) |^p p_t (x,y) d\mu(x) d\mu(y) =A(t)+B(t)$.
By \eqref{eq:HKE-non-loc} and the inequality
$|f(x)-f(y)|^{p}\leq 2^{p-1}(|f(x)|^{p}+|f(y)|^{p})$,
\begin{align}
A(t)\notag
&\leq
\frac{c_{3}}{t^{d_{H}/d_{W}}}\int_{X}\int_{X\setminus B(y,r)} \biggl(1+c_{4}\frac{d(x,y)}{t^{1/d_{W}}}\biggr)^{-d_H-d_{W}} \cdot 2^{p}|f(y)|^{p}\,d\mu(x)\,d\mu(y)
\notag\\ &=
\frac{2^{p}c_{3}}{t^{d_{H}/d_{W}}}\sum_{k=1}^{\infty}\int_{X}\int_{B(y,2^{k}r)\setminus B(y,2^{k-1}r)}\biggl(1+c_{4}\frac{d(x,y)}{t^{1/d_{W}}}\biggr)^{-d_H-d_{W}}|f(y)|^{p}\,d\mu(x)\,d\mu(y)\notag\\
&\leq
\frac{2^{p}c_{3}}{t^{d_{H}/d_{W}}}\sum_{k=1}^{\infty}\int_{X}\mu\bigl(B(y,2^{k}r)\bigr)\biggl(1+c_{4}\frac{2^{(k-1)}r}{t^{1/d_{W}}}\biggr)^{-d_H-d_{W}}|f(y)|^{p}\,d\mu(y)
\notag\\&\leq
\frac{2^{p}c_{3}}{t^{d_{H}/d_{W}}}\sum_{k=1}^{\infty}c_{2}r^{d_{H}}2^{kd_{H}}\|f\|_{L^{p}}^{p}\biggl(1+c_{4}\frac{2^{(k-1)}r}{t^{1/d_{W}}}\biggr)^{-d_H-d_{W}}
\notag\\ &=
C \|f\|_{L^{p}}^{p} \sum_{k=1}^{\infty} \Bigl(\frac{2^{k-1}r}{t^{1/d_W}}\Bigr)^{d_{H}}\biggl(1+\frac{c_{4}}{2}\frac{2^{k}r}{t^{1/d_{W}}}\biggr)^{-d_H-d_{W}}
\notag\\ &\leq
C\|f\|_{L^{p}}^{p} \sum_{k=1}^{\infty}\int_{\frac{2^{k-1} r}{t^{1/d_{W}}}}^{\frac{2^k r}{t^{1/d_{W}}}} s^{d_{H}}\biggl(1+\frac{c_{4}}{2}s\biggr)^{-d_H-d_{W}}\frac{ds}{s}
\notag\\ &\le 
C\|f\|_{L^{p}}^{p} \frac{t^{\alpha p}}{r^{\alpha d_W p}} \int_{0}^{\infty}s^{\alpha d_W p+d_{H}-1} \biggl(1+\frac{c_{4}}{2}s\biggr)^{-d_H-d_{W}}\,ds
\notag\\ &\leq 
C \frac{t^{\alpha p}}{r^{\alpha d_W p}} \|f\|_{L^{p}}^{p},
\label{eq:nonLocalHKBesov-norms-upper-proof1}
\end{align}
where the last integral is bounded as long as $\alpha \in [0,1/p)$.

On the other hand, for $B(t)$, by \eqref{eq:HKE-non-loc} we have
\begin{align}
B(t)\notag
&\leq
\frac{c_{3}}{t^{d_{H}/d_{W}}}\int_{X}\int_{B(y,r)} \biggl(1+c_{4}\frac{d(x,y)}{t^{1/d_{W}}}\biggr)^{-d_H-d_{W}} |f(x)-f(y)|^{p}\,d\mu(x)\,d\mu(y)\notag
\\&\leq
\frac{c_{3}}{t^{d_{H}/d_{W}}}\sum_{k=1}^{\infty}\int_{X}\int_{B(y,2^{1-k}r)\setminus B(y,2^{-k}r)} \biggl(1+c_{4}\frac{d(x,y)}{t^{1/d_{W}}}\biggr)^{-d_H-d_{W}} |f(x)-f(y)|^{p}\,d\mu(x)\,d\mu(y)\notag
\\&\leq
\frac{c_{3}}{t^{d_{H}/d_{W}}}\sum_{k=1}^{\infty}\int_{X}\int_{B(y,2^{1-k}r)} \biggl(1+c_{4}\frac{2^{-k}r}{t^{1/d_{W}}}\biggr)^{-d_H-d_{W}} |f(x)-f(y)|^{p}\,d\mu(x)\,d\mu(y)\notag
\\&\leq 
c_{3} \sum_{k=1}^{\infty}\frac{(2^{1-k}r)^{\alpha d_W p+d_{H}}}{t^{d_{H}/d_{W}}} \biggl(1+c_{4}\frac{2^{-k}r}{t^{1/d_{W}}}\biggr)^{-d_H-d_{W}} \frac{1}{(2^{1-k}r)^{\alpha d_W p+d_{H}}}\int_{X}\int_{B(y,2^{1-k}r)}|f(x)-f(y)|^{p}\,d\mu(x)\,d\mu(y)\notag
\\&\leq 
C t^{\alpha p}\sup_{s\in(0,r]}N^{\alpha d_W}_{p}(f,s)^{p}\sum_{k=1}^{\infty}\Bigl(\frac{2^{-k}r}{t^{1/d_W}}\Bigr)^{d_H+\alpha d_W p} \biggl(1+c_{4}\frac{2^{-k}r}{t^{1/d_{W}}}\biggr)^{-d_H-d_{W}} \notag
\\&\leq 
C t^{\alpha p}\sup_{s\in(0,r]}N^{\alpha d_W}_{p}(f,s)^{p},
\label{eq:nonLocalHKBesov-norms-upper-proof2}
\end{align}
where the last series converges for all $\alpha\in [0,1/p)$.

Combining \eqref{eq:nonLocalHKBesov-norms-upper-proof1} and \eqref{eq:nonLocalHKBesov-norms-upper-proof2}, we conclude that by taking $r\to \infty$,
\[
\sup_{t>0} \frac{1}{t^{\alpha p}} \int_X \int_X |f(x)-f(y) |^p p_t (x,y) d\mu(x) d\mu(y) \le C \sup_{s>0}N^{\alpha d_W}_{p}(f,s)^{p},
\]
and hence $\| f \|_{p,\alpha}\le CN_{p,\infty}^{\alpha d_W}(f)$.
The proof is thus complete.
\end{proof}

We can also establish several inclusions that cover the case when $\alpha\in[1/p,\infty)$.

\begin{proposition}\label{P:Inclusions_Besov_nl}
Let $p\in[ 1,\infty)$.  We have
\begin{enumerate}
\item if $\alpha \in [0,1/p]$, then $\mathfrak{B}^{\alpha d_W}_{p,p}(X) \subset \mathbf{B}^{p,\alpha }(X)$.
\item  if $\alpha \in [1/p, \infty)$, then $\mathbf{B}^{p,\alpha}(X) \subset \mathfrak{B}^{\alpha d_W}_{p,p}(X)$.
\end{enumerate}
In particular, $\mathbf{B}^{p,1/p}(X)=\mathfrak{B}^{d_W/p}_{p,p}(X)$.
\end{proposition}
\begin{proof}
We first show $(1)$. Fixing a decreasing geometric sequence $\{r_k\}_{k\in \Z}$, we can write
\[
\sum_{k\in \Z}  N_p^{\alpha}(f,r_k)^p \simeq \int_0^{\infty} N_p^{\alpha}(f,r)^p \frac{dr}{r}.
\]
Using the upper bound in \eqref{eq:HKE-non-loc}, for any $\alpha \in [0,1/p]$ we have that
\begin{align*}
&\int_X \int_X |f(x)-f(y) |^p p_t (x,y) d\mu(x) d\mu(y) 
\\ &= 
\sum_{k\in \Z} \int_X \int_{B(x,r_k)\setminus B(x,r_{k+1})} |f(x)-f(y) |^p p_t (x,y) d\mu(x) d\mu(y) 
\\ &\le
\sum_{k\in \Z}  c_{3}t^{-\frac{d_{H}}{ d_{W}}}
\brak{1+c_4\frac{r_{k+1}}{t^{1/d_W}}}^{-d_H-d_W}\int_X \int_{B(x,r_k)\setminus B(x,r_{k+1})} |f(x)-f(y) |^p d\mu(x) d\mu(y) 
 \\ &=
\sum_{k\in \Z}  c_{3} t\brak{t^{1/d_W}+c_4 r_{k+1}}^{-d_H-d_W}\int_X \int_{B(x,r_k)\setminus B(x,r_{k+1})} |f(x)-f(y) |^p d\mu(x) d\mu(y) 
 \\ &\le
 \sum_{k\in \Z}  c_{3} t \brak{t^{1/d_W}}^{-d_W+\alpha p d_W}\brak{c_4 r_{k+1}}^{-d_H-\alpha p d_W}\int_X \int_{B(x,r_k)\setminus B(x,r_{k+1})} |f(x)-f(y) |^p d\mu(x) d\mu(y) 
 \\ &\le
C\sum_{k\in \Z}  t^{\alpha p} N_p^{\alpha d_W}(f,r_{k+1})^p
\le
C t^{\alpha p} N_{p,p}^{\alpha d_W}(f)^p.
\end{align*}
This implies 
\[
\| f \|_{p,\alpha } \le C N_{p,p}^{\alpha d_W}(f),
\]
and hence $\mathfrak{B}^{\alpha d_W}_{p,p}(X) \subset \mathbf{B}^{p,\alpha }(X)$.

Next we prove $(2)$. Notice that for any $\delta\in (0,1)$ and $d_W'=d_W/\delta$, there exist two constants $c,C$ such that
\[
p_t(x,y) \ge 
c \int_{t^{1/\delta}}^{\infty} \frac{t}{s^{1+\delta}} s^{-d_H/d_W'} \exp\brak{-C\brak{\frac{d(x,y)^{d_W'}}{s}}^{1/(d_W'-1)}}ds.
\]
See for instance \cite[Section 5.4]{Gri}. Following the proof in Theorem~\ref{Besov characterization}, we obtain for $\alpha\ge 1/p $,
\begin{align*}
&  \int_X \int_X |f(x)-f(y) |^p p_s (x,y) d\mu(x) d\mu(y)  
\\ &  \ge  
c \int_{t^{1/\delta}}^{\infty} \frac{t}{s^{1+\delta}} s^{-d_H/d_W'}  \int_X \int_{B\brak{y,s^{1/d_W'}}} |f(x)-f(y) |^p d\mu(x) d\mu(y)  ds
\\ &  \ge  
c \int_{t^{1/\delta d_W'}}^{\infty} \frac{t}{r^{\delta d_W'-\alpha d_W p}} N_p^{\alpha d_W}(f,r)^p \frac{dr}{r}
\\&  \ge  
c t^{\alpha p} \int_{t^{1/d_W}}^{\infty} N_p^{\alpha d_W}(f,r)^p \frac{dr}{r},
\end{align*}
where we get the second inequality by changing the variable. It follows that
\[
\| f \|_{p,\alpha}\ge c\, \sup_{t>0} \brak{\int_{t^{1/d_W}}^{\infty} N_p^{\alpha d_W}(f,r)^p \frac{dr}{r}}^{1/p} = c\,N_{p,p}^{\alpha d_W}(f),
\]
and we conclude $(2)$.
\end{proof}

We finish this section by giving an alternative, equivalent expression for the seminorm $\|f\|_{1,1}$ that will turn useful later on.

\begin{lemma}\label{lem:BesovEquivSN}
Let $f\in \bfB^{1,1}(X)$. Then the seminorm $\|f\|_{1,1}$ is equivalent to 
\[
W_{1,1}(f):=\int_X\int_X \frac{|f(x)-f(y)|}{d(x,y)^{d_H+d_W}} d\mu(x) d\mu(y).
\]
\end{lemma}
\begin{proof}
We first show that $\|f\|_{1,1} \le CW_{1,1}(f) $, which follows directly from the heat kernel estimate \eqref{eq:HKE-non-loc}. Indeed, for every $t>0$,
\begin{align*}
&\frac1t \int_X\int_X |f(x)-f(y)|p_t(x,y) d\mu(x) d\mu(y)
\\ &\le 
\frac{c_3}{t} \int_X\int_X t^{-\frac{d_H}{d_W}} \brak{1+c_4\frac{d(x,y)}{t^{1/d_W}}}^{-d_H-d_W} |f(x)-f(y)|d\mu(x) d\mu(y)
\\ &=
\frac{c_3}{t}\int_X\int_X \frac{t}{\brak{t^{1/d_W}+c_4d(x,y)}^{d_H+d_W}} |f(x)-f(y)|d\mu(x) d\mu(y)
\le CW_{1,1}(f).
\end{align*}
Next we show the reverse inequality. For every $t>0$, 
\begin{align*}
&\int_X\int_X \frac{c_5}{\brak{t^{1/d_W}+c_6d(x,y)}^{d_H+d_W}} |f(x)-f(y)|d\mu(x) d\mu(y)
\\ &\le 
\frac{1}{t}\int_X\int_X  |f(x)-f(y)|p_t(x,y) d\mu(x) d\mu(y) \le \|f\|_{1,1}.
\end{align*}
By Lebesgue's dominated convergence theorem, we have $W_{1,1}(f) \le C \|f\|_{1,1}$. This completes the proof.
\end{proof}

\section{Coarea type estimates}\label{S:coarea_nl}
From Theorem~\ref{Besov non-local} we can derive as in the local case several estimates involving the fractional perimeter in the non-local setting, which in this case include the case when $d_W<1$. 
As in Chapter~\ref{LDsGHKU}, given a function $u\in L^1(X,\mu)$ we denote $E_s(u)=\{x\in X,\;u(x)>s\}$, for $E\subset K$ denote by $(E)_r$ the $r$-neighborhood of $E$ and by $\partial E$ its topological boundary.

\begin{theorem}\label{T:non-local_coarea}
Let $X$ be Ahlfors $d_H$-regular with a heat kernel that satisfies the upper estimate in~\eqref{eq:HKE-non-loc}. For $u\in L^1(X,\mu)$, assume that there is $0<\alpha\leq \min\{d_H/d_W,1\}$ and $R>0$ such that 
\begin{equation}\label{extended-inner-boundary-estimate-non-loc}
h_R(u,s) = \sup_{r \in(0,R]} \frac1{r^{\alpha d_W}} \mu \bigl\{ x\in E_s(u): d(x, X\setminus E_s(u)) < r  \bigr\}
\end{equation}
is in $L^1 (\mathbb{R},ds)$. Then, $u \in \mathbf{B}^{1,\alpha}(X)$ and there exist constants $C_1,C_2>0$ independent of $u$ such that
\begin{equation}\label{E:non-local_coarea}
\| u \|_{1, \alpha} \le C_1 R^{-\alpha}\|u\|_{L^1(X,\mu)}+  C_2 \int_{\mathbb{R}} h(u,s) ds.
\end{equation}
In particular, for any set $E\subset X$ we have
\begin{equation}\label{E:non-local_coarea_1E}
\| \bm{1}_E \|_{1, \alpha} \le C_1 R^{-\alpha}\mu(E)+  C_2 \sup_{r \in(0,R]} \frac1{r^{\alpha d_W}} \mu \bigl\{ x\in E\,\colon\,d(x, X\setminus E) < r  \bigr\}.
\end{equation}
\end{theorem}
\begin{proof}
This is analogous to the local case and we only give a sketch of the proof. Applying Theorem~\ref{T:non-local_coarea} with $f=\bm{1}_{E_s(u)}$ and $p=1$ we have that, for any $R>0$ and $s\in\mathbb{R}$,
\begin{multline}\label{E:non-local-coarea01}
t^{-\alpha}\int_{X\times X}|\bm{1}_{E_s(u)}(y)-\bm{1}_{E_s(u)}(x)|\,p_t(x,y)\,d\mu(y)d\mu(x)\\
\leq C_1 R^{-\alpha d_W}\|\bm{1}_{E_s(u)}\|_{L^1(X,\mu)}+C_2\sup_{r\in(0,R]}N_1^{\alpha d_W}(\bm{1}_{E_s(u)},r).
\end{multline}
Moreover, in view of definition~\eqref{eq:Besov-seminorm-r} and applying the Ahlfors $d_H$-regularity of the space $X$ we get
\begin{align*}
N_1^{\alpha d_W}(\bm{1}_{E_s(u)},r)&=\frac{1}{r^{\alpha d_W+d_H}}\mu\otimes\mu\big(\{(x,y)\in E_s(u)\times (X\setminus E_s(u))\,\colon\,d(x,y)<r\}\big)\\
&\leq \frac{1}{r^{\alpha d_W+d_H}}\int_{\{x\in E_s(u)\,\colon\,d(x,X\setminus E_s(u)\}}\mu\big(B(y,r)\big)\,d\mu(x)\\
&\leq \frac{1}{r^{\alpha p}}\mu\big(\{x\in E_s(u)\,\colon\,d(x,X\setminus E_s(u)\}\big)=h_R(u,s).
\end{align*}
Plugging this into~\eqref{E:non-local-coarea01}, integrating both sides of over $\mathbb{R}$ and noting that $\int_{\mathbb{R}}\|\bm{1}_{E_s(u)}\|_{L^1(X,\mu)}=\|u\|_{L^1(X,\mu)}$ finally yields~\eqref{E:non-local_coarea}. The last assertion follows from the fact that $E_s(\bm{1}_E)=E$ if $s\in [0,1)$ and zero otherwise.
\end{proof}
\begin{remark}
The previous theorem could be stated with $0<\alpha<1$. However, if $\alpha>d_H/d_W$, there will be no function $h_R(u,s)$ satisfying the required conditions. Analogous reasons explain this restriction for $\alpha$ in all subsequent statements.
\end{remark}
We can now derive the same corollaries as in the local case, since these do not involve heat kernel estimates. We restate them for completeness and refer for proof to the previous chapter.
\begin{corollary}\label{C:non-local-ahlfors-coarea-1E-ulc}
Let $X$ be a uniformly locally connected space with a heat kernel that satisfies the upper heat kernel estimate~\eqref{eq:HKE-non-loc}. If $E \subset X$ has finite measure and for some $0<\alpha\leq \min\{\frac{d_H}{d_W},1\}$
\begin{equation*}
\sup_{r\in(0,R]} \frac1{r^{\alpha d_W}} \mu\bigl\{ x\in E: d(x,\partial E) <r\bigr\}<\infty,
\end{equation*}
then $\bm{1}_E \in \mathbf{B}^{1,\alpha}(X) $ and
\begin{equation*}
\| \bm{1}_E \|_{1, \alpha} \le C_1 R^{-\alpha d_W} \mu(E) + C_2  \sup_{r\in(0,R]} \frac1{r^{\alpha d_W}} \mu\bigl\{ x\in E: d(x,\partial E) <r\bigr\}.
\end{equation*}
\end{corollary}

\begin{corollary}\label{C:non-local-JKLM}
Let $X$ be a uniformly locally connected space with a heat kernel that satisfies the upper estimate~\eqref{eq:HKE-non-loc} and let $u\in L^1(X,\mu)$. Assume that $\partial E_s(u)$ is finite and $s\mapsto  | \partial E_s(u) |$ is in $L^1(\mathbb{R})$, where $| \partial E_s(u) |$ denotes the cardinality of $\partial E_s(u)$. Then, if $d_H\leq d_W$, we have $u \in \mathbf{B}^{1,\frac{d_H}{d_W}}(X)$ and there exists a constant $C>0$ independent of $u$ such that
\begin{equation*}\label{E:non-local-JKLM}
 \| u \|_{1, d_H/d_W} \le C \int_{\mathbb{R}} | \partial E_s(u) | ds.
\end{equation*}
In particular, if $E \subset X$ is a set of finite measure whose boundary is finite, then $\bm{1}_E \in \mathbf{B}^{1,\frac{d_H}{d_W}}(X)$ and 
\[
\| \bm{1}_E \|_{1, d_H/d_W} \le C  | \partial E |.
\]
\end{corollary}
In the same fashion, the analogous of Proposition~\ref{P:BesovLB} will hold when the level sets $E_s(u)$ are uniformly porous at large scales.
\section{Sets of finite perimeter and fractional content of boundaries}\label{S:FP_nl}

\subsection{Characterizations of the sets of finite perimeter}
In the previous sections we have provided the non-local counterpart of all ingredients that lead to the characterization result provided by Theorem~\ref{T:FP_char} in the local case. With the corresponding restriction on the parameter $\alpha$, the same result will hold in this setting. 

\begin{theorem}\label{T:FP_char_non-local}
Under the assumptions of Theorem~\ref{Besov non-local}, for a bounded measurable set $E \subset X$  and $0<\alpha \leq\min\{1,d_H/d_W\}$
 we consider the following properties:
\begin{enumerate}
\item\label{FP_char1_nl} $\bm{1}_E \in \mathbf{B}^{1,\alpha}(X)$;
\item\label{FP_char2_nl} $\displaystyle \sup_{r>0} \frac{1}{r^{\alpha d_W+d_{H}}} (\mu \otimes \mu) \big(\{ (x,y) \in E \times E^c\, :\,
 d(x,y) < r\}\big) <+\infty$;
\item\label{FP_char3_nl} $\displaystyle \limsup_{r \to 0^+} \frac{1}{r^{\alpha d_W+d_{H}}} (\mu \otimes \mu) \big(\{ (x,y) \in E \times E^c\,
:\, d(x,y) < r\}\big) <+\infty$;
\item\label{FP_char4_nl} $\displaystyle \limsup_{r\to 0^+}\frac{1}{r^{\alpha d_W}}\mu\big(\{x\in E~\colon~d(x,E^c)<r\}\big)<+\infty$;
\item\label{FP_char5_nl} $\displaystyle \limsup_{r\to 0^+}\int_{\{x\in E\colon d(x,E^c)<r\}}\int_{B(x,r)\cap E^c}\frac{1}{d(x,y)^{d_H+\alpha d_W}}d\mu(y)d\mu(x)<+\infty$;
\item\label{FP_char6_nl} $\displaystyle \int_E \int_{E^c} \frac{1}{d(x,y)^{d_H+\alpha d_W}}d\mu(y)d\mu(x) <+\infty$.
\end{enumerate}
Then, the following relations hold:
\begin{enumerate}[label=(\roman*)]
\item \eqref{FP_char1_nl}~$\Leftrightarrow$~\eqref{FP_char2_nl}~$\Leftrightarrow$~\eqref{FP_char3_nl}; 
moreover there exist constants $c,C >0$ independent of $E$ such that
\begin{multline*}
c \sup_{r>0} \frac{1}{r^{\alpha d_W+d_{H}}} (\mu \otimes \mu) \big(\{ (x,y) \in E \times E^c\,
:\,
 d(x,y) < r\} \big) \\
 \le  \| \bm{1}_E \|_{1,\alpha} \le C \sup_{r>0} \frac{1}{r^{\alpha d_W+d_{H}}} (\mu \otimes \mu) \big(\{ (x,y) \in E \times E^c\,
 :\, d(x,y) < r\}\big);
\end{multline*}
\item \eqref{FP_char4_nl}~$\Rightarrow$~\eqref{FP_char5_nl}~$\Rightarrow$~\eqref{FP_char1_nl};
\item \eqref{FP_char6_nl}~$\Rightarrow$~\eqref{FP_char1_nl};
\item If in addition $E$ is porous, c.f.~\eqref{D:porous}, then \eqref{FP_char1_nl}~$\Rightarrow$~\eqref{FP_char4_nl}; 
\end{enumerate}
\end{theorem}
The proofs follow verbatim, replacing Theorem~\ref{Besov characterization} by Theorem~\ref{Besov non-local} and Theorem~\ref{ahlfors-coarea} by Theorem~\ref{T:non-local_coarea}. Notice that Proposition~\ref{P:local_geom_cond} is does not actually involve any heat kernel estimates. Assuming uniformly locally connectedness on the space, which includes the case of $X$ being geodesic, we obtain as in Corollary~\ref{C:ahlfors-coarea-1E-ulc} the corresponding characterization of sets with finite $\alpha$-perimeter in terms of inner Minkowski contents.

\subsection{Further results}
\begin{theorem}\label{thm-HKE-E-Ec1-non-loc}
Assume that $p_t(x,y)$ satisfies the upper heat kernel estimate~\eqref{eq:HKE-non-loc} and $d_H<d_W$. If there exists $0<\alpha<d_H/d_W$ such that for all $\varepsilon>0$,
\begin{equation}\label{E:nbd_estimateU-non-loc}
\mu\otimes\mu\{(x,y)\in E^c\times E\, :\, d(x,y)<\varepsilon\}\leqslant c_7\left(\varepsilon^{d_H+\alpha d_W}+\varepsilon^{2d_H }\right),
\end{equation}
then for all $t>0$ it holds that 
\begin{equation}\label{eq:PtE-upper-non-loc}
\left(t^{-\alpha}+t^{-{{d_H}}/d_{W}}\right)\| P_t \mathbf{1}_E -\mathbf{1}_E\|_{L^1(X,\mu)} \leqslant c_{11}<\infty,
\end{equation}
where $c_{11}$ is a constant given by~\eqref{eq:NLPtE-upperc11} that depends on $\delta,d_H$ and $d_W$.
\end{theorem}
\begin{proof}
In view of~\eqref{eq:HKE-non-loc}, $p_t(x,y)>s$ implies
\begin{equation}\label{E:def_F(t,s)}
d(x,y)<t^{1/d_W}c_4^{-1}\Big[\Big(c_3^{-1}t^{\frac{d_H}{d_W}}s\Big)^{-\frac{1}{d_W+d_H}}-1\Big]=:F(t,s).
\end{equation}
Moreover, $F(t,s)>0$ if and only if 
$s<c_3t^{-\frac{d_H}{d_W}}$ and 
the corresponding first term in the expression~\eqref{e-u-dd'} can now be bounded by
\begin{multline*}
\int_0^{c_3t^{-d_H/d_W}} c_7c_4^{-d_H+\alpha d_W}t^{\frac{d_H+\alpha d_W}{d_W}}\Big[\Big(c_3^{-1}t^{\frac{d_H}{d_W}}s\Big)^{-\frac{1}{d_W+d_H}}-1\Big]^{d_H+\alpha d_W}ds\\
=	c_7c_3c_4^{-(d_H+\alpha d_W)}t^\alpha\int_0^1\Big(u^{-\frac{1}{d_W+d_H}}-1\Big)^{d_H+\alpha d_W}du
\end{multline*}
which is finite since by assumption $0<\delta<d_H<d_W$. Thus,
\begin{equation*}
c_7\int_0^{c_3t^{-d_H/d_W}} (F(t,s))^{d_H+\alpha d_W}ds\leq c_7c_3c_4^{-(d_H+\alpha d_W)}t^{\alpha}\frac{d_W+d_H}{d_W(1-\alpha)}
\end{equation*}
and, since $d_H<d_W$, the same estimate for the second term in~\eqref{e-u-dd'} with $\alpha=d_H/d_W$ yields
\begin{equation*}
c_7\int_0^{c_3t^{-d_H/d_W}} (F(t,s))^{d_H+d_H}ds\leq c_7c_3c_4^{-2d_H}t^{\frac{d_H}{d_W}}\frac{d_W+d_H}{d_W-d_H}
\end{equation*}
so that~\eqref{eq:PtE-upper-non-loc} holds with 
\begin{equation}\label{eq:NLPtE-upperc11}
c_{11}=c_7c_3c_4^{-d_H}\max\Big\{c_4^{\alpha d_W}\frac{d_W+d_H}{d_W(1-\alpha)},c_4^{-d_H}\frac{d_W+d_H}{d_W-d_H}\Big\}.
\end{equation}
\end{proof}

\begin{theorem}\label{thm-HKE-E-Ec2-non-loc}
Assume that $p_t(x,y)$ satisfies the lower heat kernel estimate in~\eqref{eq:HKE-non-loc}. If there exist $\alpha>0$ and $\varepsilon>0$ such that
\begin{equation}\label{E:nbd_estimateL-non-loc}
\mu\otimes\mu\{(x,y)\in E^c\times E
\, :\,
d(x,y)<\varepsilon\}\geq c_8\varepsilon^{d_H+\alpha d_W},
\end{equation}
then, for $t=\varepsilon^{d_W}$, it holds that 
\begin{equation}\label{eq:PtE-lower-non-loc}
t^{-\alpha}\| P_t \mathbf{1}_E -\mathbf{1}_E \|_{L^1(X,\mu)} \geqslant 2 c_{5}c_8e^{-c_{6}} >0.
\end{equation}
\end{theorem}
\begin{corollary}
If there exists $\alpha>0$ and $\varepsilon_0>0$ such that~\eqref{E:nbd_estimateL-non-loc} holds for any $0<\varepsilon<\varepsilon_0$, then~\eqref{eq:PtE-lower-non-loc} holds for any $0<t<\varepsilon^{d_W}$.
\end{corollary}

\begin{proof}[Proof of Theorem~\ref{thm-HKE-E-Ec2-non-loc}.] 
The lower estimate in~\eqref{eq:HKE-non-loc} implies that
\begin{equation*}
p_t(x,y)\geq c_5t^{-\frac{d_H}{d_W}}c_6^{-(d_H+d_W)}
\end{equation*}
and hence assumption~\eqref{E:nbd_estimateL-non-loc} yields
\begin{align*}
\|\bm{1}_{E^c} P_t\bm{1}_{E}\|_{L^1(X,\mu)} &\geq \int_{A_{E,t^{1/d_W}}}p_t(x,y)\,\mu(dx)\,\mu(dy) \geq c_5 c_6^{-(d_H+d_W)} t^{-\frac{d_H}{d_W}}\mu\otimes\mu (A_{E,t^{1/d_W}})\\
& \geq c_5c_6^{-(d_H+d_W)}t^{-\frac{d_H}{d_W}}c_8t^{\frac{(d_H+\alpha d_W)}{d_W}}=c_5c_8c_6^{-(d_H+d_W)}t^{\alpha}.
\end{align*} 
\end{proof}

Under weaker assumptions, the following localized version of Theorem~\ref{thm-HKE-E-Ec1-non-loc} is available.
\begin{theorem}\label{thm-HKE-E-Ec1-non-loc-ball}
Let $E\subset B\subset X$ and assume that $p_t(x,y)$ satisfies the upper heat kernel estimate~\eqref{eq:HKE-non-loc}. If there exists $0<\alpha<1$ such that for all $\varepsilon>0$,
\begin{equation}\label{E:nbd_estimateU-non-loc-ball}
\mu\otimes\mu\{(x,y)\in (B\cap E^c)\times E\, :\, d(x,y)<\varepsilon\}\leqslant c_7\varepsilon^{d_H+\alpha d_W},
\end{equation}
then for all $t>0$ it holds that 
\begin{equation}\label{eq:PtE-upper-non-loc-ball}
t^{-\alpha}\| P_t \mathbf{1}_E -\mathbf{1}_E\|_{L^1(B,\mu)} \leqslant c_{11}<\infty,
\end{equation}
where $c_{11}$ is a constant given by~\eqref{eq:NLPtE-upperc11-ball} that depends on $\alpha,d_H$ and $d_W$.
\end{theorem}
\begin{proof}
Following the proof of Theorem~\ref{thm-HKE-E-Ec1-ball} verbatim we have
\begin{equation*}
\|  \mathbf{1}_{E^c} P_t \mathbf{1}_E \|_{L^1(B,\mu)}\leq\int_0^{c_3t^{-d_H/d_W}}c_7F(t,s)^{d_H+\alpha d_W}ds.
\end{equation*}
Since $0<\alpha<1$ by assumption, the same computations as in the proof of Theorem~\ref{thm-HKE-E-Ec1-non-loc} imply that~\eqref{eq:PtE-upper-non-loc-ball} holds with 
\begin{equation}\label{eq:NLPtE-upperc11-ball}
c_{11}=c_7c_3c_4^{-(d_H+\alpha d_W)}\frac{d_W+d_H}{d_W(1-\alpha)}.
\end{equation}
\end{proof}
\section{Critical exponent for $\mathbf{B}^{1,\alpha}(X)$}\label{S:criticalE_nl}
The observation that the space $\bfB^{1,\alpha}(X)$ is trivial for any $\alpha>1$ is the starting point of the discussion that will be carried out in this section, where we compute
\begin{equation}\label{E:critical_a_non-local}
\alpha^*=\inf \{ \alpha >0\,\colon\, \mathbf{B}^{1,\alpha}(X) \text{ is trivial} \}.
\end{equation}
\begin{proposition}\label{prop:NLTrivialB}
If $f\in \bfB^{1,\alpha}(X)$ with $\alpha>1$, then $f=0$.
\end{proposition}
\begin{proof}
Notice first that for any $f\in\bfB^{1,\alpha}(X)$, 
\[
 \|P_t(|f-f(\cdot)|)\|_{L^1(X,\mu)} \le t^{\alpha} \|f\|_{1,\al}.
\]
Now assume that $f$ is bounded (in fact, one can consider $f_N$ defined in the proof of Proposition \ref{prop:BesovCEp}). Then,
\begin{align*}
\|P_t(|f-f(\cdot)|)\|_{L^1(X,\mu)} &\le \int_X P_t(|f-f(x)|^2)(x) d\mu(x)\\
& \le 2\|f\|_{L^{\infty}(X)}  \|P_t(|f-f(\cdot)|)\|_{L^1(X,\mu)} \le 2t^{\alpha}\|f\|_{L^{\infty}(X)}  \|f\|_{1,\al}.
\end{align*}
On the one hand, since $\al>1$, we have
\[
\lim_{t\to 0}\frac1t \int_X P_t(|f-f(x)|^2)(x) d\mu(x)= 0.
\]
On the other hand, as in Lemma \ref{lem:energyasbesov} and Proposition \ref{prop:energyasbesov} we know that
\[
\frac1t\int_X P_t(|f-f(x)|^2)(x) d\mu(x)=2\Ecal_t(f),
\]
where $\Ecal_t(f)$ is decreasing. Hence, $\|P_t(|f-f(\cdot)|)\|_{L^1(X,\mu)}=0$ for every $t>0$ and therefore $P_t f(x)=f(x)$ for every $t>0$. The upper bound in \eqref{eq:HKE-non-loc} yields 
\[
|f(x)|=|P_t f(x)| \le \int_X p_t(x,y) |f(y)| d\mu(y) \le \frac{C}{t^{d_H/d_W}} \N{f}_{L^1(X,\mu)}
\]
and letting $t\to\infty$, we obtain that $f=0$.
\end{proof}

In order to compute the critical parameter $\alpha^*$, we shall make the following assumption on the underlying metric measure space $(X,d, \mu)$:
\begin{assumption}\label{assumption non local}
There exists a constant $C>0$,  such that for every $f \in \mathfrak{B}^{1}_{1,\infty}(X)$
\[
\sup_{\varepsilon >0} \frac{1}{\varepsilon^{d_H+1}}
 \int_{\Delta_\varepsilon}|f(y)-f(x)|\, d\mu(x)d\mu(y) \le C \limsup_{\varepsilon\to 0^+}\frac{1}{\varepsilon^{d_H+1}}
 \int_{\Delta_\varepsilon} |f(y)-f(x)| \, d\mu(x)d\mu(y),
\]
where $\Delta_\varepsilon=\{ (x,y) \in X \times X, d(x,y) < \varepsilon \}$. 
\end{assumption}

This property of the metric space $(X,d,\mu)$ is for instance satisfied if $(X,d)$ satisfies the assumptions of section \ref{weak BE metric space}, see remark \ref{chaining metric}.

One has then the following theorem:

\begin{theorem}\label{T:non-local_critical_alpha}
Under Assumption~\ref{assumption non local} and if $d_W \le 1$, then $\alpha^*=1$ and $\mathbf{B}^{1,\alpha^*}(X)=\mathfrak{B}^{d_W}_{1,1}(X)$. If $d_W  > 1$, then $\alpha^*=\frac{1}{d_W}$ and $\mathbf{B}^{1,\alpha^*}(X)=\mathfrak{B}^{1}_{1,\infty}(X)$. 
\end{theorem}
\begin{proof}
If $d_W\le1$, it follows from Proposition \ref{prop:NLTrivialB} that $\bfB^{1,\alpha}(X)$ is trivial for any $\alpha>1$. It remains to justify that $\bfB^{1,\alpha}(X)$ or equivalently $\mathfrak{B}^{\alpha d_W}_{1,\infty}(X)$ is not trivial if $\alpha <1$. This follows from the fact 4 below since one has then $  \mathfrak{B}^1_{1,\infty}(X) \subset \mathfrak{B}^{\alpha d_W}_{1,\infty}(X) $.

\medskip
If $d_W>1$, the conclusion follows by considering the following four facts.
 
{\bf Fact 1.} $\bfB^{1,\alpha}(X)$ is trivial for $\alpha>1$. This immediately follows from Proposition \ref{prop:NLTrivialB}.

\medskip
{\bf Fact 2.} $\bfB^{1,1}(X)$ is trivial. Let $f\in \bfB^{1,1}(X)$, then for any $\varepsilon>0$, Lemma \ref{lem:BesovEquivSN} yields
\begin{align*} 
\frac{1}{\varepsilon^{d_H+1}} \int_{\Delta_\varepsilon} |f(y)-f(x)|\, d\mu(x)d\mu(y)
\le\varepsilon^{d_W-1} \int_{\Delta_\varepsilon} \frac{|f(y)-f(x)|}{d(x,y)^{d_H+d_W}}\, d\mu(x)d\mu(y)
\le
C \varepsilon^{d_W-1} \|f\|_{1,1}.
\end{align*}
Thus, 
\[
 \limsup_{\varepsilon\to 0^+}\frac{1}{\varepsilon^{d_H+1}}
 \int_{\Delta_\varepsilon} |f(y)-f(x)| \, d\mu(x)d\mu(y)=0.
 \]
 From assumption \eqref{assumption non local}, this implies that $f$ is constant and thus 0.
\medskip

{\bf Fact 3.} $\bfB^{1,\alpha}(X)$ is trivial for any $\frac1{d_W}<\alpha<1$.
By Theorem \ref{Besov non-local}, we have $\bfB^{1,\alpha}(X) =\mathfrak B_{1,\infty}^{\alpha d_W}(X)$. Thus it suffices to show that $\mathfrak B_{1,\infty}^{\beta}(X)$ is trivial for $1<\beta<d_W$. For any $\varepsilon>0$,
\begin{align*} 
\frac{1}{\varepsilon^{d_H+1}} \int_{\Delta_\varepsilon} |f(y)-f(x)|\, d\mu(x)d\mu(y)
\le
\frac{\varepsilon^{\beta-1} }{\varepsilon^{d_H+\beta}} \int_{\Delta_\varepsilon} |f(y)-f(x)|\, d\mu(x)d\mu(y)
\le
C \varepsilon^{\beta-1}N_{1,\infty}^{\beta}(f).
\end{align*}
Taking the limsup when $\varepsilon$ goes to zero, again  we have that $f$ is a constant.

\medskip

{\bf Fact 4.} It remains to justify that $\bfB^{1,\frac1{d_W}}(X)$ is not trivial, or equivalently, $\mathfrak B_{1,\infty}^{1}(X)$ is not trivial.

Fix $x_0\in X$ and $R>0$. Set 
\[
\phi_{x_0,R}(x)=\max\{0,R-d(x_0,x)\}.
\]
Then $\phi_{x_0,R}$ is non-negative, $L^1$-integrable and supported in the ball $B(x_0,R)$ (see also Section 4.1 for these functions).
For any $r>0$, we have
\begin{align*}
&\iint_{\Delta_r} |\phi_{x_0,R}(x)-\phi_{x_0,R}(y)|\, d\mu(x) d\mu(y)
\\ \le& 
2\int_{B(x_0,R)}\int_{B(x_0,R)^c \cap B(y,r)} |R-d(x_0,y)| \,d\mu(x) d\mu(y)
\\ &
+\int_{B(x_0,R)}\int_{B(x_0,R)\cap B(y,r)} |d(x_0,x)-d(x_0,y)| \,d\mu(x) d\mu(y)
\\ \le& 
\,C r^{d_H+1} R^{d_H}.
\end{align*}
The second inequality above holds since $x\in B(x_0,R)^c$ and $y\in B(x_0,R)$ imply $R-d(x_0,y)\le d(x,y)\le r$. Therefore it easily follows from the definition that $\phi_{x_0,R}\in \mathfrak B_{1,\infty}^1(X)$ with semi-norm estimate depending on $R$.
\end{proof}

\section{Sobolev and isoperimetric inequalities}\label{S:Sobo_iso_nl}

\subsection{Discussion on Sobolev inequalities for non-local Dirichlet spaces}
The standard Poincar\'e type inequalities do not make sense for
non-local Dirichlet forms (although there are some forms of Poincar\'e type inequalities for Dirichlet forms, such as \cite{OS-CT,Chen2017}); however, the following global Sobolev inequality
makes sense: do there exist constants $C>0$ and $\kappa\ge 1$ such that
for every $u\in\mathcal{F}$,
\[
\left(\int_X|u|^{2\kappa}\, d\mu\right)^{\frac{1}{2\kappa}}
  \le C\l \sqrt{\mathcal{E}(u,u)}
\]
We now explore one condition under which the above inequality will hold. 
Recall that $\mathcal{E}$ is \emph{transient} if there is some almost-everywhere
positive $f\in L^1(X)$ such that 
\[
\int_0^\infty P_tf\, dt
\]
is finite almost
everywhere in $X$. Here $P_tf$ is the heat extension of $f$, namely,
\[
P_tf(x)=\int_X p_t(x,y)\, f(y)\, d\mu(y).
\]
Thus transience of the form $\mathcal{E}$ implies that $P_tf$ decays fast
as $t\to \infty$ almost everywhere in $X$. For the rest of this subsection
we will assume that $\mathcal{E}$ is a transient regular Dirichlet form on
$L^2(X)$, and that $(\mathcal{E},\mathcal{F})$ is strongly continuous.
By completing $(\mathcal{E},\mathcal{F})$ with respect to the norm
$u\mapsto\Vert u\Vert_{L^2(X)}+\sqrt{\mathcal{E}(u,u)}$ if necessary, we may
assume also that this pair is a Hilbert space (see~\cite[Theorem~1.5.2,
Lemma~1.5.5]{FOT}). We now have the following version of Sobolev
inequality:

\begin{lemma}[Theorem~1.5.3 in \cite{FOT} or Theorem 2.1.5 in \cite{ChenFukushima} ]\label{lem:g-Sobolev}
For a Dirichlet form $(\mathcal{E},\mathcal{F})$, 
  the strongly continuous   semigroup $P_t$ is 
transient if and only if 
there exists a 
bounded function $g\in L^1(X)$ such that $g$ is positive almost everywhere
in $X$ such that for all $u\in\mathcal{F}$ we have
\[
\int_X|u|\, g\, d\mu\le \sqrt{\mathcal{E}(u,u)}.
\]
\end{lemma}

While the conclusion of the above lemma looks remarkably like the
desired Sobolev inequality, the problem here is that we do not have
good control over $g$, and in general we would like to have control
over the $L^q(X,\mu)$-norm of $u$ for some $q>0$, not just over the
$L^1(X, g\, d\mu)$-norm.

For a compact set $K\subset X$ we set the variational capacity of $K$ to be
the number
\[
\text{Cap}_0(K):=\inf_u \mathcal{E}(u,u),
\]
where the infimum is over all $u\in\mathcal{F}\cap C_0(X)$ (where
$C_0(X)$ is the collection of all compactly supported continuous functions on $X$)
with $u(x)\ge 1$ for all $x\in K$, see~\cite[Section~2.4]{FOT}. 
The following capacitary type inequality holds.

\begin{lemma}[{\cite[Lemma~2.4.1]{FOT}}]
For $u\in\mathcal{F}\cap C_0(X)$ and $t>0$ let
$K_t$ denote the set $\{x\in X\, :\, |u(x)|\ge t\}$.
Then we have that
\[
\int_0^\infty 2t\, \text{\rm Cap}_0(K_t)\, dt\le 4\, \mathcal{E}(u,u).
\]
\end{lemma}

The above lemma indicates that there are plenty of compact sets in $X$
with finite capacity. Let $K\subset X$ be a compact set such that
$\text{Cap}_0(K)$ is finite. Then we can find a sequence 
$u_k\in\mathcal{F}\cap C_0(X)$ with $u_k\ge 1$ on $K$ and
$\lim_k\mathcal{E}(u_k,u_k)=\text{Cap}_0(K)<\infty$. 
Since truncations of a function in $\mathcal{F}$ do not increase its
Dirichlet energy, we know that 
$\mathcal{E}(\widehat{u_k},\widehat{u_k})\le \mathcal{E}(u_k,u_k)$
where $\widehat{u_k}=\max\{0,\min\{u_k,1\}\}$. Thus without loss of generality
we may assume that $0\le u_k\le 1$ on $X$. Therefore $\{u_k\}_k$ forms
a locally bounded sequence in $L^2(X)$, and by multiplying each
$u_k$ by a fixed $\eta\in\mathcal{F}\cap C_0(X)$ with $0\le \eta\le 1$ on $X$
and $\eta=1$ on the ball $B(x_0,n)$ for some fixed $x_0\in X$ and 
$n\in\mathbb{N}$, we get a bounded sequence in $\mathcal{F}$, which,
as $\mathcal{F}$ is a Hilbert space, gives a 
convex combination subsequence that converges to
some function $u\in\mathcal{F}$ (we do this for each positive integer $n$
and employ a Cantor-type diagonalization argument to do so) with
$\mathcal{E}(u,u)\le \lim_k\mathcal{E}(u_k,u_k)=\text{Cap}_0(K)$. 
Employing a further subsequence argument, we can ensure that 
$u\ge 1$ in $K$; however, $u$ may no longer be continuous on $X$,
see also~\cite[Lemma~2.1.1]{FOT}.
Such a function $u$ is called a \emph{$0$-th order equilibrium potential}
of $K$ in~\cite{FOT}. 

Given the above notion of capacity, the following theorem identifies a property
on $\mathcal{E}$ under which we have the desired Sobolev inequality.

\begin{theorem}[{\cite[Theorem~2.4.1]{FOT}}]\label{T:isoperim-nonlocal}
If there is some $\kappa\ge 1$ and $\Theta>0$ such that 
\begin{equation}\label{eq:isoperim-nonlocal}
\mu(K)^{1/\kappa}\le \Theta\, \text{\rm Cap}_0(K)
\end{equation}
for each compact $K\subset X$, then there is some $C>0$ with
$0<C^2\le (4\kappa)^\kappa \Theta$ such that
\[
\left(\int_X|u|^{2\kappa}\, d\mu\right)^{1/2\kappa}\le C\, \sqrt{\mathcal{E}(u,u)}.
\]
\end{theorem}

The inequality~\eqref{eq:isoperim-nonlocal} is an analog of the
\emph{isoperimetric inequality} adapted to the non-local Dirichlet form.
The optimal constant $\Theta$ is called the isoperimetric constant of $\mathcal{E}$,
and is a non-local analog of the Cheeger constant.
Indeed, \cite[Theorem~2.4.1]{FOT} claims even more, that the support of
a Sobolev type inequality is equivalent to this isoperimetric inequality.
Should $\mathcal{E}$ support such an inequality, then the total capacity of
a compact set $K$ given by
\[
\text{Cap}_1(K):=\inf_u \int_X u^2\, d\mu+\mathcal{E}(u,u)
\]
with infimum over all $u\in \mathcal{F}\cap C_0(X)$ with $u\ge 1$ on $K$
has the same class of null capacity sets as $\text{Cap}_0$.
We refer the interested reader to~\cite[Theorem~2.4.3]{FOT} for
examples of Sobolev type inequality for measures in Euclidean spaces.

We point out here that this notion of isoperimetric inequality is weaker
than the traditional one where $\text{Cap}_0(K)$ is replaced with the
relative $1$-capacity of $K$, namely infimum of the numbers
$\int_X|\nabla u|\, d\mu$ over all $u\in \mathcal{F}\cap C_0(X)$ with
$u\ge 1$ on $K$. In the case of the non-local Dirichlet form the
quantity $\int_X|\nabla u|\, d\mu$ does not make usual sense (although  see \cite{HRT} for a version of such an integral for arbitrary Dirichlet forms).
On the other hand, the results of~\cite[Theorem~2.1.1]{FOT} tell us
that $\text{Cap}_0$ has an extension to a Choquet capacity on $X$.

Fix $u\in\mathcal{F}\cap C_0(X)$, and let $t>0$. For $\varepsilon>0$
we set
\[
u_t^\varepsilon=\frac{\min\{\varepsilon, \max\{u-t,0\}\}}{\varepsilon}.
\]
Then by the Markov property of $\mathcal{E}$ we know that
$\mathcal{E}(u_t^\varepsilon,u_t^\varepsilon)\le \mathcal{E}(u,u)/\varepsilon^2$.
Observe that $u_t^\varepsilon\to \chi_{K_t}$ in $L^1(X)$ as $\varepsilon\to 0^+$.
Suppose we know that 
$\liminf_{\varepsilon\to 0^+}\mathcal{E}(u_t^\varepsilon,u_t^\varepsilon)<\infty$.
Then 
\[
\text{Cap}_0(K_t)\le 
  \liminf_{\varepsilon\to 0^+}\mathcal{E}(u_t^\varepsilon,u_t^\varepsilon)<\infty.
\]
If $\mathcal{E}$ is strongly local and supports a $2$-Poincar\'e inequality, 
then using the fact that distance functions
belong to $\mathcal{F}$ we can obtain a comparison between
$\text{Cap}_0(K_t)$
and the Minkowski co-dimension $1$ content of $\partial K_t$ for almost 
every $t>0$, see for example the arguments in~\cite{CJKS, JKY}.
For strongly local Dirichlet forms we have shown in Chapter~4 that
the correct BV class is $\mathbf{B}^{1,1/2}(X)$.
For non-local Dirichlet forms (and indeed even for forms that are not
strongly local) it is not clear which of the Besov classes 
$\mathbf{B}^{1,\alpha}(X)$ should be the correct analog of BV functions,
see Theorem~\ref{T:FP_char_non-local}. From this theorem, we see that
if $E\subset X$ is bounded and measurable such that 
$\bm{1}_E\in \mathbf{B}^{1,\alpha}(X)$, then by choosing $r>2\text{diam}(E)$,
\[
\frac{c}{r^{\alpha d_W+d_H}}\mu(E)\, \mu(B(x_0,r/2))
\le \Vert \bm{1}_E\Vert_{1,\alpha}
\]
where $x_0\in X$ such that $B(x_0,r/2)\cap E$ is empty and 
$B(x_0,r/2)\subset\bigcup_{x\in E}B(x,r)$. One can look on this as
the Besov space analog of the isoperimetric inequality.

\subsection{Besov spaces embeddings}

Note that by combining the results of Part I, Chapter 2,  with the results of this chapter, one immediately obtains the following corollaries. The proofs are similar to the ones in Section  \ref{section Sobolev local}, so will be omitted for concision.

\begin{corollary}
Let $(X,\mu,\mathcal{E},\mathcal{F})$ be a symmetric non-local Dirichlet space with  heat kernel estimates \eqref{eq:HKE-non-loc}. Let $0<\delta < d_H $. Let $1 \le p < \min \left\{  \frac{d_H}{\delta} , \frac{d_W}{\delta} \right\}$.   There exists a constant $C_{p,\delta} >0$ such that for every $f \in \mathbf{B}^{p,\delta/d_W}(X)$,
\[
\sup_{s \ge 0} s \mu \left( \{ x \in X, | f(x) | \ge s \} \right)^{\frac{1}{q}} \le C_{p,\delta} \sup_{r>0} \frac{1}{r^{\delta+d_{H}/p}}\biggl(\iint_{\Delta_r}|f(x)-f(y)|^{p}\,d\mu(x)\,d\mu(y)\biggr)^{1/p}
\]
where $q=\frac{p d_H}{ d_H -p \delta}$. Furthermore, for every $0<\delta <d_H $, there exists a constant $C_{\emph{iso},\delta}$ such that for every measurable $E \subset X$, $\mu(E) <+\infty$,
\begin{align}\label{isoperimetric intro3}
\mu(E)^{\frac{d_H-\delta}{d_H}} \le C_{\emph{iso},\delta} \sup_{r>0} \frac{1}{r^{\delta+d_{H}}} (\mu \otimes \mu) \left\{ (x,y) \in E \times E^c
\, :\,
 d(x,y) \le r\right\} 
\end{align}
\end{corollary}

\begin{corollary}
Let $(X,\mu,\mathcal{E},\mathcal{F})$ be a symmetric non-local Dirichlet space with  heat kernel estimates \eqref{eq:HKE-non-loc} and $d_W \le d_H$. Let $0<\delta < d_H $. Let  $p=\frac{d_W}{\delta}$ and assume $p \ge 1$.   There exists a constant $C_{p,\delta} >0$ such that for every $f \in \mathbf{B}^{p,\delta/d_W}(X) $,
\[
\sup_{s \ge 0} s \mu \left( \{ x \in X, | f(x) | \ge s \} \right)^{\frac{1}{q}} \le C_{p,\delta} \brak{\int_0^\infty \brak{N_p^{d_W/p}(f,r)}^{p}\, \frac{dr}{r}}^{1/p},
\]
where $q=\frac{p d_H}{ d_H -p \delta}$ and
\[
N^{d_W/p}_{p}(f,r):=\frac{1}{r^{d_W/p+d_{H}/p}}\biggl(\iint_{\Delta_r}|f(x)-f(y)|^{p}\,d\mu(x)\,d\mu(y)\biggr)^{1/p}
\]

\end{corollary}



%
%

%% file: chapter7.tex
\chapter{Sets of finite perimeter in some infinite-dimensional examples }

In this chapter, we study some infinite-dimensional examples of Dirichlet spaces $X$, associated sets of finite perimeter and discuss the connection with $\mathbf{B}^{1,1/2}(X)$. Our goal is not to study the general theory of Besov spaces related to infinite-dimensional Dirichlet spaces, but rather to look at some specific examples.

\

 In many infinite-dimensional examples, coming back to the original ideas by De Giorgi \cite{DG} it seems interesting to take advantage of the existence of nice integration by parts formulas to construct BV functions and sets of finite perimeter. To provide the reader with a motivation, we briefly present below those ideas in a general setting, avoiding at first the exact assumptions and technicalities associated with a possible infinite-dimensional setting. 

\

Let $(X,\mu,\mathcal{E},\mathcal{F}=\mathbf{dom}(\mathcal{E}))$ be a quasi regular local symmetric Dirichlet space. We assume that $\mathcal{E}$ can be written as
\[
\mathcal{E}(f,f)=\| Df \|_\mathcal{H}^2 ,
\]
where $D$ is a closed unbounded operator from $L^2(X,\mu)$ to some (separable) Hilbert space $\mathcal{H}$. In many situations (see for instance \cite{Ebe99}, {Theorem 3.9} and the examples below), we are able to think of $\mathcal{H}$ as a $L^2$ space of sections of a vector bundle over $X$: more precisely, $\mathcal{H}$ is isometrically isomorphic to some $ \int_{X}^\oplus \mathcal{H}_x \ d\mu(x)$ and for any $\eta_1,\eta_2\in \mathcal{H}$,
\[
\langle\eta_1,\eta_2\rangle_\mathcal{H} = \int_{X} \langle \eta_1,\eta_2\rangle_{\mathcal{H}_x} \ d\mu(x).
\]
In this situation, following De Giorgi in \cite{DG}, for $f \in L^1(X,\mu)$ it is then natural to define:
\[
\mathbf{Var} f = \sup\left\{ \int_X f D^* \eta d\mu
\,
:
\,
 \eta \in \mathbf{dom} (D^*), \| \eta \|_{\mathcal{H}_x} \leq 1, \mu \text{-almost everywhere} \right\},
\]
where $D^*$ is the adjoint operator of $D$ and one defines then:
\[
BV(X) := \left\{f\in L^1(X,\mu)\, :\,
\mathbf{Var} f < +\infty \right\}.
\]
This allows one to define the perimeter of a measurable set $E\subset X$ with $\mathbf 1_E\in BV(X)$ by
\[
P(E,X) = \mathbf{Var} \mathbf{1}_E.
\]
When $X$ is a complete Riemannian manifold and $D=\nabla$ is the Riemannian gradient, this construction yields the usual $BV$ space and sets with finite perimeters are Caccioppoli sets.

In the following examples, the strong Bakry-\'Emery condition is satisfied
\[
\| DP_t f \|_{\mathcal{H}_x} \le (P_t \| Df \|_{\mathcal{H}.})(x).
\]
and one shall prove:

\begin{theorem}
Let $X$ be one of the Dirichlet spaces studied in Sections \ref{Wiener} or \ref{gibbs}. Let $E$ be a measurable set with finite measure. The following are equivalent:

\begin{enumerate}
\item $P(E,X) <+\infty$;
\item The limit $\lim_{t \to 0^+} \int_X \| DP_t \mathbf{1}_E \|_{\mathcal{H}_x} d\mu(x)$ exists.
\end{enumerate}
Moreover, if $E$ is a set satisfying one of the above conditions, then $\mathbf{1}_E \in \mathbf{B}^{1,1/2}(X)$ and 
\[
\| \mathbf{1}_E \|_{1,1/2}\le 2 \sqrt{2} P(E,X) =2 \sqrt{2} \lim_{t \to 0^+} \int_X \| DP_t \mathbf{1}_E \|_{\mathcal{H}_x} d\mu(x).
\]
\end{theorem}

The theorem is stated in a class of specific examples, but the method which is developed is actually relatively general.

\section{Wiener space}\label{Wiener}

In this section, we discuss sets of finite perimeters in abstract Wiener space, The theory of BV functions is already very well established and understood in that case (see Fukushima-Hino \cite{FuHi}, and also the survey \cite{MNP} and the references therein).

Let $(X,\mathcal{H},\mu)$ be an abstract Wiener space in the sense of L. Gross \cite{Gross}, so $X$ is a separable real Banach space, $\mathcal{H}$ is a separable real Hilbert space continuously and densely embedded into $X$ and $\mu$ is a Gaussian measure satisfying for $l \in X^* \subset \mathcal{H^*} \simeq \mathcal{H}$,
\[
\int_X \exp \left( i \langle l , x \rangle \right) \mu(dx)=\exp \left(-\frac{1}{2} \| l \|^2_{\mathcal{H}} \right).
\]
We define the Ornstein-Uhlenbeck semigroup on $L^2(X,\mu)$ by the formula:
\begin{align}\label{OU semigroup}
P_t f(x) =\int_X f \left( e^{-t} x +\sqrt{1-e^{-2t} }y \right) d\mu(y), \quad f \in L^2(X,\mu), x \in X.
\end{align}
The associated Dirichlet form may be written in terms of the Malliavin derivative. We now describe this Dirichlet form. We say that a functional $F: X \to \mathbb{R}$ is a smooth cylindric functional if, for some $n$, it can be written in the form
\[
F(x)=f\left( \langle x, l_1 \rangle, \cdots, \langle x ,l_n \rangle \right),
\]
where $f$ is smooth and bounded, and $l_1,\cdots, l_n \in \mathcal{H}^*$ . Here $\langle x, l\rangle$ stands for the Wiener integral. The set of smooth cylindric functionals will be denoted by $\mathcal{C}$ and will be used as nice core algebra. Note that $\mathcal{C}$ is dense in $L^2(X,\mu)$. If $F \in \mathcal{C}$, one defines its Malliavin derivative by the formula,
\begin{align}\label{Malliavin derivative}
DF(x)=\sum_{i=1}^n \frac{\partial f}{\partial x_i} \left( \langle x, l_1 \rangle, \cdots, \langle x ,l_n \rangle \right) l_i \in \mathcal{H}.
\end{align}

We can then consider the pre-Dirichlet form $(\mathcal{E}, \mathcal{C}, \mu)$ given for $F,G \in \mathcal{C}$ by
\[
\mathcal{E} (F,G)=\int_X \langle DF (x), DG(x) \rangle_\mathcal{H} d\mu(x).
\]
The closure of this pre-Dirichlet form yields a quasi-regular local symmetric Dirichlet space $(X,\mathcal{F}, \mu)$ whose semigroup is precisely the Ornstein-Uhlenbeck semigroup (see \cite{Kusuoka}).

Note that the Dirichlet form $\mathcal{E}$ admits a carr\'e du champ and that one has for $f \in \mathcal{F}$, and $x \in X$,
\[
\Gamma(f)(x)= \| Df (x) \|_\mathcal{H}^2.
\]
where $D$ stands for the closed extension of the formula \eqref{Malliavin derivative}. One can check directly from the formula \eqref{OU semigroup} that for $f \in \mathcal{C}$,
\[
DP_t f (x) = \vec{P}_t D f (x)
\]
where $\vec{P}_t $ is a semigroup on $L^2 (X , \mathcal{H} ,\mu)$ that satisfies for every $\eta \in L^2 (X , \mathcal{H} ,\mu)$,
\begin{align}\label{BE vector}
\| \vec{P}_t \eta (x) \|_{\mathcal{H}} \le e^{-t} (P_t \| \eta \|_\mathcal{H}) (x)
\end{align}
In particular, one has the strong Bakry-\'Emery curvature condition:
\[
\| DP_t f (x) \|_\mathcal{H} \le (P_t \| Df \|_{\mathcal{H}})(x), \quad f \in \mathcal{F}.
\]
We are now ready to define sets of finite perimeter, following Fukushima-Hino \cite{FuHi}. 

One can think of $D$ as a closed unbounded operator $L^2(X,\mu) \to L^2 (X , \mathcal{H} ,\mu)$. Denote its adjoint by $D^*$ (this is the so-called Skorohod integral).
If $E \subset X$ is a measurable set, 
\[
P(E,X) = \sup\left\{ \int_E D^* \eta d\mu\, :\, 
\eta \in \mathbf{dom} (D^*), \| \eta (x) \|_{\mathcal{H}} \leq 1, \mu \text{-almost everywhere} \right\},
\]
 and we say that $E$ has a finite perimeter if $P(E,X) <+\infty$.
As everywhere in the work, for $f \in L^p(X,\mu)$, we define:
\[
\| f \|_{p,\alpha}= \sup_{t >0} t^{-\alpha} \left( \int_X P_t (|f-f(y)|^p)(y) d\mu(y) \right)^{1/p}.
\]
and
\[
\mathbf{B}^{p,\alpha}(X)=\{ f \in L^p(X,\mu)
\,
:
\,
 \| f \|_{p,\alpha} <+\infty \}.
\]

The theorem is the following:

\begin{theorem}\label{Set Wiener}
Let $E$ be a measurable set. The following are equivalent:

\begin{enumerate}
\item $P(E,X) <+\infty$;
\item The limit $\lim_{t \to 0^+} \int_X \| DP_t \mathbf{1}_E (x) \|_{\mathcal{H}} d\mu(x)$ exists.
\end{enumerate}
Moreover, if $E$ is a set satisfying one of the above conditions, then $\mathbf{1}_E \in \mathbf{B}^{1,1/2}(X)$ and 
\[
\| \mathbf{1}_E \|_{1,1/2}\le 2 \sqrt{2} P(E,X) =2 \sqrt{2} \lim_{t \to 0^+} \int_X \| DP_t \mathbf{1}_E \|_{\mathcal{H}_x} d\mu(x).
\]
\end{theorem}

The fact that (1) is equivalent to (2) is actually already essentially known (see \cite{MNP}), but we give here a proof that immediately adapts to the case considered in the following section. The proof will be divided in several lemmas.

\begin{lemma}
Let $E$ be a measurable set. If the limit $$\lim_{t \to 0^+} \int_X \sqrt{\Gamma(P_t \mathbf{1}_E) } d\mu(x)$$ exists then $\mathbf{1}_E \in \mathbf{B}^{1,1/2}(X)$ and
\[
\| \mathbf{1}_E \|_{1,1/2} \le 2 \sqrt{2} \lim_{t \to 0^+} \int_X \sqrt{\Gamma(P_t \mathbf{1}_E) } d\mu(x).
\]
\end{lemma}

\begin{proof}
The strong Bakry-\'Emery estimate implies that for every $f \in \mathcal{F} $ and $t \ge 0$, one has $\mu$-almost everywhere
\[
\| P_t f -f \|_1 \le \sqrt{2t} \int_X \sqrt{\Gamma(f)} d\mu.
\]
Let $E$ be a measurable set such that $\lim_{t \to 0^+} \int_X \sqrt{\Gamma(P_t \mathbf{1}_E) } d\mu(x)$ exists. One has for every $s \ge 0$,
\[
\| P_{t+s} \mathbf{1}_E -P_s\mathbf{1}_E \|_1 \le \sqrt{2t} \int_X \sqrt{\Gamma(P_s \mathbf{1}_E)} d\mu.
\]
Taking the limit when $s \to 0^+$ yields,
\[\
\| P_{t} \mathbf{1}_E -\mathbf{1}_E \|_1 \le \sqrt{2t} \lim_{s\to 0^+} \int_X \sqrt{\Gamma(P_s \mathbf{1}_E)} d\mu
\]
and so $\mathbf{1}_E \in \mathbf{B}^{1,1/2}(X)$ with $$ \| \mathbf{1}_E \|_{1,1/2} \le 2 \| P_{t} \mathbf{1}_E -\mathbf{1}_E \|_1 \le 2 \sqrt{2} \lim_{t \to 0^+} \int_X \sqrt{\Gamma(P_t \mathbf{1}_E) } d\mu(x).$$
\end{proof}

\begin{lemma}
For every $\eta \in L^2 (X , \mathcal{H} ,\mu)$, and $t>0$, $\vec{P}_t \eta \in \mathbf{dom} (D^*)$.
\end{lemma}

\begin{proof}
For $f \in \mathcal{F}$ and $\eta \in L^2 (X , \mathcal{H} ,\mu)$, one has
\[
\int_X \langle D f (x) , \vec{P}_t \eta (x)\rangle_\mathcal{H} d\mu(x)= \int_X \langle \vec{P}_t D f (x) , \eta (x)\rangle_\mathcal{H} d\mu (x)=\int_X \langle DP_t f (x) , \eta (x)\rangle_\mathcal{H} d\mu (x).
\]
Since the operator $DP_t$ is bounded in $L^2$, the conclusion follows.
\end{proof}

\begin{lemma}
If $f \in \mathcal{F}$, then
\[
 \sup\left\{ \int_X f D^* \eta d\mu\, :\,
  \eta \in \mathbf{dom} (D^*), \| \eta (x) \|_{\mathcal{H}} \leq 1, \mu \text{-a.e.} \right\}=\int_X \| Df (x) \|_\mathcal{H} d\mu(x).
\]
\end{lemma}

\begin{proof}
Let $f \in \mathcal{F}$. It is clear that for $\eta \in \mathbf{dom} (D^*)$ such that $ \| \eta (x) \|_{\mathcal{H}} \leq 1$, one has
\[
\int_X f D^* \eta d\mu =\int_X \langle Df , \eta \rangle_\mathcal{H} d\mu \le \int_X \| Df \|_\mathcal{H} d\mu
\]
We now prove the converse inequality. Let $f \in \mathcal{F}$. Let $s, \ve >0$ and consider 
\[
\eta= \vec{P}_s \left( \frac{Df}{\| Df \|_\mathcal{H} +\ve} \right).
\]
From the previous lemma $ \eta \in \mathbf{dom} (D^*) $ and from \eqref{BE vector}, one has
\[
\| \eta (x) \|_\mathcal{H} \le 1, \quad \mu \text{-a.e.} 
\]
Moreover,
\[
\int_X f D^* \eta d\mu=\int_X \left\langle Df, \vec{P}_s \left( \frac{Df}{\| Df \|_\mathcal{H} +\ve} \right)\right\rangle d\mu.
\]
Letting $s \to 0^+$ and $\ve \to 0$ yields
\[
\int_X \| Df \|_\mathcal{H} d\mu \le \sup\left\{ \int_X f D^* \eta d\mu
\,
:
\,
\eta \in \mathbf{dom} (D^*), \| \eta (x) \|_{\mathcal{H}} \leq 1, \mu \text{-a.e.} \right\}.
\]
\end{proof}

We are now ready to prove Theorem \ref{Set Wiener}.

\begin{proof}
It just remains to prove that 1 is equivalent to 2.

Let $E$ be a measurable set and $\eta \in \mathbf{dom} (D^*), \| \eta (x) \|_{\mathcal{H}} \leq 1 \text{ a.e.}$. We have
\[
\int_X (P_t \mathbf{1}_E) D^* \eta d\mu=\int_X \langle DP_t \mathbf{1}_E, \eta \rangle_\mathcal{H}d\mu \le \int_X \| DP_t \mathbf{1}_E \|_\mathcal{H}d\mu.
\]
Since we have in $L^2$, $\lim_{t \to 0} P_t \mathbf{1}_E=\mathbf{1}_E$, one deduces
\[
\int_E D^* \eta d\mu \le \liminf_{t \to 0^+}\int_X \| DP_t \mathbf{1}_E \|_\mathcal{H}d\mu.
\]
Thus,
\[
\sup\left\{ \int_E D^* \eta d\mu\, :\,
 \eta \in \mathbf{dom} (D^*), \| \eta (x) \|_{\mathcal{H}} \leq 1 \right\} \le \liminf_{t \to 0^+}\int_X \| DP_t \mathbf{1}_E \|_\mathcal{H}d\mu.
\]
On the other hand, again if $E$ be a measurable set and $\eta \in \mathbf{dom} (D^*), \| \eta (x) \|_{\mathcal{H}} \leq 1 \text{ a.e.}$, then we have:
\[
\int_X (P_t \mathbf{1}_E) D^* \eta d\mu=\int_X \mathbf{1}_E P_t D^* \eta d\mu=\int_X \mathbf{1}_E D^* \vec{P}_t \eta d\mu \le \sup\left\{ \int_E D^* \eta d\mu   \, :\,
 \eta \in \mathbf{dom} (D^*), \| \eta (x) \|_{\mathcal{H}} \leq 1 \right\},
\]
where in the last inequality we used the fact that $ \| \vec{P}_t \eta \|_\mathcal{H} \le 1$. Since from the previous lemma we know that
\[
\sup\left\{\int_X (P_t \mathbf{1}_E) D^* \eta d\mu
\,
:
\,
\eta \in \mathbf{dom} (D^*), \| \eta (x) \|_{\mathcal{H}} \leq 1 \right\} =\int_X \| DP_t \mathbf{1}_E (x) \|_\mathcal{H} d\mu(x),
\]
and so,
\[
\limsup_{t \to 0^+}\int_X \| DP_t \mathbf{1}_E (x) \|_\mathcal{H} d\mu(x) \le \sup\left\{ \int_E D^* \eta d\mu\,
:
\,
 \eta \in \mathbf{dom} (D^*), \| \eta (x) \|_{\mathcal{H}} \leq 1 \right\}.
\]
The proof is complete.
\end{proof}

\section{Path spaces with Gibbs measure}\label{gibbs}

In this section, we present new examples of infinite-dimensional Dirichlet spaces for which one can construct sets of finite perimeter as in the Wiener space. 
The examples are Dirichlet spaces built on the space of continuous functions $C(\mathbb{R},\mathbb{R}^d)$, where the reference measure is the Gibbs measure associated with the formal Hamiltonian:

\[
\mathcal{H}(w) =\frac{1}{2} \int_\mathbb{R} \| w'(x) \|^2 dx+ \int_{\mathbb{R}} U(w(x)) dx, \quad w \in C(\mathbb{R},\mathbb{R}^d),
\]
where $U:\mathbb{R}^d \to \mathbb{R}$ is an interaction potential.
The corresponding diffusion processes are defined through the Ginzburg-Landau type stochastic partial differential equations:
\begin{align}\label{spde}
dX_t(x)=\left( \frac{1}{2}\frac{d^2}{dx^2} X_t(x) -\nabla U(X_t(x)) \right) dt+dW_t(x), \quad t>0, x \in \mathbb{R},
\end{align}
where $(W_t)$ is a white noise process. 

In the sequel we will assume that the interaction potential $U$ satisfies the following conditions:

\begin{enumerate}
\item $U$ is two times continuously differentiable and $\nabla^2 U \ge 0$;
\item There exists $A>0$ and $p>0$ such that for every $x \in \mathbb{R}^d$:
\[
\| \nabla U (x) \| + \| \nabla^2 U(x) \| \le A (1+\| x \|^p);
\]
\item $\lim_{ \| x \| \to +\infty} U(x) =+\infty$.
\end{enumerate}

We now briefly sketch, omiiting some details, the construction of the relevant Dirichlet space. For the details, we refer to Kawabi \cite{Kawabi} or Kawabi-R\"ockner \cite{KR}. 

We denote $X=C(\mathbb{R},\mathbb{R}^d)$\footnote{Strictly speaking, the construction needs to be performed on a subspace $X$ of $C(\mathbb{R},\mathbb{R}^d)$.} . We first define a set of smooth cylindric functionals. We say that a functional $F: X\to \mathbb{R}$ is a smooth cylindric functional if, for some $n$, it can be written in the form
\[
F(w)=f\left( \langle w, \phi_1 \rangle, \cdots, \langle w ,\phi_n \rangle \right),
\]
where $f$ is smooth and bounded, and $\phi_1,\cdots,\phi_n$ are smooth and compactly supported. Here, we use the notation $\langle w ,\phi_i \rangle=\int_{\mathbb{R}} \langle w (x) ,\phi_i (x) \rangle dx$. The set of smooth cylindric functionals will be denoted by $\mathcal{C}$ and will be used as nice core, as in the Wiener space example. 

We denote by $\mathcal{H}$ the space of square integrable functions $\mathbb{R} \to \mathbb{R}^d$. If $F \in \mathcal{C}$, one defines its derivative by the formula,
\[
DF=\sum_{i=1}^n \frac{\partial f}{\partial x_i}\left( \langle w, \phi_1 \rangle, \cdots, \langle w ,\phi_n \rangle \right) \phi_i \in \mathcal{H}.
\]

Let now $\mu$ be the $U$-Gibbs measure defined by (2.1) in \cite{Kawabi}. Note that $\mathcal{C}$ is dense in $L^2(X,\mu)$. We can then consider the pre-Dirichlet form $(\mathcal{E}, \mathcal{C}, \mu)$ given for $F,G \in \mathcal{C}$ by
\[
\mathcal{E} (F,G)=\int_X \langle DF (w), DG(w) \rangle_\mathcal{H} d\mu(w).
\]
The closure of this pre-Dirichlet form yields a quasi-regular local symmetric Dirichlet space $(X,\mathcal{F}, \mu)$. We shall denote by $\{ P_t,t\ge 0\}$ the semigroup of this Dirichlet form. It is worth noting that $\{ P_t,t\ge 0\}$ is conservative and is indeed the semigroup associated to the spde \eqref{spde} (see \cite{KR}).

Note that the Dirichlet form $\mathcal{E}$ admits a carr\'e du champ and that one has for $f \in \mathcal{F}$,
\[
\Gamma(f)(w)= \| Df (w) \|_\mathcal{H}^2.
\]
In this framework, one has the strong Bakry-\'Emery estimate (see Proposition 2.1 in \cite{Kawabi06}):
\[
\| DP_t f (w) \|_\mathcal{H} \le (P_t \| Df \|_{\mathcal{H}})(w), \quad f \in \mathcal{F},
\]
Actually, according to Lemma 3.1. in \cite{Kawabi}, as in the Wiener space example, for $f \in \mathcal{F}$,
\[
DP_t f (w) = \vec{P}_t D f (w)
\]
where $\vec{P}_t $ is a semigroup on $L^2 (X , \mathcal{H} ,\mu)$ that satisfies for every $\eta \in L^2 (X , \mathcal{H} ,\mu)$,
\begin{align}\label{BE vector2}
\| \vec{P}_t \eta (w) \|_{\mathcal{H}} \le (P_t \| \eta \|_\mathcal{H}) (w).
\end{align}

From there, the situation is exactly similar to the Wiener space case that was studied in the previous section and, if $E \subset X$ is a measurable set, we set
\[
P(E,X) = \sup\left\{ \int_E D^* \eta d\mu
\,
:
\,
 \eta \in \mathbf{dom} (D^*), \| \eta (x) \|_{\mathcal{H}} \leq 1, \mu \text{-almost everywhere} \right\},
\]
 and we say that $E$ has a finite perimeter if $P(E,X) <+\infty$. The same proof that we gave in the Wiener space example works and Theorem \ref{Set Wiener} therefore holds.

%% file: source.bbl
\providecommand{\bysame}{\leavevmode\hbox to3em{\hrulefill}\thinspace}
\providecommand{\MR}{\relax\ifhmode\unskip\space\fi MR }
\providecommand{\MRhref}[2]{%
  \href{http://www.ams.org/mathscinet-getitem?mr=#1}{#2}
}
\providecommand{\href}[2]{#2}
\begin{thebibliography}{BCLSC95}

\bibitem[ACDH04]{ACDH}
Pascal Auscher, Thierry Coulhon, Xuan~Thinh Duong, and Steve Hofmann,
  \emph{Riesz transform on manifolds and heat kernel regularity}, Ann. Sci.
  {\'E}cole Norm. Sup. (4) \textbf{37} (2004), no.~6, 911--957. \MR{2119242}

\bibitem[ADMG17]{ambrosio2017perimeter}
Luigi Ambrosio, Simone Di~Marino, and Nicola Gigli, \emph{Perimeter as relaxed
  {M}inkowski content in metric measure spaces}, Nonlinear Anal. \textbf{153}
  (2017), 78--88. \MR{3614662}

\bibitem[AF03]{AR}
Robert~A. Adams and John J.~F. Fournier, \emph{Sobolev spaces}, second ed.,
  Pure and Applied Mathematics (Amsterdam), vol. 140, Elsevier/Academic Press,
  Amsterdam, 2003. \MR{2424078}

\bibitem[AFP00]{AFP}
Luigi Ambrosio, Nicola Fusco, and Diego Pallara, \emph{Functions of bounded
  variation and free discontinuity problems}, Oxford Mathematical Monographs,
  The Clarendon Press, Oxford University Press, New York, 2000. \MR{1857292}

\bibitem[AMP04]{AMP}
L.~Ambrosio, M.~Miranda, Jr., and D.~Pallara, \emph{Special functions of
  bounded variation in doubling metric measure spaces}, Calculus of variations:
  topics from the mathematical heritage of {E}.\ {D}e {G}iorgi, Quad. Mat.,
  vol.~14, Dept. Math., Seconda Univ. Napoli, Caserta, 2004, pp.~1--45.
  \MR{2118414}

\bibitem[AS61]{AS}
N.~Aronszajn and K.~T. Smith, \emph{Theory of {B}essel potentials. {I}}, Ann.
  Inst. Fourier (Grenoble) \textbf{11} (1961), 385--475.

\bibitem[Bar98]{Ba98}
Martin~T. Barlow, \emph{Diffusions on fractals}, Lectures on probability theory
  and statistics ({S}aint-{F}lour, 1995), Lecture Notes in Math., vol. 1690,
  Springer, Berlin, 1998, pp.~1--121.

\bibitem[Bar03]{Ba03}
\bysame, \emph{Heat kernels and sets with fractal structure}, Heat kernels and
  analysis on manifolds, graphs, and metric spaces ({P}aris, 2002), Contemp.
  Math., vol. 338, Amer. Math. Soc., Providence, RI, 2003, pp.~11--40.

\bibitem[Bar13]{BaASC}
M.~T. Barlow, \emph{Analysis on the {S}ierpinski carpet}, Analysis and geometry
  of metric measure spaces, CRM Proc. Lecture Notes, vol.~56, Amer. Math. Soc.,
  Providence, RI, 2013, pp.~27--53. \MR{3060498}

\bibitem[BB89]{BB89}
Martin~T. Barlow and Richard~F. Bass, \emph{The construction of {B}rownian
  motion on the {S}ierpi\'nski carpet}, Ann. Inst. H. Poincar\'e Probab.
  Statist. \textbf{25} (1989), no.~3, 225--257. \MR{1023950}

\bibitem[BB92]{BB92}
\bysame, \emph{Transition densities for {B}rownian motion on the
  {S}ierpi{\'n}ski carpet}, Probab. Theory Related Fields \textbf{91} (1992),
  no.~3-4, 307--330. \MR{1151799}

\bibitem[BB99]{BB99}
\bysame, \emph{Brownian motion and harmonic analysis on {S}ierpinski carpets},
  Canad. J. Math. \textbf{51} (1999), no.~4, 673--744. \MR{1701339}

\bibitem[BB12]{BB}
Fabrice Baudoin and Michel Bonnefont, \emph{Log-{S}obolev inequalities for
  subelliptic operators satisfying a generalized curvature dimension
  inequality}, J. Funct. Anal. \textbf{262} (2012), no.~6, 2646--2676.
  \MR{2885961}

\bibitem[BB16]{BB2}
\bysame, \emph{Reverse {P}oincar\'e inequalities, isoperimetry, and {R}iesz
  transforms in {C}arnot groups}, Nonlinear Anal. \textbf{131} (2016), 48--59.

\bibitem[BBBC08]{BBBC}
Dominique Bakry, Fabrice Baudoin, Michel Bonnefont, and Djalil Chafa\"\i,
  \emph{On gradient bounds for the heat kernel on the {H}eisenberg group}, J.
  Funct. Anal. \textbf{255} (2008), no.~8, 1905--1938. \MR{2462581}

\bibitem[BBG14]{BBG}
Fabrice Baudoin, Michel Bonnefont, and Nicola Garofalo, \emph{A
  sub-{R}iemannian curvature-dimension inequality, volume doubling property and
  the {P}oincar\'e inequality}, Math. Ann. \textbf{358} (2014), no.~3-4,
  833--860. \MR{3175142}

\bibitem[BBKT10]{BBKT}
Martin~T. Barlow, Richard~F. Bass, Takashi Kumagai, and Alexander Teplyaev,
  \emph{Uniqueness of {B}rownian motion on {S}ierpi\'nski carpets}, J. Eur.
  Math. Soc. (JEMS) \textbf{12} (2010), no.~3, 655--701. \MR{2639315}

\bibitem[BBST99]{BenBassatStrichartzTeplyaev}
Oren Ben-Bassat, Robert~S. Strichartz, and Alexander Teplyaev, \emph{What is
  not in the domain of the {L}aplacian on {S}ierpinski gasket type fractals},
  J. Funct. Anal. \textbf{166} (1999), no.~2, 197--217. \MR{1707752}

\bibitem[BCLSC95]{BCLS}
D.~Bakry, T.~Coulhon, M.~Ledoux, and L.~Saloff-Coste, \emph{Sobolev
  inequalities in disguise}, Indiana Univ. Math. J. \textbf{44} (1995), no.~4,
  1033--1074.

\bibitem[BG17]{BG}
Fabrice Baudoin and Nicola Garofalo, \emph{Curvature-dimension inequalities and
  {R}icci lower bounds for sub-{R}iemannian manifolds with transverse
  symmetries}, J. Eur. Math. Soc. (JEMS) \textbf{19} (2017), no.~1, 151--219.

\bibitem[BGL14]{BGL}
Dominique Bakry, Ivan Gentil, and Michel Ledoux, \emph{Analysis and geometry of
  {M}arkov diffusion operators}, Grundlehren der Mathematischen Wissenschaften
  [Fundamental Principles of Mathematical Sciences], vol. 348, Springer, Cham,
  2014. \MR{3155209}

\bibitem[BGPV14]{BeGriPittVess14}
A.~D. Bendikov, A.~A. Grigor'yan, K.~Pitt\`e, and V.~V\"ess, \emph{Isotropic
  {M}arkov semigroups on ultra-metric spaces}, Uspekhi Mat. Nauk \textbf{69}
  (2014), no.~4(418), 3--102. \MR{3400536}

\bibitem[BH91]{BouleauHirsch}
Nicolas Bouleau and Francis Hirsch, \emph{Dirichlet forms and analysis on
  {W}iener space}, De Gruyter Studies in Mathematics, vol.~14, Walter de
  Gruyter \& Co., Berlin, 1991.

\bibitem[BH97]{BobkovHoudre}
Serguei~G. Bobkov and Christian Houdr\'e, \emph{Some connections between
  isoperimetric and {S}obolev-type inequalities}, Mem. Amer. Math. Soc.
  \textbf{129} (1997), no.~616, viii+111. \MR{1396954}

\bibitem[BK14]{BK14}
Fabrice Baudoin and Bumsik Kim, \emph{Sobolev, {P}oincar\'e, and isoperimetric
  inequalities for subelliptic diffusion operators satisfying a generalized
  curvature dimension inequality}, Rev. Mat. Iberoam. \textbf{30} (2014),
  no.~1, 109--131. \MR{3186933}

\bibitem[BK17]{BK}
Fabrice Baudoin and Daniel~J Kelleher, \emph{Differential one-forms on
  {D}irichlet spaces and {B}akry-{{E}}mery estimates on metric graphs},
  arXiv:1604.02520, to apear in Trans. Amer. Math. Soc. (2017), 1--42.

\bibitem[BP88]{BP}
Martin~T. Barlow and Edwin~A. Perkins, \emph{Brownian motion on the
  {S}ierpi\'nski gasket}, Probab. Theory Related Fields \textbf{79} (1988),
  no.~4, 543--623.

\bibitem[BSC00]{bendikov-saloff-coste}
Alexander Bendikov and L~Saloff-Coste, \emph{On-and off-diagonal heat kernel
  behaviors on certain infinite dimensional local dirichlet spaces}, American
  Journal of Mathematics \textbf{122} (2000), no.~6, 1205--1263.

\bibitem[Bus82]{Bu}
Peter Buser, \emph{A note on the isoperimetric constant}, Ann. Sci. \'Ecole
  Norm. Sup. (4) \textbf{15} (1982), no.~2, 213--230. \MR{683635}

\bibitem[CCFR17]{LiChen2}
Li~Chen, Thierry Coulhon, Joseph Feneuil, and Emmanuel Russ, \emph{Riesz
  transform for {$1\le p\le 2$} without {G}aussian heat kernel bound}, J. Geom.
  Anal. \textbf{27} (2017), no.~2, 1489--1514. \MR{3625161}

\bibitem[CCH]{LiChen}
Li~Chen, Thierry Coulhon, and Bobo Hua, \emph{Riesz transforms for bounded
  laplacians on graphs}, arXiv:1802.02410.

\bibitem[CF12]{ChenFukushima}
Zhen-Qing Chen and Masatoshi Fukushima, \emph{Symmetric {M}arkov processes,
  time change, and boundary theory}, London Mathematical Society Monographs
  Series, vol.~35, Princeton University Press, Princeton, NJ, 2012.
  \MR{2849840}

\bibitem[Che70]{Ch}
Jeff Cheeger, \emph{A lower bound for the smallest eigenvalue of the
  {L}aplacian}, Problems in analysis ({P}apers dedicated to {S}alomon
  {B}ochner, 1969), Princeton Univ. Press, Princeton, N. J., 1970,
  pp.~195--199. \MR{0402831}

\bibitem[CJKS17]{CJKS}
Thierry Coulhon, Renjin Jiang, Pekka Koskela, and Adam Sikora, \emph{Gradient
  estimates for heat kernels and harmonic functions}, arXiv:1703.02152 (2017).

\bibitem[CW17]{Chen2017}
Xin Chen and Jian Wang, \emph{Weighted poincar{\'e} inequalities for non-local
  dirichlet forms}, Journal of Theoretical Probability \textbf{30} (2017),
  no.~2, 452--489.

\bibitem[Dav89]{EBD}
E.~B. Davies, \emph{Heat kernels and spectral theory}, Cambridge Tracts in
  Mathematics, vol.~92, Cambridge University Press, Cambridge, 1989.
  \MR{990239}

\bibitem[Dav97]{Da97}
\bysame, \emph{Non-{G}aussian aspects of heat kernel behaviour}, J. London
  Math. Soc. (2) \textbf{55} (1997), no.~1, 105--125. \MR{1423289}

\bibitem[DG54]{DG}
Ennio De~Giorgi, \emph{Su una teoria generale della misura
  {$(r-1)$}-dimensionale in uno spazio ad {$r$} dimensioni}, Ann. Mat. Pura
  Appl. (4) \textbf{36} (1954), 191--213.

\bibitem[DGL77]{DeLe}
E.~De~Giorgi and G.~Letta, \emph{Une notion g\'en\'erale de convergence faible
  pour des fonctions croissantes d'ensemble}, Ann. Scuola Norm. Sup. Pisa Cl.
  Sci. (4) \textbf{4} (1977), no.~1, 61--99. \MR{0466479}

\bibitem[Dun08]{Dungey}
Nick Dungey, \emph{A {L}ittlewood-{P}aley-{S}tein estimate on graphs and
  groups}, Studia Math. \textbf{189} (2008), no.~2, 113--129. \MR{2449133}

\bibitem[Ebe99]{Ebe99}
Andreas Eberle, \emph{Uniqueness and non-uniqueness of semigroups generated by
  singular diffusion operators}, Lecture Notes in Mathematics, vol. 1718,
  Springer-Verlag, Berlin, 1999. \MR{1734956}

\bibitem[EKS15]{EKS}
Matthias Erbar, Kazumasa Kuwada, and Karl-Theodor Sturm, \emph{On the
  equivalence of the entropic curvature-dimension condition and {B}ochner's
  inequality on metric measure spaces}, Invent. Math. \textbf{201} (2015),
  no.~3, 993--1071. \MR{3385639}

\bibitem[Eld10]{Eldredge}
Nathaniel Eldredge, \emph{Gradient estimates for the subelliptic heat kernel on
  {$H$}-type groups}, J. Funct. Anal. \textbf{258} (2010), no.~2, 504--533.
  \MR{2557945}

\bibitem[Fal03]{Fal03}
K.~Falconer, \emph{Fractal geometry}, second ed., John Wiley \& Sons, Inc.,
  Hoboken, NJ, 2003, Mathematical foundations and applications.

\bibitem[Fed69]{Fed69}
H.~Federer, \emph{Geometric measure theory}, Die Grundlehren der mathematischen
  Wissenschaften, Band 153, Springer-Verlag New York Inc., New York, 1969.

\bibitem[FH01]{FuHi}
Masatoshi Fukushima and Masanori Hino, \emph{On the space of {BV} functions and
  a related stochastic calculus in infinite dimensions}, J. Funct. Anal.
  \textbf{183} (2001), no.~1, 245--268.

\bibitem[FK12]{FK12}
U.~Freiberg and S.~Kombrink, \emph{Minkowski content and local {M}inkowski
  content for a class of self-conformal sets}, Geom. Dedicata \textbf{159}
  (2012), 307--325.

\bibitem[FOT11]{FOT}
Masatoshi Fukushima, Yoichi Oshima, and Masayoshi Takeda, \emph{Dirichlet forms
  and symmetric {M}arkov processes}, extended ed., De Gruyter Studies in
  Mathematics, vol.~19, Walter de Gruyter \& Co., Berlin, 2011. \MR{2778606}

\bibitem[GHH17]{GrigHuHu17}
Alexander Grigor'yan, Eryan Hu, and Jiaxin Hu, \emph{Lower estimates of heat
  kernels for non-local {D}irichlet forms on metric measure spaces}, J. Funct.
  Anal. \textbf{272} (2017), no.~8, 3311--3346. \MR{3614171}

\bibitem[GHL03]{GHL:TAMS2003}
Alexander Grigor'yan, Jiaxin Hu, and Ka-Sing Lau, \emph{Heat kernels on metric
  measure spaces and an application to semilinear elliptic equations}, Trans.
  Amer. Math. Soc. \textbf{355} (2003), no.~5, 2065--2095. \MR{1953538}

\bibitem[GHL14]{GHL14}
A.~Grigor'yan, J.~Hu, and K.-S. Lau, \emph{Estimates of heat kernels for
  non-local regular {D}irichlet forms}, Trans. Amer. Math. Soc. \textbf{366}
  (2014), no.~12, 6397--6441.

\bibitem[GK08]{GK08}
A.~Grigor'yan and T.~Kumagai, \emph{On the dichotomy in the heat kernel two
  sided estimates}, Analysis on graphs and its applications, Proc. Sympos. Pure
  Math., vol.~77, Amer. Math. Soc., Providence, RI, 2008, pp.~199--210.

\bibitem[GK17]{GrigKajino17}
Alexander Grigor'yan and Naotaka Kajino, \emph{Localized upper bounds of heat
  kernels for diffusions via a multiple {D}ynkin-{H}unt formula}, Trans. Amer.
  Math. Soc. \textbf{369} (2017), no.~2, 1025--1060. \MR{3572263}

\bibitem[GKS10]{GKS}
Amiran Gogatishvili, Pekka Koskela, and Nageswari Shanmugalingam,
  \emph{Interpolation properties of {B}esov spaces defined on metric spaces},
  Math. Nachr. \textbf{283} (2010), no.~2, 215--231. \MR{2604119}

\bibitem[GL15]{GrigLiu15}
Alexander Grigor'yan and Liguang Liu, \emph{Heat kernel and {L}ipschitz-{B}esov
  spaces}, Forum Math. \textbf{27} (2015), no.~6, 3567--3613. \MR{3420351}

\bibitem[Gri03]{Gri}
Alexander Grigor'yan, \emph{Heat kernels and function theory on metric measure
  spaces}, Heat kernels and analysis on manifolds, graphs, and metric spaces
  ({P}aris, 2002), Contemp. Math., vol. 338, Amer. Math. Soc., Providence, RI,
  2003, pp.~143--172. \MR{2039954}

\bibitem[Gro67]{Gross}
Leonard Gross, \emph{Abstract {W}iener spaces}, Proc. {F}ifth {B}erkeley
  {S}ympos. {M}ath. {S}tatist. and {P}robability ({B}erkeley, {C}alif.,
  1965/66), {V}ol. {II}: {C}ontributions to {P}robability {T}heory, {P}art 1,
  Univ. California Press, Berkeley, Calif., 1967, pp.~31--42. \MR{0212152}

\bibitem[Hei01]{Hei}
Juha Heinonen, \emph{Lectures on analysis on metric spaces}, Universitext,
  Springer-Verlag, New York, 2001. \MR{1800917}

\bibitem[Hin13a]{Hino2}
Masanori Hino, \emph{Measurable {R}iemannian structures associated with strong
  local {D}irichlet forms}, Math. Nachr. \textbf{286} (2013), no.~14-15,
  1466--1478. \MR{3119694}

\bibitem[Hin13b]{Hino3}
\bysame, \emph{Upper estimate of martingale dimension for self-similar
  fractals}, Probab. Theory Related Fields \textbf{156} (2013), no.~3-4,
  739--793. \MR{3078285}

\bibitem[Hin16]{Hino1}
\bysame, \emph{Some properties of energy measures on {S}ierpinski gasket type
  fractals}, J. Fractal Geom. \textbf{3} (2016), no.~3, 245--263. \MR{3549797}

\bibitem[HKM18]{HinzKochMeinert}
Michael Hinz, Dorina Koch, and Melissa~Meinert Meinert, \emph{Sobolev spaces
  and calculus of variations on fractals}, preprint arXiv:1805.04456 (2018).

\bibitem[HKST15]{HKST}
Juha Heinonen, Pekka Koskela, Nageswari Shanmugalingam, and Jeremy~T. Tyson,
  \emph{Sobolev spaces on metric measure spaces}, New Mathematical Monographs,
  vol.~27, Cambridge University Press, Cambridge, 2015, An approach based on
  upper gradients. \MR{3363168}

\bibitem[HRT13]{HRT}
Michael Hinz, Michael R\"ockner, and Alexander Teplyaev, \emph{Vector analysis
  for {D}irichlet forms and quasilinear {PDE} and {SPDE} on metric measure
  spaces}, Stochastic Process. Appl. \textbf{123} (2013), no.~12, 4373--4406.
  \MR{3096357}

\bibitem[Jia15]{Jiang15}
Renjin Jiang, \emph{The {L}i-{Y}au inequality and heat kernels on metric
  measure spaces}, J. Math. Pures Appl. (9) \textbf{104} (2015), no.~1, 29--57.
  \MR{3350719}

\bibitem[JKY14]{JKY}
Renjin Jiang, Pekka Koskela, and Dachun Yang, \emph{Isoperimetric inequality
  via {L}ipschitz regularity of {C}heeger-harmonic functions}, J. Math. Pures
  Appl. (9) \textbf{101} (2014), no.~5, 583--598.

\bibitem[Kaj10]{Kajino}
Naotaka Kajino, \emph{Spectral asymptotics for {L}aplacians on self-similar
  sets}, J. Funct. Anal. \textbf{258} (2010), no.~4, 1310--1360. \MR{2565841}

\bibitem[Kaw06]{Kawabi06}
Hiroshi Kawabi, \emph{A simple proof of log-{S}obolev inequalities on a path
  space with {G}ibbs measures}, Infin. Dimens. Anal. Quantum Probab. Relat.
  Top. \textbf{9} (2006), no.~2, 321--329. \MR{2235553}

\bibitem[Kaw08]{Kawabi}
\bysame, \emph{Topics on diffusion semigroups on a path space with {G}ibbs
  measures}, Proceedings of {RIMS} {W}orkshop on {S}tochastic {A}nalysis and
  {A}pplications, RIMS K\^oky\^uroku Bessatsu, B6, Res. Inst. Math. Sci.
  (RIMS), Kyoto, 2008, pp.~153--165. \MR{2407561}

\bibitem[Kig95]{KigamiDendrites}
Jun Kigami, \emph{Harmonic calculus on limits of networks and its application
  to dendrites}, J. Funct. Anal. \textbf{128} (1995), no.~1, 48--86.
  \MR{1317710}

\bibitem[Kig01]{KigB}
\bysame, \emph{Analysis on fractals}, Cambridge Tracts in Mathematics, vol.
  143, Cambridge University Press, Cambridge, 2001. \MR{1840042}

\bibitem[Kig12]{Kig:RFQS}
\bysame, \emph{Resistance forms, quasisymmetric maps and heat kernel
  estimates}, Mem. Amer. Math. Soc. \textbf{216} (2012), no.~1015, vi+132.
  \MR{2919892}

\bibitem[Kot07]{Kotschwar}
Brett~L. Kotschwar, \emph{Hamilton's gradient estimate for the heat kernel on
  complete manifolds}, Proc. Amer. Math. Soc. \textbf{135} (2007), no.~9,
  3013--3019. \MR{2317980}

\bibitem[KR07]{KR}
Hiroshi Kawabi and Michael R\"ockner, \emph{Essential self-adjointness of
  {D}irichlet operators on a path space with {G}ibbs measures via an {SPDE}
  approach}, J. Funct. Anal. \textbf{242} (2007), no.~2, 486--518. \MR{2274818}

\bibitem[KRS03]{KRS}
Pekka Koskela, Kai Rajala, and Nageswari Shanmugalingam, \emph{Lipschitz
  continuity of {C}heeger-harmonic functions in metric measure spaces}, J.
  Funct. Anal. \textbf{202} (2003), no.~1, 147--173. \MR{1994768}

\bibitem[KST04]{Koskela-Shanmugalingam-Tyson04}
Pekka Koskela, Nageswari Shanmugalingam, and Jeremy~T. Tyson, \emph{Dirichlet
  forms, {P}oincar\'e inequalities, and the {S}obolev spaces of {K}orevaar and
  {S}choen}, Potential Anal. \textbf{21} (2004), no.~3, 241--262. \MR{2075670}

\bibitem[KSZ14]{Koskela-Shanmugalingam-Zhou}
Pekka Koskela, Nageswari Shanmugalingam, and Yuan Zhou, \emph{Geometry and
  analysis of {D}irichlet forms ({II})}, J. Funct. Anal. \textbf{267} (2014),
  no.~7, 2437--2477. \MR{3250370}

\bibitem[Kus82]{Kusuoka}
Shigeo Kusuoka, \emph{Dirichlet forms and diffusion processes on {B}anach
  spaces}, J. Fac. Sci. Univ. Tokyo Sect. IA Math. \textbf{29} (1982), no.~1,
  79--95.

\bibitem[Kus89]{Kus89}
\bysame, \emph{Dirichlet forms on fractals and products of random matrices},
  Publ. Res. Inst. Math. Sci. \textbf{25} (1989), no.~4, 659--680.

\bibitem[KY92]{KZ}
Shigeo Kusuoka and Zhou~Xian Yin, \emph{Dirichlet forms on fractals:
  {P}oincar\'e constant and resistance}, Probab. Theory Related Fields
  \textbf{93} (1992), no.~2, 169--196.

\bibitem[KZ12]{Koskela-Zhou}
Pekka Koskela and Yuan Zhou, \emph{Geometry and analysis of {D}irichlet forms},
  Adv. Math. \textbf{231} (2012), no.~5, 2755--2801. \MR{2970465}

\bibitem[Led94]{Le}
M.~Ledoux, \emph{A simple analytic proof of an inequality by {P}. {B}user},
  Proc. Amer. Math. Soc. \textbf{121} (1994), no.~3, 951--959. \MR{1186991}

\bibitem[Led96]{Ledoux}
Michel Ledoux, \emph{Isoperimetry and {G}aussian analysis}, Lectures on
  probability theory and statistics ({S}aint-{F}lour, 1994), Lecture Notes in
  Math., vol. 1648, Springer, Berlin, 1996, pp.~165--294. \MR{1600888}

\bibitem[Led03]{Ledoux2003}
M.~Ledoux, \emph{On improved {S}obolev embedding theorems}, Math. Res. Lett.
  \textbf{10} (2003), no.~5-6, 659--669. \MR{2024723}

\bibitem[Lin90]{Lind}
Tom Lindstr{\o}m, \emph{Brownian motion on nested fractals}, Mem. Amer. Math.
  Soc. \textbf{83} (1990), no.~420, iv+128.

\bibitem[LP06]{LP06}
M.~L. Lapidus and E.~P.~J. Pearse, \emph{A tube formula for the {K}och
  snowflake curve, with applications to complex dimensions}, J. London Math.
  Soc. (2) \textbf{74} (2006), no.~2, 397--414.

\bibitem[LPW11]{LPW11}
M.~L. Lapidus, E.~P.~J. Pearse, and S.~Winter, \emph{Pointwise tube formulas
  for fractal sprays and self-similar tilings with arbitrary generators}, Adv.
  Math. \textbf{227} (2011), no.~4, 1349--1398.

\bibitem[LSV09]{lenz2009allegretto}
Daniel Lenz, Peter Stollmann, and Ivan Veselic, \emph{The
  allegretto-piepenbrink theorem for strongly local dirichlet forms}, Documenta
  Mathematica \textbf{14} (2009), 167--189.

\bibitem[LSV11]{lenz2011generalized}
Daniel Lenz, Peter Stollmann, and Ivan Veseli{\'c}, \emph{Generalized
  eigenfunctions and spectral theory for strongly local dirichlet forms},
  Spectral theory and analysis, Springer, 2011, pp.~83--106.

\bibitem[Mat95]{Mattila}
Pertti Mattila, \emph{Geometry of sets and measures in {E}uclidean spaces},
  Cambridge Studies in Advanced Mathematics, vol.~44, Cambridge University
  Press, Cambridge, 1995, Fractals and rectifiability. \MR{1333890}

\bibitem[Mir03]{Mr}
Michele Miranda, Jr., \emph{Functions of bounded variation on ``good'' metric
  spaces}, J. Math. Pures Appl. (9) \textbf{82} (2003), no.~8, 975--1004.
  \MR{2005202}

\bibitem[MMS16]{MMS}
Niko Marola, Michele Miranda, Jr., and Nageswari Shanmugalingam,
  \emph{Characterizations of sets of finite perimeter using heat kernels in
  metric spaces}, Potential Anal. \textbf{45} (2016), no.~4, 609--633.
  \MR{3558354}

\bibitem[MNP15]{MNP}
Michele Miranda, Jr., Matteo Novaga, and Diego Pallara, \emph{An introduction
  to {$BV$} functions in {W}iener spaces}, Variational methods for evolving
  objects, Adv. Stud. Pure Math., vol.~67, Math. Soc. Japan, [Tokyo], 2015,
  pp.~245--294. \MR{3587453}

\bibitem[MPPP07]{MPPP}
Michele Miranda, Jr., Diego Pallara, Fabio Paronetto, and Marc Preunkert,
  \emph{Short-time heat flow and functions of bounded variation in {$\mathbb
  R^N$}}, Ann. Fac. Sci. Toulouse Math. (6) \textbf{16} (2007), no.~1,
  125--145. \MR{2325595}

\bibitem[OSCT08]{OS-CT}
Kasso~A. Okoudjou, Laurent Saloff-Coste, and Alexander Teplyaev, \emph{Weak
  uncertainty principle for fractals, graphs and metric measure spaces}, Trans.
  Amer. Math. Soc. \textbf{360} (2008), no.~7, 3857--3873. \MR{2386249}

\bibitem[PP10]{P-P10}
K.~Pietruska-Pa{\l}uba, \emph{Heat kernel characterisation of
  {B}esov-{L}ipschitz spaces on metric measure spaces}, Manuscripta Math.
  \textbf{131} (2010), no.~1-2, 199--214.

\bibitem[PW14]{PW14}
D.~Pokorn\'y and S.~Winter, \emph{Scaling exponents of curvature measures}, J.
  Fractal Geom. \textbf{1} (2014), no.~2, 177--219.

\bibitem[RZ12]{RZ12}
J.~Rataj and M.~Z\"ahle, \emph{Curvature densities of self-similar sets},
  Indiana Univ. Math. J. \textbf{61} (2012), no.~4, 1425--1449.

\bibitem[SC92]{Sal-Cos}
L.~Saloff-Coste, \emph{A note on {P}oincar\'e, {S}obolev, and {H}arnack
  inequalities}, Internat. Math. Res. Notices (1992), no.~2, 27--38.
  \MR{1150597}

\bibitem[SC02]{saloff2002}
Laurent Saloff-Coste, \emph{Aspects of {S}obolev-type inequalities}, London
  Mathematical Society Lecture Note Series, vol. 289, Cambridge University
  Press, Cambridge, 2002. \MR{1872526}

\bibitem[Sha00]{Shan2000}
Nageswari Shanmugalingam, \emph{Newtonian spaces: an extension of {S}obolev
  spaces to metric measure spaces}, Rev. Mat. Iberoamericana \textbf{16}
  (2000), no.~2, 243--279. \MR{1809341}

\bibitem[ST12]{Str}
Robert~S. Strichartz and Alexander Teplyaev, \emph{Spectral analysis on
  infinite {S}ierpi\'nski fractafolds}, J. Anal. Math. \textbf{116} (2012),
  255--297. \MR{2892621}

\bibitem[Str03a]{ST}
Robert~S. Strichartz, \emph{Fractafolds based on the {S}ierpi\'nski gasket and
  their spectra}, Trans. Amer. Math. Soc. \textbf{355} (2003), no.~10,
  4019--4043.

\bibitem[Str03b]{Str03}
\bysame, \emph{Function spaces on fractals}, J. Funct. Anal. \textbf{198}
  (2003), no.~1, 43--83.

\bibitem[Str06]{StrB}
\bysame, \emph{Differential equations on fractals. a tutorial}, Princeton
  University Press, Princeton, NJ, 2006.

\bibitem[Stu94]{St-I}
Karl-Theodor Sturm, \emph{Analysis on local {D}irichlet spaces. {I}.
  {R}ecurrence, conservativeness and {$L^p$}-{L}iouville properties}, J. Reine
  Angew. Math. \textbf{456} (1994), 173--196. \MR{1301456}

\bibitem[Stu95]{St-II}
\bysame, \emph{Analysis on local {D}irichlet spaces. {II}. {U}pper {G}aussian
  estimates for the fundamental solutions of parabolic equations}, Osaka J.
  Math. \textbf{32} (1995), no.~2, 275--312. \MR{1355744}

\bibitem[Stu96]{St-III}
K.~T. Sturm, \emph{Analysis on local {D}irichlet spaces. {III}. {T}he parabolic
  {H}arnack inequality}, J. Math. Pures Appl. (9) \textbf{75} (1996), no.~3,
  273--297. \MR{1387522}

\bibitem[Tai64]{Taibleson}
Mitchell~H. Taibleson, \emph{On the theory of {L}ipschitz spaces of
  distributions on {E}uclidean {$n$}-space. {I}. {P}rincipal properties}, J.
  Math. Mech. \textbf{13} (1964), 407--479. \MR{0163159}

\bibitem[Tep08]{T08}
Alexander Teplyaev, \emph{Harmonic coordinates on fractals with finitely
  ramified cell structure}, Canad. J. Math. \textbf{60} (2008), no.~2,
  457--480. \MR{2398758}

\bibitem[Tri78]{Trie}
H.~Triebel, \emph{Interpolation theory, function spaces, differential
  operators}, VEB Deutscher Verlag der Wissenschaften, Berlin, 1978.
  \MR{500580}

\bibitem[Win15]{W15}
S.~Winter, \emph{Minkowski content and fractal curvatures of self-similar
  tilings and generator formulas for self-similar sets}, Adv. Math.
  \textbf{274} (2015), 285--322.

\bibitem[WZ13]{WZ13}
S.~Winter and M.~Z\"ahle, \emph{Fractal curvature measures of self-similar
  sets}, Adv. Geom. \textbf{13} (2013), no.~2, 229--244.

\bibitem[Yos95]{Yosida}
K\=osaku Yosida, \emph{Functional analysis}, Classics in Mathematics,
  Springer-Verlag, Berlin, 1995, Reprint of the sixth (1980) edition.
  \MR{1336382}

\end{thebibliography}
